\documentclass[12pt]{amsart}
\usepackage{amsfonts}
\usepackage{amsmath}
\usepackage{amsxtra}
\usepackage{amssymb,latexsym}
\usepackage[mathcal]{eucal}
\usepackage{amscd}
\usepackage{tikz-cd}
\usepackage{hyperref}

\makeindex

\newtheorem{theo}{{\bfseries Theorem}}[section]
\newtheorem{prop}[theo]{{\bfseries Proposition}}
\newtheorem{lem}[theo]{{\bfseries Lemma}}
\newtheorem{cor}[theo]{{\bfseries Corollary}}
\newtheorem{df}[theo]{{\bfseries Definition}}

\newtheorem{ex}{{\bfseries Example}}
\newtheorem{exes}[ex]{{\bfseries Examples}}
\newtheorem{ques}[theo]{{\bfseries Question}}

\newtheorem{add}[theo]{{\bfseries Addendum}}

\newtheorem{back}[theo]{{\bfseries Background}}

\def \N {\mathbb N}

\def \Z {\mathbb Z}
\def \R {\mathbb R}

\def \A {\mathcal A}
\def \B {\mathcal B}
\def \CC {\mathcal C}

\def \U {\mathcal U}

\def \a {\alpha }
\def \alp {\aleph}
\def \b {\beta}

\def \ep {\epsilon}
\def \om {\omega}
\def \th {\theta}
\def \d {\delta}

\def \r {\rho}
\def \s {\sigma}
\def \t {\tau}
\def \ha {\rightharpoonup}
\def \lha {\leftharpoonup}

\def \1  {{\mathbf 1}}
\def \0  {{\mathbf 0}}

\usepackage{amssymb,latexsym}
\usepackage[mathcal]{eucal}
\usepackage{amscd}
\usepackage{amsmath}
\usepackage{graphicx}
\usepackage{amsfonts}
\usepackage{amscd}

\numberwithin{equation}{section}
\begin{document}

\begin{titlepage}
\large
\title {\bfseries Topological Tournaments}
\author{Ethan Akin}
 \vspace{.7cm}

\address{Mathematics Department \\
    The City College \\ 137 Street and Convent Avenue \\
       New York City, NY 10031, USA     }
\email{ethanakin@earthlink.net}

\date{March, 2023}

\begin{abstract} A directed graph $R^{\circ}$ on a set $X$ is a set of ordered pairs of
distinct points called \emph{arcs}. It is a tournament when  every pair of distinct points is connected by an arc in one direction or
the other (and not both). We can describe a tournament $R \subset X \times X$ as a total, antisymmetric relation, i.e. $R \cup R^{-1} = X \times X$
and $R \cap R^{-1}$ is the diagonal $1_X = \{ (x,x) : x \in X \}$. The set of arcs is $R^{\circ} = R \setminus 1_X = (X \times X) \setminus R^{-1}$.
A topological tournament on a compact Hausdorff space $X$ is a tournament $R$ which is a closed subset of $X \times X$. We construct uncountably many
non-isomorphic examples on the Cantor set $X$ as well as examples of arbitrarily large cardinality. We also describe compact Hausdorff spaces which
do not admit any topological tournament.

\end{abstract}



\end{titlepage}
\maketitle

\tableofcontents

\vspace{1cm}
\setcounter{page}{1}
\section{ \textbf{Introduction}}\vspace{.5cm}

A directed graph (or just \emph{digraph})  consists of a non-empty finite set
$X$ of elements called \emph{vertices} and a finite set $R^{\circ}$ of ordered pairs of
distinct vertices called \emph{arcs}. In addition, we assume that $(x,y) \in R^{\circ}$ implies $(y,x) \not\in R^{\circ}$. That is, for
any pair of distinct vertices there is at most one arc between them. A digraph is called a \emph{tournament} if for
any pair of distinct vertices there is  exactly one arc between them.  That is, either $(x,y) \in R^{\circ}$ or $(y,x) \in R^{\circ}$ but
not both. Digraphs have been the object of considerable study, see e.g. \cite{B-JG}. For the special case of tournaments, see \cite{HM} and \cite{M}.
When the tournament $R$ is understood, we will write $x \ha y$ when $(x,y) \in R^{\circ}$.

In considering tournaments on infinite set $X$, it will be convenient to attach the diagonal set $1_X = \{ (x,x) \in X \}$. So we
will call $R$ a tournament on $X$ when it is an anti-symmetric, total relation on $X$.  That is, $R \subset X \times X$ with
$R \cap R^{-1} = 1_X$ and $R \cup R^{-1} = X \times X$, where $R^{-1} = \{ (x,y) : (y,x) \in R \}$.  The set of arcs is
$$R^{\circ} \ = \ R \setminus 1_X \ = \ (X \times X) \setminus R^{-1}.$$

A tournament $(X,R)$ is a \emph{trivial tournament} when $X$ is a singleton set, and is an \emph{arc tournament} when $X$ is a two point set.
If $A \subset X$, then $(A,R|A)$ is called the \emph{restriction} to $A$ where $R|A = R \cap (A \times A)$.

We will call $R$ a topological tournament
on a topological space $X$
when it is a tournament, closed as a subset of $X \times X$. Of course, when $X$ is given the discrete topology, any tournament on $X$ is a
topological tournament. We will be primarily interested in the case when $X$ is compact.

All our spaces are assumed to be Hausdorff, but as they need not be metrizable, we will use the convergence theory of nets. These are
analogues of sequences, indexed by directed sets instead of by the natural numbers $\N$. For the theory of nets, see \cite{K} Chapter 2 on
Moore-Smith Convergence.

For a topological tournament $R$ and a point $x \in X$, the \emph{outset} is  $R^{\circ}(x) = R(x) \setminus \{ x \}$, and the
\emph{inset} is $R^{\circ -1}(x)$.
We call a point \emph{right balanced} if it is in the closure of its outset and \emph{left balanced} if it is in the
closure of its inset. It is \emph{balanced} if
it is both left and right balanced. A point is neither left nor right balanced if and only if it is an isolated point. We call a point $x$ a
\emph{cycle point} when every neighborhood of $x$ contains a $3-$cycle which includes $x$. Clearly a cycle point is balanced and we will see that
in the compact case a cycle point is a $G_{\d}$ point and so has a countable neighborhood base. Of course, balanced points of any sort only
occur when the space $X$ is infinite.

\begin{prop}\label{propintro01b} If $x$ is a cycle point for a tournament $R$ on a compact space $X$, then $x$ is a $G_{\d}$ point and so
has a countable neighborhood base in $X$. \end{prop}
\vspace{.5cm}

A tournament is \emph{arc cyclic} when every arc is contained in a $3-$cycle. We call a topological tournament \emph{weakly arc cyclic} or just \emph{wac}
when every non-isolated point is a cycle point. As the name suggests, an arc cyclic topological tournament is wac. A finite tournament is always
wac and although not all finite tournaments are arc cyclic, many are.

For compact topological tournaments $(X_1,R_1)$ and $(X_2,R_2)$ a \emph{quotient map} $h : (X_2,R_2) \to (X_1,R_1)$ is a surjective continuous
map $h : X_2 \to X_1$ such that $(h \times h)(R_2) \subset R_1$.  When $h$ is injective, it is an \emph{isomorphism} with inverse
$h^{-1} : (X_1,R_1) \to (X_2,R_2)$. It is an \emph{automorphism} when, in addition, $(X_2,R_2) = (X_1,R_1)$.

If $(X_1,R_1)$ and $\{ (Y_x,S_x) : x \in X_1 \}$  are tournaments, then the \emph{lexicographic product}
$$(X_2,R_2) \ = \ (X_1,R_1) \ltimes \{ (Y_x,S_x) : x \in X_1 \}$$
is the tournament with $X_2 = \bigcup \{ \{x \} \times Y_x : x \in X_1 \}$ and $((x,y),(x',y')) \in R_2$ when
$(x,x') \in R_1^{\circ}$ or $x = x'$ and $(y,y') \in S_x$. The map $\pi : X_2 \to X_1$ is the projection to the first coordinate.
The product is called a \emph{compact topological lexicographic product} when the following conditions are satisfied.
 \begin{itemize}
 \item[(i)] The space $X_1$ and each $Y_x$ is compact.

\item[(ii)] If $x \in X_1$  is non-isolated point of $X_1$, then $(Y_x,S_x)$ is a  trivial tournament.

\item[(iii)] The space $X_2$ is given the topology with basis $\B$ where $U \in \B$ when
either $U = \pi^{-1}(V)$ for $V$ some open  subset of $X_1$, or
$U = \{ x \} \times V$ for $x$ isolated in $X_1$ and $V$ some open subset of $Y_x$.
\end{itemize}

In that case, $(X_2,R_2)$ is a compact topological tournament and $\pi : (X_2,R_2) \to (X_1,R_1)$ is an open quotient map.

A sequence $\{ (X_i,R_i,f_i) : i \in \N \}$ is a \emph{compact inverse sequence} when each $(X_i,R_i)$ is a compact topological tournament
and each $f_i : (X_{i+1},R_{i+1}) \to (X_i,R_i)$ is a quotient map. The \emph{inverse limit} is the compact tournament
$(X,R) \ = \ \overleftarrow{Lim} \{ (X_i,R_i,f_i) \}$ with
\begin{displaymath}\begin{split}
X \ = \ \{ x \in \prod_i X_i : x_i = f_i(x_{i+1}) \ \ \text{ for all } \ i \in \N \}, \hspace{1cm}\\
R \ = \ \{ (x,x') \in X \times X : (x_i,x'_i) \in R_i \ \ \text{ for all } \ i \in \N \}.
\end{split}\end{displaymath}
The projection $\pi_i : (X,R) \to (X_i,R_i)$ given by $\pi_i(x) = x_i$ is a quotient map for each $i \in \N$.

\begin{prop}\label{propintro01a} Let $h : (X_2,R_2) \to (X_1,R_1)$  be a quotient map.

(a) If $y \in X_1$ is a cycle point for $R_1$, then $h^{-1}(y)$ is a singleton $\{ x \} \subset X_2$ and $x$ is a cycle point for $R_2$.
In particular, if every point of $X_1$ is a cycle point, then $h$ is an isomorphism.

(b) If $(X_2,R_2)$ is an arc cyclic (or wac) tournament, then $(X_1,R_1)$ is arc cyclic (resp. wac) and there is
an isomorphism
$$ q : (X_1,R_1) \ltimes \{ (h^{-1}(y),R_2|h^{-1}(y)) : y \in X_1 \} \ \to \ (X_2,R_2) $$
such that $\pi = h \circ q$. In particular, $h$ is an open map.

(c) The inverse limit $ \overleftarrow{Lim} \{ (X_i,R_i,f_i) \}$ is an arc cyclic (or wac) tournament, if and only if
each $(X_i,R_i)$ is arc cyclic (resp. wac). In particular, if each $X_i$ is finite, then $ \overleftarrow{Lim} \{ (X_i,R_i,f_i) \}$ is wac.
\end{prop}
\vspace{.5cm}

If $X$ is a group with identity $e$, a \emph{game subset} $A$ is a subset of $X$ such that $A \cap A^{-1} = \{ e \}$ and $A \cup A^{-1} = X$, where
$A^{-1} = \{ x^{-1} : x \in A \}$. That is, for every $x \not = e$ in $X$, exactly one member of the pair $\{ x, x^{-1} \}$ lies in $A$. Thus, $X$
admits a game subset if and only if it contains no elements of order two. When $X$ is finite, this means $X$ has odd order.
Associated with a game subset $A$
is the tournament $\widehat{A} = \{ (x,y) : x^{-1}y \in A \}$. If $X$ is a topological group and $A$ is a closed game subset,
then $\widehat{A}$ is a topological
tournament on $X$.

\begin{theo}\label{theointro01} Let $X$ be an infinite, compact topological group with no elements of order two.

(a) There exists a closed game subset $A$ for $X$ if and only if the space $X$ is totally disconnected and metrizable, i.e. it is a Cantor set.

(b) If $A$ is a closed game subset for $X$, then the tournament $\widehat{A}$ is arc cyclic and so every point of $X$ is a cycle point.
\end{theo}
\vspace{.5cm}

Notice that a topological group contains an isolated point if and only if it is discrete. Hence, any infinite, compact topological group
has no isolated points.

On the one hand, there are many different topological group structures on the Cantor set.  For example, the $p-$adic integers for any prime
$p$ including $2$ admits closed game subsets.

On the other hand, this illustrates that the existence of a topological tournament is a demanding condition.  For example, if $X$ is an uncountable product
of finite groups of odd order, then because the product topology is not metrizable, the product group does not admit a closed game subset.

A transitive tournament is just a linear order. A linear order on $X$ is a topological tournament when $X$ is given the order topology.
Conversely, if $L$ is a transitive topological tournament on a compact space $X$, then, as we will see, the topology on $X$ is the order
topology associated with $L$, i.e. $X$ is a compact \emph{LOTS} (= linearly ordered topological space).

A topological tournament $R$ on $X$ is called \emph{nowhere locally transitive} if no restriction of $R$ to a nonempty open set is transitive. This
is equivalent to the condition that every nonempty open
subset contains a $3-$cycle.

Clearly, if $X$ contains a dense set of cycle points, then $R$ is nowhere locally transitive. Conversely, we have

\begin{theo}\label{theointro02} Let $(X,R)$ be a compact topological tournament.

(a) If $R$ is nowhere locally transitive, then $X$
is totally disconnected and contains a dense set of cycle points.

(b) If $R$ is balanced and the space $X$ is totally disconnected, then $R$ is nowhere locally
transitive.
\end{theo}
\vspace{.5cm}

Despite these limitations, it is possible to construct big examples by using lexicographic products of LOTS and inverse limits.

\begin{theo}\label{theointro03} Let $\alp$ be an arbitrary uncountable cardinal. There exists a compact, totally disconnected LOTS
 $X$ such that every nonempty open subset has cardinality at least that of $\alp$  and $X$ admits a balanced topological tournament $R$.
The set of cycle points for $R$ is a dense $G_{\d}$ subset of $X$, while
 the set of non $G_{\d}$ points  is also  a dense subset of $X$. Furthermore, no open subset is separable. \end{theo}
\vspace{.5cm}

A nontrivial, compact topological tournament $(Y,P)$ is a \emph{prime tournament} when the only quotient maps $h : (Y,P) \to (Y_1,P_1)$ with
$(Y_1,P_1)$ nontrivial are isomorphisms.  For example, an arc tournament is prime.
A quotient map $(X,R) \to (Y,P)$ with $(Y,P)$ prime is called a \emph{prime quotient map} and $(Y,P)$ is
called a \emph{prime quotient} of $(X,R)$.

\begin{theo}\label{theointro04} If $(X,R)$ is a wac tournament, then it has a prime quotient which is unique up to isomorphism.
If the prime quotient $(Y,P)$ is not an arc, then the prime quotient map is unique up to isomorphism.  That is,
If $h : (X,R) \to (Y,P)$ and $h_1 : (X,R) \to (Y_1,P_1)$ are prime quotient maps and $(Y,P)$ is not an arc, then there exists
an isomorphism $q : (Y,P) \to (Y_1,P_1)$ such that $q \circ h = h_1$. \end{theo}
\vspace{.5cm}

If $(Y,L)$ is an order on a nontrivial finite set, then a \emph{maximum order quotient map} $h: (X,R) \to (Y,L)$ is a quotient map
such that for all $y \in Y$, the restriction $(h^{-1}(y),R|h^{-1}(y))$ does not have an arc quotient.

\begin{add}\label{addintro04a} If $(X,R)$ is a wac tournament with an arc quotient, then it has a maximum order quotient map unique
up to isomorphism. \end{add}
\vspace{.5cm}

Thus, every wac tournament $(X,R)$ has a base quotient map $h: (X,R) \to (Y,P)$, unique up to isomorphism, as follows
\begin{itemize}
\item If $(X,R)$ is nontrivial and does not have an arc quotient, then $h$ is a prime quotient map.

 \item If $(X,R)$ is nontrivial and has an arc quotient, then $h$ is a maximum order quotient map.

 \item If $(X,R)$ is trivial, then $(Y,P)$ is trivial.
 \end{itemize}

  \begin{df}\label{defintro05} For a wac tournament $(X,R)$ the \emph{classifier system}
  is an inverse system $\{(X_i,R_i,f_i) \}$ of
 topological tournaments, together with quotient maps $h_i : (X,R) \to (X_i,R_i)$ which satisfy the
 following properties.
 \begin{itemize}
 \item[(i)]  $h_i = f_i \circ h_{i+1}$ for all $i \in \N$.
 \item[(ii)] $h_1 : (X,R) \to (X_1,R_1)$ is a base quotient map.
 \item[(iii)] For each $x_i \in X_i$, the restriction $(h_i^{-1}(x_i),R|h^{-1}_i(x_i))$ is a wac tournament and
 the map $$h_{i+1} : (h_i^{-1}(x_i),R|h^{-1}_i(x_i)) \to (f_i^{-1}(x_i),R_{i+1}|f^{-1}_i(x_i))$$ is a base quotient map.
 \end{itemize}
 \end{df}

  \begin{theo}\label{theointro06} If $(X,R)$ is a wac tournament, then it has a classifier system which is unique up to isomorphism
  and the product map $\prod_i h_i : (X,R) \to  \overleftarrow{Lim} \{ (X_i,R_i,f_i) \}$ is an isomorphism. Furthermore,
  the classifier can be constructed so that $(X_{i+1},R_{i+1}) = (X_i,R_i) \ltimes \{ (Y_{ix_i}, S_{ix_i} : x_i \in X_i \}$
  with each $(Y_{ix_i}, S_{ix_i})$ either prime, a nontrivial finite order or  trivial. \end{theo}
  \vspace{.5cm}

  Using the uniqueness of the classifier system, we are able to construct an uncountable number of arc cyclic tournaments on the
  Cantor set with each $(X_i,R_i)$ a finite arc cyclic tournament, such that no two are isomorphic.

  On the compact group $\Z[2]$ of $2-$adic integers, there exists a closed game subset $A$ such that $\widehat{A}$ is a prime
  tournament. Using it, we are able to construct an uncountable number of prime, arc cyclic tournaments on the Cantor set, as well
  as arc cyclic, prime tournaments with countable sets of isolated points and with Cantor subsets.

\begin{back}\label{back} \end{back} We briefly review some standard results about
compact spaces which we will be using.  All of our spaces, compact or not, are
assumed to be Hausdorff.
\begin{enumerate}
\item If $\{ A_n \}$ is a decreasing sequence of compact sets in a space $X$ with intersection $A$ and $U$
is an open set with $A \subset U$, then for sufficiently
large $n$, $A_n \subset U$, because $\{ U \} \cup \{ X \setminus A_n \}$ is an open cover of $A_1$ and so has a finite subcover.  In particular, if
$A$ is clopen (= open as well as closed), then, using $U = A$, $A_n = A$ for sufficiently large $n$.

\item A component $A$ in a compact space is the intersection of the clopen subsets which meet and therefore contain it. Hence, if $X$ is totally
disconnected, i.e. the only connected subsets are singletons, then the clopen subsets form a base for the topology.

\item If $X$ is compact and metrizable, and so has a countable base,  then there are only countably many clopen subsets because each clopen set is
a finite union of members of the base.

\item A Cantor set is a compact, metrizable, totally disconnected space with no isolated points. Any Cantor set is homeomorphic to the
product $\{ 0, 1 \}^{\N}$ and so to any other Cantor set. In particular, it is homeomorphic to the classical Cantor Set $C$ contained in the
unit interval of $\R$.

\item For any compact, metrizable, totally disconnected space $X$, the product $X \times C$ is a Cantor set, homeomorphic to $C$ itself and so
$X$ can be embedded in $C$.

\item A countable compact space $X$ is  totally disconnected since any non-trivial connected compact space maps onto the unit interval and so is uncountable.
The diagonal $1_X = \{ (x,x) : x \in X \}$ is clearly a $G_{\d}$ subset. In a compact space, a closed subset is
a $G_{\d}$ set if and only if it has a countable base of neighborhoods. A compact space is metrizable if and only if the diagonal is $G_{\d}$
because then the uniformity of neighborhoods of the diagonal has a countable base, see, e.g. \cite{K} Theorem 6.13. Hence,
the countable compact space $X$ can be embedded in
a Cantor set.

\item The countably infinite product of non-trivial, compact, metrizable, totally disconnected spaces is a compact, metrizable,
totally disconnected space with no isolated points, i.e. a Cantor set.

\item On $\{ 0, 1 \}^{\N}$ the metric $u$ defined to by $u(x,x') = \max_i 2^{-i}|x_i - x'_i|$ is compatible with the product topology and
is an \emph{ultra-metric}\index{ultra-metric}.  That is, it satisfies the strengthening of the triangle inequality:
$ u(x,x'') \ \le \ \max (u(x,x'),u(x',x''))$. For any ultra-metric $u$ and $\ep > 0$ the set $V_{\ep} = \{ (x,x') : u(x,x') < \ep \}$ is an equivalence
relation with finitely many clopen equivalence classes, namely, the $\ep$ balls $V_{\ep}(x)$. Hence,
$V_{\ep} = \bigcup_x \{ V_{\ep}(x) \times V_{\ep}(x) \}$ is clopen.
Any compact, metrizable, totally disconnected space admits such an ultrametric.

\item If $X$ is a general compact, totally disconnected space and $V$ is a neighborhood of the diagonal $1_X$, then there exists a
clopen equivalence relation $E$ on $X$ such that $E \subset V$. We can choose a finite cover $\{U_1, \dots, U_n \}$ of $X$ by clopen sets
such that $U_i \times U_i \subset V$. With $U_0 = \emptyset $ we let $U'_i = U_i \setminus \bigcup_{j< i} U_j$ to get a clopen partition
and then let $E = \bigcup_i U_i' \times U_i'$.

\item Let $X$ be a compact metric space with metric $d$. If $G$ is a compact topological group with a continuous action $(g,x) \mapsto gx$ on $X$,
then $d_G(x,x') = \max \{ d(gx,gx') : g \in G \}$ is a $G$ invariant metric compatible with the topology on $X$. Notice that if
 $\ep > 0$, then $\min \{ d(g^{-1}x,g^{-1}x') : g \in G, d(x,x') \ge \ep \} = \d > 0$ and so  $d(x,x') < \d$ implies $d_G(x,x') < \ep$.
 If $d$ is an ultra-metric, then $d_G$ is a $G$ invariant ultra-metric.
 \end{enumerate}

\vspace{1cm}

\section{ \textbf{Topological Tournaments}}\vspace{.5cm}

Following \cite{A93} we will use the language of relations.  For  sets $X, Y$ (not necessarily finite)
a \emph{relation}\index{relation} $F$ from $X$ to $Y$
is just a subset of the product set
set $X \times Y$ of ordered pairs. We let $\pi_1$ and $\pi_2$ denote the coordinate projections.

We define for a relation $F$ from $X$ to $Y$ and $x \in X, A \subset X, B \subset Y$:
\begin{equation}\label{eq1.01}
\begin{split}
F(x) \ = \ \{ y \in Y : (x,y) \in F \}, \hspace{2cm} \\
F(A) \ = \ \bigcup \{ F(x) : x \in A \} \ = \ \pi_2((A \times Y) \cap F) \\
F^{-1} \ = \ \{ (y,x) \in Y \times X : (x,y) \in F \}, \hspace{1.5cm} \\
F^*(B) \ = \ X \setminus F^{-1}(Y \setminus B).  \hspace{2cm}
\end{split}
\end{equation}
The \emph{reverse relation}\index{relation!reverse} $F^{-1}$ is a relation from $Y$ to $X$.

Notice that
\begin{equation}\label{eq1.01aa}
\begin{split}
F^{-1}(B) \ = \ \{ x \in X : F(x) \cap B \not= \emptyset \}, \hspace{1.5cm} \\
F^*(B) \ = \ \{ x \in X : F(x) \subset B \}. \hspace{1.5cm}
\end{split}
\end{equation}

We think of a relation as a generalization of a mapping.  The relation $F$ is a function from $X$ to $Y$
when for every $x \in X$ the set $F(x)$ is a singleton,
i.e. $|F(x)| = 1$ where we use $|A|$ to denote the cardinality of a finite set $A$.  For example, the identity map $1_X$ is the relation
$\{ (x,x) : x \in X \}$ on $X$.
If $F$ is a mapping, then $F^{-1}(B) = F^*(B)$ is the usual pre-image of $B$.

The relation $F$ is called \emph{surjective}\index{relation!surjective}
when for all $x \in X, y \in Y$, $F(x) \not= \emptyset$ and $F^{-1}(y) \not= \emptyset$, or, equivalently, for every $x \in X$ there exists
$y \in Y$ such that $y \in F(x)$ and for every $y \in Y$ there exists $x \in X$ such that $y \in F(x)$.
When $F$ is a function, it is a surjective relation exactly when it is a surjective function.

If $F$ is a relation from $X$ to $Y$ and $G$ is a relation from $Y$ to $Z$,
the \emph{composition}\index{relation!composition} is the relation
$G \circ F$ from $X$ to $Z$  defined by
\begin{equation}\label{eq1.01a}
\begin{split}
G \circ F \ = \ \pi_{13}((F \times Z) \cap (X \times G)) \ = \ \hspace{2cm} \\
\{ (x,z) \in X \times Z : \text{there exists }  \ y \in Y  \ \text{such that}  \ (x,y) \in F, (y,z) \in G \},
\end{split}
\end{equation}
where $\pi_{13}$ is the coordinate projection from $X \times Y \times Z$ to $X \times Z$.

Thus, for any subset $A$ of $X$, $(G \circ F)(A) = G(F(A))$.  Clearly, $ (G \circ F)^{-1} = F^{-1} \circ G^{-1}$. As with functions, composition
of relations is associative.

When $X = Y$ $F$ is called a \emph{relation on} $X$.

If $R_1$ is a relation on $X_1$ and $R_2$ is a relation on $X_2$, then a function $h : X_2 \to X_1$ maps $R_2$ to $R_1$ when
$(x,x') \in R_2$ implies $(h(x),h(x')) \in R_1$. That is,
\begin{equation}\label{eq1.02}
(h \times h)(R_2) \subset R_1, \quad \text{or, equivalently,} \quad R_2 \subset (h \times h)^{-1}(R_1),
\end{equation}
where $h \times h : X_2 \times X_2 \to X_1 \times X_1$ is the product map induced by $h$. It clearly follows that $h$ maps
$R_2^{-1}$ to $R_1^{-1}$.

We define
the \emph{product relation}\index{relation!product}  $R_1 \times R_2$ on $X_1 \times X_2$ by
\begin{equation}\label{eq1.03}
R_1 \times R_2 \ = \ \{ (x_1,x_2),(y_1,y_2)) : (x_1,y_1) \in R_1 \ \ \text{and} \ \ (x_2,y_2) \in R_2 \}.
\end{equation}
That is, we identify $(X_1 \times X_1) \times (X_2 \times X_2)$ with $(X_1 \times X_2) \times (X_1 \times X_2)$.

If $Y \subset X$ and $R$ is a relation on $X$, then the \emph{restriction}\index{relation!restriction} of $R$ to $Y$ is $R|Y = R \cap (Y \times Y)$.

A relation $R$ on $X$ is \emph{reflexive}\index{relation!reflexive} when $1_X \subset R$,
\emph{symmetric}\index{relation!symmetric} when $R = R^{-1}$
and \emph{transitive}\index{relation!transitive} when $R \circ R \subset R$.

For $n > 1$ an \emph{$n-$cycle}\index{$n-$cycle} for the relation $R$ on $X$ is a sequence $\{ x_1, \dots, x_n \}$ of distinct points of
$X$ such that $(x_i,x_{i+1}) \in R$ for $i = 1, \dots, n$ (with addition mod $n$).

A \emph{closed relation}\index{relation!closed} (or an \emph{open relation})\index{relation!open} is a
 relation $F$ between Hausdorff topological spaces $X$ and $Y$ with $F$ a closed subset
(resp. an open subset) of $X \times Y$. Clearly, for a closed relation $F$,
the reverse relation $F^{-1}$ is closed and for each $x \in X, y \in Y$ the sets $F(x)$ and $F^{-1}(y)$ are closed.  The product and restriction
of closed relations are closed relations.
Similarly, for an open relation $F$ the sets $F^{-1}$, $F(x), F^{-1}(y)$  are open and the product and restriction of open relations are open relations.

As all of our spaces  are assumed to be Hausdorff,
 any continuous map between spaces is  a closed relation. If the spaces are compact, then the converse holds. That is, if a mapping between
  compact spaces  is a closed
 relation, then it is a continuous map. For compact spaces the composition of closed relations
 is closed and the image of a closed subset by a closed relation
 is closed. Furthermore, in the compact case, if $B$ is open, then $F^{*}(B)$ is open.

Thus, a digraph is a relation $R^{\circ}$ on a finite set $X$ such that
$R^{\circ} \cap (R^{\circ})^{-1} = \emptyset$. It is a tournament when, in addition,
$R^{\circ} \cup (R^{\circ})^{-1} = X \times X \setminus 1_X$.  A tournament is said
to be \emph{regular}\index{tournament!regular} when for every $x \in X$
the \emph{inset}\index{tournament!inset} $(R^{\circ})^{-1}(x)$
and the \emph{outset}\index{tournament!outset} $R^{\circ}(x)$ have the same cardinality, i.e. $|(R^{\circ})^{-1}(x)| = |R^{\circ}(x)|$ for all $x \in X$.
A regular tournament exists on a finite set $X$
only when the cardinality $|X|$ is odd. Conversely, as we will see from the group examples below,
a finite set of odd cardinality admits regular tournaments.

Notice that an $n-$cycle for a digraph has length $n$ greater than $2$.

It will be convenient for our purposes
to attach the identity $1_X$ to $R^{\circ}$.

\begin{df}\label{def1.01} For an arbitrary nonempty set $X$, a \emph{tournament}\index{tournament}
on $X$ is a relation $R$ on $X$ which is anti-symmetric and total, i.e.
\begin{equation}\label{eq1.04}
R \cap R^{-1} \ = \ 1_X, \quad \text{and} \quad R \cup R^{-1} \ = \ X \times X.
\end{equation}
We denote by $R^{\circ}$ the \emph{arc-set}\index{tournament!arc-set} $R \setminus 1_X$.

The tournament $1_X$ on a singleton set $X$ is called a \emph{trivial tournament}\index{tournament!trivial}.
A tournament on a two point set is called an \emph{arc tournament}\index{tournament!arc} or simply an arc.

A \emph{topological tournament}\index{tournament!topological}\index{topological tournament} is a
tournament on a topological space which is a closed relation.\end{df}

We will call a pair $(X,R)$ a tournament when $R$ is a tournament on the set $X$. The pair is a topological tournament when
$X$ is a topological space and $R$ is closed, and it is a compact topological tournament when, in addition, the space $X$ is compact.
A finite tournament is a tournament on a finite set, always a compact topological tournament with the discrete topology on $X$.

If $R$ is a topological tournament on $X$, then
\begin{itemize}
\item The reverse relation $R^{-1}$ is a topological tournament on $X$ with $(R^{-1})^{\circ} = (R^{\circ})^{-1}$ which we will therefore write as
$R^{\circ -1}$.

\item The arc-set relation $R^{\circ} = R \setminus 1_X = (X \times X) \setminus R^{-1}$  is an open relation.

\item For each $x \in X$, the sets $R(x)$ and $R^{-1}(x)$ are closed subsets and the \emph{outset}\index{tournament!outset}
 $R^{\circ}(x) = R(x) \setminus \{ x \}$ and
the \emph{inset}\index{tournament!inset} $(R^{-1})^{\circ}(x) = R^{-1}(x) \setminus \{ x \}$ are open subsets of $X$.
\end{itemize}

When the tournament $R$ is understood we will write $ x \ha y$ or $y \lha x$ when $(x,y) \in R^{\circ}$ and we write
$x \ \underline{\ha} \ y$ when $(x,y) \in R$. For subsets $A, B$ of $X$, we will write $A \ha B$, when $x \ha y$ for all
$x \in A, y \in B$, or, equivalently, when $A \times B \subset R^{\circ}$.

We will call a tournament $(X,R)$ \emph{arc cyclic}\index{tournament!arc cyclic} (or \emph{point cyclic} \index{tournament!point cyclic})
when every arc $(x_1,x_2) \in R^{\circ}$
 (respectively, every point $x_1 \in X$) is contained
in a $3$-cycle $\{x_1, x_2, x_3 \}$ for $R$.

Every finite regular tournament is arc cyclic, see, e.g. \cite{A20} Proposition 2.2 or
 \cite{CH} Proposition 5.1.  However, there exist finite tournaments which are arc cyclic but not regular.
A non-trivial arc cyclic tournament is clearly point cyclic. On the other hand, a trivial tournament is not point cyclic but is vacuously
arc cyclic.

  For a tournament $(X,R)$ we will call $A \subset X$ an \emph{arc cyclic subset}\index{arc cyclic subset}  when every arc $(x_1,x_2) \in R^{\circ}$
  with $x_1, x_2 \in A$ is contained in a $3$-cycle $\{x_1, x_2, x_3 \}$ for $R$. Note that $x_3$ need not be in $A$.
 Thus, if the restriction $R|A$ is an arc cyclic tournament, then $A$ is an arc cyclic subset, but the converse need not be true.
Clearly, $(X,R)$ is an arc cyclic tournament when $X$ is an arc cyclic subset.

We will call a topological tournament $(X,R)$ \emph{locally arc cyclic}\index{tournament!locally arc cyclic}\index{tournament!arc cyclic!locally}
when every point $x \in X$ has a neighborhood $U$ which is an arc cyclic subset. Since a trivial tournament is arc cyclic, the
singleton set containing an isolated point is an arc cyclic neighborhood of the point. In particular, every finite tournament is
locally arc cyclic. Of course, an arc cyclic topological tournament is locally arc cyclic.

For a tournament $(X,R)$ we call $x \in X$ a \emph{terminal point}\index{terminal point}
(or a \emph{initial point})\index{initial point} for $R$ when $R^{\circ}(x) = \emptyset$ (resp.
$R^{\circ -1}(x) = \emptyset$). A tournament has
 at most one terminal point since $R^{\circ}(x) = \emptyset$ and $x \not= y $ implies $x \in R^{\circ}(y)$
because  $R$ is total.  Similarly, there is at most one
 initial point. A tournament $R$ is a surjective relation if and only if it has neither a terminal point nor an initial point.

For a topological tournament $R$ on $X$ we define for a point $x \in X$
\begin{align}\label{eq1.05}\begin{split}
x \ \ \text{is right balanced} \quad &\Longleftrightarrow \quad \overline{R^{\circ}(x)} = R(x) \\
x \ \ \text{is left balanced} \quad &\Longleftrightarrow \quad \overline{R^{\circ -1}(x)} = R^{-1}(x)\\
x \ \ \text{is balanced} \quad &\Longleftrightarrow \quad x \ \ \text{is both left and right balanced}.
\end{split}\end{align}
\index{balanced point}\index{balanced point!left}\index{balanced point!right}

Note that if $x$ is not right balanced, if and only if $\overline{R^{\circ}(x)} = R^{\circ}(x)$ and so $R^{\circ}(x)$ is closed as well as open.
Hence,  $x$ is neither left nor right balanced, if and only if $\{ x \}$ is clopen and so $x$ is an isolated point.

We will call a topological tournament $(X,R)$  \emph{balanced}\index{tournament!balanced} when every point of $X$ is balanced.

We will call a topological tournament $(X,R)$ \emph{regular}\index{tournament!regular} when
for every $x \in X$ there is a homeomorphism $h_x$ from $X$ onto $X$ such that
$h_x(x) = x$ and $h_x(R(x)) = R^{-1}(x)$. Note that when $X$ is finite any bijection on
$X$ is a homeomorphism and so this concept agrees with the usual
notion of regularity when $X$ is finite.

If $(X_1,R_1)$ and $(X_2,R_2)$ are tournaments, and $h$ is a function from
$ X_2$ to $X_1$, then we call $h : (X_2,R_2) \to (X_1,R_1)$ a \emph{tournament map}\index{tournament map} (or, equivalently, $h$ is a tournament map
from $R_2$ to $R_1$)
when $h$ maps the relation $R_2$ to the relation $R_1$. Since
$h$ then maps $R_2^{-1}$ to $R_1^{-1}$ it follows that $(h(x),h(x')) \in R_1^{\circ}$ implies $(x,x') \in R_2^{\circ}$ or, equivalently
 \begin{equation}\label{eq1.06}
R_2^{\circ} \supset (h \times h)^{-1}(R_1^{\circ}).
\end{equation}

Thus, the preimage of $R_1^{\circ}$ is contained in $R_2^{\circ}$ and
$R_2$ is contained in the union of the preimage of $R_1^{\circ}$  and that of
$1_{X_1}$.

If $h$ is a bijection, then the inverse map $h^{-1} : X_1 \to X_2$ maps $R_1$ to $R_2$ and we call $h$ a
\emph{tournament isomorphism}\index{tournament!isomorphism}. When $(X_1,R_1)$ and $(X_2,R_2)$ are topological tournaments
and $h$ is a homeomorphism we call it a \emph{ topological tournament isomorphism}. \index{topological tournament!isomorphism}
It is a \emph{topological  tournament automorphism}\index{topological tournament!automorphism} when the domain and the range are the same.

If $Y \subset X$ and $R$ is a tournament on $X$, then the restriction $R|Y$ is a tournament on $Y$
and the inclusion map from $Y$ to $X$ is a tournament map
from $R|Y$ to $R$.  Conversely, if $h$ is a tournament map with $h: X_2 \to X_1$ is  injective,
then since $1_{X_2} = (h \times h)^{-1}(1_{X_1})$ it follows that $h$ is an
isomorphism from $R_2$ onto the restriction of $R_1$ to the image of $h$.

We will call a topological tournament \emph{rigid} \index{topological tournament!rigid} when the identity is the only automorphism.

\begin{prop}\label{prop1.01a} Let $h : (X_2,R_2) \to (X_1,R_1)$ be a tournament map and let $x_1, x_2, x_3$ be distinct points of $X_2$
with $y_1 = h(x_1), y_2 = h(x_2), y_3 = h(x_3)$.

If $\{ y_1, y_2, y_3 \}$ is a $3-$cycle in $X_1$, then $\{ x_1, x_2, x_3 \}$ is a $3-$cycle in $X_2$.

Conversely, if
$\{ x_1, x_2, x_3 \}$ is a $3-$cycle in $X_2$, then either $\{ y_1, y_2, y_3 \}$ is a $3-$cycle in $X_1$ or else $y_1 = y_2 = y_3$.
\end{prop}

\begin{proof} That a cycle lifts to a cycle follows from (\ref{eq1.06}). If $\{ x_1, x_2, x_3 \}$ is a $3-$cycle in $X_2$,
then $y_1 \ \underline{\ha} \ y_2 \ \underline{\ha} \ y_3 \ \underline{\ha} \ y_1$. If two of the points are equal, then all three are.
For example, if $y_1 = y_2$
then $y_1  \ \underline{\ha} \  y_3$ and $y_3  \ \underline{\ha} \  y_1$ and so $y_1 = y_3$ by anti-symmetry.

\end{proof} \vspace{.5cm}

\begin{cor}\label{cor1.01b} Let $h : (X_2,R_2) \to (X_1,R_1)$ be a surjective tournament map and let $A$ be a subset of $X_1$.

The subset $h^{-1}(A)$ is an arc cyclic subset of $X_2$ if and only if
\begin{itemize}
\item[(i)] $A$ is an arc cyclic subset, and
\item[(ii)]  the restriction  $R_2|h^{-1}(y)$ is arc cyclic for every $y \in A$.
\end{itemize}
In particular,
the tournament $(X_2,R_2)$ is arc cyclic, if and
only if $(X_1,R_1)$ is arc cyclic and, in addition, the restriction  $R_2|h^{-1}(y)$ is arc cyclic for every $y \in X_1$.

The tournament $(X_2,R_2)$ is point cyclic, if either $(X_1,R_1)$ is point cyclic, or the restriction
$R_2|h^{-1}(y)$ is point cyclic for every $y \in X_1$.\end{cor}

\begin{proof} Assume $h^{-1}(A)$ is an arc cyclic subset.

If $y_1 \ha y_2$ with $y_1, y_2 \in A$, then because $h$ is surjective there exist
$x_1, x_2 \in h^{-1}(A)$ such that $y_1 = h(x_1), y_2 = h(x_2)$.
Because $h$ is a tournament map $x_1 \ha x_2$.  Because $h^{-1}(A)$ is an arc cyclic subset, there exists $x_3$ such that
$\{ x_1, x_2, x_3 \}$ is a $3-$cycle in $X_2$. Since $y_1 \not= y_2$, it follows from Proposition \ref{prop1.01a} that with $y_3 = h(x_3)$
$\{ y_1, y_2, y_3 \}$ is a $3-$cycle in $X_1$.

If $x_1 \ha x_2 $ and $h(x_1) = y = h(x_2)$ with $y \in A$, then  any  $3-$cycle $\{ x_1, x_2, x_3 \}$ in $X_2$ is contained in $h^{-1}(y)$
by Proposition \ref{prop1.01a}.
Since $h^{-1}(A)$ is an arc cyclic subset, it follows that $R_2|h^{-1}(y)$ is arc cyclic.

For the converse, suppose $ x_1 \ha x_2 $ with $x_1, x_2 \in h^{-1}(A)$. If $h(x_1) = y = h(x_2)$,
then there exists a $3-$cycle $\{ x_1, x_2, x_3 \}$  in $h^{-1}(y)$
by assumption. If $h(x_1) = y_1$ and $h(x_2) = y_2$ are distinct, then $y_1 \ha y_2$ with $y_1, y_2 \in A$. Because $A$ is
an arc cyclic subset, there exists $\{ y_1, y_2, y_3 \}$ a $3-$cycle in $X_2$ and so there exists  a $3-$cycle lift $\{ x_1, x_2, x_3 \}$ by
Proposition \ref{prop1.01a}again.

The point cyclicity result is obvious from Proposition \ref{prop1.01a}. Notice that if $(X_1,R_1)$ is trivial, then it is not
point cyclic even when $(X_2,R_2)$ is.

\end{proof} \vspace{.5cm}

The condition that $h$ be a continuous surjective tournament map between topological tournaments is rather restrictive.

\begin{theo}\label{theo1.02} With $(X_1,R_1)$ and $(X_2,R_2)$ topological tournaments, assume that
$h$ is a continuous tournament map from $(X_2,R_2)$ to $(X_1,R_1)$.  Let $y \in X_1$ and define
$h^*(y) = X_2 \setminus \overline{h^{-1}(R_1^{\circ}(y) \cup R_1^{\circ -1}(y))}$.
\begin{itemize}
\item[(i)] $h^*(y)$ is an open subset of $X_2$ with $h^*(y) \subset h^{-1}(y)$.

\item[(ii)] If $y$ is right balanced with respect to $R_1$, then
there exists at most one point $M \in \overline{h^{-1}(R_1^{\circ}(y))} \cap h^{-1}(y)$.
If the point $M$ exists, then it is a terminal point for the restriction $R_2|h^{-1}(y)$. If, in addition, $y$ is not left balanced, then
the open set $h^*(y)$ is $h^{-1}(y) \setminus \{ M \}$ or $h^{-1}(y)$ if $M$ does not exist.

\item[(iii)] If $y$ is left balanced with respect to $R_1$, then there exists at most one point
$m \in \overline{h^{-1}(R_1^{\circ -1}(y))} \cap h^{-1}(y)$.
If the point $m$ exists, then it is a initial point for the restriction $R_2|h^{-1}(y)$. If, in addition, $y$ is not right balanced, then
the open set $h^*(y)$ is $h^{-1}(y) \setminus \{ m \}$ or $h^{-1}(y)$ if $m$ does not exist.

\item[(iv)] If $y$ is balanced, then the open set $h^*(y)$ is $h^{-1}(y)$ with $m$ and $M$ removed when either exists.

\item[(v)] If $y$ is isolated, then $h^*(y) = h^{-1}(y)$ is a clopen subset of $X_2$.
\end{itemize}
\end{theo}

\begin{proof} (i) is clear since $\{ y \} = X_1 \setminus (R_1^{\circ}(y) \cup R_1^{\circ -1}(y))$.

(ii) If $M \in \overline{h^{-1}(R_1^{\circ}(y))} \cap h^{-1}(y)$ and $x \in h^{-1}(y)$, then for any
$z \in h^{-1}(R_1^{\circ}(y))$, $x \ha z$. Since $R_2$ is closed, $x \ \underline{\ha} \ M$. Hence, $M$
is the terminal point of $h^{-1}(y)$. If $y$ is not left balanced, then $R_1^{\circ -1}(y)$ is closed and so
$h^*(y) = h^{-1}(y) \setminus [\overline{h^{-1}(R_1^{\circ}(y))} \cap h^{-1}(y)]$.

(iii) Since $h$ maps $R_2^{-1}$ to $R_1^{-1}$, this follows from (ii).

(iv) and (v) are obvious.

\end{proof}

\begin{add}\label{add1.03} Let $h$ be a surjective continuous tournament map from the topological
tournament $(X_2,R_2)$  to $(X_1, R_1)$, with $X_2$ compact.

If $y \in X_1$ is right balanced (or left balanced) with respect to $R_1$, then $\overline{h^{-1}(R_1^{\circ}(y))} \cap h^{-1}(y)$
(resp. $\overline{h^{-1}(R_1^{\circ -1}(y))} \cap h^{-1}(y)$) is nonempty and so is a singleton.\end{add}

\begin{proof} By continuity and compactness, the surjective map $h$ sends
$\overline{h^{-1}(R_1^{\circ}(y))}$ onto a closed set which contains $R_1^{\circ}(y)$. If $y$ is
right balanced, then $\overline{R_1^{\circ}(y)} = R_1(y)$ which contains
$y$.  Thus, $\overline{h^{-1}(R_1^{\circ}(y))}$ meets $h^{-1}(y)$ and
from Theorem \ref{theo1.02} (ii) we see that the intersection is a singleton.

\end{proof}
\vspace{.5cm}

\begin{theo}\label{theo1.02a}  If $h : (X_2,R_2) \to (X_1,R_1)$ is a continuous tournament map with $X_1, X_2$ compact metric spaces,
then for every $\ep > 0$, the set
$\{ y \in X_1 : diam \ h^{-1}(y) \ge \ep \}$ is finite.\end{theo}

\begin{proof} Suppose there is is a sequence of triples  $\{ (y_n,x_n,z_n) \in X_1 \times X_2 \times X_2 \}$
with $\{ y_n \}$ distinct points, $h(z_n) = y_n = h(x_n)$ and
$d(z_n,x_n) \ge \ep $.  By going to a subsequence, we may assume that the sequence converges to $(y,x,z)$ so that $h(z) = y = h(x)$ and
$d(z,x) \ge \ep $. By going to a further subsequence, we may assume $y_n \in R_1^{\circ}(y)$ or $y_n \in R_1^{\circ -1}(y)$ for all $n$.
Suppose the latter holds.  Then $z_n \ha x$ and $x_n \ha z$ for all $n$ and so in the limit $z \ha x$ and $x \ha z$ contradicting anti-symmetry.

\end{proof}
\vspace{1cm}

\section{ \textbf{Lexicographic Products}}\vspace{.5cm}

Let $(X_1,R_1)$ be a tournament.  Assume that for each $x \in X_1$, $(Y_x,S_x)$ is a tournament.
The \emph{lexicographic product}\index{lexicographic product} is defined by:
\begin{align}\label{eqlex01}\begin{split}
X_2 \ = \ X_1 \times \{ Y_x \} \ = \  &\ \bigcup_{x \in X} \ \{ x \} \times Y_x, \\
R_2 \ = \ R_1 \ltimes \{ S_x \} \quad \text{where} \ &\text{for} \ (x_1,y_1), (x_2,y_2) \in X_2, \\
((x_1,y_1), (x_2,y_2)) \in R_2 \quad \Longleftrightarrow &\quad \begin{cases} (x_1,x_2) \in R_1^{\circ} \quad \text{or}\\
x_1 = x_2 \ \text{and} \ (y_1,y_2) \in S_{x_1}. \end{cases}
\end{split} \end{align}
It is clear that $(X_2,R_2)$ is a tournament and the first coordinate projection $\pi : X_2 \to X_1$
is a surjective tournament map from $(X_2,R_2)$ to $(X_1,R_1)$.
We call $X_2$ the \emph{total space}\index{total space}\index{space!total}, $X_1$ the \emph{base space}\index{base space}\index{space!base}
and the $Y_x$'s the \emph{fibers} \index{space!fiber} of the product.

\begin{prop}\label{proplex01a} If $R_1$ and each $S_x$ is transitive, then $R_2$ is transitive. \end{prop}

 \begin{proof} Assume $((x_1,y_1), (x_2,y_2)), ((x_2,y_2), (x_3,y_2))  \in R_2$. If either \\$(x_1,x_2) \in R_1^{\circ}$ or $(x_2,x_3) \in R_1^{\circ}$
 then by transitivity of $R_1$, $(x_1,x_3) \in R_1^{\circ}$ and so $((x_1,y_1), (x_3,y_3)) \in R_2$. Otherwise, $x_1 = x_2 = x_3$ and
 $(y_1,y_2), (y_2,y_3) \in S_{x_1}$. By transitivity of $S_{x_1}$, $(y_1,y_3) \in S_{x_1}$ and so $((x_1,y_1), (x_3,y_3)) = ((x_1,y_1), (x_1,y_3)) \in R_2$.

 \end{proof}

 We will write $(X_2,R_2)$ as $(X_1,R_1) \ltimes \{ (Y_x,S_x) \}$.
 \vspace{.5cm}

For the special case when $(Y_x,S_x) = (Y,S)$ for all $x$, we have $X_2 \ = \ X_1 \times Y$ and we write $R_2 \ = \ R_1 \ltimes S$ and
 $(X_2,R_2) = (X_1,R_1) \ltimes (Y,S)$ is called the  \emph{lexicographic product} of $(X_1,R_1)$ and $(Y,S)$, see \cite{GM}. In that case,
 \begin{equation}\label{eqlex02}
R_1 \ltimes S \ = \ [R_1^{\circ} \times (Y \times Y)] \ \cup \ [1_X \times S], \hspace{1cm}
\end{equation}
i.e. the union of two product relations as in (\ref{eq1.03}).
If $R_1$ and $S$ are topological tournaments with $Y$ non-trivial, then $R_1 \ltimes S$ is
closed if and only if $R_1^{\circ}$ is closed and so $1_X$ is clopen,
which means that $X$ is discrete, i.e. every point of $X$ is isolated.

In particular, if $X_1$ is finite and $(Y,S)$ is a topological tournament, then $(X_1,R_1) \ltimes (Y,S)$ is a topological tournament
where $X_1 \times Y$ has the product topology.   I emphasize the topology on $X_1 \times Y$ because we will deal with the problem of obtaining
a closed relation for the lexicographic product by
adjusting the topology on the total space.

\begin{df}\label{dflex01a} Let $(X_1,R_1)$ and the members of $\{ (Y_x,S_x) : x \in X \}$ be topological tournaments. The product tournament
$(X_2,R_2) = (X_1,R_1) \ltimes \{ (Y_x,S_x) \}$ is  called the
\emph{topological lexicographic product}\index{lexicographic product!topological} when
the following conditions hold:
\begin{itemize}
\item[(i)] For each $x \in X_1$, either $x$ is an isolated point of $X_1$ or else $(Y_x,S_x)$ is a  trivial tournament.

\item[(ii)] The total space $X_2$ is given the topology with basis $\B$ where $U \in \B$ when
either $U = \pi^{-1}(V)$ for $V$ some open  subset of $X_1$, or
$U = \{ x \} \times V$ for $x$ isolated in $X_1$ and $V$ some open subset of $Y_x$.
\end{itemize} \end{df}

In particular, if $x \in X_1$ is non-isolated, then $\pi^{-1}(x)$ is a singleton subset of $X_2$ which we will identify with $\{ x \}$.

\begin{theo}\label{theolex01} The topological lexicographic product $(X_2,R_2) =$ \\ $ (X_1,R_1) \ltimes \{ (Y_x,S_x) \}$ is a
topological tournament which satisfies the following properties.

\begin{enumerate}
\item[(a)] The projection map
 $\pi$ is a continuous, open surjection mapping $R_2$ to $R_1$.

\item[(b)] A point $(x,y) \in X_2$ is isolated if and only if $x$ is isolated in $X_1$ and $y$ is isolated in $Y_x$.

\item[(c)] Assume $x \in X_1$ is an isolated point. The map $y \mapsto (x,y)$ is a homeomorphism from $Y_x$ onto the clopen subset $\{ x \} \times Y_x$ of $X_2$,
mapping $S_x$ isomorphically to the restriction $R_2|\{ x \} \times Y_x$. In particular, a point $y \in Y_x$ is left (or right) balanced for $S_x$
if and only if $(x,y)$ is left (resp. right) balanced for $R_2$.

\item[(d)] If $X_1$ and each $Y_x$ is compact, then $X_2$ is compact.

\item[(e)] If $X_1$ and each $Y_x$ is countable, then $X_2$ is countable.

\item[(f)] If $R_1$ and each $S_x$ is transitive, then $R_2$ is transitive.

\item[(g)] If $R_1$ and each $S_x$ is arc cyclic, then $R_2$ is arc cyclic.
If $R_1$ is point cyclic, then $R_2$ is point cyclic.
\end{enumerate}
\end{theo}

\begin{proof} (a): It is clear that the collection $\B$ is closed under intersection and so forms a basis for a topology with $\pi_X$ continuous.
Since $\pi(U)$ is open in $X_1$ for each $U \in \B$, it follows that $\pi$ is an open map. From (i) and (ii) it easily follows that the topology on
$X_2$ is Hausdorff.

Now suppose that $\{ ((x_k,y_k),(u_k,v_k)) \} $ is a net in $R_2$ converging to $((x,y),(u,v)) \in X_2 \times X_2$. Since $\pi$ is continuous,
$\{ (x_k,u_k) \}$ converges to $(x,u)$ in $X_1 \times X_1$.

Case 1: ($x \not= u$) If $(x,u) \in R_1^{\circ -1}$ then eventually $(x_k,u_k) \in R_1^{\circ -1}$ and so $((x_k,y_k),(u_k,v_k)) \in R_2^{\circ -1}$
contrary to hypothesis.  Hence, $(x,u) \in R_1^{\circ}$ and so $((x,y),(u,v)) \in R_2^{\circ}$.

Case 2a: ($x = u$ is isolated) In that case, eventually, $x_k = x$ and $u_k = u = x$ and so $\{(y_k,v_k) \}$ is eventually a net in
$S_x$. Hence, the limit point $(y,v) \in S_x$ which implies $((x,y),(u,v)) = ((x,y),(x,v))\in R_2$.

Case 2b: ($x = u$ is not isolated) In that case, $Y_x$ is a singleton and so $y = v$. That is, $(x,y) = (u,v)$ and so
$((x,y),(u,v)) \in R_2$.

Thus, $R_2$ is a closed relation and so $(X_2,R_2)$ is a topological tournament.

(b): Clearly, if $x$ is isolated in $X_1$ and $y$ is isolated in $Y_x$, then $\{ x \} \times \{ y \}$ is a basic open set in $X_2$ and so $(x,y)$ is isolated.

If $x$ is isolated and $\{ y_k \}$ is a net in $Y_x \setminus \{ y \}$ converging to $y$,
then $\{ (x,y_k) \}$ converges to $(x,y)$ and so $(x,y)$ is not isolated.

If $\{ x_k \}$ is a net in $X_1 \setminus \{ x \}$ converging to $x$ and $Y_x = \{ y \}$,
then for any $y_k \in Y_{x_k}$, the net $\{ (x_k,y_k) \}$ in $X_2$ converges to $(x,y)$ and so
$(x,y)$ is not isolated.

(c): That the injection from $Y_x$ onto $\{ x \} \times Y_x \subset X_2$ is a homeomorphism onto a clopen subset follows using the basis $\B$ in (ii).

(d): Now assume that $X_1$ and the $Y_x$'s are compact and that $\U$ is an open cover of $X_2$. Let $\U_1$ be the open cover of $X_1$ consisting of
the singleton isolated points together with open sets $V$ such that $\pi^{-1}(V)$ is contained in some member of $\U$. Because $X_1$ is compact,
$\U_1$ has a finite subcover consisting of
finitely many isolated point singletons $\{ x_j \}$ together with finitely many open sets $V_i$ with $\pi^{-1}(V_i) \subset U_i$.
For each $\{x_j \}$ there is a finite cover $\{ V_{jk} \}$ of $Y_{x_j}$ consisting of open sets with $\{x_j \} \times V_{jk}$ contained in
some member $U_{jk}$ of $\U$.  Then, $\{ U_i \}$ together with $\{ U_{jk} \}$ for each $x_j$ is a finite cover of $X_2$ by elements of $\U$.
It follows that $X_2$ is compact.

(e): The countability result is obvious.

(f): The transitivity result follows from Proposition \ref{proplex01a}.

(g): The cyclicity results follow from Corollary \ref{cor1.01b}

\end{proof}

\begin{add}\label{addlex01a} If $\{z_k \}$ is a net in $X_2$ and $z \in X_2$ with $\pi(z)$ non-isolated, then $\{ z_k \}$ converges
to $z$ if and only if the net $\{ \pi(z_k) \}$ in $X_1$ converges to $\pi(z) \in X_1$. In particular, $x = \pi(z)$ is left (or right) balanced for $R_1$
if and only if, identified
with the point in $\pi^{-1}(x)$  it is left (resp. right) balanced for $R_2$.
\end{add}

\begin{proof} Suppose $\{ \pi(z_k) \}$ converges to $\pi(z)$. If $U \subset X$ is an open set with $\pi(z) \in U$, then
eventually $\pi(z_k) \in U$ and so $z_k \in \pi^{-1}(U)$.  From the definition of the topology on $X_2$, it follows that
$\{ z_k \}$ converges to $z$.

The converse is obvious from the continuity of $\pi$.

The balance results follow because $(z_k,z) \in R_2^{\circ}$ if and only if \\ $(\pi(z_k),\pi(z)) \in R_1^{\circ}$.

\end{proof}

\vspace{.5cm}

For a topological lexicographic product $(X_2,R_2) = (X_1,R_1) \ltimes \{ (Y_x,S_x) \}$ a \emph{section}\index{section} is a
function $\xi : X_1 \to X_2$ such that $\pi \circ \xi \ = \ 1_{X_1}$.  That is, $\xi$ is essentially a choice function $\tilde \xi$ for
the family $\{ Y_x \}$ with $\xi(x) = (x,\tilde \xi(x))$.

\begin{lem}\label{lemlex01b} For a topological lexicographic product $(X_2,R_2) = (X_1,R_1)$ \\ $ \ltimes \{ (Y_x,S_x) \}$ any section
is continuous. Furthermore any section $\xi$ is a topological tournament isomorphism from $(X_1,R_1)$ onto the
restriction of $R_2$ to the image $\xi(X_1)$. \end{lem}

\begin{proof} Continuity at $x$ when $x$ is  isolated is obvious.  When $x$ is non-isolated, continuity follows from
Addendum \ref{addlex01a}. It is clear that the injection $\xi$ maps $R_1$ to $R_2$.

\end{proof}

\vspace{.5cm}

\begin{df}\label{dflex02} We call a topological tournament $(X,R)$ a \emph{brick}\index{brick} when it satisfies the following conditions.
\begin{itemize}
\item[(i)] The space $X$ is compact and the isolated points are dense in $X$.

\item[(ii)] If $x \in X$  is not isolated, then the point $x$ is balanced for $R$.
\end{itemize}

We call a brick \emph{isolated point cyclic}\index{brick!isolated point cyclic}, or \emph{ip cyclic}\index{brick!ip cyclic} when it satisfies, in addition,
\begin{itemize}
\item[(iii)] If $x$ is an isolated point, then there exists a $3-$cycle for $R$ which contains $x$.
 \end{itemize}
 \end{df}

 From the density of the isolated points, it follows that the $3-$cycle in (iii) can be chosen to consist of isolated points.
 \vspace{.5cm}

 If $X$ is finite then $(X,R)$ is a brick and if, in addition, $(X,R)$ is regular, then it is ip cyclic by \cite{HM} Theorem 7.

 \begin{theo}\label{theolex03} If $(X_1,R_1)$ and the members of $\{ (Y_x,S_x) : x \in X \}$ are all bricks, with $(Y_x,S_x)$ trivial when $x$ is
 not isolated in $X_1$, then the topological lexicographic
 product $(X_2,R_2) = (X_1,R_1) \ltimes \{ (Y_x,S_x) \}$ is a brick.  If, in addition, for each isolated point
 $x$, the brick $(Y_x,S_x)$ is ip cyclic, then
$(X_2,R_2)$ is ip cyclic. \end{theo}

\begin{proof} Compactness follows from Theorem \ref{theolex01}.

If $x$ is isolated but $y \in Y_x$ is not, then since $y$ is assumed balanced in $Y_x$, it follows that $(x,y)$ is balanced in $X_2$ by
Theorem \ref{theolex01} (c). Since $y$ is a limit of isolated points in $Y_x$, $(x,y)$ is a limit of  points isolated in $\{ x \} \times Y_x$ and
hence in $X_2$.

If $x$ is not isolated in $X_1$, then since $x$ is balanced in $X_1$, it is balanced in $X_2$ by  Addendum \ref{addlex01a}. If $x$ is the limit
of a net $\{ x_k \}$ of isolated points in $X_1$ and $y_k$ is an isolated point in $Y_{x_k}$ then $\{ (x_k,y_k) \}$ is a net of isolated points
in $X_2$ which converges to $x$ in $X_2$ by Addendum \ref{addlex01a} again.

Thus, $(Z,T)$ is a brick.

If $(x,y)$ is an isolated point and $(Y_x,S_x)$ is ip cyclic, then $y$ is contained in a $3-$cycle $\{ y, y', y'' \}$ in $Y_x$
and so $(x,y)$ is contained in the  $3-$cycle of  points $\{ (x,y), (x,y'), (x,y'') \}$ in $X_2$.

\end{proof}

In general, we will call a topological tournament $(X,R)$ \emph{ip cyclic}\index{tournament!ip cyclic} when every isolated point of $X$ is
contained in a $3-$cycle.

\vspace{1cm}

\section{ \textbf{Inverse Limits}}\vspace{.5cm}

An \emph{inverse system}\index{inverse system}  $\{ (X_i,f_i): i \in \N \}$ is a sequence with $f_i$ a function from $X_{i+1}$ to $X_i$ for all $i \in \N$.
The \emph{inverse limit}\index{inverse limit} $X = \overleftarrow{Lim} \{ (X_i,f_i) \}$ is given by
  \begin{equation}\label{eqinv01}
  X \ = \ \{ x \in \prod_{i \in \N} \ X_i :  f_i(x_{i+1}) \ = \ x_i \ \ \text{for all} \ \ i \in \N \},
  \end{equation}
The functions $\pi_i : X \to X_i$ and $\pi_{i,i+1} : X \to X_i \times X_{i+1}$ are the projection mappings. Clearly, for all $i$:
 \begin{equation}\label{eqinv01a}
 f_i \circ \pi_{i+1} \ = \ \pi_i \quad \text{on} \ \ X. \hspace{1cm}
 \end{equation}
 We call $\{ (X_i,f_i): i \in \N \}$ a \emph{surjective inverse system}\index{inverse system!surjective} when each $f_i$ is a surjective map.

\begin{prop}\label{propinv01} If $\{ (X_i,f_i): i \in \N \}$ is a surjective inverse system, then
for all $i \in \N$ $\pi_{i,i+1}$ maps $X$ onto $f_i^{-1}$ and
$\pi_i$ maps $X$ onto $X_i$. \end{prop}

\begin{proof} It is clear that $\pi_{i,i+1}$ maps into $f_i^{-1}$.

Let $(x_i,x_{i+1}) \in f_i^{-1}$. Inductively,
for $j$ with $1 \le j < i$, let $x_{i-j}$ be the point such that
$ f_{i-j}(x_{i-j+1}) \ = \ x_{i-j}.$ Because each $f_k$ is surjective, for $j$ with $1 < j$  we can choose, inductively, a point $x_{i+j}$ such that
 $f_{i+j-1}(x_{i+j}) \ = \ x_{i+j-1}$.
Thus, $\pi_{i,i+1}$ maps onto $f_i^{-1}$.  Since $f_i$ is surjective, it clearly follows that $\pi_i : X \to X_i$ is onto as well.

\end{proof}\vspace{.5cm}

If each $f_i$ is a continuous map, then $X$ is a closed subset of $\prod_{i \in \N} \ X_i$ with the latter given the product topology.
If, in addition, the spaces $X_i$ are compact, then the inverse limit space $X$ is compact by the Tychonoff Product Theorem.
In any case, the projection maps are continuous.

\begin{theo}\label{theoinv02} Assume that $\{ (X_i,f_i) \}$ is an inverse system
with  inverse limit $X$.
If, for each $i \in \N$, $R_i$ is a relation on
$X_i$ such that $f_i$ maps $R_{i+1}$ to $R_i$, then $\{ (R_i,f_i \times f_i) \}$ is an inverse system with inverse limit which we label $R$.

Identifying $\prod_{i \in \N} (X_i \times X_i)$ with $ (\prod_{i \in \N} X_i) \times (\prod_{i \in \N} X_i)$ we can regard $R$ as a relation on $X$ with
 \begin{equation}\label{eqinv02}
 R \ = \ \bigcap_{i \in \N} \ (\pi_i \times \pi_i)^{-1}(R_i). \hspace{2cm}
 \end{equation}

 If each $R_i$ is a tournament on $X_i$, then $R$ is a tournament on $X$ with $\pi_i$ mapping $R$ to $R_i$. For $x,x' \in X$, we have
 $(x,x') \in R^{\circ}$ if and only if there exists $i \in \N$ such that $x_j = x'_j$ for all
$j < i$ and $(x_i,x'_i) \in R_i^{\circ}$.

 If each $f_i : X_{i+1} \to X_i$ is a continuous map of topological spaces and each $R_i$ is a topological tournament on $X_i$, then $R$
 is a topological tournament on $X$.

  If each $R_i$ is a transitive tournament, then $R$ is transitive.
 \end{theo}

 \begin{proof} It is clear that $\{ (R_i,f_i \times f_i) \}$ is an inverse system and with the above identification we can regard $R$ as a
 relation on $X$ such that $\pi_i$ maps $R$ to $R_i$. Hence, $R \subset \bigcap_{i} \ (\pi_i \times \pi_i)^{-1}(R_i)$. On the other hand,
 if $(x,x') \in  \bigcap_{i} \ (\pi_i \times \pi_i)^{-1}(R_i)$, then $(x_i,x'_i) \in R_i$ and $x, x' \in X$ implies
 $(x_i,x'_i) = (f_i \times f_i)(x_{i+1},x'_{i+1})$.  Hence, $(x,x') \in R$, proving (\ref{eqinv02}).

 Now assume that each $R_i$ is a tournament.
 \begin{align}\label{eqinv03}\begin{split}
 R \cap R^{-1} \ = \ &\bigcap_{i} \ (\pi_i \times \pi_i)^{-1}(R_i \cap R_i^{-1}) \\
  = \ \bigcap_{i} \ &(\pi_i \times \pi_i)^{-1}(1_{X_i}) \ = \ 1_X.
 \end{split}\end{align}
 Therefore, $R$ is anti-symmetric.

 Now assume that $(x,x') \in (X \times X) \setminus R$. From (\ref{eqinv02}) it follows that for some $i_0, \ (x_{i_0},x'_{i_0}) \in R_{i_0}^{\circ -1}$.

 If for some $i_1, (x_{i_1},x'_{i_1}) \in R_{\circ i_1}$, then applying (\ref{eq1.06}) to the appropriate composition of the maps $f_k$
 with would obtain, with $i = \max(i_0,i_1)$ that $(x_{i},x'_{i}) \in R_{i}^{\circ -1}\cap R_{i}^{\circ}$ which is impossible.
 Hence, for all $i_1 \in \N, (x_{i_1},x'_{i_1}) \in R_{i_1}^{-1}$ and thus $(x,x') \in R^{-1}$.  That is, $R$ is total and so is a tournament.

 If $(x,x') \in R^{\circ}$, then since $\pi_j$ maps $R$ to $R_j$, we have $(x_j,x'_j) \in R_j$ for all $i$.  So if $i = \min \{ j : x_j \not= x'_j \}$,
 then $(x_i,x_i') \in R_i^{\circ}$. Conversely, if $(x_i,x_i') \in R_i^{\circ}$ then since $\pi_i$ maps $R$ to $R_i$, we have we have
 $(x,x') \in R^{\circ}$.

 Given the topological assumptions, it is clear that $R$ is closed and so is a topological tournament.

 Now assume that each $R_i$ is transitive and that $(x,x'),(x',x'') \in R$. We show that $(x,x'') \in R$. Clearly, we may assume that the
 three points are distinct so that $(x,x'),(x',x'') \in R^{\circ}$ There exists $i \in \N$
 such that $x_k = x_k'$ for all $k < i$ and $(x_i,x_i') \in R_i^{\circ}$. Similarly,
 for some $j, x_k' = x_k''$ for all $k < j$ and $(x_j',x_j'') \in R_j^{\circ}$. If $i > j$, then $x_k = x_k' = x_k''$ for all $k < j$ and
 $(x_j,x_j'') = (x_j',x_j'') \in R_j^{\circ}$ so that $(x,x'') \in R^{\circ}$. Similarly, if $i < j$, $(x,x'') \in R^{\circ}$.
 If $i = j$, then $x_k = x_k' = x_k''$ for all $k < i$ and $(x_i,x_i'), (x_i',x_i'') \in R_i^{\circ}$. By transitivity of
 $R_i$, we have $(x_i,x_i'') \in R_i$. Antisymmetry and  $(x_i,x_i'), (x_i',x_i'') \in R_i^{\circ}$ implies that $x_i \not= x_i''$.
 Hence, in this case as well $(x,x'') \in R^{\circ}$.

 \end{proof}

 We will call a sequence $\{ (X_i,R_i,f_i): i \in \N \}$ an inverse system of tournaments, when $\{ (X_i,f_i): i \in \N \}$ is an inverse system,
$R_i$ is a tournament on $X_i$ and $f_i$ maps $R_{i+1}$ to $R_i$ for all $i$. We call the
 tournament $(X,R)$ the inverse limit of this system when $X = \overleftarrow{Lim} \{ (X_i,f_i) \}$ and
 $R = \overleftarrow{Lim} \{ (R_i,f_i \times f_i) \}$.

 For every inverse system of tournaments, $\{ (X_i,R_i,f_i): i \in \N \}$ it will be convenient to assume that there
 is a zero level with $(X_0,R_0,f_0)$  with $(X_0,R_0)$ a trivial tournament, i.e. $X_0$ is a singleton,
 and $f_0$ is the unique function from $X_1$ to $X_0$.
 \vspace{.5cm}

 \begin{prop}\label{propinv01a} If $\{ (X_i,R_i,f_i): i \in \N \}$ is a surjective inverse system of tournaments, then
the limit tournament $(X,R)$ is arc cyclic, if and only if $(X_i,R_i)$ is arc cyclic for every $i \in \N$.\end{prop}

\begin{proof} If $(X,R)$ is arc cyclic, then since $\pi_i$ maps $(X,R)$ onto $(X_i,R_i)$, the latter is arc cyclic by Corollary \ref{cor1.01b}.

Now assume that all the $(X_i,R_i)$'s are arc cyclic. If $(x,x') \in R^{\circ}$, then there exists $i \in \N$ such that $x_j = x'_j$ for all
$j < i$ and $(x_i,x'_i) \in R_i^{\circ}$. Since $(X_i,R_i)$ is arc cyclic, there exists $z \in X_i$ such that $\{ x_i,x'_i,z \}$ is a $3-$cycle
in $X_i$.  Since $\pi_i$ is surjective, there exists $x''$ such that $\pi_i(x'') = z$.  Then $\{x,x',x'' \}$ is a $3-$cycle in $X$.

 \end{proof} \vspace{.5cm}

 We call $\{(X_i,R_i,f_i) : i \in \N \}$ an \emph{inverse system of topological tournaments}\index{inverse system of topological tournaments}
  when $\{ (X_i,R_i, f_i) \}$ is an inverse system of tournaments with each $(X_i,R_i)$ a topological tournament and each $f_i$ continuous.
The inverse limit $(X,R)$ is then topological.

 \begin{add}\label{addinv03} Assume that $\{ (X_i,R_i,f_i) \}$ and $\{ (Y_i,S_i,g_i)\}$ are inverse systems of topological
 tournaments with limits $(X,R)$ and $(Y,S)$. If for each $i$, the continuous function $h_i : X_i \to Y_i$
maps $R_i$ to $S_i$ and $g_i \circ h_{i+1} = h_i \circ f_i$, then the product map
$\prod_{i} \ h_i : \prod_i X_i \to \prod_i Y_i$  defined by $h(x)_i = h_i(x_i)$ restricts to a continuous function
$h : X \to Y$ which maps  $(X,R)$ to $(Y,S)$.
\end{add}

\begin{proof} Just as the family $\{ h_i \}$ maps  $\{ (X_i,f_i) \}$ to $\{ (Y_i,g_i)\}$, the family $\{ h_i \times h_i \}$ maps
 $\{ (R_i,f_i \times f_i) \}$ to $\{ (S_i,g_i \times g_i)\}$ and hence $(h \times h)(R) \subset S$.

 \end{proof}  \vspace{.5cm}

\begin{exes}\label{exes01} \end{exes} (a) Let $\{ K_i \}$ be a decreasing sequence of subsets of a set $X$ with $k_i : K_{i+1} \to K_i$ the inclusion map.
If $K \ = \ \bigcap_{i \in \N} \ K_i$, then map which associates to $x \in K$ the constant sequence at $x$ is an identification of
$K$ with the inverse limit of $\{ (K_i,k_i) \}$.  The inverse map for this identification equals $\pi_i$ for every $i$.

If $R$ is a tournament on $X$, then the identification is an isomorphism from the restriction $R|K$ to the inverse limit of the system
$\{ (K_i, R|K_i, k_i) \}$. \vspace{.25cm}

(b) With $\{ Y_i \}$ a sequence of spaces, let $X_i = \prod_{1 \le j \le i} \ Y_j$, $f_i : X_{i+1} \to X_i$ be the projection on the first
$i$ coordinates and $g_i : X_i \to Y_i$ be the $i^{th}$ coordinate projection. Let $Y \ = \ \prod_{i \in \N} \ Y_i$.
 The map $q : X \to Y$ defined by $q(x)_i = g_i(\pi_i(x))$ is an identification of $Y$ with
the inverse limit $X$ of $\{ (X_i,f_i) \}$. \vspace{.5cm}

For $\{(X_i,R_i,f_i) : i \in \N \}$ an inverse system of topological tournaments with limit the topological tournament $(X,R)$ we let
$IS = \{ x \in X : x_i $  is an isolated point of $X_i$ for all $i \in \N \}$ \index{IS}.
\vspace{.5cm}

Now assume that  $(X_1,R_1)$ is a compact topological tournament.
Inductively we define  $(X_{i+i},R_{i+1})$ to be the topological lexicographic product
$(X_i,R_i) \ltimes \{ (Y_{iz},S_{iz}) : z \in X_i \}$ with each $(Y_{iz},S_{iz})$ a compact tournament and
with $(Y_{iz},S_{iz})$ trivial when $z$ is not isolated in
$X_i$.  Let $f_i : X_{i+1} \to X_i$ be the first coordinate projection. By
Theorem \ref{theolex01} each $(X_i,R_i)$ is a compact topological tournament. Thus,
$\{ (X_i,R_i,f_i) \}$ is a surjective inverse system of topological tournaments which we will call a \emph{lexicographic inverse system}
\index{lexicographic inverse system}.\index{inverse system!lexicographic}
The limit system $(X,R)$ is a compact topological tournament by Theorem \ref{theoinv02}.

If $(X_1,R_1)$ and each $(Y_{iz},S_{iz})$ is a brick, then, inductively,$(X_i,R_i)$ is a brick and
we will call  $\{ (X_i,R_i,f_i) \}$  a \emph{lexicographic inverse system of bricks}
\index{lexicographic inverse system of bricks}\index{inverse system!lexicographic!of bricks}.

\begin{theo}\label{theoinv03} Assume that $\{(X_i,R_i,f_i) \}$ is a lexicographic inverse system with limit tournament $(X,R)$.
\begin{itemize}
\item[(a)] For each $i \in \N$ the projection map $\pi_i : X \to X_i$ is a continuous, open surjection.

\item[(b)]  If $z \in X_i$ is not isolated, then $\{ \pi_i^{-1}(z) \}$ is a singleton subset $ \{ x \} $ of $X$, and a net $\{ x_k \}$
 in $X$ converges to $x$ in $X$ if and only if $\{ \pi_i(x_k) \}$ converges to $z$ in $X_i$.
\end{itemize}

Now assume that $\{(X_i,R_i,f_i) \}$ is a lexicographic inverse system of bricks.
\begin{itemize}
\item[(c)] The set $IS$  is residual in $X$.  That is, it is a dense $G_{\d}$ subset of $X$.

 \item[(d)]  If for infinitely many   $i \in \N$ and  the tournament $(Y_{iz},S_{iz})$ has no
terminal  point for each isolated point $z \in X_i$, and for infinitely many   $i \in \N$  the tournament $(Y_{iz},S_{iz})$ has no
 initial point each isolated point $z \in X_i$, then the limit tournament $(X,R)$ is balanced.
If, in addition, $X_1$ and each $Y_{iz}$ is countable, then $X$ is a Cantor Set.

\end{itemize}
\end{theo}

\begin{proof}  (a): The $\pi_j$'s are surjective by Proposition \ref{propinv01}.  The basic open subsets of $X$
can be written $\pi_j^{-1}(U)$ for $j$ arbitrarily large and $U$  open in $X_j$.  Because each $\pi_j$ is
surjective by Proposition \ref{propinv01}, $\pi_j(\pi_j^{-1}(U)) = U$. Choose $j > i$.  Using (\ref{eqinv01a}) and induction we see that
$\pi_i(U) = f_i \circ \dots \circ f_{j-1}(U)$. This is open because each $f_j$ is an open map by Theorem \ref{theolex01}(a). Hence, $\pi_i$ is an open
map.  It is clearly continuous.

(b):  If $z \in X_i$ is not isolated, then $Y_{iz}$ is a singleton and so $\{ f_i^{-1}(z) \}$ is a singleton $\{ z' \}$ in $X_{i+1}$, and by
Theorem \ref{theolex01} $z'$ is not isolated in $X_{i+1}$. Proceeding upwards by induction we see that there is only one point
$x$ with $x_i = z$.  If $\{ x_k \}$ is a net in $X$ such that $\{ \pi_i(x_k) \}$ converges to $z$ in $X_i$, then Addendum
\ref{addlex01a}  implies that $\{ \pi_{i+1}(x_k) \}$ converges to $z'$.  Of course, by continuity
$\{ \pi_{i-1}(x_k) = f_{i-1}(\pi_i(x_k)) \}$ converges to $ f_{i-1}(\pi_i(x)) = \pi_{i-1}(x)$.  Proceeding upwards and downwards by induction
we see that $\{ x_k \}$ converges coordinatewise to $x$.

(c): If $Iso(X_i)$ is the set of isolated points of $X_i$, then it is a open subset of $X_i$ which is dense in $X_i$ because $X_i$ is a brick.
Because $\pi_i$ is continuous and open, the set $\pi_i^{-1}(Iso(X_i))$ is open and dense in $X$. By the Baire Category Theorem, the set
$IS \ = \ \bigcap_i \  \pi_i^{-1}(Iso(X_i))$ is a dense $G_{\d}$ subset of $X$.

(d):  Let $x \in X$.

If $x_i \in X_i$ is not isolated, then since $(X_i,R_i)$ is a brick,  $x_i \in X_i$ is balanced and so we can choose a net $ \{ z_k \}$ in $R_i^{\circ}(x_i)$ which
converges to $x_i$ in $X_i$. Choose $x_k$ so that $\pi_i(x_k) = z_k$.  Then $\{ x_k \}$ is a net in $X$ which converges to $x$. Because
$\pi_i$ maps $R$ to $R_i$, we have $x_k \in R^{\circ}(x)$.  Hence, $x$ is right balanced and similarly it is left balanced.

Thus, if $x \not\in IS$, then it balanced in any case.

Now assume that $x \in IS$.

Fix $i$ arbitrarily large such that no $(Y_{iz},S_{iz})$ with $z$ isolated has a terminal point
 and let $z = x_i$. The point $x_{i+1} = (z,y)$ with $y \in Y_{iz}$. Since
$(Y_{iz},S_{iz})$ has no terminal point, there exists $y' \in S_{iz}^{\circ}(y)$ and so $(z,y') \in R_{i+1}^{\circ}(z,y)$.
There exists $x' \in X$ with $x'_{i+1} = (z,y')$ and so $x' \in R^{\circ}(x)$. Furthermore, $x'_i = z = x_i$ and so
$x_j = x'_j$ for all $j \le i$. As $i$ was arbitrarily large, $x'$ is arbitrarily close to $x$ and so $x$ is right balanced.
Similarly, $x$ is left balanced.

If $X_1$ and each $Y_{iz}$ is countable, then Theorem \ref{theolex01} and induction imply each $X_i$ is a countable brick and certainly $X_j$
is not trivial for $j > 1$.  Hence, the space $\prod_i \ X_i$ is a countable product of compact, metrizable, totally disconnected spaces and so is
a Cantor set.  The subset $X$ is therefore a compact, metrizable, totally disconnected space. Since the tournament $R$ is balanced, $X$ has no
isolated points and so is itself a Cantor set.

\end{proof}
\vspace{.5cm}

When $X_1$ and all the $Y_{iz}$'s are finite, then each $X_i$ is finite and so consists of isolated points. We consider the case when
for each $i$ the $(Y_{iz},S_{iz})$'s are the same for all $z \in X_i$.

 On a finite set of cardinality $n$ there are $2^{n(n-1)/2}$ tournaments.  When $n \ge 4$ the majority of these have no terminal nor initial point.

 Let $S_i$ for $ i \in \\Z_+ = \{ 0 \} \cup \N $ be a tournament  on a finite set $Y_i$ with infinitely many having no
 terminal point and with infinitely many having no initial point.
  On the infinite product $Y =\prod_{i \in \Z_+} Y_i$ define $S \ = \ \ltimes_{i \in \Z_+} \ S_i$ by
    \begin{equation}\label{eq3.01}
  (y,z) \in S^{\circ} \quad \Longleftrightarrow \quad (y_i,z_i) \in S_i \ \ \text{for} \ \ i = \min \{ j : y_j \not= z_j \},
  \end{equation}
 for $y, z$ distinct points of $Y$.

 On the other hand, we can let $(X_1,R_1) = (Y_0, S_0)$  and inductively for $i \in \N$ define
 $(X_{i+1},R_{i+1}) = (X_{i},R_{i}) \ltimes (Y_{i},S_{i})$  with $f_{i} : X_{i+1} \to X_{i}$
  the first coordinate projection. Let $g_1 : X_1 \to Y_0$ be the identity and for $i \in \N$ let
  $g_{i+1} : X_{i+1} \to Y_i$ be the second coordinate projection.

  It is clear that $\{ (X_i,R_i,f_i) \}$ an inverse system of topological tournaments which is a lexicographic inverse system of bricks
  Let $(X,R)$ be the limit. Thus, by Theorem \ref{theoinv03} $(X,R)$ is
  a balanced tournament on a Cantor set. Furthermore, the following is easy to check.

 \begin{theo}\label{theo3.01}
 If we define $q: X \to Y$ by
  \begin{equation}\label{eq3.01a}
  q(x)_{i} \ = \ g_{i+1}(\pi_{i+1}(x)), \ \ \text{for} \ \ i \in \Z_+
  \end{equation}
  then $q$ is a homeomorphism from $X$ onto $Y$ which is a tournament isomorphism from $R$ to $S$.

  Thus, the relation $S$ is a balanced topological tournament on the Cantor set $Y$.
 \end{theo}
%
%
  \vspace{1cm}

\section{ \textbf{Connectedness and Compactness}}\vspace{.5cm}

A \emph{ linear order}\index{linear order} is exactly a transitive tournament.
When the space is connected, a topological tournament is necessarily a linear order.

\begin{theo}\label{theo2.01} Let $R$ be a topological tournament on a space $X$.
\begin{itemize}
\item[(a)] If $A$ is a connected subset of $X$ and $x \in X \setminus A$, then either $A \subset R^{\circ}(x)$ or $A \subset R^{\circ -1}(x)$.

\item[(b)] If $X$ is connected, then for all $x \in X$ the sets $R(x)$ and $R^{-1}(x)$ are connected.  Furthermore, $R$ is transitive and so is
a linear order on $X$.
\end{itemize}
\end{theo}

\begin{proof} (a) Since $x \not\in A$, $A$ is the disjoint union of the relatively open subsets $A \cap R^{\circ}(x)$ and $A \cap R^{\circ -1}(x)$.
So if $A$ is connected, one of these is empty.

(b) If $R(x)$ is not connected, then it contains a proper subset $A$ which is clopen in the relative topology on $R(x)$.  Replacing $A$ by its
complement if necessary, we may assume $x \not\in A$. Since $A$ is a closed subset of $R(x)$, it is closed in $X$.
Since $A$ is an open subset of $R^{\circ}(x)$, it is open in $X$. Since $A$ is nonempty, $X$ is not connected. Applying the result to $R^{-1}$ we
see that $R^{-1}(x)$ is connected as well when $X$ is connected.

If $y \in R^{\circ}(x)$, then $R(y)$ is a connected set which meets  $R^{\circ}(x)$.  By anti-symmetry, $x \not\in R(y)$. So (a) implies
that $R(y) \subset R^{\circ}(x)$. Hence, $R$ is transitive and so is a linear order.

\end{proof}

It follows that if $R$ is a topological tournament on $X$, then the restriction of $R$ to any connected subset $A$ of $X$ is a linear order on $A$.

\begin{cor}\label{cor2.02} For a compact space $X$, the set $E = \{ (x,y) : x,y \in A$ with $A$ a connected
subset of $X \}$ is a closed equivalence relation with equivalence classes the components of $X$.
The quotient space $X/E$ is totally disconnected.  Let $\pi : X \to X/E$ be the quotient map.

Assume that $X$ admits a topological tournament $R$.  The relation $R_E = (\pi \times \pi)(R) \subset X/E \times X/E$ is a topological
tournament on $X/E$ with $\pi$ a continuous, surjective
tournament map from $R$ to $R_E$. For every non-trivial component $A$ of $X$ there is an open subset $A^{\circ}$ of
$X$ which is contained in $A$ and with the cardinality
of $A \setminus A^{\circ}$ at most two.

If, in addition, $X$ is metrizable, then for every $\ep > 0$ the set of
components of $X$ with diameter at least $\ep$ is finite and so
the set of non-trivial components of $X$ is countable.

If $R$ is balanced, then $R_E$ is balanced and so $X/E$ has no isolated points. If, in addition, $X$ is metrizable, then $X/E$ is a Cantor set.
\end{cor}

\begin{proof} The equivalence classes of $E$ are clearly the components of $X$. For each
component $A$, the collection of clopen sets which contain $A$ form a base for the neighborhood system of $A$.  Any component which meets a
clopen set is contained in it.  It follows that $E = \bigcap \{ B \times B \cup (X \setminus B) \times (X \setminus B) \}$ where $B$ varies over the
clopen subsets of $X$.  Hence, $E$ is closed and $X/E$ is totally disconnected.

Obviously $R_E \cup R_E^{-1} = X/E$.  By Theorem \ref{theo2.01}(a) if
$A$ and $B$ are distinct components of $X$ then either $A \times B \subset R^{\circ}$ or $B \times A \subset R^{\circ}$.
It follows that $R \cup E = (\pi \times \pi)^{-1}(R_E)$. Thus, $(\pi \times \pi)^{-1}(R_E \cap R_E^{-1}) = E$.  It follows that
$R_E$ is a topological tournament and that $\pi$ is a tournament map.

We apply Theorem \ref{theo1.02} to the surjective map $\pi$.  Assume for $x \in X/E, \pi^{-1}(x)$ is a non-trivial component $A$.

It then follows that
$A^{\circ} = \pi^*(x)$ is an open subset of $X$ which differs from $A$ by at most two points. Thus, the
collection $\{ A^{\circ} \}$ with $A$ varying over the non-trivial
components of $X$ is a pairwise disjoint collection of nonempty open subsets.  If $X$ is metrizable, then it is totally bounded and so for any
$\ep > 0$ for at most finitely many $A$ is it true that $diam A^{\circ} \ge \ep$. Since $A$ is connected and so has no isolated points,
$diam A^{\circ} = diam A$. The metric result also follows directly from Theorem \ref{theo1.02a}.

If  $M$ is the maximum  for the compact linear order $R|A$, then $M$ is a terminal point for $R|A$ and so $R^{\circ}(M) = \pi^{-1}(R_E^{\circ}(x))$.
If $x$ is not right balanced, then $R_E^{\circ}(x)$, and hence
$R^{\circ}(M)$ as well, are clopen sets and so $M$ is not right balanced. With a similar argument when
$x$ is not left balanced, we see that if $R$ is balanced,
then $R_E$ is balanced. If $X$ is metrizable, then then $X_E$ is metrizable, see \cite{K} Theorem 5.20.
Since $X/E$ is totally disconnected, it is a Cantor set when it has no
isolated points.

\end{proof}
\vspace{.5cm}

For a closed relation $R$ on a compact metric space $X$, the map  $\overrightarrow{R} : X  \to 2^X$ defined
by $x \mapsto R(x)$ is upper semicontinuous, where $2^X$ is the compact space of closed subsets of $X$ equipped
with the Hausdorff metric, see, e.g.
\cite{A93} Proposition 7.11.

\begin{theo}\label{theo2.03} If $R$ is a topological tournament on a compact metric space, then the map $\overrightarrow{R}$
is an embedding, i.e. it is a homeomorphism onto its image in $2^X$.  In particular, it is
lower semicontinuous as well as upper semicontinuous. \end{theo}

\begin{proof} If $\{ x_n \}$ is a sequence converging to $x \in X$ and $y \in R^{\circ}(x)$, then since $R^{\circ}$ is open,
eventually $(x_n,y) \in R^{\circ}$.  On the other hand, if $y = x$, then $(x_n,x_n) \in R$. Each sequence converges to $(x,y)$.
It follows from \cite{A93} Exercise 7.4 and Proposition 7.11 that the map $\overrightarrow{R}$
is lower semicontinuous and so is continuous.

If $y \in R^{\circ}(x)$, then $x \in R(x) \setminus R(y)$ by anti-symmetry and so
$R(x) \not= R(y)$.  It follows that the map $\overrightarrow{R}$ is injective and so is a homeomorphism onto its image by compactness.

\end{proof}
\vspace{.5cm}

\begin{theo}\label{theo2.03a} If $R$ is a topological tournament on a Cantor set $X$, then $R$ is regular if and only if it is balanced. \end{theo}

\begin{proof} Clearly, if $R$ admits a terminal point or a initial point, then it is neither regular nor balanced. So we may assume
that for every $x \in X$, $R^{\circ}(x)$ and $R^{\circ -1}(x)$ are nonempty open subsets and since
$x$ is not isolated, it is either left or right balanced.

If $x$ is right balanced but not left balanced then $R(x)$ is a Cantor set while $R^{-1}(x)$ consists of the Cantor set $R^{\circ -1}(x)$ together
with an isolated point $x$.  Hence, $R$ is not regular.  Similarly, if there exists a point which is left balanced but not right balanced.
It follows that if $R$ is regular, then it is balanced.

Finally, if $x$ is balanced, then  $R(x)$ and $R^{-1}(x)$ are Cantor sets and so there is a homeomorphism $h_x : R(x) \to R^{-1}(x)$ with $h_x(x) = x$.
Define $h_x$ on $R^{-1}(x)$ to be $h_x^{-1}$.

\end{proof}

We will see that, in contrast with the finite case, an infinite regular topological tournament need not be arc cyclic.
\vspace{.5cm}

For a topological tournament, $R$ on $X$ and $x \in X$ the set $R(x)$ is clopen if and only if
$\overline{R^{\circ -1}(x)} \not= R^{-1}(x)$, i.e. $x$ is not left balanced.

\begin{theo}\label{theo2.04} Let $R$ be a topological tournament on a compact metric space $X$.
The set of points $x$ which are not left balanced, i.e. for which $ R(x)$ is clopen, is countable. Similarly the set of points which are
not right balanced is countable. If $X$ has no isolated points, then the set of balanced points is
residual, i.e. it is a dense $G_{\d}$ subset of $X$. \end{theo}

\begin{proof} As described in Background \ref{back} (3) a compact metric space
 has only countably many clopen subsets. By Theorem \ref{theo2.03} the map $\overrightarrow{R}$ is injective and so
$\{ x : R(x)$ is clopen$\}$ is countable. The union of this set and the corresponding set for $R^{-1}$ is countable and so if $X$ has no
isolated points, the complement is a dense $G_{\d}$ set by the Baire Category Theorem.

\end{proof}

Without metrizability  this result  may fail.

On the real line $\R$, the linear order $L_{\R} = \{(t,t') : t \le t' \}$ is a transitive topological
tournament. Its restriction to $\{ \pm 1 \} (= \{-1, +1 \})$ is an arc.
By Proposition \ref {proplex01a} $L = L_{\R} \ltimes (L_{\R}|\{ \pm 1 \})$ is a linear order on $\R \times \{ \pm 1\}$.
When we use the associated order topology, instead of the product topology, we obtain  the
\emph{Sorgenfrey Double Arrow}\index{Sorgenfrey Double Arrow}. For every $t \in \R$ let $t+ = (t,+1), t- = (t,-1)$.
Each $L(t+)$ and each $L^{-1}(t-)$ is clopen. It follows that the space is not metrizable, see Background \ref{back} (3).
The first coordinate projection to $\R$ is a continuous surjective tournament mapping from $L$
to $L_{\R}$. The subset $X = L(0+) \cap L^{-1}(1-)$ is compact with no isolated points and the first coordinate projection is a
continuous surjective tournament mapping from the restriction of $L|X$ to the restriction of $L_{\R}|I$ with $I$ the
unit interval in $\R$.

The set of left balanced points and the set of right balanced points are disjoint. Each is dense and the union is
all of $X$. In particular, there are no balanced points.

On the subset $\R \times \{ -1 \}$ the relative topology is not the order topology.  Instead the basis consists of half-open intervals
$(s_-,t_-]$ with $s < t$. The space is non-metrizable and non-compact, but it is separable and with no isolated points.
When we restrict $L$ to this subset, we obtain a topological tournament such that every $L^{-1}(t_-)$ is clopen. That is, there are no
right balanced points. Every point is left balanced.

On the other hand, we do have the following result in the general compact case.

\begin{theo}\label{theo2.05} Let $R$ be a topological tournament on a compact space $X$. If the isolated points are not dense, then
the set of right balanced points is nonempty, in fact, it is dense in the complement of the closure of the set of isolated points.
Similarly, the set of left balanced points is dense in the complement of the closure of the set of isolated points.
\end{theo}

\begin{proof}  Let $U$ be a nonempty open subset of $X$ which contains no isolated points and let $U_1$ be a nonempty open subset with
$\overline{U_1} \subset U$. We show that there exists $x \in U$ such that $R(x) = \overline{R^{\circ}(x)}$
or, equivalently, $R^{-1}(x)$ is not clopen.

We may assume that $G = \{ x \in U_1:  R^{-1}(x)$ is clopen $ \}$ is dense in $U_1$.  For if not, the required $x$ exists in $U_1$.

Choose $x_1 \in G $. Assume we have constructed inductively $x_1, x_2, \dots, x_n $ distinct points in $G$ such that for each $i$ with $1 < i \le n$,
$x_i \in \bigcap_{j < i} R^{-1}(x_j)$. Hence, $U_1 \cap \bigcap_{j \le n} R^{-1}(x_j)$ is an open subset of $U_1$ which
contains $x_n$. Since $U_1$ contains no isolated
points, there exists $x_{n+1} \in G \cap [(U_1 \cap \bigcap_{j \le n} R^{-1}(x_j)) \setminus \{ x_1, \dots, x_n \}]$.

Let $x$ be a limit point of the sequence $\{ x_n \}$ so that $x \in U \supset \overline{U_1}$.
By excluding one $x_i$ if necessary, we may assume $x \not= x_n$ for any $n$.
Since $x_j \in R^{-1}(x_n)$ for all $j > n$, it follows that $x \in R^{-1}(x_n)$ and so $x_n \in R^{\circ}(x)$. Thus, $x \in \overline{R^{\circ}(x)}$.
Thus, $x$ is right balanced.

\end{proof}

 \vspace{.5cm}

   We conclude this section with a useful tool.

  \begin{df}\label{def2.06} Let $(X,R)$ be a topological tournament and let $F = \{x_1, x_2, \dots, x_n \}$ be a list of distinct points in $X$.
  A \emph{thickening}\index{thickening} of $F$ is a list $U_F = \{ U_1, U_2, \dots, U_n \}$ of open subsets of $X$ such that
  \begin{itemize}
  \item For $i = 1,\dots,n, \ \ x_i \in U_i$.

   \item For $i,j = 1,\dots,n,$ with $i \not= j$ and $z_i \in U_i, z_j \in U_j$, we have $z_i \ha z_j$ if
   $x_i \ha x_j$.
   \end{itemize}
   In particular, the open sets in $U_F$ are pairwise disjoint.

   We call $U_F$ a \emph{clopen thickening}\index{thickening!clopen} when every $U_i$ is clopen.
   \end{df}

 \vspace{.5cm}

  If $(X,R)$ is a compact topological tournament and $F = \{x_1, x_2, \dots, x_n \}$ is any list of distinct points in $X$,
 then there exists a thickening for $F$. In fact, the thickening can be chosen uniformly.

 \begin{theo}\label{theo2.07} If $(X,R)$ is a compact topological tournament, then
for any neighborhood $V_1$ of the diagonal $1_X$
 there exists a neighborhood $V$ of the diagonal such that
 whenever $F$ is a finite subset such  that $(x_i, x_j) \not\in V_1$ when $i \not= j$, then
 $\{ V(x_1), \dots, V(x_n) \}$ is a thickening of $F$.

 When $X$ is metrizable with metric $d$, then for every $\ep >0$, there
exists $\d > 0$ such that $d(x_i, x_j) \ge \ep$ when $i \not= j$ implies that $\{ V_{\d}(x_1), \dots, V_{\d}(x_n) \}$ is a thickening of $F$.

 If $X$ is totally disconnected, then we may choose $V$ to be a clopen equivalence relation and so obtain a clopen thickening. \end{theo}

 \begin{proof} We use induction on $n$.  The result for $n = 1$ is vacuous.  We may use any neighborhood of $1_X$.

Now assume that $V_2 \subset V_1$ is a neighborhood of the diagonal
such that $\{ V_2(x_1), \dots, V_2(x_{n-1}) \}$  is a thickening of $\{ x_1, \dots, x_{n-1} \}$
whenever $(x_i,x_j) \not\in V_1$ for $i \not= j \le n-1$. Note that the set of diagonal neighborhoods $V \subset V_2$ is directed with intersection
the diagonal.  Suppose there existed $F = \{x_1, \dots, x_n \}$ such that no $V$ exists.
Then for any such $V$ because $\{ V(x_1), \dots, V(x_{n-1}) \}$  is a thickening of $\{ x_1, \dots, x_{n-1} \}$ there must exist
$x_1(V),\dots, x_n(V)$ such that $(x_i,x_i(V)) \in V$ for $i = 1, \dots, n$, but for some $j_V < n$,
$x_{j_V}(V) \ha x_n(V)$  while $x_{j(V)} \lha x_n$  or vice-versa. Assume the
first. By restricting to a cofinal subset
we may assume that for some fixed $j < n$, $j_V = j$ for all $V$. Each net $\{ x_i(V) \}$ has limit $x_i$ since $(x_i,x_i(V)) \in V$. By assumption,
$x_{j} \lha x_n $. But  $x_{j}(V) \ha x_n(V)$ implies, in the limit, $x_{j} \ \underline{\ha} \ x$ violating
anti-symmetry. The argument for the reverse assumption is similar.

When $X$ is totally disconnected, the clopen equivalence relations form a neighborhood base for the diagonal and so we may choose $V$ to be such.

\end{proof}

 \vspace{1cm}

  \section{ \textbf{Group Tournaments}}\vspace{.5cm}

  For a subset $A$ of a group $G$ we let $A^{-1} = \{ x^{-1} : x \in A \}$. We let $e$ denote the identity element.

\begin{df}\label{def3.02} For a group $G$, a \emph{game subset}\index{game subset} $A$ for $G$ is a subset such that
\begin{equation}\label{eq3.03}
A \cap A^{-1} \ = \ \{ e \}, \quad \text{and} \quad A \cup A^{-1} \ = \ G.
\end{equation}
We let $A^{\circ} = A \setminus \{ e \}$.\end{df}

If $A$ is a game subset, then $A^{-1}$ is the reverse game subset. If $G$ is a topological group and $A$ is closed,
then $A^{\circ} = X \setminus A^{-1}$ is open.

Clearly, a group admits a game subset if and only if it has no elements of order two. In the finite case this says that $G$ has odd order.

If $h : G_2 \to G_1$ is a group homomorphism and $A_1, A_2$ are game subsets for $G_1$ and $G_2$, respectively, then $h$ maps $A_2$ to $A_1$
when $h(A_2) \subset A_1$ or, equivalently, $A_2 \subset h^{-1}(A_1)$.  Since $h$ maps $A_2^{-1}$ to $A_1^{-1}$ it follows that
 \begin{equation}\label{eq3.04}
A_2^{\circ} \  \supset \ h^{-1}(A_1^{\circ}). \hspace{2cm}
\end{equation}
Thus, the preimage of $A_1^{\circ}$ is contained in $A_2^{\circ}$ and $A_2$ is contained in the union of the preimage of $A_1^{\circ}$  and the
kernel $h^{-1}(e)$ of $h$.

If $H$ is a subgroup of $G$, then $H \cap A$ is a game subset for $H$ and the inclusion maps $H \cap A$ to $A$.

If $A$ is a game subset for a group $G$ then the associated tournament $\widehat{A}$ is defined by
 \begin{align}\label{eq3.05} \begin{split}
\widehat{A} \ = \ \{ (x,y) : x^{-1}y &\in A \} \quad \text{so that} \ \ \widehat{A}^{-1} \ = \ \widehat{A^{-1}},\\
 \text{and so} \ \  \widehat{A}^{\circ} \  &= \ \{ (x,y) : x^{-1}y \in A^{\circ} \}.
 \end{split}\end{align}

 Thus, $A = \widehat{A}(e)$ and $A^{\circ} = \widehat{A}^{\circ}(e)$. If $h : G_2 \to G_1$ is a group homomorphism, then
 $h$ maps the game subset $A_2$ to the game subset $A_1$ if and only if it is a tournament map from $\widehat{A_2}$ to  $\widehat{A_1}$.

 If $G$ is a topological group, then $A$ is a closed game subset if and only if $\widehat{A}$ is a topological tournament.

 For the results on the finite case of group games, see, e.g. \cite{A20}.

 The tournament $\widehat{A}$ is regular.  Define for $x,y \in G$
  \begin{align}\label{eq3.06} \begin{split}
 h_x(y) \ = \ &x y^{-1} x \quad \text{so that} \ \ h_x(x) \ = \ x, \\
 \text{and} \quad &x^{-1}h_x(y) \ = \ y^{-1} x, \\
 \text{and} \quad &h_x \circ h_x = 1_G.
 \end{split}\end{align}
Thus, $(x,y) \mapsto (x,xy^{-1}x)$ maps $\widehat{A}$ to $\widehat{A}^{-1}$ and so $h_x(\widehat{A}(x)) = \widehat{A}^{-1}(x)$.

 The tournament $\widehat{A}$ on the group $G$ is \emph{homogeneous}\index{tournament!homogeneous}.
 For $x \in G$, the left translation map $\ell_x$, defined
 by $\ell_x(y) = xy$, is an automorphism of $\widehat{A}$. That is, $\ell_x$ is a bijection on $G$ mapping $\widehat{A}$ to itself.

 For a topological group, the maps $h_x$ and $\ell_x$ are homeomorphisms.

\begin{prop}\label{prop3.03} Let $h : G_2 \to G_1$ be a group homomorphism and $A_1$ be a game subset for $G_1$. Then $A_2 \subset G_2$
 is a game subset for $G_2$ which is mapped to $A_1$ by $h$ if and only if $A_2$ is the union of the disjoint sets $h^{-1}(A_1^{\circ})$ and
 $B$ with $B$ a game subset for the kernel of $h$, $H \ = \ h^{-1}(e)$.

 If, in addition, $h$ is surjective, then there exists a retraction $p : G_2 \to H$ such that
 the product map $h \times p : G_2 \to G_1 \times H$ is a bijection mapping $\widehat{A_2}$ isomorphically
 onto the lexicographic product $\widehat{A_1} \ltimes \widehat{B}$.

 If $h$ is a continuous group homomorphism between topological groups with non-trivial kernel and $A_1$ is closed, then
 $A_2 = B \cup h^{-1}(A_1^{\circ})$ is closed if and only if $B$ is closed and, in addition, the kernel $H$ is a clopen subgroup.\end{prop}

 \begin{proof} It is easy to check that if $B$ is a game subset for the kernel of $h$, then $A_2 = h^{-1}(A_1^{\circ}) \cup B$ satisfies
 the conditions of (\ref{eq3.03}) and is mapped by $h$ to $A_1$.

 Conversely, if $A_2$ is a game subset for $G_2$, then $B = A_2 \cap h^{-1}(e)$ is a game subset for the kernel and if $h$ maps $A_2$ to
 $A_1$, then (\ref{eq3.04}) implies that $A_2$ contains the game subset $h^{-1}(A_1^{\circ}) \cup B$.  Clearly, if one game subset
 for $G_2$ includes another such, then the two are equal.

If $h$ is surjective, we can define a (not necessarily continuous) map $j : G_1 \to G_2$ such that $h \circ j = 1_{G_1}$ with $j(e_1) = e_2$.
Define $p(x) = j(h(x))^{-1}x$ so that $p$ maps $G_2$ into $H$ with $p = 1_H$ on $H$. Since $j(h(x))p(x) = x$, the inverse map
to $h \times p$ is given by $(z,b) \mapsto j(z)b$. So
 $h \times p : G_2 \to G_1 \times B$ is a bijection.

If $h(x) \not= h(y)$, then $x^{-1}y \in A_2$ if and only if $h(x^{-1}y) = h(x)^{-1}h(y) \in A_1$,
i.e. $(x,y) \in \widehat{A_2}^{\circ}$ if and only if $(h(x),h(y)) \in \widehat{A_1}^{\circ}$.

If $h(x) = h(y)$,
then $j(h(x)) = j(h(y))$ and so $x^{-1}y = p(x)^{-1}p(y)$. Hence, $(x,y) \in \widehat{A_2}$ if and only if $(p(x),p(y)) \in \widehat{B}$.

It follows that $h \times p$ maps $\widehat{A_2}$ isomorphically
 onto $\widehat{A_1} \ltimes \widehat{B}$.

In the topological case, the kernel $h^{-1}(e)$ is a closed subgroup because of our standing assumption that all spaces are Hausdorff.
If it is not open then there exists a net $\{ a_k \}$ in $G_2 \setminus h^{-1}(e)$ which converges to a point $x$ in the kernel.
Replacing $a_k$ by $a_k^{-1}$ if necessary and by going to a subnet we may assume that $a_k \in h^{-1}(A_1^{\circ})$ for all $k$.
If $x, y$ lie in the kernel with $x$ the limit point of the net, then $\{ yx^{-1}a_i \}$ is a net in $h^{-1}(A_1^{\circ})$ which
converges to $y$. Thus, all of $h^{-1}(e)$ is contained in the closure of  $h^{-1}(A_1^{\circ})$ which is contained in $A_2$ when
the latter is closed. If the kernel is non-trivial, then $B^{\circ -1}$ is a nonempty subset of the kernel which is disjoint
from $A_2$.  The contradiction shows that the kernel must be clopen.

\end{proof}

\begin{prop}\label{propcycle01} For $x \in G$ with $G$ a compact topological group, the set
\begin{equation}\label{eqcycle01}
\om(x) \ = \ \bigcap_{n \in \N} \ \overline{\{ x^i : i \ge n \}}
\end{equation}
is a nonempty closed subgroup of $G$. \end{prop}

\begin{proof} Since $\om(x)$ is the intersection of a decreasing sequence of non-empty compacta, it is nonempty and compact.
It consists of the set of limit points of the sequence $\{ x^i : i \in \N \}$. So if $z \in \om(x)$, then $x^{-1}z$ is also a
limit point of the sequence and so lies in $\om(x)$. Thus, $\{ y : y^{-1}z \in \om(x) \}$ is closed and contains $x^i$ for
all $i \in \N$. In particular, it contains $\om(x)$.  That is, $\om(x)^{-1} \om(x) \subset \om(x)$ and so $\om(x)$ is a subgroup.

\end{proof}
\vspace{.5cm}

\begin{theo}\label{theocycle02} If $A$ is a closed game subset on a compact topological group $G$, then $(G,\widehat{A})$ is
an arc cyclic tournament. \end{theo}

\begin{proof} By homogeneity it suffices to consider arcs with $x = e$ and so $y \in A^{\circ}$. The arc $(e,y)$ is contained in a $3-$cycle
if and only if $yA$ meets $A^{\circ -1}$.  Assume now that $y \in A$ with  $yA$ is disjoint from $A^{\circ -1}$ and so
$yA \subset A$. Inductively,
for all $i \in \N, \ y^i \in y^i A \subset y^{i-1} A$. In particular, the sequence $\{ y^i \}$ is contained in $yA$ and so $\om(y)$ is contained
in the closed set $yA $. However, Proposition \ref{propcycle01} implies that $\om(y)$ is a subgroup and this yields
$e \in yA$ or, equivalently, $y \in A^{-1}$. Since $y \in A \cap A^{-1}$ we have $y = e$.  Thus, if $y \in A^{\circ}$ it must happen that
$yA$ meets $(A^{-1})^{\circ}$.

\end{proof} \vspace{.5cm}

The following is a topological version of the proof of \cite{M} Theorem 3, which in turn is an extension of \cite{HM} Theorem 7.

 \begin{cor}\label{corcycle03} If $A$ is a closed game subset on an infinite compact topological group $G$, then for every $n \ge 3$ each point
 of $G$ is contained in an $n-$cycle.  \end{cor}

\begin{proof} The result for $n = 3$ follows from Theorem \ref{theocycle02}. Now assume that $C = \{x_1, \dots, x_n \}$ is an $n-$cycle with $n \ge 3$.
We may assume, by multiplying by $x_1^{-1}$ if necessary, that $x_1 = e$. We will construct an $n+1-$cycle through $x_1$. \vspace{.25cm}

Case 1 Assume there exists $x \in G \setminus C$ such that $\widehat{A}(x)$ and $\widehat{A}^{-1}(x)$ both meet $C$.
By renumbering we may assume
$x_1 \ha x$. Let $k = \max \{ i : x_j \ha x$
for all $j \le i \}$. By assumption, $k < n$ and by definition
$x \ha  x_{k+1} $. Hence,
$\{ x_1, \dots, x_k, x, x_{k+1}, \dots, x_n \}$ is an $n+1-$cycle which contains all the points of $C$ and
so includes the point previously labelled $x_1$. \vspace{.25cm}

Case 2: Assume instead that with $Z_+ = \{ x : C \subset \widehat{A}(x)\}$ and $Z_- = \{ x : C \subset \widehat{A}^{-1}(x)\}$ we have
$Z_+ \cup Z_- = G \setminus C$. Notice that in any case $Z_+ \cup Z_-$ is disjoint from $C$ since the points of $C$ lie on a cycle.

If $Z_-$ were empty, then for every point $x_i \in C$, we would have $\widehat{A}(x_i) = x_i A \subset C$. Thus would imply that $A$ is finite
and so $G = A \cup A^{-1}$ is finite.  Similarly, $Z_+$ is nonempty.

Choose $z_1 \in Z_+, z_2 \in Z_-$. We may assume that $z_2 \ha z_1$. If instead $z_1 \ha z_2$, then
Corollary  \ref{theocycle02} there exits $z_3 \in G$ such that $\{ z_1, z_2, z_3 \}$ is a $3-$cycle. Because $z_2 \ha z_3 $ it cannot
happen that $z_3 \in C$. If $z_3 \in Z_+$ then replace $z_1$ by $z_3$. If $z_3 \in Z_-$, then replace $z_2$ by $z_3$.

Assuming that $z_2 \ha z_1$ we obtain $\{ x_1, z_2, z_1,x_3, \dots x_n \}$ (omitting $x_2$) an $n+1-$cycle containing $x_1$.

\end{proof}

 \begin{theo}\label{theocycle04} If $A$ is a closed game subset on an infinite compact topological group $G$, then $\widehat{A}$ is a balanced
 topological tournament. \end{theo}

 \begin{proof} By homogeneity it suffices to show that $e$ is a balanced point. If it were not then either $A^{\circ}$ or
 $A^{\circ -1}$ would be clopen and so both would be clopen since the map $x \mapsto x^{-1}$ is a homeomorphism.
  In that case $e$ is an isolated point. By homogeneity all the points
 of $X$ would be isolated and so, by compactness, $X$ would be finite.

 \end{proof}

 Now let $\{ G_i : i \in \N \} $ be a sequence of finite groups of odd order with $f_i : G_{i+1} \to G_{i}$
 surjective group homomorphisms each with
 non-trivial kernel $H_{i+1}$ so that the sequence of orders $\{ |G_i| \}$ is strictly increasing. Let $H_1 = G_1$.
 Choose $A_1 = B_1$ a game subset for $G_1 = H_1$
 and $B_{i+1}$ a game subset
 for the kernel $H_{i+1}$. Inductively, let $A_{i+1} \ = \ B_{i+1} \cup (f_i)^{-1}(A_i^{\circ})$ which is a game subset
 for $G_{i+1}$ mapped onto $A_i$ by $f_i$.

  \begin{theo}\label{theo3.04} The sequence $\{ (G_i,\widehat{A_i},f_i) : i \in \N \}$ is a surjective inverse
  system of topological tournaments with limit
 $(G,R)$ a compact, topological
 tournament and $G$ a Cantor set.

 The space $G$ is a closed subgroup of the product topological group $\prod_{i \in \N} \ G_i$ with closed game subset
 \begin{equation}\label{eq3.07}
 A \ = \ \bigcap_{i \in \N} \ \pi_{i}^{-1}(A_i) \ = \ \overleftarrow{Lim} \{ A_i \}
 \end{equation}
 such that $R = \widehat{A}$.
 \end{theo}

 \begin{proof} That $R$ is a topological tournament on the inverse limit $G$ follows from Theorem \ref{theo3.01}.

 It is clear that $G$ is a closed subgroup of the product group.  It is easy to check that the closed subset $A$ is
 a game subset for $G$ and that $R = \widehat{A}$.

 \end{proof}

\begin{ex}\label{exes02}The $3-$adic integers.\end{ex} Consider the $3-$adic integers, with $ \Z/3^i \Z$ and the projection $f_i$ reduction mod $3^i$. The
kernel of each $f_i$ is isomorphic to $\Z/3 \Z = \{ 0, 1, 2 \}$. We can identify
 $ \Z/3^i \Z$ with the product $ \{ 0, 1, 2 \}^{ \{1,\dots,i  \} }$ with addition of two sequences pointwise (mod 3) but with carrying to the right.
 The projection $f_i : \Z/3^{i+1} \Z \to \Z/3^i \Z$ is a surjective group homomorphism.  So $\{ (\Z/3^i \Z,f_i) \}$ is
 an inverse system of finite groups. As an additive topological group, the inverse limit is identified with $ \{ 0. 1, 2 \}^{\N}$ with addition of
 two sequences pointwise (mod 3) but with carrying to the right. We label this, the group of $3-$adic integers by $\Z[3]$. \index{$\Z[3]$}

  The identity element $e$ has $e_i = 0$ for all $i$.

 An example of a closed game subset $A$, let $A^{\circ} = \{ y \in G \setminus e : y_j = 1 $ for $ j = \min \{ k : y_k \not= 0 \} \}$.
For each $ \Z/3^i \Z$ we let $A_i^{\circ} = \{ y \in  \Z/3^i \Z \setminus e : y_j = 1 $ for $ j = \min \{ k : y_k \not= 0 \} \}$.

  Equipped with this game subset we will
  refer to  the tournament $(\Z[3], \widehat{A})$  as the \emph{standard $3-$adic example} \index{standard $3-$adic example}.
  It is the inverse limit of the system $\{ ( \Z/3^i \Z,\widehat{A_i},f_i) \}$.

  \vspace{.5cm}

 \begin{lem}\label{lem3.04a} If $w$ is a homeomorphism on a Cantor set $X$ which induces a free $\Z/2\Z$ action, i.e. $w \circ w = 1_X$ and
 $w(x) \not= x$ for all $x \in X$, then there exists a clopen subset $A$ of $X$ such that $X$ is the disjoint union of $A$ and $w(A)$. \end{lem}

 \begin{proof} We may choose a $w$ invariant ultra-metric $u$ on $X$, see Background \ref{back} (9).

Because $u$ is an ultra-metric, the relation $V_{\ep} = \{ (x,y) : u(x,y) < \ep \}$ is a clopen equivalence relation for every $\ep > 0$.
Because $u$ is $w$ invariant, we have $h(V_{\ep}(x)) = V_{\ep}(w(x))$.

Choose $\ep$ so that $0  <  \ep  <  \min_{x \in X} \ u(x,w(x))$.  The equivalence classes $\{ V_{\ep}(x) : x \in X \}$ form a finite
cover of $X$ by clopen sets.  By choice of $\ep, \ V_{\ep}(w(x))$ is disjoint from $V_{\ep}(x)$. So we can partition the cover by the
collection of pairs $\{ \{ V_{\ep}(x), V_{\ep}((w(x)) \} \}$. Choose one member from each pair and take the union to define $A$. Observe that
there are $2^n$ choices leading to distinct sets $A$ with $2n = |\{ V_{\ep}(x) \}|$.  By shrinking $\ep$ we can increase the number of
alternative sets $A$.

\end{proof} \vspace{.5cm}

\begin{theo}\label{theo3.04b} Let $G$ be a topological group with the underlying space a Cantor set. There exists
a closed game subset $A$ for $G$ if and only if $G$ contains no
elements of order $2$. \end{theo}

\begin{proof} Clearly if $G$ contains an element of order $2$, then there is no game subset.  Now assume there are no such elements
so that $w(x) = x^{-1}$ defines a homeomorphism of $X$ which induces a free $\Z/2\Z$ action except at the point $e$ where $w(e) = e$.

Choose $\{U_i : i \in \N \}$ a decreasing sequence of clopen neighborhoods of $e$ with intersection $e$.  For example, with $u$ the
ultrametric of the previous proof we may use $U_i = V_{1/i}(e)$. Replacing $U_i$ by $U_i \cap U_i^{-1}$ for all $i$, we may assume that
$(U_i)^{-1} = U_i$ for all $i$. Let $U_0 = G$.  By renumbering we may assume that the sequence $\{ U_0, U_1, \dots \}$ is strictly decreasing
so that $\{ X_i = U_{i-1} \setminus U_{i} : i \in \N \}$ is a sequence of nonempty clopen subsets which partition $G \setminus \{ e \}$ and
each of which is $w$ invariant.

For each $i$ use Lemma \ref{lem3.04a} to choose $A_i$ clopen in $X_i$ with $\{ A_i, w(A_i) \}$ a partition of $X_i$.
Let $A^{\circ} \ = \ \bigcup_i \ A_i$. This is an open subset of $X$ with $A = \overline{A^{\circ}} = A^{\circ} \cup \{ e \}$. Thus,
$A$ is a closed game subset for $X$.

\end{proof}\vspace{.5cm}

\begin{ex}The $2-$adic integers.\label{exes03}\end{ex} Consider the $2-$adic integers, with $ \Z/2^i \Z$ and the projection $f_i$ reduction mod $2^i$. The
kernel of each $f_i$ is isomorphic to $\Z/2 \Z = \{ 0, 1 \}$. We can identify
 $ \Z/2^i \Z$ with the product $ \{ 0, 1 \}^{ \{1,\dots,i  \} }$ with addition of two sequences pointwise (mod 2) but with carrying to the right.
 The projection $f_i : \Z/2^{i+1} \Z \to \Z/2^i \Z$ is a surjective group homomorphism.  So $\{ (\Z/2^i \Z,f_i) \}$ is
 an inverse system of finite groups. As an additive topological group, the inverse limit is identified with $ \{ 0. 1 \}^{\N}$ with addition of
 two sequences pointwise (mod 2) but with carrying to the right. We label this, the group of $2-$adic integers by $\Z[2]$. \index{$\Z[2]$}
   Note that since $ \Z/2^i \Z$ has even order it does not admit a game subset.

   With $\bar 0 = 1, \bar 1 = 0$ we define $\bar y$ for $y$ in  $\Z[2]$ by $(\bar y)_i = \overline{y_i}$.
   With $\1 = 1000\dots$, it is clear that
   $y + \bar y + \1 = \0 $ where $\0 = 0000\dots$ is  the zero element of the additive group. So if $y = 0^{i-1} 1 z$, then $-y = 0^{i-1} 1 \bar z$.
   Define $A_i = \{ 0^{i-1} 1 0 z : z \in \Z[2] \}$. This is a clopen
   subset with $-A_i = \{ 0^{i-1} 1 1 \bar z : z \in Y \}$. So $A = \{ e \} \cup (\bigcup_{i} A_i)$ is a game subset.

    Equipped with this game subset we will
  refer to  the tournament $(\Z[2], \widehat{A})$   as the \emph{standard $2-$adic example}\index{standard $2-$adic example}.

   The map $m_k$, multiplication by $k$ on $\Z[2]$, for any $k \in \N$  odd, is an automorphism of the additive group
   $Y$. 

   It follows that if $h: G \to H$ is a surjective group homomorphism with $H$ a finite group, then the order of $H$ is a power of $2$.
   For if not, since it is necessarily abelian, it has a quotient group of odd order and so we may assume that $H$ has odd order $k$.
   If $x \not= 0$ in $H$, then there exists $y_1 \in G$ with $h(y_1) = x$ and since $m_k$ is an automorphism of $G$ there exists
   $y_2 \in Y$ with $k y_2 = y_1$.  then $k h(y_2) = h(y_1) = x$. On the other hand, $k z = 0$ for all $z \in H$.

   While every topological group on a Cantor set is an inverse limit of a sequence of finite quotient groups, the $2-$adics provides an
   example where no game subset can be obtained as a limit of game subsets from a sequence of quotient groups.

   \vspace{1cm}

  \section{ \textbf{Cycle Points}}\vspace{.5cm}

  A tournament $R$ is transitive, and so is a linear order, if and
  only if contains no $3-$cycle.

  \begin{df}\label{def3.05} Let $(X,R)$ be a topological tournament.

  We say that $R$ is
  \emph{nowhere locally transitive} when there does not exist a nonempty open subset $U$ of
  $X$ such that the restriction $R|U$ is transitive, or, equivalently, when every
  nonempty open subset of $X$ contains a $3-$cycle.

  We call $x \in X$ a \emph{cycle point} when
 every open set containing $x$ contains a $3-$cycle which includes $x$. \end{df} \vspace{.5cm}

  Clearly a
 cycle point is balanced.

 \begin{lem}\label{lemcycle05} Let $(X,R)$ be a compact topological tournament.

If $V_1$ is a neighborhood of the diagonal $1_X$, then there exists a neighborhood of the diagonal
 $V$ such that if $\{ x, y, z\}$ is a $3-$cycle with $(x,y) \in V$, then $(x,z), (y,z) \in V_1$.
  If $X$ is metrizable with metric $d$ and $\ep > 0$ there exists $\d > 0$ such that if $\{ x, y, z\}$ is a $3-$cycle with $d(x,y) < \d$,
 then $d(x,z)  < \ep$ and $d(y,z)  < \ep$

 Assume that $(x_k, y_k,z_k)$ is a net in $X \times X \times X$
 such that for each $k$, $\{ x_k,y_k,z_k \}$ is a $3-$cycle. If $\{ y_k \}$ and $\{x_k \}$ both converge to a point $x$,
 then $\{ z_k \}$ converges to $x$ as well.
 \end{lem}

 \begin{proof} Suppose instead that
   for some $V_1 > 0$,  we could construct for each $V$ a $3-$cycles $\{x_V, y_V,z_V \}$ with
 $(x_V,y_V) \in V$  but with $(x_V,z_V) \not\in V_1$.  The collection of neighborhoods $V$ is directed by
 inclusion with intersection the diagonal $1_X$. So we can regard $\{ (x_V,y_V,z_V) \}$ as a net indexed by $V$. A limit
 point $(x,y,z)$ would satisfy $x = y$ but $x \not= z$.  Since $(y,z), (z,x) \in R$, this would violate anti-symmetry.
 In the metric case, the neighborhoods $V_{\ep} = \{ d(x,y) < \ep \}$ generate the neighborhoods of the diagonal.

 For the net $\{ (x_k, y_k, z_k) \}$, eventually $\{ (x_k,y_k) \}$ enters $V$ and so eventually $\{ (x_k, z_k) \}$ enters $V_1$.
 Since $\{ x_k \}$ converges to $x$, $\{ z_k \}$ does as well.

 \end{proof} \vspace{.5cm}

 \begin{theo}\label{theocycle06} Let $(X,R)$ be a compact topological tournament. If  $x$ is a non-isolated point of $X$ and it has an
 arc cyclic neighborhood,   then it is a cycle point and so is balanced. So if $(X,R)$ is
 locally arc cyclic and $X$ has no isolated points, then every point is a cycle point and $(X,R)$ is balanced.
 \end{theo}

 \begin{proof} If $x$ is non-isolated, then it is either left or right balanced. If $x \in X$ is right balanced,
 there exists a net $\{ y_k \}$ in $R^{\circ}(x)$ which
 converges to $x$ and we may assume the net lies in an arc cyclic neighborhood $U$.
  Because $U$ is an arc cyclic subset, we can choose for each $k$, a point $z_k \in U$ such that $\{ x, y_k, z_k \}$ is a
 $3-$cycle. By Lemma \ref{lemcycle05}, $\{ z_k \}$ converges to $x$. So if $U_1$ is any neighborhood of $x$, eventually, the
 cycle $\{ x, y_k, z_k \}$ is contained in $U_1$. Thus, $x$ is a cycle point. Similarly, if $x$ is left balanced, it is a cycle point.
 Since a cycle point is balanced, it follows that $(X,R)$ is balanced when it is locally arc cyclic and there are no isolated points.

 \end{proof}\vspace{.5cm}

\begin{cor} \label{corcycle06a} If $A$ is a closed game subset for an infinite compact group $X$, then every point of $X$ is a cycle point.\end{cor}

\begin{proof} Immediate from Theorem \ref{theocycle02}, Theorem \ref{theocycle04} and Theorem \ref{theocycle06}.

 \end{proof}\vspace{.5cm}

We have the following sharpening of Theorem \ref{theoinv03}.

\begin{theo}\label{theo3.07a} Assume that $\{(X_i,R_i,f_i) \}$ is a lexicographic inverse system of topological tournaments with limit tournament $(X,R)$.
If for infinitely many $i \in \N$  the fiber $(Y_{iz},S_{iz})$ is ip cyclic for each $z$ an isolated point of $X_i$, then every point $x$ of
the  subset $IS$ of $X$ is a cycle point.

In particular, if  $X_1$ is finite, and every $(Y_{iz},S_{iz})$ is finite and point cyclic, then every point of $X = IS$ is
a cycle point. \end{theo}

\begin{proof} Let $x \in IS$.  Fix $i$ arbitrarily large so that the fibers $(Y_{iz},S_{iz})$ are ip cyclic for the isolated points
$z$ of $X_i$, and let $z = \pi_i(x)$.
The point $\pi_{i+1}(x) = (z,y)$ with $y$ an isolated point in $Y_{iz}$. Since
$(Y_{iz},S_{iz})$ is  ip cyclic, there exist $y', y'' \in Y_{iz}$ so
that $\{ y', y, y'' \}$ is a $3-$cycle for $(Y_{iz},S_{iz})$.
There exist $x', x'' \in X$ with $\pi_{i+1}(x') = (z,y'), \pi_{i+1}(x'') = (z,y'')$ and so $\{ x', x, x'' \}$ is a $3-$cycle for
$(X,R)$. Furthermore, $\pi_i(x'') = \pi_i(x') = z = \pi_i(x)$ and so
$\pi_j(x) = \pi_j(x') = \pi_j(x'')$ for all $j \le i$. As $i$ was arbitrarily large,
$x'$ and $x''$ are arbitrarily close to $x$ and so $x$ is
a cycle point.

\end{proof}

Recall that if $\{(X_i,R_i,f_i) \}$ is a lexicographic inverse system of bricks, then $IS$ is a dense $G_{\d}$ subset of $X$.
\vspace{.5cm}

\begin{theo}\label{theo3.08} Let $h : (X_2,R_2) \to (X_1,R_1)$  be a continuous, surjective tournament map of compact tournaments.
 If $y \in X_1$ is a cycle point, then
$h^{-1}(y)$ is a singleton $\{ x \}$ and $x \in X_2$ is a cycle point.

If every point of $X_1$ is a cycle point, then $h$ is a homeomorphism mapping $R_2$ isomorphically onto $R_1$. \end{theo}

\begin{proof} Assume that $\{ (y'_k,y''_k) \in X_1 \times X_1 \}$ is a net converging to $(y,y)$ with $\{ y'_k, y, y''_k \}$ a $3-$cycle for all $k$.
Since $h$ is surjective, we can choose $(x'_k,x''_k) \in  X_2 \times X_2$ with $h(x'_k) = y'_k, h(x''_k) = y''_k$.  Since $h$ is a tournament map,
$x''_k \ha x'_k $.  By Theorem \ref{theo1.02} and Addendum \ref{add1.03} $h^{-1}(y)$ has a terminal point $M$ and a initial point $m$ and
every convergent subnet of $\{ x'_k \}$ converges to $m$ and so, by compactness, $\{x'_k \}$ converges to $m$.  Similarly, $\{ x''_k \}$ converges to $M$.
Since $x''_k \ha x'_k $ it follows that $M \ \underline{\ha} \ m$. But $m$ is a initial point for $h^{-1}(x)$ and so $m \ \underline{\ha} \ M$.
It follows from
anti-symmetry that $m = M$ and so $h^{-1}(y)$ is a singleton.

If $h^{-1}(y) = \{ x \}$ and $U$ is an open set containing $x$, then, by compactness, there exists an open set $U_1$ containing $y$ with
$h^{-1}(U_1) \subset U$. Any $3-$cyle containing $y$ in $U_1$ lifts to a $3-$cycle in $U$ containing $x$. Hence, $x$ is a cycle point.

If every point of $X$ is a cycle point, then $h$ is a bijection and so is a homeomorphism by compactness.

\end{proof}

\vspace{.5cm}

\begin{theo}\label{theocycle07} Let $(X,R)$ be a compact topological tournament.  If $x$ is a cycle point, then the singleton $\{ x \}$ is a $G_{\d}$
set which is a component of  $X$. \end{theo}

\begin{proof} Let $\pi : X \to X_E$ be the quotient map of Corollary \ref{cor2.02}. From Theorem \ref{theo2.01} it follows that a $3-$cycle in $X$
can meet a component in at most one point.  Hence, if $x$ is a cycle point in $X$, then $\pi(x)$ is a cycle point in $X/E$. From
 Theorem \ref{theo3.08} it follows that $\pi^{-1}(\{ \pi(x) \})$ is the singleton $\{ x \}$ and so $\{ x \}$ is a component.

 Now assume that $x$ is a cycle point. There exists a $3-$cycle $\{ a_1, x, b_1 \}$ in $X$. $R^{\circ}(a_1) \cap R^{\circ -1}(b_1)$ is an
 open set which contains $x$. Let $U_1$ be an open set with $x \in U_1$ and with closure contained in $R^{\circ}(a_1) \cap R^{\circ -1}(b_1)$.
 Thus, for every $z \in \overline{U_1}, \ \{ a_1, z, b_1 \}$ is a $3$-cycle,  Inductively, we define points $\{a_1, \dots, a_n \}, \{b_1, \dots, b_n \}$
 and open sets $\{ U_1, \dots, U_n \}$ such that for $i = 2, \dots n,$
 \begin{align}\label{eqcycle09} \begin{split}
  a_i, b_i \in &U_{i-1},\quad \overline{U_{i}} \subset U_{i-1}, \\
  \{ a_i, z, b_i \} \ \text{ is a } \ &3-\text{cycle, for all} \ z \in \overline{U_i}.
  \end{split}\end{align}
  Then choose $\{ a_{n+1}, x, b_{n+1} \}$ a $3-$cycle in $U_n$ and thicken $x$ to an open set $U_{n+1}$ with closure contained in
   $U_n \cap R^{\circ}(a_{n+1}) \cap R^{\circ -1}(b_{n+1})$.

   Let $(a,b)$ be a limit point of the sequence $\{ (a_n,b_n) \}$  in $X \times X$ and let $K = \bigcap_n U_n = \bigcap_n \overline{U_n}$.
   Since $a_i, b_i \in U_n$ for all $i > n$ it follows that $a, b \in K$. For all $z \in K \subset U_n, \ \{a_n, z, b_n \}$ is a $3-$cycle.
   So in the limit $(a,z), (z,b), (b,a) \in R$ for all $z \in K$.  In particular, since $a, b \in K, \ (a,b), (b,a) \in R$ and so $a = b$ by
   anti-symmetry. Similarly, $(a,z), (z,b) \in R$ and $a = b$ implies $a = b = z$ for all $z \in K$.  That is, $K$ is a singleton.  Since
   $x \in K, \ K = \{ x \}$.  Thus, $\{ x \}$ is a $G_{\d}$ set.

   \end{proof}

   For a compact space, a point is a $G_{\d}$ point if and only if it has a countable neighborhood base.

   \vspace{.5cm}

  \begin{theo}\label{theo3.06} Assume $(X,R)$ is a  compact topological tournament.  If $R$ is nowhere
  locally transitive, then $X$ is a totally disconnected space with no isolated points and
  every nonempty open set contains a compact subset $K$ such that $R|K$ is isomorphic to the standard $3-$adic example. Every
  point of $K$ is a cycle point and so is $G_{\d}$ point.

  If, in addition, $x \in X$ is a cycle point, then for every open set $U$ with $x \in U$, the compact set
  $K$ can be chosen with $x \in K$.

  Conversely, if the cycle points for $R$ are dense in $X$, then $R$ is nowhere locally transitive. \end{theo}

  \begin{proof} Any non-trivial component of $X$ contains a nonempty open
  subset of $X$ by Corollary \ref{cor2.02} and by Theorem \ref{theo2.01} the restriction of $R$ to this
  open set is transitive. Hence, for a nowhere locally transitive tournament every component is trivial. If $x$ were an isolated point, then
  $\{ x \}$ would be an open subset on which $R$ is trivially transitive. Hence,
 $X$ has no isolated points.

  Let $U$ be a nonempty open subset.

  Because $R|U$ is not transitive and $R^{\circ}$ is open, we can choose a $3-$cycle in $U$ and thicken it, using
  to get disjoint, nonempty, clopen sets $K_1^{\ep} \subset U$ for
  $\ep = 0, 1, 2$ so that $x \mapsto \ep$ for $x \in K_1^{\ep}$ defines a function $h_1$ from $K_1 = \bigcup_{\ep = 0,1,2} K_1^{\ep}$ to
  $\Z/ 3 \Z  = \{ 0, 1, 2 \}$ which maps $R|K_1$ to $\widehat{A_1}$ on $\Z/ 3 \Z $.

  Assume that, inductively, we have defined $K_i$ a disjoint union of nonempty clopen
  subsets $K_i^y$ for $y \in \Z/ 3^i \Z  = \{ 0, 1, 2 \}^{ \{1,\dots,i  \} }$
 so that $x \mapsto y$ for $x \in K_i^{y}$ defines a function $h_i : K_i \to \Z/ 3^i \Z $ mapping $R|K_i$ to $\widehat{A_i}$
 and for $y = z\ep$ with $z \in \Z/ 3^{i-1} \Z $
 and $\ep = 0,1,2$  $K_i^{y} \subset K_{i-1}^z$.

 For the inductive step, for each $y \in Y_i$ choose a $3-$cycle in $K_i^y$ and thicken it, using Theorem \ref{theo2.07}, to
 obtain disjoint nonempty clopen subsets
 $K_{i+1}^{y\ep} \subset K_i^y$ for $\ep \in Y_1$ such that $x \mapsto \ep$ defines a function from $\bigcup_{\ep = 0,1,2} \ K_{i+1}^{y\ep}$ to $\Z/ 3 \Z $
 which maps $R|\bigcup_{\ep = 0,1,2} \ K_{i+1}^{y\ep}$ to $\widehat{A_1}$. With $K_{i+1} = \bigcup_{y\ep \in Y_{i+1}} \ K_{i+1}^{y\ep}$
 $x \mapsto y\ep$ for $x \in  K_{i+1}^{y\ep}$ defines the required function from $K_{i+1}$ to $\Z/ 3^{i+1} \Z $ taking $R|K_{i+1}$ to $\widehat{A_{i+1}}$.

 Let $K $ be the intersection $ \bigcap_{i \in \N} \  K_i \subset U$.

 If $x \in U$ is a cyclic point, then we can make the choice so that $x \in K_i^{0^i}$. In that case, $x \in K$.

 With $K  =  \bigcap_i \  K_i $ we have that the restriction $R|K$ is identified with the inverse limit of $ (R|K_i,k_i \times k_i)$
 with $k_i$ the inclusion map from $K_{i+1}$ to $K_i$. Hence, the maps $h_i : K_i \to \Z/ 3^i \Z $ defines the continuous limit map
 $h : K \to \Z[3]$ which maps $R|K$ to $\widehat{A}$.

 In the standard $3-$adic example every point $y \in \Z[3]$ is a cycle point by Corollary \ref{corcycle06a}.
 Hence, by Theorem \ref{theo3.08}, $h$ is a homeomorphism
 mapping $R|K$ isomorphically onto the standard $3-$adic example.  From it follows that
 every point of $K$ is a cycle point for $R|K$. The $3-$cycles
 in $K$ through a point $x \in K$ are $3-$cycles in $X$ and so each point of $K$ is a cycle point for $R$.

 The converse result is obvious.

  \end{proof}  \vspace{.5cm}

  \vspace{.5cm}

\begin{theo}\label{theo3.07} If $(X,R)$ is a balanced, compact topological tournament, then $R$ is nowhere locally transitive if and only if
the space $X$ is totally disconnected.
\end{theo}

\begin{proof}
If a compact tournament $(X,R)$ is nowhere locally transitive, then
by Theorem \ref{theo3.06} $X$ is totally disconnected.

Conversely, if $R$ is balanced and $U$ is a clopen subset of $X$, then the restriction
$R|U$ is balanced and so has no terminal or initial point.  In particular,
since $U$ is compact, $R|U$ is not transitive. If the compact space $X$ is totally disconnected, then every nonempty open subset contains
a nonempty clopen subset and so $R$ is nowhere locally transitive.

\end{proof}

This completes the proof of Theorem \ref{theointro02}.
  \vspace{.5cm}

\begin{cor}\label{cor3.08} If an infinite compact group $X$ admits a closed game subset $A$, then $X$ is a Cantor set.
\end{cor}

\begin{proof}  That $X$ is totally disconnected with no isolated points
follows from Corollary \ref{corcycle06a} together with Theorem \ref{theocycle07}, which also implies that the points of $X$ are $G_{\d}$ points.
Hence, $e$ has a countable neighborhood base of clopen subsets $U_n$. It follows that $V_n = \{ (x,y) : x^{-1}y \in U_n \}$ is a countable neighborhood base
for $1_X$ by clopen subsets of $X \times X$. For a compact space $X$ the set of neighborhoods of $1_X$ is a uniformity which is metrizable if
it has a countable base, see \cite{K} Chapter 6 and in particular, Theorem 6.13. Since $X$ is metrizable, it is a Cantor set.

\end{proof}

 Together with Corollary \ref{corcycle06a}  and Theorem \ref{theo3.04b} this completes the proof of Theorem \ref{theointro01}.

 \vspace{.5cm}

 It follows that if $G$ is a nontrivial finite group of odd order and $K$ is an uncountable set, then the product group $G^{K}$ is totally disconnected,
 with no isolated points, and with no elements of order two, but since it is not metrizable, it does not admit a closed game subset.
 Of course, since there is no element of order two,
 there are many game subsets (none of which is closed). In fact since such a product contains no $G_{\d}$ points, it follows from
 Theorem \ref{theo3.07} and Theorem \ref{theo3.06} that it admits no balanced tournament. When the cardinality of $K$ at least ${\mathbf c}$,
the cardinality of the continuum,
 we will see below that the product group admits no topological tournaments at all.

 \begin{ques}\label{ques01} Let $(X,R)$ be a compact topological tournament.  If every point of $X$ is a cycle point, does it follow that
 $X$ is metrizable and so is a Cantor set? \end{ques}

 I conjecture that the answer is affirmative.

   \vspace{1cm}

  \section{ \textbf{LOTS Constructions}}\vspace{.5cm}

We have seen that a linear order on a set is exactly a tournament which is transitive. If $L$ is a linear order on a set $X$ (usually written $\le$),
then $L^{\circ}(x)$ is the set of points larger than $x$,  and
$L^{\circ -1}(x)$ is the set of points smaller than $x$. We omit the usual interval notation to avoid confusion with ordered pairs.

  A linearly ordered topological space, or \emph{LOTS}, $X$, is a space  with a linear order $L$,
equipped with the \emph{order topology}\index{order topology} which has subbase $\{ L^{\circ}(x) : x \in X \} \cup \{ L^{\circ -1}(x) : x \in X \}$.

\begin{theo}\label{theo4.01a} If $L$ linear order on $X$, then the order topology is Hausdorff and with respect to the order topology $L$ is closed,
and so is a topological tournament on $X$.

If $X$ is compact and $L$ is closed, i.e. it is a topological tournament which is transitive, then the topology on $X$ is the
order topology obtained from $L$.  In particular, $X$ is a LOTS. \end{theo}

\begin{proof} Assume $b \in L^{\circ}(a)$. The pair $a,b$ is a \emph{gap pair}\index{gap pair} when there is no point between them,
i.e. $L^{\circ}(a) \ \cap \ L^{\circ -1}(b) \ = \ \emptyset$.  In that case, $L^{\circ -1}(b) = L^{-1}(a)$ and
$L^{\circ}(a) = L(b)$ are disjoint neighborhoods of $a$ and $b$, respectively. Furthermore, $L^{\circ -1}(b) \times L^{\circ}(a) \subset L^{\circ}$.

If $c \ \in \ L^{\circ}(a) \ \cap \ L^{\circ -1}(b)$ then $L^{\circ -1}(c)$ and $L^{\circ}(c)$
are disjoint neighborhoods of $a$ and $b$, respectively and $L^{\circ -1}(c) \times L^{\circ}(c) \subset L^{\circ}$.

Thus, the LOTS $X$ is Hausdorff and $L^{\circ}$ is open. Hence, $L = (X \times X) \setminus L^{\circ -1}$ is closed.

Conversely, if $L$ is a topological tournament, then each $L^{\circ}(x)$ and $L^{\circ -1}(x)$ is an open subset of $X$.
If $X_{ord}$ is the set $X$ with the order topology, then the identity map $X \to X_{ord}$ is a continuous bijection.  If $X$ is compact, then since
$X_{ord}$ is Hausdorff, the map is a homeomorphism.  That is, $X$ has the order topology.

\end{proof}
\vspace{.5cm}

A LOTS is \emph{complete} when every bounded, nonempty subset $A$, has a supremum $\sup A$ and an infimum $\inf A$. The LOTS $X$ is compact if and only
if it is complete and has a maximum point $M$ and a minimum point $m$. For details about LOTS, see, e.g. \cite{AH02} and its extension \cite{AH21}.
Regarding the order $L$ as a topological tournament, a maximum is a terminal point and a minimum is an initial point.

\begin{lem}\label{lem4.01aa} If $X$ is a complete LOTS, then every bounded sequence in $X$ has a convergent subsequence. \end{lem}

\begin{proof} It suffices to recall the proof that a sequence $\{ x_n \}$ in $X$ has a monotone subsequence.

Call $n \in \N$ dominating in the sequence, if for all $m > n \ \ x_n \ha x_m$.

If there are infinitely many dominating indices, then the restriction to those indices is a monotone decreasing sequence.
If there are only finitely many dominating indices and $N$ is the largest such, then let $n_1 = N+1$ and inductively
choose $n_{k+1} > n_k$ with $x_{n_{k+1}} \ \underline{\ha} \ x_{n_{k}}$ which exists because $n_k$ is not dominating.
This is a monotone non-decreasing sequence.

A bounded monotone sequence converges to its supremum or infimum.

\end{proof}

\begin{theo}\label{theo4.01ab} If $Y$ is a non-trivial compact space and $I$ has cardinality at least ${\mathbf c}$,
the cardinality of the continuum, then the compact
product space $Y^I$ does not admit any topological tournament. \end{theo}

\begin{proof} Let $P$ be the power set of $\N$. There is a surjection from $I$ onto $P$ and an injection from the two point set $\{ 0, 1 \}$
into $Y$. This induces a continuous embedding of $X = \{ 0, 1 \}^P$ into $Y^I$. It suffices to show that $X$ does not admit a topological tournament.

Observe that $X$ contains no $G_{\d}$ points. By Theorem \ref{theo3.06} is will suffice to show that any topological tournament on $X$ would have to be
nowhere locally transitive.

Suppose instead that on some non-empty clopen subset of $X$ there exists a closed, transitive tournament. By restricting further to a basic open set
obtained by fixing finitely many coordinates we obtain a subset homeomorphic to $X$ itself. It suffices to contradict the assumption that $X$ admits
a closed transitive tournament. By
Theorem \ref{theo4.01a} the topology on $X$ is the associated LOTS topology.
From Lemma \ref{lem4.01aa} it will suffice to produce
a sequence in $X$ with no convergent subsequence.

Each $A \in P$ is a subset of $\N$. Define $\{ x_n \in X \}$ by
\begin{equation}\label{eqnotourn}
(x_n)_A \ = \ \begin{cases} 1 \ \ \text{if} \ \ n \in A, \\ 0 \ \ \text{if} \ \ n \not\in A. \end{cases}
\end{equation}

 Suppose that some subsequence $\{ x_{n_k} \}$ converges. By going to a further subsequence, we may assume $\{ n_k \}$ is strictly
 increasing varying over a subset $B$ of $\N$. The sequence  $\{ (x_{n_k})_A \}$ converges to $1$ if and only if
 $A \cap B$ is a cofinite subset of $B$ and  $\{ (x_{n_k})_A \}$ converges to $0$ if and only if
 $(\N \setminus A) \cap B$ is a cofinite subset of $B$. Write $B$ as the disjoint union of two infinite subsets $B_1$ and $B_2$.
 Let $A = B_1$. Since neither $A \cap B = B_1$ nor $(\N \setminus A) \cap B = B_2$ is a cofinite subset of $B$, it follows that
 $\{ (x_{n_k})_A \}$ does not converge.

\end{proof}
\vspace{.5cm}

In a LOTS $X$ let $\{ x_k \}$ be a net indexed by the directed set $D$ and converging to $x$.  The index set $D$ is partitioned by three subsets:
$\{ k : x_k \in L^{\circ -1}(x)\}, \ \{ k : x_k \in L^{\circ}(x) \}, \ \{ k : x_k = x \}$.  At least one of these sets is
cofinal in $D$ and so by going to a subnet we
may assume that either
\begin{itemize}
\item[$-$]:  $x_k \in L^{\circ -1}(x)$ for all $k$ in which case the net converges to $x$ from the left and $x$ is left balanced for $L$.
\item[+]:  $x_k \in L^{\circ}(x)$ for all $k$ in which case the net converges to $x$ from the right and $x$ is right balanced for $L$.
\item[0]:  $x_k = x$ for all $k$ so that the net is constant at $x$.
\end{itemize}

If $(X_1,L_1)$ and $\{ (Y_x,L_{x}) : x \in X \}$ are LOTS, then we will denote by $(X_2,L_2)$ the lexicographic product
$(X_1,L_1) \ltimes \{ (Y_x, L_x) \}$ as in (\ref{eqlex01}). From Theorem \ref{theolex01}
 the product tournament  is transitive and
so, when equipped with the order topology, $X_2$ is a LOTS.

\begin{prop}\label{prop4.01}  If $(X_1,L_1)$ is a complete LOTS and for each $x \in X_1, (Y_x,L_x)$ is a
compact LOTS with minimum $m_x$ and maximum $M_x$, then
the LOTS $(X_2,L_2) = (X_1 \times \{ Y_x \}, L_1 \ltimes \{ L_{x}) \})$ is complete and the
 projection map $\pi : X_2 \to X_1$ is a continuous, surjective topological tournament map from $L_2$ on
  to $L_1$.\end{prop}

\begin{proof} For $A$  a bounded subset of $X_2$, the set $A_1 = \pi(A)$ is a bounded subset of $X_1$ and so it has a supremum, $a_1$.

If $a_1 \in A_1$, then $\{ y \in Y_{a_1}: (a_1,y) \in A \}$ is nonempty and so has a supremum $y_1 \in Y_{a_1}$.
In that case, $(a_1,y_1)$ is the supremum of $A$.

If $a_1 \not\in A_1$, then $(a_1,m_{a_1})$ is the supremum of $A$ where $m_{a_1}$ is the minimum of $Y_{a_1}$.

The $\pi$ preimage of $L_1^{\circ}(x) \subset X_1$ is $L_2^{\circ}(x,M_x) \subset X_2$ and the preimage of $L_1^{\circ -1}(x)$ is $L_2^{\circ -1}(x,m_x)$.
Hence, $\pi$ is continuous.  It is clearly a tournament map.

\end{proof}

Notice that , in contrast with the topological
lexicographic products of Theorem \ref{theolex01} the LOTS $(Y_x,L_{x})$ can be non-trivial for any point $x$. If, however, the LOTS
is only non-trivial when $x$ is isolated in $X$, then the order topology agrees with the topology in Theorem \ref{theolex01} when
$X_1$ is complete and each $Y_x$ is compact.

From the definition of the lexicographic product and the order topology, the following is obvious.

\begin{add}\label{add4.02} If in $X_2$ a net $\{ (x_k,y_k) \}$ converges to $(x,y)$, then $\{ x_k \}$ converges to $x$ in $X_1$.

\begin{itemize}
\item[$-$]: If $\{ x_k \}$ converges to $x$ from the left, then $\{ (x_k,y_k) \}$ converges to $(x,y)$ if and only if $y = m_x$, i.e. $y$ is the
minimum for $Y_x$.
\item[+]: If $\{ x_k \}$ converges to $x$ from the right, then $\{ (x_k,y_k) \}$ converges to $(x,y)$ if and only if $y = M_x$, i.e. $y$ is the
maximum for $Y_x$.
\item[0]: If  $\{ x_k \}$ is constant at $x$, then $\{ (x_k,y_k) \}$ converges to $(x,y)$ if and only if $\{ y_k \}$ converges to $y$ in $Y_x$.
\end{itemize}
\end{add}

\vspace{.5cm}

When $(Y_x,L_{x}) = (Y, L)$ for all $x \in X$, then we write $X_1 \ltimes Y$ for the product set equipped with linear order $L_1 \ltimes L$.  Notice
that the order topology is usually not the same
as the product topology.  For example, if $b \in L^{\circ}(a) $ in $Y$, then for any $x \in X$, the set
$\{ x \} \times (L^{\circ}(a) \cap L^{\circ -1}(b)$ is the
interval $L^{\circ}(x,a) \cap L^{\circ -1}(x,b))$ in $X \ltimes Y$
and so is open. The Sorgenfrey Double Arrow described above is an example of such a lexicographic product.
\vspace{.5cm}

For our construction we begin with $S$ the topological tournament on a Cantor set $C$ obtained from a closed game subset for a topological
group structure on $C$. From Corollary \ref{corcycle06a} it follows that every point of $C$ is a cycle point for $S$.

Next, use the linear order $L_{\N}$ on the discrete set of natural numbers $\N$.
To the topological tournament $L_{\N} \ltimes S$ on $\N \times C$ we adjoin a terminal point
$M$ to obtain the one-point compactification of $\N \times C$.  Every point of the resulting
topological tournament $S_1$ is a cycle point except the terminal point $M$
which is only left balanced. For its inverse $S_0$ every point is a cycle point except for the initial point which is right balanced.

Because the one-point compactification is itself a Cantor set, we can use a homeomorphism to move $S_0$ and $S_1$ and so
obtain the tournaments on the standard Cantor Set $C$ in the unit interval
with the maximum $1$ the terminal point for the tournament $S_1$ and with the minimum $0$ the initial point for the tournament $S_0$.

Similarly,  use the linear order $L_{\Z}$ on the discrete set of integers $\Z$. To the topological tournament
$L_{\Z} \ltimes S$ on $\Z \times C$ we adjoin a terminal point
$M$ and initial point $m$ to obtain the two-point compactification of $\Z \times C$. We can
use a homeomorphism to obtain the topological tournament
$S_{01}$ on the standard Cantor Set $C$ with initial point equal to the minimum $0$ and terminal point equal to the maximum $1$.  Every point is
a cycle point except for the left balanced point $1$ and the
right balanced point $0$.

Thus, on the Cantor Set $C$ we have four topological tournaments: $S, \ S_0, \ S_1, \ S_{01}$.  Every point
is a cycle point, and so is balanced,  for each of these tournaments except the the
right balanced initial point $0$ for $S_0$ and $S_{01}$ and the left balanced terminal point $1$ for $S_1$ and $S_{01}$.

\begin{df}\label{df4.01a} Let $(Y,L)$ be a non-trivial, compact LOTS with maximum $M$ and minimum $m$. For each of the following types,
$S$ is assumed to be a topological tournament on $Y$ such that every point of $Y$ is balanced except for the terminal or initial points for $S$
when such exist.
\begin{itemize}
\item The tournament $S$ is type 0 when $S$ has no terminal point and the minimum $m$ is an initial point for $S$ which is right balanced with respect to $S$.
\item The tournament $S$ is type 1 when $S$ has no initial point and the maximum $M$ is a terminal point for $S$ which is left balanced with respect to $S$.
\item The tournament $S$ is type 01 when $M$ is a terminal point for $S$ which is left balanced with respect to $S$
 and $m$ is a initial point for $S$ which is right balanced with respect to $S$.
\end{itemize} \end{df}

The existence of tournaments of each of these types for $Y$ requires that  the minimum $m$ be right balanced with respect to the order $L$ on $Y$
and that the maximum $M$ be left balanced with respect to $L$, i.e. neither extremum is isolated.
\vspace{.5cm}

Let $X_1$ be a compact LOTS with  order $L_1$ and with minimum $m$ and maximum $M$. Of course, $m$ is not left balanced and $M$ is not right balanced
for $L_1$.
For each point $x \in X$ we choose  a compact LOTS $Y_x$ with order $L_x$ and a topological
tournament $S_x$ on  $Y_x$ which satisfies the following rules:
\begin{itemize}
\item[(i)]: If $x$ is balanced for $L_1$, then either $Y_x$ is a singleton with $Y_x = \{ m_x \} = \{ M_x \}$ and so $L_x$ and $S_{x}$
are trivial, or else $S_x$ is of type 01.
\item[(ii)] If $x$ is left balanced for $L_1$, but not right balanced for $L_1$, then $S_x$ is of type 0.
\item[(iii)] If $x$ is right balanced, but not left balanced for $L_1$, then $S_x$ is of type 1.
\item[(iv)] If $x$ is an isolated point in $X$, then $S_x$ is balanced and so has no terminal or initial point.
\end{itemize}

Thus, a terminal (or initial) point for the tournament $S_x$, when it exists, coincides with the maximum (resp. the minimum) for the order $L_x$.

Now on the LOTS $(X_2,L_2) = (X_1,R_1) \ltimes \{ (Y_x, L_x) \}$ we define the tournament $R_2$ by:
\begin{equation}\label{eq4.02}
((x,t),(y,s)) \in R_2^{\circ} \quad \Longleftrightarrow \quad \begin{cases} \qquad (x,y) \in L_1^{\circ}, \quad \text{or} \\
x = y \ \  \text{and} \ \ (t,s) \in S_x^{\circ}. \end{cases}
\end{equation}

It is clear the $R_2$ is just the tournament $L_1 \ltimes \{ S_x \}$ on $X_2$ and
the continuous  surjection $\pi : X_2 \to X_1$
is a tournament map from  $R_2$  to $L_1$. Notice  that on $X_2$ we are using the order topology obtained from $L_2$.

\begin{theo}\label{theo4.03} The tournament $R_2$ is a balanced topological tournament on the compact LOTS $X_2$.\end{theo}

\begin{proof}  Let $\{ ((x_k,t_k),(y_k,s_k))\}$ be a net in  $ R_2^{\circ}$ which converges to
\\$((x,t),(y,s))$ in $X_2 \times X_2$.

First observe that it cannot happen that $y \in L_1^{\circ -1}(x)$ since $L_1^{\circ}$ is open and this would imply that
 eventually $y_k \in L_1^{\circ -1}(x_k)$ and
so eventually $\{ ((x_k,t_k),(y_k,s_k)) \in R_2^{\circ -1}$ violating anti-symmetry.

If $(x,y) \in L_1^{\circ}$, then $((x,t),(y,s)) \in R_2^{\circ}$ as required.

We are left with the cases when $x = y$. It cannot happen that at the same time $\{ y_k \} $ converges to $y = x$ from the left
and $\{ x_k \}$ converges to
$x$ from the right, because then $y = x$ and transitivity imply $y_k \in L_1^{\circ -1}(x_k)$ for all $k$
and so again $\{ ((x_k,t_k),(y_k,s_k)) \in R_2^{\circ -1}$.
Similarly, it cannot happen that
$\{ x_k \}$ is constant and $\{ y_k \}$ converges from the left, and it cannot happen that $\{ x_k \}$
converges from the right and $\{ y_k \}$ is constant.

If $\{ x_k\}$ converges to $x$ from the left, then by Addendum \ref{add4.02} $t = m_x$ .  Because $x$ is left balanced for $L_1$, $S_x$ is of type 0
or type 01 and so
$m_x$ is an initial point for $S_x$.  Hence, $((x,t),(y,s)) = ((x,m_x),(x,s)) \in R_2$.

Similarly, if $\{ y_k \}$ converges to $y$ from the right, then $((x,t),(y,s)) = ((y,t),(y,M_x)) \in R_2$.

The remaining possibility is that both $\{ x_k \}$ and $\{ y_k \}$ are constant at $x$. In that case, $\{(t_k,s_k) \in S_x^{\circ}\}$ converges to
$(t,s)$ in $Y_x \times Y_x$ by Addendum \ref{add4.02} and so $(t,s) \in S_x$ which implies $(x,t),(y,s) \in R_2$.

We have proved that $R_2$ is closed.

Now let $(x,t) \in X_2$. If $t$ is neither a initial nor a terminal point for $S_x$, then it is balanced
for $S_x$ and so $(x,t)$ is balanced for $R_2$. Note
that on each $Y_x$ the relative topology induced from $X_2$ is that of $Y_x$.

If $t = m_x$ is a initial point for $S_x$, then $x$ is left balanced for $L_1$ and so there exists a
net $\{ x_k \}$ which converges to $x$ from the left.
For any $y_k \in Y_{x_k}$, $(x_k,x) \in L_1^{\circ}$ implies $\{((x_k,y_k),(x,m_x)) \in R_2^{\circ}\}$ and $\{ (x_k,y_k) \}$ converges to $(x,m_x)$.
Hence, $(x,m_x)$ is left balanced.

Similarly, if $t = M_x$ is a terminal point for $S_x$, then $(x,M_x)$ is right balanced for $R_2$.

In particular, if $Y_x$ is trivial, then $(x,m_x) = (x,M_x)$ is balanced for $R_2$.

Finally, assume that $Y_x$ is non-trivial.

In that case, the initial $m_x$ for $S_x$, when it exists,  is right balanced for  $S_x$
and so there exists a net $t_k \in S_x^{\circ}(0)$ which converges to $m_x$ which implies that the net
$\{ (x,t_k) \in R_2^{\circ}((x,m_x)) \}$ converges to $(x,m_x)$.  That is, $(x,m_x)$ is right balanced and so is balanced.

Similarly, if $t = M_x$ is a terminal point for $S_x$, then $(x,M_x)$ is balanced for $R_2$.

\end{proof}
\vspace{.5cm}

We can make the following alterations in our choice for $S_x$
\begin{itemize}
\item[(v)]: If the minimum $m$ is right balanced for $L_1$, let $S_m$ be type 01
instead of type 1. If $m$ is isolated, let $S_m$ be type 0 instead of balanced.
\item[(vi)]: If the maximum $M$ is left balanced for $L_1$,
let $S_M $ be type 01 instead of type 0.  If $M$ is isolated, let $S_M$ be type 1 instead of balanced.
\end{itemize}

It is easy to check the following.

\begin{add}\label{add4.04}  If we alter our choices according to (v) we obtain a topological tournament $R_0$ on $X_2$ which is
of type 0.

If we alter our choices according to (vi) we obtain a topological tournament $R_1$ on $X_2$ which is
of type 1.

If we alter our choices according to both (v) and (vi) we obtain a topological tournament $R_{01}$ on $X_2$ which is
of type 01.
\end{add}
\vspace{.5cm}

\begin{exes}\label{exes04}  Nonseparable Examples\end{exes}

(a) In \cite{AH02} and \cite{AH21} there is a rich supply of connected, complete, first countable LOTS $X$ which are
not separable. For example, let $I$ be the closed interval in $\R$ with end-points $\pm 1$. The LOTS
$(X,L) = (\R, L_{\R}) \ltimes (I,L_{\R}|I)$ is the product set
$\R \times I$ equipped with the order topology from the lexicographic product order $L$. Restrict to the compact subset $X_1$ which is the closed
interval in $\R \times I$ with minimum $m = (0,1)$ and maximum $M = (1,0)$. For every $t \in \R$ with $0 < t < 1$, let
$A_t$ be the interval in $\{ t \} \times I$ with end-points $(t,-\frac{1}{2})$ and $(t,+\frac{1}{2})$. Thus, $\{ A_t \}$ is an
uncountable collection of pair-wise disjoint, non-trivial intervals, illustrating that $X_1$ is not separable. The LOTS is connected, equivalently
every point of $X_1$ is balanced with respect to the order, except for the right balanced minimum and the left balanced maximum.

For all $x \in \ A \ = \ \bigcup_t \ A_t$ we let $Y_x$ be trivial set $\{ 0 \}$. For $x \in X_1 \setminus A$ we let $Y_x$ be
the standard Cantor Set $C$ with order $L_C$ and tournaments $S_0, S_1$ and $S_{01}$ chosen as above.
There are no isolated points in $X$ and so rule (iv) does not apply.
Let $(X_2,L_2) $ be the LOTS $(X_1,L|X_1) \ltimes \{ (Y_x,L_x \}$ and let $R_2$ be the tournament on $X_2$ given by (\ref{eq4.02}).
For each $t$ between $0$ and $1$,
$A_t \times \{ 0 \}$ is a closed interval in $X_2$ on which the tournament $R_2$ is isomorphic to the order $L_1$ on $A_t$ and each of these
is a component of $X_2$.  Contrast this with the
countability result in the metric case given in Corollary \ref{cor2.02}.
\vspace{.5cm}

(b) Let $X_1$ be the unit interval in $\R$ with end-points $m = 0$ and $M = 1$. For all $x \in X$ we let $Y_x$ be
the standard Cantor Set $C$ which is a LOTS with order $L_C$ inherited from $\R$. Let $S_0, S_1$ and $S_{01}$ be
tournaments on $C$ chosen as above.  Let $(X_2,L_2)$ be the LOTS $(X_1,L_{\R}|X_1) \ltimes (C,L_C)$ and let $R_2$ be
the tournament on $X_2$ given by (\ref{eq4.02}). Every point which is not either equal to $(t,0)$ for some $t$ with $0 < t \le 1$ or equal to
$(t,1)$ for some $t$ with $0 \le t < 1$ is a cycle point for $R_2$ and each of these points has a Cantor set neighborhood in $X_2$.
Each of the remaining, exceptional, points is balanced and with a countable neighborhood base, but with no separable neighborhood.

We can alter $R_2$ to convert some of these exceptional points to cycle points.

Fix $0 < t < 1 $ and a strictly decreasing sequence $\{ t_n \}$ in $X_1$ converging to $t$. Let $\{ A_n \}$ be a sequence of pair-wise disjoint
clopen sets in $C$ arranged with $A_n < A_{n+1}$ and so that the sets converge to the point $1$. Define the subset  $Q$ of
$R^{\circ} \subset X_2 \times X_2$ by
\begin{equation}\label{eq4.03}
Q \ = \ \bigcup_{n=2}^{\infty} \ ( \{ t \} \times A_n) \times (\{ t_n \} \times A_1).
\end{equation}

Observe that $Q$ is a clopen subset of $R^{\circ}$ with closure in $X_2 \times X_2$ given by $Q \cup \{ ((t,1),(t,1)) \}$. So with
\begin{equation}\label{eq4.04}
R_Q \ = \ (R \setminus Q) \cup Q^{-1}, \hspace{2cm}
\end{equation}
we obtain a topological tournament on $X_2$. All of the cycle points for $R$ are still cycle points for $R_Q$, but in addition if $a_n \in A_n$
for each $n$, then for $n \ge 2, \ \{ (t,a_n), (t,1), (t_n,a_1) \}$ is a $3-$cycle for $R_Q$ and so $(t,1)$ is a cycle point with no
separable neighborhood.

It is possible to use this procedure to convert a discrete countable collection of exceptional points to become cycle points. However,
this method can't be used to convert all of the exceptional points to cycle points.  Notice that if the question \ref{ques01} has an affirmative answer, then
no topological tournament on the non-metric space $X_2$ could consist entirely of cycle points.
\vspace{.5cm}

A space $X$ has a LOTS topology if there exists a linear order on $X$ such that the topology on $X$ is the order topology. A Cantor set and a
finite set have LOTS topologies and the results of our lexicographic product and inverse limit constructions all have LOTS topologies. We
saw in the proof of Theorem \ref{theo4.01ab} that if $Y$ is a non-trivial compact space and $I$ has cardinality at least ${\mathbf c}$,
the cardinality of the continuum, then the compact
product space $Y^I$ does not have a LOTS topology and does not admit any topological tournament.

\begin{ques}\label{ques02} If $(X,R)$ is a compact tournament, does the underlying space $X$ have to have a LOTS topology?\end{ques}

\vspace{1cm}

 \section{ \textbf{Big Examples}}\vspace{.5cm}

In this section we perform the construction leading to Theorem \ref{theointro03}.
We use ordinal numbers, see \cite{J} or \cite{H}. What we need is also in \cite{AH21}.

 An ordinal number is a well-ordered set which is equal to the set of its predecessors, beginning with $0 = \emptyset$.  That is, $\a = \{ \b < \a \}$. The
 successor ordinal $\a + 1 = \a \cup \{ \a \}$. With its order topology a successor ordinal $\a + 1$ is a compact LOTS with minimum $0$ and
 maximum $\a$. The successor ordinals in $\a + 1$ form a dense set of isolated points. The remaining, limit, ordinals are left balanced with respect
 to the order.  By well-ordering, no point is right balanced.

 Let $\alp$ be a  limit ordinal so that the successor $\alp + 1$ is a compact LOTS with minimum $0$ and maximum
 $\alp$. We write $L_{\alp}$ for the order on $\alp + 1$. Let $A = \{ 0 \} \cup \{ 1/n : n \in \N \}$ regarded as a compact LOTS with the order $L_A$
 from $\R$.
 Let $A_{\a} = \{ 0 \}$ for any non-limit ordinal, i.e. $0$ and
 all the successor ordinals less than $\alp$ and let $A_{\a} = A$ for all limit ordinals contained in $\alp + 1$, including $\alp$ itself.
 Let $(X_0,L_0)$ be the LOTS $ (\alp + 1, L_{\alp}) \ltimes \{ (A_{\a},L_A)  \}$. The projection
 map to $\alp + 1$ is continuous, and in this case, the injective map
 $\a \to (\a,0)$ is continuous as well. So we will identify $\alp + 1$ with the subset $(\alp + 1) \times \{ 0 \} \subset X_0$. Thus,
 $X_0$ consists of $\alp + 1$ together with, for each limit ordinal $\a$, a decreasing sequence of isolated points converging to  $\a$.
 In the LOTS $X_0$ the isolated points are dense and each non-isolated point is balanced with respect to the
 order $L_0$.  Thus, the LOTS $(X_0,L_0)$ is a brick in the sense of
 Definition \ref{dflex02}.

 Now let $B$ be a finite set with odd cardinality and $L_B$ be a linear order on $B$. Fix $S_B$ an arc cyclic tournament on $B$
  so that $(B,S_B)$ is an ip cyclic brick. The LOTS $(B,L_B)$ is also a brick, but not, of course, ip cyclic.

  For each isolated point $y \in X_0$ let $(B_y,L_y) = (B,L_B)$ and $(B_y,S_y) = (B,S_B)$.
 for each non-isolated point $y$ we let $(B_y,S_y) = (B_y,L_y)$ be trivial. Let the compact LOTS $(X_1,L_1)$ be the  lexicographic product
 $(X_0, L_0) \ltimes \{ (B_y, L_y)  \}$. In this case it is also topological lexicographic product.
 On $X_1$ we define $R_1 =  L_0 \ltimes \{ S_y \}$ the topological lexicographic product using the cycle tournaments.
 By Theorem \ref{theolex03} $(X_1,R_1)$ is a ip cyclic brick. We call $(X_1,R_1)$ a Big Brick. The underlying space $X_1$ is a compact LOTS
 with the ordering $L_1$.

 If $p : X_1 \to X_0$ is the first coordinate projection, then for each limit ordinal $\a$ in $X_0$, $p^{-1}(\a)$ is a singleton set
 and we label the point in this set by $\a$ as well. That is, we regard the limit ordinals $\a \le \alp$ as points of $X_1$. Because the tournament
 on $X_0$ is an order, it follows that no limit ordinal $\a$ in $X_1$ is contained in a cycle in $X_1$.

 We now perform a lexicographic inverse system of bricks using the Big Brick.

Begin with $(X_1,R_1)$ as above so that $X_1$ is a LOTS with order $L_1$. Inductively, we are given the LOTS $(X_i,L_i)$
 with the tournament $R_i$ so that  $(X_i,R_i)$ is a brick. We let $(Y_{iz},S_{iz},L_{iz}) = (X_1,R_1,L_1)$ for all $z$ isolated in $X_i$ and
 $(Y_{iz},S_{iz},L_{iz})$ trivial for all $z$ non-isolated in $X_i$.

 We let $(X_{i+1},L_{i+1}) = (X_i,L_i) \ltimes \{ (Y_{iz}, L_{iz} )\}$,
 i.e. the LOTS which is the lexicographic product. Since $(Y_{iz},S_{iz},L_{iz})$ trivial for all $z$ non-isolated in $X_i$, this
 lexicographic product is  at the same time the topological lexicographic product. $(X_{i+1},R_{i+1}) = (X_i,R_i) \ltimes \{ (Y_{iz}, R_{iz}) \}$.
 Since both of these are topological lexicographic products, the space $X_{i+1}$ is the same for both products and so the topology on
 $X_{i+1}$ is the LOTS topology obtained from $L_{i+1}$.
Let $f_i:  X_{i+1} \to X_i$ be the first coordinate projection which is an open, continuous surjection mapping $L_{i+1}$ to $L_i$ and $R_{i+1}$ to $R_i$.

 We let $(X,R,L)$ be the inverse limit of the inverse sequence \\ $\{ (X_i,R_i,L_i,f_i) \}$, i.e. $(X,L)$ is the inverse limit of the
 tournament inverse sequence $\{ (X_i,L_i,f_i) \}$ and $(X,R)$ is the inverse limit of the
 tournament inverse sequence $\{ (X_i,R_i,f_i) \}$.

 \begin{theo}\label{5.01} The ordered space $(X,L)$ is a compact, totally disconnected LOTS with no isolated points and $R$ is a balanced
 topological tournament on $X$.  The set of cycle points for $R$ is the dense $G_{\d}$ subset $IS$ of points
 $x \in X$ such that each $x_i$ is isolated in $X_i$.
 In addition, every nonempty open subset of $X$ has cardinality at least that of $\alp$.

If the cardinality of $\alp$ is countable, then $X$ is a Cantor set.

If $\alp$ is uncountable, then
there is a dense set of points which are not $G_{\d}$ points and no open subset is separable.

\end{theo}

 \begin{proof} By Theorem \ref{theoinv03} $(X,R)$ is a balanced tournament with $X$ compact and totally disconnected and similarly $L$ is a
 transitive topological tournament on $X$ so that the topology on $X$ is the order topology by Theorem \ref{theo4.01a}.

 By Theorem \ref{theoinv03} again the subset
 $IS$ is a dense $G_{\d}$ set and it consists of cycle points by Theorem \ref{theo3.07a}.
Also by Theorem \ref{theoinv03} each projection map $\pi_i : X \to X_i$ is
 open as well as continuous.

 A nonempty open subset contains some $\pi_i^{-1}(U)$ with $U$ an open nonempty subset of $X_i$. There exists an $IS$ point $x$ in $\pi_i^{-1}(U)$.
 So $x_i$ is isolated in $X_i$ and for every $y\in Y_{ix_i}$, there exists a point $x' \in X$ with $x'_i = x_i$ and $x'_{i+1} = (x_i,y)$.
 Thus, $\pi_i^{-1}(x) \subset \pi_i^{-1}(U)$ contains a set of cardinality at least that of $\alp$.

 Notice that if $x_i$ is an isolated point of $X_i$, then the map $y \mapsto (x_i,y)$ is a homeomorphism from $X_1$ onto a clopen subset
 of $X_{i+1}$, inducing a tournament isomorphism from $(X_1,R_1)$ onto the restriction $R_{i+1}|(\{ x_i \} \times Y_{ix_i})$.

We saw above that if $\a$ is a limit ordinal in $Y$ then it is not contained in a cycle in $Y$. It follows that the unique point $x' \in X$ with
$x'_{i+1} = (x_i,\a)$ is not in any cycle contained in the clopen set $\pi_i^{-1}(\{ x_i \})$. Thus, $x'$ is not a cycle point.

It follows that the points of $IS$ are the only cycle points. By Theorem \ref{theocycle07} every cycle point is a $G_{\d}$ point.

Furthermore, because $\pi_{i+1}$ is a continuous, open map and because for the above point $x'$, $\pi_{i+1}^{-1}(x')$ is a singleton, it
follows that $x'$ is a $G_{\d}$ point if an only if $x'_{i+1}$ is a $G_{\d}$ point in the clopen set $\pi_i^{-1}(\{ x_i \})$ and so
if and only if $\a$ is a $G_{\d}$ point in $X_1$. This is true if and only if the limit ordinal $\a$ is countable. Such a point is a $G_{\d}$ point
but is not a cycle point.

If $\alp$ is countable, then the bricks are countable and so $X$ is a Cantor set by Theorem \ref{theoinv03}.

If $\alp$ is uncountable, the point $x'$ with $x'_{i+1} = (x_i,\alp)$ is not a $G_{\d}$ point.
Thus, the points which are not $G_{\d}$ form a dense subset of $X$.

A nonempty clopen subset of $X$  is a compact LOTS with respect to the restriction of the order $L$.
A compact, separable LOTS can be embedded in $\R$ and so every
  interior point would be a $G_{\d}$ point. It follows that no open subset is separable when $\alp$ is uncountable.
 \end{proof}

For example, if $\alp$ is an uncountable ordinal but with cardinality less than or equal to ${\mathbf c}$, the cardinality of the continuum,
then $X$ has cardinality ${\mathbf c}$, that of the Cantor set, but contains no separable nonempty open subset.

\vspace{1cm}

  \section{\textbf{WAC Tournaments and Prime Quotients}}\vspace{.5cm}

  We begin with an extension of the concept of arc cyclicity.

   \begin{df}\label{wacdef01} A topological tournament $(X,R)$ is called \emph{weakly arc cyclic},\index{tournament!weakly arc cyclic}
   \index{tournament!arc cyclic!weakly}
   or just a \emph{wac tournament},\index{wac tournament}\index{tournament!wac} when $X$ is
   compact and every non-isolated point of $X$ is a cycle point for $R$.

  The topological tournament $(X,R)$ an \emph{almost wac tournament}\index{almost wac tournament}\index{tournament!almost wac} when $X$ is compact and
   every point of $X$ is either isolated, initial,
  terminal or a cycle point.  \end{df} \vspace{.5cm}

  We will write that the tournament $R$ is wac or almost wac when the underlying space is understood.

   \begin{theo}\label{wactheo02} (a) If $(X,R)$ is an almost wac tournament, then $X$ is totally disconnected.

   (b) If $R$ is a compact locally arc cyclic tournament,e.g. a finite tournament, then it is wac.

   (c) If $R$ is a wac (or almost wac) tournament, then the reverse tournament $R^{-1}$ is wac (resp. almost wac).

      (d) If $R$ is a wac tournament (or an almost wac tournament) and $A$ is a nonempty clopen subset of $X$,
      then the restriction $R|A$ is  wac (resp. almost wac).  If $R$ is a wac tournament and $x \in X$, then the
      restriction $R|R(x)$ is almost wac with initial point $x$.

   (e) If $R$ is a wac tournament, then any initial or terminal point for $R$ is an isolated point.

   (f) Let $h : (X_2,R_2) \to (X_1,R_1)$ be a surjective continuous map of topological tournaments. If $(X_2,R_2)$  is wac (or almost wac), then
   $(X_1,R_1)$ is wac (resp. almost wac) and for every non-isolated point $y$ of $X_1$, the set $h^{-1}(y)$ is a singleton subset $\{ x \}$  with $x$ non-isolated in $X_2$.

   If $x$ is a terminal (or initial) point for $R_2$, then $h(x)$ is a terminal (resp. initial) point for $R_1$.
   If $y$ is a terminal (or initial) point for $R_1$ and $y$ is not isolated, then $h^{-1}(y) = \{ x \}$ with $x$ terminal (resp. initial) for $R_2$.

   If $y$ is an isolated point of $X_1$, then the restriction $R_2|h^{-1}(y))$ is wac (resp. almost wac).

   There
    is an isomorphism  to the topological lexicographic product
    $$\hat h : (X_2,R_2)\ \to \ (X_1,R_1)  \ltimes  \{ (h^{-1}(y),R_2|h^{-1}(y)) : y \in X_1 \}$$
    such that
    $\pi \circ \hat h = h$ where $\pi$ is the coordinate projection to $(X_1,R_1)$. In particular, $h$ is an open map.

    If $(X_2,R_2)$ is arc cyclic or locally arc cyclic, then $(X_1,R_1)$ satisfies the corresponding condition.
   \end{theo}

   \begin{proof} (a):  If $A$ were a non-trivial component of $X$, then it contains no isolated points and by Corollary \ref{cor2.02}
   it contains a non-empty open subset $U$ of $X$. Furthermore,
   the restriction $R|A$ is an order and so the infinite open set $U$ contains no cycle points. As at most two points are initial or terminal,
   the tournament cannot be almost wac.

   (b): A compact locally arc cyclic tournament is wac by Theorem \ref{theocycle06}.

   (c): is obvious.

   (d): If $x \in A$ and $A$ is an arbitrary subset, then $x$ non-isolated in $A$ or $x$  a cyclic point for $R|A$ implies the corresponding
   condition for $X$. If $A$ is open, then the converse holds. If $A$ is clopen, then it is compact as well.

   So a point $x' \in R^{\circ}(x)$ is isolated or a cycle point in $R(x)$ if and only if satisfies the corresponding property in $X$. Clearly,
   $x$ is initial for the restriction to $R(x)$. Hence, $R|R(x)$ is almost wac when $R$ is wac.

   (e): An initial or terminal point is not balanced and so is not a cycle point. Such a point in a wac tournament is therefore isolated.

   (f): Clearly, if $M$ is terminal for $X_2$, then $x \ \underline{\ha} \ M$ for all $x \in X$ implies
  $h(x) \ \underline{\ha} \ h(M)$.  Since $h$ is surjective, $h(M)$ is terminal for $X_1$. Similarly, $m$ initial for $X_2$ implies
  $h(m)$ is initial for $X_1$.

   If $y$ is an isolated point in $X_1$, then $h^{-1}(y)$ is clopen and so the restriction  $R_2|h^{-1}(y)$ is wac or almost wac by (c).

   For $y$ a non-isolated point of $X_1$, let $U$ be an open set containing $y$.
  We may assume $y$ is left balanced as the right balanced case is similar.  As it is left balanced, it is not an initial point for $X_1$. By
   Theorem \ref{theo1.02} and Addendum \ref{add1.03}, there is a unique point $x \in h^{-1}(y) \cap \overline{h^{-1}(R_1^{\circ}(y))}$ and
   $x$ is an initial point for the restriction $R_2|h^{-1}(y)$. So $x$ is a non-isolated point for the wac tournament $(X_2,R_2)$.
   Since it is left balanced, it is not an initial point for $X_2$.

   Case 1 ($y$ is not terminal in $X_1$):  In that case, $x$ is not terminal in $X_2$. As it is not initial and not isolated, it is a cycle point
 Therefore, there  exists a $3-$cycle $\{ x, x', x'' \}$ in $h^{-1}(U)$. Because $x$ is an initial point for $R_2|h^{-1}(y)$, the
   cycle is not contained in $h^{-1}(y)$. So by Proposition \ref{prop1.01a} it is mapped by $h$ to a $3-$cycle $\{ y, y', y'' \}$ in $U$.
   Thus, $y$ is a cycle point and from Theorem \ref{theo3.08} it follows that $h^{-1}(y)$ is the singleton $\{ x \}$.
\vspace{.25cm}

Case 2 ($y = M_1$ is terminal in $X_1$):  If $h^{-1}(M_1)$ were to contain more than one point, then $x$ is not terminal
  in $X_2$ since it is initial in $h^{-1}(M_1)$. Because $x$ is not isolated and neither initial nor terminal, it would have to be
  a cycle point. As in Case 1, we would obtain a $3-$cycle $\{ M_1, y', y'' \}$ in $U$. This is impossible because $M_1$ is terminal.
  Hence, $h^{-1}(M_1)$ is the singleton $\{ x \}$ in this case as well. Finally, for all $x' \in X_2 \setminus h^{-1}(M_1)$,
  $h(x') \ha M_1$ implies $x' \ha x''$ for all $x'' \in h^{-1}(M_1)$.  As the latter set is the singleton $\{ x \}$ it follows that
  $x$ is terminal in $X_2$.
  \vspace{.25cm}

   Since $h^{-1}(y)$ is a singleton whenever $y$ is non-isolated,  the topological lexicographic product
    $(X_1,R_1) \ltimes \{ (h^{-1}(y),R_2|h^{-1}(y)) \}$ can be defined according to
   Definition \ref{dflex01a}. The map $\hat h$ defined by $\hat h(x) = (h(x),x)$ with inverse $(y,x) \mapsto x$ is a bijection
   providing a tournament isomorphism. From the definition of the basis in Definition \ref{dflex01a} it is clear that $\hat h$ is continuous.
   So it is a homeomorphism by compactness.  Form Theorem \ref{theolex01}(a) it follows that $h$ is an open map.

   If a neighborhood $U$ of the point in $h^{-1}(y)$ is an arc cyclic subset, then since $h$ is an open map,
   $h(U)$ is a neighborhood of $y$ and Corollary \ref{cor1.01b} implies that it is an arc cyclic subset of $X_1$.
   In particular, if $X_2$ is arc cyclic, then $X_1 = h(X_2)$ is arc cyclic.

   \end{proof} \vspace{.5cm}

   When $h : (X,R) \to (Y,S)$ is a surjective, continuous map of compact topological
   tournaments, we will call it a \emph{quotient map}\index{quotient map}.

  Notice that if $(X,R)$ is a wac tournament
   and $(Y,S)$ is any arc cyclic, compact topological tournament on
  a Cantor set, e.g. a tournament obtained from a closed game subset on an infinite, compact topological group, then we can perform a topological
  lexicographic product with base $(X,R)$ and with $(Y_x,S_x) = (Y,S)$ for every isolated point $x \in X$. In the resulting compact tournament
  every point is a cycle point and the tournament maps onto
  $(X,R)$. If the answer to Question \ref{ques01} is affirmative, then the lift is metrizable and so $(X,R)$ is metrizable as well. Thus, it
  would then follow that for any wac tournament the underlying space is
  metrizable and so has only countably many isolated points.
\vspace{.5cm}

  We will later see that there exist wac tournaments which are not locally arc cyclic. We pause to consider the stronger condition.
Recall that for a tournament $(X,R)$ a subset $A $ of $X$ is arc cyclic when every arc contained in $A$ is contained in a $3-$cycle in $X$.
Since the $3-$cycle need not be contained in $A$, this condition is weaker than the assumption that the restriction $R|A$ is an arc cyclic tournament
on $A$.

  \begin{prop}\label{arccycprop01} Let $(X,R)$ be a topological tournament.

  (a) If $A$ is an arc cyclic subset, then its closure $\overline{A}$ is arc cyclic. Any subset of $A$ is arc cyclic.

  (b) If $\{ A_i \}$ is a monotone family of arc cyclic subsets, then the union $\bigcup_i A_i$ is arc cyclic.

  (c) If $A$ is an arc cyclic subset and $x \in A$ has an arc cyclic neighborhood, then there exists a neighborhood $U$ of $x$
  such that $A \cup U$ is arc cyclic.

  (d) Any arc cyclic subset, e.g. a singleton set, is contained in a maximal arc cyclic subset which is a closed subset of $X$.
   If $A$ is a maximal arc cyclic subset and $x \in A$ has an arc cyclic neighborhood, then $x$ is in the interior of $A$.
  \end{prop}

  \begin{proof} (a): If $x \ha x'$ with $x, x' \in \overline{A}$, then there is a net
  $\{(x_i,x_i') \in R^{\circ} \cap (A \times A)$ converging to $(x,x')$. By arc cyclicity, there exists $x_i'' \in X$
  such that $\{ x_i, x'_i, x_i'' \}$ is a $3-$cycle. By compactness we may assume $\{ x_i'' \}$ converges to $x'' \in  X$.
  Since $R$ is closed, $(x', x''), (x'',x) \in R$ and so by asymmetry, $x''$ cannot equal either $x$ or $x'$.  Hence,
  $\{ x, x', x'' \}$ is a $3-$cycle. The subset result is obvious.

  (b): If $x \ha x'$ with $x, x' \in \bigcup_i A_i$, then for some $i$, $x, x' \in A_i$ by monotonicity.
 Since $A_i$ is arc cyclic there exists $x''$ such that $\{ x, x', x'' \}$ is a $3-$cycle.

  (c): Suppose that $V_1$ is a neighborhood of the diagonal so that $V_1(x)$ is an arc cyclic subset. By Lemma \ref{lemcycle05}
    there exists a symmetric diagonal neighborhood
  $V_2 \subset V_1$ such that if $\{ x, y, z \}$ is a $3-$cycle with $(x,y) \in V_2$, then $(x,z), (y,z) \in V_1$ and so $(x,z) \not\in V_1$ or
  $(y,z) \not\in V_1$   implies $(x,y) \not\in V_2$. By Theorem \ref{theo2.07} there exists a symmetric neighborhood $V$ of the diagonal such that
  $(x,y),(y,z), (x,z) \not\in V_2$ implies that $\{ V(x), V(y), V(z) \}$ is a thickening of $\{x, y, z\}$. I claim that $V(x) \cup A$ is an arc cyclic
  subset.

  If $x_1, y_1 \in V_1(x)$, then there exists $z_1$ such that $\{x_1, y_1, z_1 \}$ is a $3-$cycle, because $V_1(x)$ is an arc cyclic set.
  Hence, $V(x) \cup (A \cap V_1(x)) \subset V_1(x)$ is an arc cyclic subset. Now suppose that $y \in A \setminus V_1(x)$. Because
  $A$ is an arc cyclic subset, there exists $z$ such that $\{x, y, z \}$ is a $3-$cycle. It cannot happen that $(x,z) \in V_2$ or $(y,z) \in V_2$ because
  either would imply $(x,y) \in V_1$. Hence, $(x,y),(y,z), (x,z) \not\in V_2$. Hence,  $\{ V(x), V(y), V(z) \}$ is a thickening of $\{x, y, z\}$.
  Thus, if $x_1 \in V(x)$, $\{x_1, y, z \}$ is a $3-$cycle. That is, $V(x) \cup A$ is an arc cyclic
  subset.

   (d): Immediate from (b) and Zorn's Lemma. By maximality, (a) implies that a maximal arc cyclic set is closed.
  If $x \in A$ has a neighborhood which is an arc cyclic subset, then, by maximality (c) implies that $A$ contains some neighborhood of $x$.

  \end{proof}

  \begin{cor}\label{arccyccor02} If $(X,R)$ be a locally arc cyclic topological tournament, then any maximal arc cyclic subset is clopen. \end{cor}

  \begin{proof} This is immediate from Proposition \ref{arccycprop01}(d).

  \end{proof}

  \begin{cor}\label{arccyccor03} If $h : (X_2,R_2) \to (X_1,R_1)$ is a quotient map  such that
  $h^{-1}(y)$ is an arc cyclic subset of $X_2$ for all $y \in X_1$, then $(X_2,R_2)$ is locally arc cyclic if and only if $(X_1,R_1)$ is locally
  arc cyclic.   \end{cor}

  \begin{proof}  Note first that if $h^{-1}(y)$ is an arc cyclic subset and $\{ x, x', x'' \}$ is a $3-$cycle with $x, x' \in h^{-1}(y)$, then
  by Proposition \ref{prop1.01a}, $x'' \in h^{-1}(y)$. Thus the restriction $R_2|h^{-1}(y)$ is an arc cyclic tournament.

  If $U$ is an open subset of $X_1$, then Corollary \ref{cor1.01b} and the restriction assumptions imply that $U$ is
  an arc cyclic subset of $X_1$ if and only if $h^{-1}(U)$ is an arc cyclic subset of $X_2$. Thus, if $(X_1,R_1)$ is locally arc cyclic, then
  $(X_2,R_2)$ is because $h$ is surjective.

  If $y \in X_1$, then the arc cyclic subset $h^{-1}(y)$ is contained in a maximal arc cyclic subset $A$ of $X_2$.   By Corollary
  \ref{arccyccor02} $A$ is clopen if $(X_2,R_2)$ is locally arc cyclic.  It then follows that there exists an open subset
  $U$ of $X$ with $y \in U$ and such that $h^{-1}(U) \subset A$. Thus, $h^{-1}(U)$ is arc cyclic.  By Corollary \ref{cor1.01b} again it follows that
  $U$ is arc cyclic.  Hence, $(X_1,R_1)$ is locally arc cyclic.

  \end{proof} \vspace{.5cm}

  For a compact topological tournament $(X,R)$ we define the subsets $Q, Q^{\circ}$ of $ X \times X \times X$;
  \begin{align}\label{qeq01}\begin{split}
  Q \ = \ \{ (x,y,z) : &(x,z), (z,y) \in R \} \ \cup \\ &\{ (x,y,z) : (y,z),(z,x)  \in R \}, \\
 Q^{\circ} \ = \ \{ (x,y,z) :  &(x,z), (z,y) \in R^{\circ} \} \ \cup \\ &\{ (x,y,z) : (y,z),(z,x)  \in R^{\circ} \}
 \end{split}\end{align}
 We regard $Q$ and $Q^{\circ}$ as relations from $X \times X$ to $X$. Clearly, $Q$ is closed and $Q^{\circ}$ is open.

 The following properties are easy to check.
 \begin{align} \begin{split}
 \{ x \} \ = \ \ &Q(x,x), \qquad \emptyset \ = \ Q^{\circ}(x,x). \\
  \{ x, y \} \ \subset \ \ &Q(x,y) \ = \ Q(\{ x,y \} \times \{ x,y \}). \\
  Q(x,y) \setminus \{ x,y \} \ = \ \ &Q^{\circ}(x,y) \ = \ Q^{\circ}(\{ x,y \} \times \{ x,y \}).
 \end{split}
 \end{align}

 For any subset $A$ of $X, \ \ Q^{\circ}(A \times A) $ is open. $Q(A \times A) \ = \ Q^{\circ}(A \times A) \cup A$.
 So, if $A$ is closed, then $Q(A \times A)$ is closed and
 $Q(A \times A) \setminus A = Q^{\circ}(A \times A) \setminus A$ is open.

 \begin{lem}\label{qlem01} $Q^{\circ}(\overline{A} \times \overline{A}) = Q^{\circ}(A \times A)$. \end{lem}

 \begin{proof} If $R^{\circ}(z)$ and $R^{\circ -1}(z)$ meet $\overline{A}$, then they meet $A$.

 \end{proof}

\begin{lem}\label{qlem01a} For a compact tournament $(X,R)$ assume  $x \ha y$ in $X$.

If $y$ is left balanced or $x$ is right balanced, then
the open set $R^{\circ}(x) \cap R^{\circ -1}(y)$ is nonempty and is contained in $Q^{\circ}(x,y)$.
\end{lem}

\begin{proof} The open set $R^{\circ}(x)$ contains $y$.  If $y$ is left balanced, then $R^{\circ -1}(y)$ meets $R^{\circ}(x)$ and
the intersection is clearly contained in $Q^{\circ}(x,y)$.  Similarly, if $x$ is right balanced.

\end{proof}
\vspace{.5cm}

 \begin{df}\label{qdef02} For a compact topological tournament $(X,R)$, a subset $A$ of $X$ is called \emph{$Q$ invariant}\index{$Q$ invariant subset}
 when $Q(A \times A) \subset A$, or, equivalently, when $Q^{\circ}(A \times A) \subset A$. \end{df}
 \vspace{.5cm}

 Clearly, every singleton subset is $Q$ invariant.

 \begin{prop}\label{qprop02} Let $(X,R)$ be a compact topological tournament.
 \begin{itemize}
 \item[(a)]  A subset $A$ is $Q$ invariant if and only if for all $z \in X \setminus A$, either $A \subset R(z)$ or $A \subset R^{-1}(z)$.

 \item[(b)] If $A$ is $Q$ invariant, then the closure $\overline{A}$ is $Q$ invariant.

 \item[(c)] If $\{ A_i \}$ is a family of $Q$ invariant sets, then the intersection $\bigcap \{ A_i \}$ is $Q$ invariant.

 \item[(d)] If $\{ A_i \}$ is a monotone family of $Q$ invariant sets,
 then the union $\bigcup \{ A_i \}$ is $Q$ invariant.

 \item[(e)] If $(X,R)$ is arc cyclic and $A$ is a  $Q$ invariant set, then the restriction $(A,R|A)$ is arc cyclic.

  \end{itemize}
 \end{prop}

 \begin{proof} (a) is obvious and (b) follows from Lemma \ref{qlem01}.

 (c) is obvious, and for (d) if $x,y \in \bigcup \{ A_i \}$, then monotonicity implies that for some $i$, $x,y \in A_i$.
 Hence, $Q(x,y) \subset A_i$.

 (e): If $(x,y)$ is an arc in $A$, then there is a $3-$cycle $\{ x, y, z \}$ in $X$. By $Q$ invariance, $z \in A$.

 \end{proof}

  \begin{theo}\label{qtheo03} If $h : (X_2,R_2) \to (X_1,R_1)$ is  continuous map of compact topological tournaments,
  then for every $Q$ invariant subset $B$ of $X_1$, the pre-image $h^{-1}(B)$ is a $Q$ invariant subset of $X_2$. In particular
  for every $y \in X_1$, the set $h^{-1}(y)$ is a closed, $Q$ invariant subset of $X_2$.

  If $A \subset X_2$ is $Q$ invariant and $h$ is surjective, then the image $h(A)$ is a $Q$ invariant subset of $X_1$. If $A$ is a proper subset of $X_2$
  and $X_1$ has no initial nor terminal point, then $h(A)$ is a proper subset of $X_1$.

  Conversely, if $A$ is a nonempty, closed subset of a compact space $X$, then $1_X \cup (A \times A)$  is a closed equivalence relation, with
  $\pi_A : X \to X/A$ the quotient map to the space with $A$ smashed to a point.\index{smashed to a point}
   If $A$ is clopen, then the point $\pi_A(A)$ is isolated in $X/A$.
  If $A$ is $Q$ invariant for the topological tournament
  $R$ on $X$, then $R_A = (\pi_A \times \pi_A)(R)$ is a topological tournament on $X/A$ and $\pi_A$ maps $R$ to $R_A$.
  \end{theo}

  \begin{proof}  If $x, x' \in h^{-1}(B)$ and $z \in X_2$ with $x' \ha z \ha x $, then \\$h(x') \ \underline{\ha} \  h(z)  \ \underline{\ha} \  h(x)$ and
  so $h(z) \in B$ because $B$ is $Q$ invariant. Hence, $z \in h^{-1}(B)$.

  Now assume $h$ is surjective. If  $x \in X_2 \setminus A$ and $z = h(x)$, then since
  $A$ is $Q$ invariant
  either $A \subset R_2(x)$ or $A \subset R_2^{-1}(x)$. Hence, $h(A) \subset R_1(z)$ or $h(A) \subset R_1^{-1}(z)$.
  If $h(A) = X_1$, then $z = h(x)$ is either an initial or terminal point for $X_1$. If
  $z \in X_1 \setminus h(A)$, then there exists $x $ such that $h(x) = z$ and, necessarily, $x \in X_2 \setminus A$.
  So $h(A) \subset R_1(z)$ or $h(A) \subset R_1^{-1}(z)$ implies that $h(A)$ is $Q$ invariant.

  If $A$ is a closed, $Q$ invariant subset, then (a) of Proposition \ref{qprop02} implies that $R_A$ is a tournament. It is closed by compactness.
  If $A = \pi_A^{-1}(\pi_A(A))$ is clopen, it follows that the point $\pi_A(A)$ is clopen by definition of the quotient topology.

  \end{proof}
 \vspace{.5cm}

 Now let $(X,R)$ be a compact tournament and let $A_0$ be a non-trivial subset of $X$.  Inductively, define $A_{n+1} = Q(A_n \times A_n)$.
 If $A_0$ is closed, then, inductively, we see that $A_n$ is closed for all $n$.  Clearly, if $A_0$ is $Q$ invariant, then
 $A_n = A_0$ for all $n$.

  \begin{prop}\label{qprop03} For $(X,R)$ a compact tournament and  $A_0$ be a non-trivial  subset of $X$,  $\{ A_n \}$
  is a non-decreasing sequence. The union $\bigcup_n \{ A_n \}$ is the
  smallest $Q$ invariant subset of $X$ which contains $A_0$, its closure $\overline{\bigcup_n \{ A_n \}}$ is the smallest
  closed, $Q$ invariant subset of $X$ which contains $A_0$.

  If $A_0$ is closed, then all of the $A_n$'s are closed and each $A_n \setminus A_0$ is open.

\end{prop}

 \begin{proof} Since $A$ is contained in  $Q(A \times A)$, it follows that $\{ A_n \}$ is an
  increasing sequence.

  If $z \in Q(x,y)$ with $x, y \in \bigcup_n \{ A_n \}$, then for some $n$ $x, y \in A_n$ and so $z \in A_{n+1}$. It follows
  that $\bigcup_n \{ A_n \}$ is $Q$ invariant and is clearly the  smallest $Q$ invariant subset of $X$ which contains $A_0$.
  Hence,$\overline{\bigcup_n \{ A_n \}}$ is the smallest
  closed, $Q$ invariant subset of $X$ which contains $A_0$.

   If  $A$ is closed, then $A$ is contained in the closed set $Q(A \times A)$ and $Q(A \times A) \setminus A
  = Q^{\circ}(A \times A) \setminus A$ so that $Q(A \times A) \setminus A$ is open.

  Thus, if $A_0$ is closed, then $\{ A_n \}$ is an
  increasing sequence of closed sets and each $A_{n} \setminus A_{n-1}$ is open.
  Hence, $A_n \setminus A_0 = \bigcup_{i=1}^n A_i \setminus A_{i-1}$ is open.

 \end{proof}

   \begin{lem}\label{qlem03a} Let $(X,R)$ be a compact tournament and $x$ be an isolated point or a cycle point of $X$.
  For all $n \ge 3$ if $x$ is
   contained in the closure of  $A_{n-3}$, then $x$ is contained
   in the interior of $A_n$.

   In particular, if $(X,R)$ is wac, then the closure of
    of  $A_{n-3}$ is contained
   in the interior of $A_n$ for all $n\ge 3$. \end{lem}

   \begin{proof} If $n > 3$ and we define $A_0' = A_{n-3}$ then in the associated sequence $A_n'$ we have $A_3' = A_n$ and so it suffices to
   prove the result for $n = 3$.

   Let $x \in \overline{A_0}$. If $x$ is isolated, then is in the interior of $A_0$ and so of that of $A_{3}$.

   Assume $x$ is a cycle point and choose $y \in A_0$ with $y \not= x$. Assume that $ y \ha x $. There exists a
   $3-$cycle $\{ x, x', x'' \}$ contained in $R^{\circ}(y)$.
   By  Theorem \ref{theo2.07}, there exists $\{ U_x, U_{x'}, U_{x''} \}$  a
   thickening contained in $R^{\circ}(y)$. Let $z \in A_0 \cap U_x$. Because
   $U_{x''} \subset R^{-1}(z) \cap R(y)$ it follows that $U_{x''} \subset A_1$. Since $\{ z, z', x'' \}$ is a $3-$cycle
   for any $z' \in U_{x'}$ it follows that $U_{x'} \subset A_2$. Similarly, $U_x \subset A_3$. Hence, $x$ is in the
   interior of $A_3$.

    \end{proof}
  \vspace{.5cm}

   \begin{theo}\label{qtheo04} If  $(X,R)$ is a wac tournament, then any non-trivial $Q$ invariant subset $A$ of $X$ is clopen.

   If $A_0$ is a non-trivial closed subset of $X$, then for the increasing sequence $\{ A_n \}$ of closed sets with $A_n = Q(A_{n-1} \times A_{n-1})$
   for $n \ge 3, \ \ A_n $ is clopen. Furthermore, for  sufficiently large $n$, $A_n$ is $Q$ invariant and
  so equals  $\overline{\bigcup_n \{ A_n \}}$.\end{theo}

  \begin{proof} If $A$ is $Q$ invariant, then with $A_0 = A$, $A_3 = A$. So Lemma \ref{qlem03a} implies that the closure of $A$ is contained in
  the interior of $A$, i.e. $A$ is clopen.

  If $A_0$ is an arbitrary non-trivial closed set, then for the sequence of closed sets, $A_n$,
  each $A_n \setminus A_0$ is open by Proposition \ref{qprop03}. But from Lemma \ref{qlem03a}, for $n \ge 3$
  $A_n = A_n \setminus A_0 \cup Int A_3$ and so it is open and therefore clopen.

  The union $\bigcup_n \{ A_n \}$ is  $Q$ invariant and so   it is clopen and  equals its closure.
 The sequence $\{ A_n \}$ is an open cover of the closed set $\bigcup_n \{ A_n \}$ and so it has a finite subcover.
 Hence, for large enough $n$, $A_n = \bigcup_n \{ A_n \} = \overline{\bigcup_n \{ A_n \}}$.

     \end{proof}
  \vspace{.5cm}

 Recall that $(X,R)$  is almost wac when   every point of $X$ is either isolated, initial,
  terminal or a cycle point.

   \begin{add}\label{qadd04a} Let   $(X,R)$ be an almost wac tournament. Assume that $A$ is a non-trivial
   $Q$ invariant subset $A$ of $X$. If for $x$ terminal or initial, either $x \in A$ or
   $x \not\in \overline{A}$, then $A$ is clopen.  In particular, if $A$ is closed, then it is clopen. \end{add}

   \begin{proof} By Lemma \ref{qlem03a} if $x$ is an isolated point or cycle point with $x \in \overline{A}$,
   then $x$ is in the interior of $A$.  By assumption on
   $(X,R)$ this applies to every point which is not initial or terminal. If $M$ is a terminal point in  $\overline{A}$ then
   by assumption $M \in A$. If $x \in A \setminus \{ M \}$, then for every point $x' \in R^{\circ}(x) \setminus \{ M \}$ we
   have $x \ha x' \ha M$ and so $x' \in A$ because $A$ is $Q$ invariant. Hence, $M \in R^{\circ}(x) \subset A$ and so
   $M$ is in the interior of $A$. Similarly, for an initial point $m$.

       \end{proof}

      \textbf{ Remark:} Note that if $M$ is a terminal point, then $X \setminus \{ M \}$ is a proper $Q$ invariant subset which
      is not closed unless $M$ is isolated.
  \vspace{.5cm}

 \begin{df}\label{qdef05} For $(X,R)$ a non-trivial compact tournament, a subset $A$ is a \emph{maximal $Q$ invariant subset},
 \index{$Q$ invariant subset!maximal} when it is a proper, closed, $Q$ invariant subset of $X$ such
 that $X$ is the only closed, $Q$ invariant subset which
 properly contains $A$, i.e. $A \subset A'$ with $A'$ a closed, $Q$ invariant subset, then either $A' = A$ or $A' = X$.\end{df}

 Since $\emptyset$ is contained in every singleton and $X$ is non-trivial, $\emptyset$ is never maximal.
  \vspace{.5cm}

  \begin{theo}\label{qtheo06} Let $h : (X_2,R_2) \to (X_1,R_1)$ be a quotient map with $X_1$ non-trivial.

If $A \subset X$ is maximal $Q$ invariant of $X_2$ , then either the
image $h(A)$ equals $X_1$ or else $h(A)$ is a maximal $Q$ invariant of  $X_1$.
  In the latter case, $A = h^{-1}(h(A))$. If $X_1$ has no initial or terminal point, then $h(A)$
  is a maximal $Q$ invariant of  $X_1$.

  If $B$ is a maximal $Q$ invariant subset $B$ of $X_1$, then  any closed $Q$ invariant subset $A$ of $X_2$ which properly
  contains the pre-image $h^{-1}(B)$ maps onto $X_1$, i.e. $h(A) = X_1$. In particular, if $X_1$ has no initial or
  terminal point, then $h^{-1}(B)$ is a maximal $Q$ invariant subset of $X_2$.
 \end{theo}

  \begin{proof} Assume $A$ is maximal. If $B$ is a proper, closed, $Q$ invariant subset of $X_1$ which contains $h(A)$, then $h^{-1}(B)$ is a proper,
  closed $Q$ invariant subset of $X_2$ which contains $A$ and so equals $A$. Applied to $B = h(A)$ when it is a proper subset of $X_1$,
  we obtain $A = h^{-1}(h(A))$. If $X_1$ has no initial or terminal point, then by Theorem \ref{qtheo03}, $h(A)$ is a proper subset of $X_1$.

  If $A$ is a closed, $Q$ invariant set which properly contains $h^{-1}(B)$, then $h(A)$ is a closed, $Q$
  invariant set of $X_1$ which properly contains $B$.  Hence, by maximality
  $h(A) = X_1$. If $X_1$ has no initial or terminal point, then by Theorem \ref{qtheo03} again, $A$ cannot be a proper subset of $X_2$. Hence,
  $h^{-1}(B)$ is maximal.

 \end{proof}
  \vspace{.5cm}

  \begin{theo}\label{qtheo07} For $(X,R)$ a non-trivial, wac tournament, every proper  $Q$ invariant subset is contained in
  a maximal $Q$ invariant subset. \end{theo}

  \begin{proof} By Theorem \ref{qtheo04} every non-trivial $Q$ invariant subset is clopen. If $\{ A_i \}$ is a monotone family of proper
  non-trivial $Q$ invariant subsets, then by Proposition \ref{qprop02}(d) the union $\bigcup \{ A_i \}$ is $Q$ invariant and so is clopen.
  The cover  $\{ A_i \}$ has a finite subcover and so, by monotonicity the union equals $A_i$ for some $i$. This implies that
  the union is proper.  It follows that if $A$ is a proper non-trivial $Q$ invariant subset, we can apply Zorn's Lemma to the family
  of proper  $Q$ invariant subsets which contain $A$ and so obtain a maximal element.

  For a singleton $\{ x \}$ either it is contained in a non-trivial $Q$ invariant subset which is then contained in a maximal subset, or else
  the singleton itself is maximal.

 \end{proof}
  \vspace{.5cm}

  \begin{df}\label{qdef08} A tournament $(Y,P)$ is called a \emph{prime topological tournament}\index{topological tournament!prime}
  \index{prime topological tournament} when it is compact and non-trivial and
  every singleton subset is a maximal $Q$ invariant subset.  That is, $Y$ itself is the
  only closed, non-trivial $Q$ invariant subset.

  If $(X,R)$ is a compact topological tournament, then a surjective map $\pi : (X,R) \to (Y,P)$ of topological tournaments with
  $(Y,P)$ prime is called a \emph{prime quotient map}\index{prime quotient map} and $(Y,P)$ is called a
  \emph{prime quotient}\index{prime quotient}\index{quotient!prime} for $(X,R)$. \end{df}
    \vspace{.5cm}

    An arc $(Y_0,P_0)$, i.e. a tournament on a two point set,   is prime. When $(X,R)$ admits a quotient map onto an arc then we say it has an
    \emph{arc quotient}\index{arc quotient}\index{quotient!arc}.

    \begin{prop}\label{qprop09a} If $(X,R)$ has an  initial or terminal point which is either isolated or contained in a
    non-trivial, proper, closed $Q$ invariant subset, then $(X,R)$ has an arc quotient.  In particular,
    if $(X,R)$ is wac and has an initial or terminal point, then it has an arc quotient. \end{prop}

    \begin{proof} If $M$ is an isolated terminal point, then $M \mapsto 1$ and $x \mapsto 0$ for all $x \not= M$ defines a quotient map to the
    arc on $\{ 0, 1 \}$ with $0 \ha 1$. If $A$ is a non-trivial, proper, closed $Q$ invariant subset which contains $M$, then by Addendum \ref{qadd04a},
     $A$ is clopen.
    $x \mapsto 1$ for $x \in A$ and $x \mapsto 0$ otherwise defines a quotient map to the arc. Similarly for an initial point $m$.

     An initial or terminal point for a wac tournament is isolated by Theorem \ref{wactheo02}(d).

 \end{proof}

\begin{theo}\label{qtheo09b} Assume $(X,R)$ is a non-trivial almost wac tournament with no arc quotient.

(a) Every non-trivial, proper, closed  $Q$ invariant subset is contained in a maximal $Q$ invariant subset which is clopen and does not
contain an initial or terminal point. In particular, if $x$ is an initial or terminal point, then $\{ x \}$ is a
maximal $Q$ invariant set.

(b) $h : (X,R) \to (X_1,R_1)$ is a quotient map with $(X_1,R_1)$ non-trivial, then $(X_1,R_1)$ is almost wac. If $x$ is a
terminal (or initial) point of $X$, then $h(x)$ is terminal (resp. initial) point of $X_2$. If $y$ is a terminal
(or initial) point of $X_1$, then $h^{-1}(y)$ is a singleton $\{ x \}$ with $x$ terminal (resp. initial) point of $X$.

 If $A $ is a proper, closed $Q$ invariant subset of $X$ , then $h(A)$ is a proper, closed $Q$ invariant subset of $X_1$.
\end{theo}

\begin{proof} (a): If $X$ has no initial or terminal point then it is wac and we apply Theorem \ref{qtheo07} directly.
We may assume that $X$ has a terminal point $M$.

Proposition \ref{qprop09a} implies that an initial or terminal point is not contained in a non-trivial proper, closed $Q$
invariant set since $(X,R)$ does not admit an arc quotient.  In particular, an initial or terminal point is not isolated. Hence,
for such a point $x, \ \ \{ x \}$  is a maximal $Q$ invariant subset.

As in the proof of Theorem \ref{qtheo07} we consider $\A = \{ A_i \}$ a monotone family of proper
  non-trivial, closed $Q$ invariant subsets. By Addendum \ref{qadd04a} each $A_i$ is clopen.
  By Proposition \ref{qprop02}(b) and (d) the closure of the union
  $\overline{\bigcup  A_i }$ is $Q$ invariant. When we show it is proper, we can apply Zorn's Lemma as in
  Theorem \ref{qtheo07}. Assume instead that $\overline{\bigcup_i  A_i } = X$. Fix $y \in \bigcap_i A_i$.

  Let $x$ be a point of $ \overline{\bigcup_i  A_i }$.

  If $x \in X$ is isolated, then $x \in \overline{\bigcup_i  A_i }$ implies $x \in \bigcup_i  A_i $. If $x$ is a cycle point,
  we may assume $y \ha x$.  We follow the proof of Lemma \ref{qlem03a}. There exists a
   $3-$cycle $\{ x, x', x'' \}$ contained in $R^{\circ}(y)$.
   By  Theorem \ref{theo2.07}, there exists $\{ U_x, U_{x'}, U_{x''} \}$  a thickening contained in $R^{\circ}(y)$. For some $i$
   there exists  $z \in A_i\cap U_x$. It then follows as in the Lemma that $U_x \subset A_i$ and, in particular, $x \in A_i$.
   Thus, the union contains every point of $X$ except for terminal and initial points.

  Observe first that the terminal point $M$ is not in $A_i$ for any $i$ since these are proper
  non-trivial, closed $Q$ invariant subsets.

  Since $M$ is terminal, $M \in R^{\circ}(y)$. Define the clopen sets
    \begin{displaymath}\begin{split}
  B_i = R^{\circ}(y) \setminus A_i = R(y) \setminus A_i, \\
   B'_i = R^{\circ -1}(y) \setminus A_i = R^{-1}(y) \setminus A_i.
   \end{split} \end{displaymath},

 Because the family $\A = \{ A_i \}$ is monotone, it is directed by inclusion.

  For all $i,$ $M \in B_i$. If $m$ is an initial point, then $m \in B_i'$ for all $i$. If there is no initial point, then
 eventually $B'_i$ is empty because then $\{ A_i \}$ is a covering of $R^{-1}(y)$ and so has a finite subcover.  By monotonicity
 $R^{-1}(y)$ will then be contained in $A_i$ for some $i$.

 Because $A_i$ is $Q$ invariant
  we have for every $z \in B_i$ either $A_i \subset R^{\circ}(z)$ or $A_i \subset R^{\circ -1}(z)$. I claim that for some
  $i_0, \ \ A_i \subset R^{\circ -1}(z)$ for all $z \in B_i$ and for all $A_i \supset A_{i_0}$.

  If not, then for a cofinal collection of  $A_i$'s there exists $z_i \in B_i$
  such that $z_i \ha x$. The only possible limit points of this net are $M$ or an initial point. However, an initial
  point does not lie in the closed set $R(y)$. Hence, the net $\{ z_i \}$ (indexed by the cofinal subset of the monotone family $\A$)
  converges to $M$. Since $z_i \ha x$, this would yield $M \ha x$, violating the condition that $M$ be terminal.
  It then follows that for $A_i \supset A_{i_0}, \ \ x' \ha z$ for all $x' \in A_i$ and $z \in B_i$.

  Now I claim that for some $i_1$ with $A_{i_1} \supset A_{i_0}$, $z' \ha z$ for all $z' \in B'_i$ and $z \in B_i$. This is vacuously
  true if there is no initial point and so $B'_i$ is eventually empty. If not, then
  we can choose $(z_i, z'_i) \in R^{\circ} \cap (B_i \times B'_i)$, indexed again by a cofinal subset of $\A$. The limit
  is $(M,m)$ with $m$ an initial point.  However, this pair is in $R^{\circ -1}$ rather than in $R$.

  Thus, for $A_i \supset A_{i_1}$ we have $x \ha z$ for all $(x,z) \in (X \setminus B_i) \times B_i$. This implies that such $B_i$'s
  are proper clopen $Q$ invariant subsets which contain $M$. This contradicts the assumption that $(X,R)$ has no arc quotient.

  This contradiction
  implies that $\overline{\bigcup_i  A_i }$ is a proper subset of $X$.  So, at long last, we may apply Zorn's Lemma and show that
 every non-trivial, proper, closed  $Q$ invariant subset is contained in a maximal $Q$ invariant subset. Since the maximal $Q$ invariant
 subsets are closed, they are clopen by Addendum \ref{qadd04a}.

   (b):   By Theorem \ref{wactheo02}(f), $(X_1.R_1)$  is almost wac when $(X,R)$ is. Since $(X,R)$ does not have an arc quotient, $(X_1,R_1)$
       does not. Hence an initial point or terminal point in $X_1$ is not isolated. It follows from Theorem \ref{wactheo02}(f) again
       that the pre-image of a terminal point (or initial point) is a singleton terminal point (resp. a singleton initial point).

      If $A$ is a proper, closed $Q$ invariant subset of $X$, then by Theorem \ref{qtheo03} $h(A)$ is a $Q$ invariant subset of $X_1$.  If
      $A$ is a singleton, then since $X_1$ is non-trivial, $h(A)$ is a proper subset. So we may assume $A$ is non-trivial and so it is
      clopen by Addendum \ref{qadd04a}. If
      $h(A) = X_1$, then the proof of Theorem \ref{qtheo03} shows that for $x \not\in A, \ \ h(x)$ is either an initial or terminal point of
      $X_1$. We have seen that the pre-image of an initial or terminal point is a singleton.  This would imply that
      $X \setminus A$ consists of at least one and at most two points and these are isolated.  Thus, $X$ would have an isolated
      initial or terminal point and so $(X,R)$ would have an arc quotient.

 \end{proof}

  \begin{theo}\label{qtheo09}(a) A compact, non-trivial tournament $(Y,P)$ is prime if and only if whenever
  $h : (Y,P) \to (Z,T)$ is a quotient map with $(Z,T)$ non-trivial, $h$ is a homeomorphism and so
  is an isomorphism from $(Y,P)$ to $(Z,T)$.

  (b) Assume that $h : (X_1,R_1) \to (X_2,R_2)$ is a quotient map. Let
  $(Y,P)$ is a prime tournament such that either $Y$ had no initial or terminal point, or else
  $(Y,P)$ is almost wac but not an arc. If $\pi_1 : (X_1,R_1) \to (Y,P)$ is a
  quotient map, then $\pi_1$ factors through $h$ to uniquely
  define the continuous surjection $\pi_2 : X_2 \to Y$ such that $\pi_2 \circ h = \pi_1$. Furthermore, $\pi_2 : (X_2,R_2) \to (Y,P)$ is a
  quotient map.

  (c) If a compact tournament $(X,R)$ admits a prime quotient map $\pi : (X,R) \to (Y,P)$ such that either $(Y,P)$ has no initial nor
  terminal point or else $(Y,P)$ is almost wac but not an arc, then $\pi$ is unique up to isomorphism.
   That is, if $\pi_1 : (X,R) \to (Y_1,P_1)$ is a prime quotient map,
  then there exists a homeomorphism $h : Y_1 \to Y$ such that $\pi = h \circ \pi_1$ and $h : (Y_1,P_1) \to (Y,P)$ is
  a tournament isomorphism.
  \end{theo}

  \begin{proof} (a): Assume $(Y,P)$ is prime.  For $y \in Y$, $\{ y \}$ is a maximal $Q$ invariant subset and so by
  Theorem \ref{qtheo06}, $h(y) \not= Z$ (because $Z$ is not trivial) implies $y = h^{-1}(h(y))$ and so $h$ is bijective.
  By compactness it is a homeomorphism and so is an isomorphism from $(Y,P)$ to $(Z,T)$.

  If $(Y,P)$ is not prime, then it contains a proper, closed $Q$ invariant subset $A$. The projection $\pi_A$ obtained by smashing $A$
  to a point as in Theorem \ref{qtheo03} provides a quotient map on $(Y,P)$ which is not an isomorphism.

  (b): For $y \in X_2, \ h^{-1}(y)$ is closed and $Q$ invariant in $X_1$ and so $\pi_1(h^{-1}(y))$ is closed and $Q$ invariant in $Y$, see
  Theorem \ref{qtheo03}. Because $Y$ is prime either $\pi_1(h^{-1}(y))$  is a singleton or else $\pi_1(h^{-1}(y)) = Y$. By Theorem \ref{qtheo03}
  $\pi_1(h^{-1}(x))$ is a proper subset of $Y$ if it has no initial nor terminal point. By Theorem \ref{qtheo09b}
  $\pi_1(h^{-1}(x))$ is a proper subset of $Y$ if $(Y,P)$ is wac and is not an arc (and so does not have an arc quotient by (a)).
  By assumption on $Y$ it follows that $\pi_1(h^{-1}(y))$ is a singleton for
  every $y \in X_2$. Hence, there is a, necessarily unique, map $\pi_2 : X_2 \to Y$ such that $\pi_2 \circ h = \pi_1$. By compactness
  $\pi_2$ is continuous.  It is clearly surjective and maps $R_2$ to $P$.

  (c): The existence of the continuous surjection $h$ follows from (b).
  Since $(Y_1,P_1)$ is prime, (a) implies that $h$ is an isomorphism.

 \end{proof}

    \begin{prop}\label{qprop09aa} If $(Y,P)$ is  a prime topological tournament, then $Y$ is totally disconnected. \end{prop}

    \begin{proof} From Theorem \ref{theo2.01} it follows that any component $A$ of $Y$ is a
    $Q$ invariant subset on which $P$ restricts to an order.
    If $A$ were not trivial, then it would contain a non-trivial proper subinterval, $B$, which is itself a closed, $Q$ invariant subset.
   This would imply that $(Y,P)$ is not prime.

 \end{proof}

  \begin{theo}\label{qtheo10} Assume that $\pi : (X,R) \to (Y,P)$ is a prime quotient map with $(X,R)$ a wac tournament with no arc quotient.
  \begin{itemize}
  \item[(a)]  The tournament $(Y,P)$ is wac with no initial or terminal point.

  \item[(b)]  If $(X,R)$ is arc cyclic, then $(Y,P)$ is arc cyclic.

  \item[(c)] If $y$ is an isolated point of $Y$, then $\pi^{-1}(y)$ is a maximal $Q$ invariant subset of $X$.

  \item[(d)]  If $y$ is a non-isolated point of $Y$, then $\pi^{-1}(y)$ is a singleton subset $\{ x \}$
  which is a maximal $Q$ invariant subset of $X$. The point $x$ is non-isolated in $X$ and so the points
  $x$ and $y$ are cycle points.
  \end{itemize}
  \end{theo}

  \begin{proof} (a), (b): By Theorem \ref{wactheo02}(f) , $(Y,P)$ is wac.  If it had an initial or terminal point, then it would have an
  arc quotient by Proposition \ref{qprop09a}.  Composing with $\pi$ we would obtain an arc quotient map for $(X,R)$, contra assumption.

  Corollary \ref{cor1.01b} implies that $(Y,P)$ is arc cyclic when $(X,R)$ is.

  (c), (d): The set $\pi^{-1}(y)$ is a proper $Q$ invariant subset of $X$ by Theorem \ref{qtheo03}. Let $A$ be a proper $Q$ invariant subset of $X$
  which contains $\pi^{-1}(y)$. By Theorem \ref{qtheo03}, again, $\pi(A)$ is a proper $Q$ invariant subset of $Y$ which
  contains $y$. Because $Y$ is prime, it equals $\{ y \}$ and so $A = \pi^{-1}(y)$. Thus, $\pi^{-1}(y)$ is maximal.
  The remaining results follow from Theorem \ref{wactheo02}(f)  again.

 \end{proof}

 \begin{theo}\label{qtheo11} Assume $(X,R)$ is a almost wac tournament which does not have an arc quotient.
 The tournament $(X,R)$ admits a prime quotient map,
 unique up to isomorphism.

 If $(X,R)$ is arc cyclic, then it does not have an arc quotient and its prime quotient is arc cyclic.\end{theo}

 \begin{proof} If $(X,R)$ contains no proper, non-trivial $Q$ invariant subset, then it is already prime and the identity is a
 prime quotient map. Otherwise Theorem \ref{qtheo09b}(a) implies there exists a maximal, non-trivial $Q$ invariant subset $A$.
 We can use Zorn's Lemma to obtain a
 maximal collection $\A = \{ A_i \}$ of pair-wise disjoint, maximal non-trivial $Q$ invariant subsets which contains $A$.
 By Theorem \ref{qtheo09b}(a) the sets $A_i$ are clopen. Since $(X,R)$ does
 not have an arc quotient, no $A_i$ contains a terminal or initial point
by Proposition \ref{qprop09a}.
 \vspace{.25cm}

 Claim: Let $d$ be any continuous pseudo-metric on $X$. For any $\ep > 0$, there are only finitely many $A_i$'s with $d-$diameter greater than $\ep$.
\vspace{.25cm}

 \begin{proof} If not
 we can choose two sequences $\{ x_n \}, \{ y_n \}$ such that
 \begin{itemize}
 \item $x_n, y_n \in A_{i_n}$ with $A_{i_n}$ disjoint from $A_{i_m}$ when $n \not= m$.

 \item $d(x_n,y_n) \ge \ep$

 \item The sequence of pairs $\{ (x_n,y_n) \}$ has a limit point $(x,y)$  with $d(x,y) \ge \ep$ and so we may assume $x \in R^{\circ}(y)$.
  \end{itemize}

   First we eliminate the possibility that $y$ is initial and $x$ is terminal. Were this so, then for
   any $z \in X \setminus \{ x, y \}$, e.g. $z \in A$, we would have $(x,y) \in R^{\circ}(z) \times R^{\circ -1}(z)$. Then for
   infinitely many $n, \ \ (x_n,y_n) \in R^{\circ}(z) \times R^{\circ -1}(z)$ which implies $z \in A_{i_n}$ since $A_{i_n}$ is $Q$ invariant.
   This is impossible since the $A_n$'s are pairwise disjoint.

   Now assume that $x$ is not terminal. We follow the proof of Lemma \ref{qlem03a}. As $x$ is not initial
   because $y \ha x$ and it is not isolated as it is a limit point, it is
   a cycle point and we may choose a $3-$cycle $\{ x, x', x'' \}$ contained in $R^{\circ}(y)$.
 We thicken $\{y,  x, x', x'' \}$ to $\{ U_y,  U_x, U_{x'}, U_{x''} \}$. Then for
   infinitely many $n, \ \  (x_n, y_n) \in U_x \times U_y$ . For any $z \in U_{x''}$, $z \in R^{\circ}(y_n) \cap R^{\circ -1}(x_n)$. Because these
   $A_{n}$'s are  $Q$ invariant, we have $z \in A_{n}$ for infinitely many  $n$.  Again, this is impossible.

 We use a similar argument if $y$ is not initial and so is a cycle point.

 \end{proof}
\vspace{.25cm}

  Now define the equivalence relation $E_{\A} = 1_X \cup \bigcup_i \{ A_i \times A_i \}$. If $\{ (x_k,y_k) \}$ is a net in
  $E_A$ converging to $(x,y)$, then either for some cofinal set of indices $k$ the pairs $(x_k,y_k)$
  lie in a single $A_i \times A_i$, in which case $(x,y) \in E_{\A}$
  since the $A_i$'s are closed, or else eventually the  sequence leaves any finite
  collection of $A_i$'s. The Claim implies that for any continuous pseudo-metric $d \ \ \lim_n \ d(x_n,y_n) = 0$ and so $d(x,y) = 0$.
 Since the continuous pseudo-metrics generate the unique uniformity on $X$, see \cite{K} Chapter 6, it follows that
   $x = y$ and so $(x,y) \in E_{\A}$.

  Now as in Theorem \ref{qtheo03} define the quotient map $\pi$ from $ X$ to $X_{\A}$, the space of $E_{\A}$ equivalence classes with the quotient topology.
  Let $R_{\A} = (\pi \times \pi)(R)$ and as before it is a topological tournament with $\pi : (X,R) \to (X_{\A},R_{\A})$ a continuous surjective
  tournament map. Since each $A_i$ is proper, $X_{\A}$ is non-trivial and by Theorem \ref{qtheo09b}(b)
  it is almost wac.  Since $(X,R)$ has no arc quotient,
  neither does $(X_{\A},R_{\A})$.

  By definition of the quotient topology, $A_i = \pi^{-1}(\pi(A_i))$ clopen implies that the point
  $\pi(A_i)$ is an isolated point. Because $A_i$ is maximal, Theorem \ref{qtheo06} and Theorem \ref{qtheo09b}(b) implies that $\{ \pi(A_i) \}$ is maximal.

  Now assume that $B$ is a proper, closed $Q$ invariant subset of $X_{\A}$. If $B$ contains some $\pi(A_i)$, then it equals
  $\{ \pi(A_i) \}$ by maximality.

  Now suppose that $B$ is disjoint from all of the $\pi(A_i)$'s. Assume $B$ is non-trivial, so that  $\pi^{-1}(B)$ is  a non-trivial proper $Q$ invariant
  subset of $X$. Hence, it is contained in a maximal $Q$ invariant subset $A'$.  It follows from Theorem \ref{qtheo06} and
  Theorem \ref{qtheo09b}(b) again that $\pi(A')$ is a
   maximal $Q$ invariant subset of $X_{\A}$. As it contains
  $B$, it is not equal to any $\pi(A_i)$ and so is disjoint from all of them. It follows that $A'$ is a non-trivial maximal $Q$ invariant subset
  of $X$ which is disjoint from all the $A_i$'s.  This contradicts the maximality of the family $\{ A_i \}$. So we see that $B$ had to be trivial.

  Thus, every proper $Q$ invariant subset of $X_{\A}$ is a singleton and so $(X_{\A},R_{\A})$ is prime.

  Uniqueness up to isomorphism follows from Theorem \ref{qtheo09}.

  If $(X,R)$ is arc cyclic, then every quotient is arc cyclic by Corollary \ref{cor1.01b}. Since the arc is not arc cyclic, $(X,R)$ does not
  have an arc quotient and its prime quotient is arc cyclic.

   \end{proof}
  \vspace{.5cm}

  From the proof we obtain the following.

   \begin{add}\label{qadd12} If $(X,R)$ is an almost wac tournament which does not
   admit an arc quotient, then the maximal $Q$ invariant subsets
   of $X$ are the elements of $\A = \{ A_i \}$ together with the singletons $\{ x \}$ for $x \in X \setminus \bigcup_i \{ A_i \}$.
   In particular, any two distinct maximal $Q$ invariant subsets are disjoint.\end{add}

   \begin{proof} Let $A'$ be an arbitrary non-trivial maximal $Q$ invariant subset and let $\A'$ be a maximal collection of pair-wise disjoint
   non-trivial maximal $Q$ invariant subsets which includes $A'$. We obtain the
   prime quotient map $\pi' : X \to X_{\A'}$ as before.  By Theorem \ref{qtheo09} we obtain
   the homeomorphism $h : X_{\A'} \to X_{\A}$ such that $\pi = h \circ \pi'$. It follows that $(\pi')^{-1}(h^{-1}(z)) = \pi^{-1}(z)$
   for all $z \in X_{\A}$. Hence, $\A = \A'$ and so $A' \in \A$. We saw in the proof above that for each $x \not\in \bigcup_i A_i$
   the singleton $\{ x \}$ is maximal.

   \end{proof}
  \vspace{.5cm}

  Now we consider what happens when a wac tournament has an arc quotient.

    \begin{theo}\label{qtheo13} If a wac tournament $(X,R)$ has an arc quotient, then any prime quotient of $(X,R)$ is an arc. \end{theo}

    \begin{proof}  Suppose that $\pi_0 : (X,R) \to (Y_0,P_0)$ and $\pi_1 (X,R) \to (Y_1,P_1)$ are prime quotient maps with $(Y_0,P_0)$ an
    arc.

    We first observe that $(Y_1,P_1)$ has an initial or a terminal point.  Were this not true then by Theorem \ref{qtheo09}(b) there would
    be a surjective map from $(Y_0,P_0)$ onto $(Y_1,P_1)$. Since $|Y_0| = 2$ and $Y_1$ is non-trivial, the map would have to be an isomorphism
    and so $(Y_1,P_1)$ would have both an initial and a terminal point.

    Since $(Y_1,P_1)$ is wac, the existence of an initial or terminal point implies that $(Y_1,P_1)$ has a quotient map onto an arc $(Y_3,P_3)$.
    Since $(Y_1,P_1)$ is prime, this quotient map is an isomorphism and so $(Y_1,P_1)$ is an arc.

   \end{proof}
  \vspace{.5cm}

  Thus, in any case, a wac  tournament has a prime quotient, unique up to isomorphism.  However, when the tournament has an arc quotient,
  the quotient map need not be unique up to isomorphism.

   \begin{df}\label{qdef14}Suppose that $L$ is a linear order on a non-trivial finite set $I$
   so that $(I,L)$ is a non-trivial finite, transitive tournament. A
   quotient map $\pi: (X,R) \to (I,L)$ is called an \emph{order quotient map}\index{order quotient map}\index{quotient!order}.
   It is called a \emph{maximum order quotient map}\index{maximum order quotient map}\index{order quotient map!maximum} when for
   each $i \in I$, the restriction $R|\pi^{-1}(i)$ does not have an arc quotient. \end{df}
  \vspace{.5cm}

  \begin{theo}\label{qtheo15} Assume that for a compact tournament $(X,R), \ \ \pi: (X,R) \to (I,L)$ and $\pi_1: (X,R) \to (I_1,L_1)$
  are order quotient maps with $\pi$ maximum. There exists a surjective tournament map (i.e. an order-preserving surjection)
  $h : (I,L) \to (I_1,L_1)$ such that $\pi_1 = h \circ \pi$. In particular, $|L| \ge |L_1|$.

  If $\pi_1$ is also maximum, then $h$ is an isomorphism and so $|I| = |I_1|$.  Conversely, $h$ is an isomorphism if $|I| = |I_1|$.

   Thus, the maximum quotient map, if it exists, is unique up to isomorphism.\end{theo}

  \begin{proof} If $\pi_1(\pi^{-1}(i))$ contains more than one point for any $i \in I$, then clearly, the restriction $R|\pi^{-1}(i)$
  admits an arc quotient. Because $\pi$ is maximum, it follows that each $\pi_1(\pi^{-1}(i))$ is a singleton and so the map $h$ is defined as usual.

  Since $h$ is a surjection between finite sets, it follows that $|I| \ge |I_1|$. If $\pi_1$ is also maximum, then $|I_1| \ge |I|$
  and so $|I| =|I_1|$. Since $h$ is a surjection, it is a bijection (and so an isomorphism) if and only if $|I| = |I_1|$.

   \end{proof}
  \vspace{.5cm}

  \begin{theo}\label{qtheo16} If $(X,R)$ is a wac tournament which admits an arc quotient, then it admits a maximum order quotient map, unique
  up to isomorphism. \end{theo}

  \begin{proof} Begin with $\pi_0 : (X,R) \to (I_0,L_0)$, with $(I_0,L_0)$ an arc. If it is not maximum, then we construct, inductively, a finite or
  infinite sequence of finite orders $(I_k,L_k)$ surjective, but not bijective tournament maps $f_k : (I_{k+1},L_{k+1})) \to (I_k,L_k)$ and
  continuous tournament surjections $\pi_k : (X,R) \to (I_k,L_k)$ such that $f_k \circ \pi_{k+1} = \pi_k$.

  If at stage $k$, the map $\pi_k$ is not maximum, then for some $i \in L_k$, the restriction $(\pi_k^{-1}(i), R|\pi_k^{-1}(i))$ admits
  an arc quotient.  It is easy to see that we can split the point $i$, to obtain $f_k :(I_{k+1},L_{k+1}) \to (I_k,L_k)$ with $|f_k^{-1}(i)| = 2$
  and a lift $\pi_{k+1} :  (X,R) \to (I_{k+1},L_{k+1})$.

  This process terminates when $\pi_k$ is maximum.

  In fact it must terminate.  If it did not, then $\{ (I_k,L_k,f_k) \}$ would be an inverse system of finite, transitive tournaments with
  the inverse limit $(I,L)$ an infinite compact LOTS. Furthermore, the maps $\pi_k$ would induce a quotient map
  $\pi : (X,R) \to (I,L)$. A LOTS has no cycle points. An infinite compact LOTS has some non-isolated points. Since the continuous surjective
  image of a wac tournament is wac, it follows that a wac tournament cannot map onto an infinite compact LOTS.

   \end{proof}

   Notice that $X \ = \ \{ 0 \} \cup \{ 1/n : n \in \N \}$ with the order $L$ inherited from $\R$ is an almost wac tournament which does not
   admit a maximum order quotient.

\vspace{1cm}

  \section{ \textbf{Classification of WAC Tournaments}}\vspace{.5cm}

  We first separate the class of wac tournaments into three types \index{tournament!wac!types}\index{wac tournament!types}.
  \begin{itemize}
  \item  Type 1: $(X,R)$ is Type 1 when it is non-trivial and does not have an arc quotient.

  \item Type 2:  $(X,R)$ is Type 2 when it has an arc quotient (and so is non-trivial).

  \item Type 3: $(X,R)$ is Type 3 when it is trivial.
  \end{itemize}

 Each wac tournament has a so-called \emph{base quotient map} \index{quotient map!base} a continuous, surjective
 surjective tournament map $\pi : (X,R) \to (Y,P)$ which is unique up to isomorphism. If $(X,R)$ is Type 1, then its base quotient map is its
prime quotient map, as in Theorem \ref{qtheo11}.  If $(X,R)$ is Type 2, then the base quotient map is its
  maximum order quotient map, as in Theorem \ref{qtheo16}.  If $(X,R)$ is Type 3, then the base quotient map is the isomorphism onto any trivial
  tournament.

 For a quotient map $h :(X,R) \to (Y,P)$ of topological tournaments, if $(X,R)$ is wac, then $(Y,P)$ is wac by Theorem \ref{wactheo02}(f)
 and for every  $y \in Y$, the restriction $R|h^{-1}(y)$ is wac because if the $Q$ invariant set $h^{-1}(y)$ is not trivial, then it is clopen.

 \begin{df}\label{classdef01} For a wac tournament $(X,R)$ the \emph{classifier system}\index{wac tournament!classifier system}
 \index{tournament!wac!!classifier system}\index{inverse system!classifier system} is an inverse system $\{(X_i,R_i,f_i) \}$ of
 topological tournaments, together with quotient maps $h_i : (X,R) \to (X_i,R_i)$ which satisfy the
 following properties.
 \begin{itemize}
 \item[(i)]  $h_i = f_i \circ h_{i+1}$.
 \item[(ii)] $h_1 : (X,R) \to (X_1,R_1)$ is a base quotient map.
 \item[(iii)] For each $x_i \in X_i$, the restriction $(h_i^{-1}(x_i),R|h^{-1}_i(x_i))$ is a wac tournament and
 the map $$h_{i+1} : (h_i^{-1}(x_i),R|h^{-1}_i(x_i)) \to (f_i^{-1}(x_i),R_{i+1}|f^{-1}_i(x_i))$$ is a base quotient map.
 \end{itemize}
 \end{df}
 \vspace{.5cm}

  \begin{theo}\label{classtheo02} A wac tournament $(X,R)$ admits a classifier system $\{(X_i,R_i,f_i) \}$ with maps $\{ h_i \}$.

  If  $\{(X_i,R_i,f_i) \}$ is a classifier system with maps $\{ h_i \}$, then the map $h : (X,R) \to \overleftarrow{Lim} \{ (X_i,R_i,f_i) \}$
  given by $h(x)_i = h_i(x)$ is an isomorphism.

  If $\{(X_i',R_i',f_i') \}$ with maps $\{ h_i' \}$ is another classifier system for $(X,R)$, then there exist isomorphisms
  $q_i : (X_i,R_i) \to (X_i',R_i')$ such that for all $i$:
  \begin{equation}\label{classeq01}
   f'_i \circ q_{i+1} \ = \ q_i \circ f_i, \quad \text{and} \quad q_i \circ h_i \ = \ h_i'.
   \end{equation}
   \end{theo}

   \begin{proof} Begin with $h_1 : (X,R) \to (X_1,R_1)$ a base quotient map.
Assume that $(X_i,R_i)$ with maps $h_i$ have been constructed for $i \le n$ and with projection $f_i$ for $i \le n-1$ so that
conditions (i) and (iii) of Definition \ref{classdef01} hold for $i \le n-1$.

If $x_n \in X_n$ is a non-isolated point, then by Theorem \ref{wactheo02}(f) , $h_n^{-1}(x_n)$ is a singleton set and
so $(h_n^{-1}(x_n),R|h^{-1}_n(x_n))$ is a type 3, trivial, wac. If $x_n$ is isolated, then $h_n^{-1}(x_n)$ is clopen and
$(h_n^{-1}(x_n),R|h^{-1}_n(x_n))$ is a wac. For any $x_n$, let
$\pi_{nx_n}: (h_n^{-1}(x_n),R|h^{-1}_n(x_n)) \to (Y_{nx_n},P_{nx_n})$ be a base quotient map.
Let $(X_{n+1},R_{n+1}) = (X_n,R_n) \ltimes \{ (Y_{nx_n},P_{nx_n}) \}$ be the lexicographic product. It is a topological lexicographic product
because $(Y_{nx_n},P_{nx_n})$ is trivial whenever $x_n$ is non-isolated in $X_n$. The map $f_n$ is the first coordinate projection
for the lexicographic product. The map $h_{n+1}$ is defined by
 \begin{equation}\label{classeq02}
  h_{n+1}(x) \ = \ (h_n(x),\pi_{nh_n(x)}(x)) \quad \text{for all} \ \ x \in X.
  \end{equation}

  Because $h_n$ and each $\pi_{xn}$ is a quotient map, it easily follows that $h_{n+1}$ is a quotient map.
  It follows from Theorem \ref{wactheo02}(f) again that $(X_{n+1},R_{n+1})$ is a wac tournament.
  Finally, conditions (i) and (iii) of Definition \ref{classdef01} hold for $i = n$.

  Thus, by inductive construction, we obtain  $\{(X_i,R_i,f_i) \}$  a classifier system with maps $\{ h_i \}$.

  Clearly, the map $h : (X,R) \to \overleftarrow{Lim} \{ (X_i,R_i,f_i) \}$ is a quotient map and so the inverse limit space is wac.  It suffices
  to prove that $h$ is injective, i.e. the pre-image of every point is a singleton. If $x \in \overleftarrow{Lim} \{ (X_i,R_i,f_i) \}$  is
  non-isolated, this follows from Theorem \ref{wactheo02}(f) again. In particular, if  $x_n \in X_n$ is non-isolated for any
  $n$, then with $g_i : \overleftarrow{Lim} \{ (X_i,R_i,f_i) \} \to (X_i,R_i)$
  is the projection given by $x \mapsto x_i$, the sets $g_n^{-1}(x_n)$ and $h_n^{-1}(x_n) = h^{-1}(g_n^{-1}(x_n))$ are singletons consisting
  of non-isolated points by Theorem \ref{wactheo02}(f) yet again.

  Now assume that $x \in \overleftarrow{Lim} \{ (X_i,R_i,f_i) \}$  is isolated so that $x_i$ is isolated in $X_i$ for all $i$.
  For any $x \in \overleftarrow{Lim} \{ (X_i,R_i,f_i) \}, \ \  \{ g_i^{-1}(x_i)\} $ is a decreasing sequence of closed sets with intersection
  $\{ x \} $. When $x$ is isolated, $\{ x \}$ is clopen and so the sequence stabilizes and so for some $n$ , $g_i^{-1}(x_i) = \{ x \}$ for
  all $i \ge n$. In particular, $(Y_{nx_n},P_{nx_n})$ is trivial because $g_{n+1}$ maps the singleton $\{ x \}$ onto
  $f_n^{-1}(x_n) = \{ x_n \} \times Y_{nx_n}$. Now by construction $\pi_{nx_n} : (h_n^{-1}(x_n),R|h_n^{-1}(x_n)) \to (Y_{nx_n},P_{nx_n})$
  is a base quotient map. Since $(Y_{nx_n},P_{nx_n})$ is trivial, the tournament $(h_n^{-1}(x_n),R|h_n^{-1}(x_n))$ is type 3 and so
  $h_n^{-1}(x_n)$, which contains (and so equals) $h^{-1}(x)$ is a singleton.

  Thus, we can use $h$ to identify $(X,R)$ with the inverse limit \\ $\overleftarrow{Lim} \{ (X_i,R_i,f_i) \}$ so that $h_i$ is identified with
  the projection map $g_i$.  That is, for $x \in X$, $h_i(x) = x_i$ in $X_i$.

Given two different classifiers, the construction of the maps $\{ q_i \}$ is an obvious induction
using the uniqueness up to isomorphism of the base quotient maps.
Observe that by Theorem \ref{wactheo02}(f), the quotient maps $f_i$ and $f'_i$ are isomorphic to projections from topological lexicographic products
and so they are open maps. Hence,
continuity of $q_{i+1}$ follows from that of $q_i$.

\end{proof}

\textbf{Remark:} Thus, the classifier system for a wac tournament is unique up to isomorphism.

If $(X,R)$ is an almost wac tournament which does not have an arc quotient, so that an initial or terminal point is non-isolated,
then $(X,R)$ also has a classifier, unique up to isomorphism.

\vspace{.5cm}

  \begin{add}\label{classadd03} For a wac tournament $(X,R)$ the topological inverse system of topological lexicographic products
used to construct the classifier satisfies the following properties for every point $x$ in $X$ (which we identify with
the inverse limit space) and for every $n \in \N$, .
\begin{itemize}
\item[(i)] If  $(Y_{(n-1)x_{n-1}},P_{(n-1)x_{n-1}})$ is
a non-trivial, finite linear order, then $(Y_{nx_n},P_{nx_n})$ does not have an arc quotient.

\item[(ii)] If  $(Y_{(n-1)x_{n-1}},P_{(n-1)x_{n-1}})$ is
a trivial tournament, then $(Y_{nx_n},P_{nx_n})$ is a trivial tournament.
\end{itemize} \end{add}

\begin{proof} (i): If  $(Y_{(n-1)x_{n-1}},P_{(n-1)x_{n-1}})$ is
a non-trivial, finite linear order, then
$$\pi_{(n-1)x_{n-1}}: (h_{n-1}^{-1}(x_{n-1}),R|h^{-1}_{n-1}(x_{n-1})) \to (Y_{(n-1)x_{n-1}},P_{(n-1)x_{n-1}})$$
is a maximum order quotient and with $x_n = (x_{n-1},y)$ we
have $h_n^{-1}(x_n) = (\pi_{(n-1)x_{n-1}})^{-1}(y)$. The restriction of $R$ to this
set does not have an arc quotient by definition of the maximum order quotient.
Since $\pi_{nx_n} : (h_n^{-1}(x_n),R|h_n^{-1}(x_n)) \to (Y_{nx_n},P_{nx_n})$
  is a base quotient map, it follows that $(Y_{nx_n},P_{nx_n})$ does not have an arc quotient.

  (ii): If  $(Y_{(n-1)x_{n-1}},P_{(n-1)x_{n-1}})$ is
trivial, i.e. type 3, then since
$$\pi_{(n-1)x_{n-1}}: (h_{n-1}^{-1}(x_{n-1}),R|h^{-1}_{n-1}(x_{n-1})) \to (Y_{(n-1)x_{n-1}},P_{(n-1)x_{n-1}})$$
is a base quotient map, it follows that $(h_{n-1}^{-1}(x_{n-1}),R|h^{-1}_{n-1}(x_{n-1}))$ is trivial and so also is $(Y_{nx_n},P_{nx_n})$.
 If $(X_1,R_1)$ is trivial, then $(X,R)$ is trivial and so every $(Y_{nx_n},P_{nx_n})$ is trivial.

 \end{proof}
 \vspace{.5cm}

\begin{cor}\label{classcor04} Let $(X,R)$ be a wac tournament with  classifier system $\{(X_i,R_i,f_i) \}$ with maps $\{ h_i \}$.
The following conditions are equivalent.
\begin{itemize}
\item[(i)]  The tournament $(X,R)$ is arc cyclic.

\item[(ii)] For every $i \in \N$ the tournament $(X_i,R_i)$ is arc cyclic.

\item[(iii)] The base tournament $(X_1,R_1)$ is arc cyclic and for every $x \in X$ and $i \in \N$,
the restriction $(h_i^{-1}(x_i),R|h_i^{-1}(x_i))$ has an arc cyclic base quotient.

\item[(iv)] For every non-trivial $Q$ invariant subset $A$ of $X$, the restriction $(A,R|A)$ has an arc cyclic prime quotient.
\end{itemize}

The following conditions are equivalent.
\begin{itemize}
\item[(i')]  The tournament $(X,R)$ is locally arc cyclic.

\item[(ii')] There exists $i \in \N$ such that $(X_i,R_i)$ is locally arc cyclic and for every $x \in X$ and $j \ge i$,
the restriction $(h_j^{-1}(x_j),R|h_j^{-1}(x_j))$ has an arc cyclic base quotient.

\item[(iii')] There exists $i \in \N$ such that such that $(X_i,R_i)$ is locally arc cyclic and for every $x \in X$
the restriction $(h_i^{-1}(x_i),R|h_i^{-1}(x_i))$ is arc cyclic.
\end{itemize}
\end{cor}

\begin{proof} (i) $\Rightarrow$ (iv):  By Proposition \ref{qprop02}(e) the restriction $R|A$ is arc cyclic when
$R$ is arc cyclic and $A$ is $Q$ invariant. Hence, when $A$ is non-trivial it has a unique prime quotient which is arc cyclic by
Corollary \ref{cor1.01b}.

(iv)  $\Rightarrow$ (iii): Since $X$ itself is $Q$ invariant, the base quotient $(X_1,R_1)$ is arc cyclic.
Since a trivial tournament is vacuously arc cyclic and  the restriction to any of
 $Q$ invariant subsets $h_i^{-1}(x_i)$ has an arc cyclic base quotient.

(iii) $\Rightarrow$ (ii): By uniqueness we may assume that $\{(X_i,R_i,f_i) \}$ is given by the inductive construction in the proof of
Theorem \ref{classtheo02}. Proceed by induction.  By assumption the base quotient $(X_1,R_1)$ is arc cyclic.

Now assume that $(X_n,R_n)$ is arc cyclic. For every $x \in X$, the base for  $(h_n^{-1}(x_n),R|h_n^{-1}(x_n)$ is
$(Y_{nx},P_{nx})$ which is an arc cyclic by assumption.  From Corollary \ref{cor1.01b} applied to $f_n$ it follows $(X_{n+1},R_{n+1})$ is
arc cyclic.

(ii)  $\Rightarrow$ (i):  By Proposition \ref{propinv01a} the inverse limit of an inverse system of  arc cyclic tournaments is arc cyclic.

(i') $\Leftrightarrow$ (iii'): If $(X_i,R_i)$ is locally arc cyclic, and each of the restrictions is arc cyclic, then
$(X,R)$ is locally arc cyclic by Corollary \ref{arccyccor03} applied to the quotient map $h_i$.

Now assume that $(X,R)$ is locally arc cyclic. Let $\{ U_1,\dots, U_m \}$ be an open cover of $X$ by arc cyclic subsets. Define $G_j =
\{ x \in X : h_j^{-1}(x_j) \subset U_t$ for some $t = 1, \dots, n \}$. Since
$$G_j \ = \ \bigcup_{t=1}^{n} \ X \setminus h_j(X \setminus U_t)$$
it follows that $G_j$ is open. Since $h_{j+1}^{-1}(x_{j+1}) \subset h_j^{-1}(x_j)$ it follows that $G_j \subset G_{j+1}$.
Because $x = \bigcap_j h_j^{-1}(x_j)$ it follows that each $x$ is contained in some $G_j$.

From compactness, it follows that for some $i \ \ G_i = X$. That is, each $h_i^{-1}(x)$ is a $Q$ invariant arc cyclic subset and so
the restriction $(h_i^{-1}(x_i),R|h_i^{-1}(x_i))$ is arc cyclic. By Corollary \ref{arccyccor03} again it follows that $(X_i,R_i)$ is locally arc cyclic.

(iii') $\Leftrightarrow$ (ii'): It is clear that we can construct the
classifier for $(h_i^{-1}(x_i),R|h_i^{-1}(x_i))$ by starting with $(f_i^{-1}(x_i),R_{i+1}|f_i^{-1}(x_i))$, which is isomorphic to
the  base for $(h_i^{-1}(x_i),R|h_i^{-1}(x_i))$ and then by using the same choices for the restricted lexicographic
construction which uses the bases for the restrictions  $(h_j^{-1}(z_j),R|h_j^{-1}(z_j))$ with $z \in h_i^{-1}(x_i)$.
The equivalence then follows from the equivalence of (i) and (iii) applied to $(h_i^{-1}(x_i),R|h_i^{-1}(x_i))$.

 \end{proof}
 \vspace{.5cm}

 Above we began with a wac tournament.  Now we would like to build the classifier system directly, achieving the wac tournament as the limit.

  \begin{lem}\label{classlem05a} Let  $(X_2,R_2)$ be the topological lexicographic product $(X_1,R_1) \ltimes \{ (Y_x,S_x) : x \in X \}$.
  The tournament $(X_2,R_2)$ is wac if and only if $(X_1,R_1)$ and each $(Y_x,S_x)$ is wac. \end{lem}

  \begin{proof}  If $(X_2,R_2)$ is wac, then the quotient $(X_1,R_1)$ and the restriction to the clopen subsets
  $\{ x \} \times Y_x$ are wac by Theorem \ref{wactheo02}.

  Now assume that $R_1$ and each $S_x$  is wac. If $x$ is non-isolated in $X_1$ then it is a cycle point and each $3-$cycle containing $x$
  lifts to a $3-$cycle containing the unique point in $\pi^{-1}(x)$. If $x$ is isolated, then $(x,y)$ is non-isolated in $X_2$ if and
  only if $y$ is non-isolated in $Y_x$ and if $\{y, y', y'' \}$ is a $3-$cycle in
  $Y_x$ then $\{(x,y), (x,y'), (x,y'') \}$ is a $3-$cycle in $X_2$.

  \end{proof}

 Recall that when $X$ is the limit of an inverse system the set $IS$ consists of those points $x \in X$ such that $x_i$ is isolated for all $i \in \N$.

  \begin{theo}\label{classtheo05}  Let $\{ (X_i,R_i,f_i) \} $ be an inverse lexicographic system so that for each
 $i \in \N, \ \ (X_{i+1},R_{i+1})$ is the topological lexicographic product $(X_i,R_i) \ltimes \{ (Y_{ix_i},P_{ix_i}) \}$ with $f_i$
 the first coordinate projection.  Thus,
 if $x_i$ is non-isolated in $X_i$, then $(Y_{ix_i},P_{ix_i}) $ is trivial.

 We assume that conditions (i) and (ii) of Addendum \ref{classadd03}
 hold and, in addition,
 \begin{itemize}
 \item[(iii)] The tournaments  $(X_1,R_1)$ and each tournament $(Y_{ix_i},P_{ix_i})$ is either trivial, a
 non-trivial finite order, or a prime wac which not an arc.
 \end{itemize}

 Let $(X,R) = \overleftarrow{Lim} \{ (X_i,R_i,f_i) \}$  with $h_i : (X,R) \to (X_i,R_i)$ the coordinate projection map.

 A point $x \in IS$ is a cycle point if and only if it satisfies the following:
    \begin{itemize}
 \item[(iv)] For  infinitely many $i \in \N$, with $x_{i+1} = (x_i,y_i), \ \ y_i$ is contained in a $3-$cycle in $Y_{ix_i}$.
 \end{itemize}

 A point $x \in IS$ is an isolated point in $X$ if and only if it satisfies the following:
  \begin{itemize}
 \item[(iv')]  There exists $i \in \N$ such that $(Y_{ix_i},P_{ix_i})$ is trivial.
 \end{itemize}

The limit tournament $(X,R)$ is wac if and only if (iv) or (iv') holds for every $x \in IS$.
 In particular, if $(Y_{ix_i},R_{ix_i})$ is either trivial or point cyclic for every $x \in X$ and $i \in \N$, then $(X,R)$ is wac.

 If there exists $k \in N$  such that for every $x \in X$ and $i \ge k, \ \ (Y_{ix_i},P_{ix_i})$ is arc cyclic, $(X,R)$ is
 locally arc cyclic and so is wac.

 If $(X_1,R_1)$ is arc cyclic and for every $x \in X$ and $i \in \N, \ \ (Y_{ix_i},P_{ix_i})$ is arc cyclic, $(X,R)$ is arc cyclic and so is wac.

 When $(X,R)$ is wac, the inverse system $\{ (X_i,R_i,f_i) \}$  is a classifier for $(X,R)$.
 \end{theo}

 \begin{proof} First assume that $x \not\in IS$ and let $n$ be the smallest value such that $x_n$ is not isolated in $X_n$.
 If $n = 1$, then since $(X_1,R_1)$ is wac, it follows that $x_1$ is a cycle point in $X_1$. If $n > 1$, then
 $x_{n-1}$ is isolated in $X_{n-1}$ and with $x_n = (x_{n-1},y)$ we have that $y$ is non-isolated in the wac
 tournament $(Y_{(n-1)x_{n-1}},P_{(n-1)x_{n-1}})$ and so $y$ is a cycle point in $Y_{(n-1)x_{n-1}}$. Consequently,
 $x_n$ is a cycle point in $X_n$. By Theorem \ref{theo3.08}, $\{ x \} = h_n^{-1}(x_n)$ and $x$ is a cycle point in $X$.

Now let $x$ be a point of $IS$.

 If  for infinitely many $i \in \N$, with $x_{i+1} = (x_i,y_i), \ \ y_i$ is contained in a $3-$cycle in $Y_{ix_i}$,
 then it  follows, as in the proof of Theorem \ref{theo3.07a} that $x$ is a cycle point.

Conversely, assume that $x$ is a cycle point. For $j$ arbitrarily large, $x_j$ isolated implies that
 $h_j^{-1}(x_j)$ is a neighborhood of $x$ and so contains a $3-$cycle $\{ x, x', x'' \}$. Let $k+1$ be the minimum index $i$ such that
 $x_i \not= x'_i$ so that $k \ge j$. Because $h_{k+1}^{-1}(x_{k+1})$ and $h_{k+1}^{-1}(x'_{k+1})$ are disjoint and $Q$ invariant,
 it cannot happen that $x''$ lies in
 either them. Hence, $\{ x_{k+1}, x'_{k+1}, x''_{k+1} \}$ is a $3-$cycle in $X_{k+1}$. If $x_{k+1} = (x_k,y), x'_{k+1} = (x_k,y')$ and
 $x''_{k+1} = (x_k,y'')$, then $\{ y, y', y'' \}$ is a $3-$cycle in $Y_{kx_k}$ as required.

 If, instead, $(Y_{ix_i},P_{ix_i})$ is trivial for some $i$, then Condition (ii) implies, inductively, that
$(Y_{jx_j},P_{jx_j})$ is trivial for all $j \ge i$. It follows that $\{ x \} = h_i^{-1}(x_i)$ and so
  $x$ is isolated in $X$.

 Conversely, if $x \in IS$ is isolated, i.e. $\{ x \}$ is clopen,
 then $\{ x \} = \bigcap_i h_i^{-1}(x_i)$ implies
 that for some $i, \ \ \{ x \} = h_j^{-1}(x_j)$ for all $j \ge i$.  So
 $\{ x_j \} \times Y_{jx_j} = f_j^{-1}(x_j) = \{ x_{j+1} \}$ for all $j \ge i$.
 Thus, $(Y_{jx_j},P_{jx_j})$ is trivial for all $j \ge i$.

 If $(Y_{ix_i},P_{ix_i})$ trivial, then $f_i^{-1}(x_i)$ is a singleton. If $(Y_{ix_i},P_{ix_i})$ is point cyclic, then
 with $x_{i+1} = (x_i,y_i), \ \ y_i$ is contained in a $3-$cycle in $Y_{ix_i}$.  Hence, the point cyclic assumption implies
 Condition (iv).

 The arc cyclicity results follow from
 Corollary \ref{classcor04}.

 If $\{ (X'_i,R'_i,f'_i) \}$  is a classifier for $(X,R)$, we use Conditions (i) and (ii) to inductively construct the isomorphisms
  $q_i : (X_i,R_i) \to (X_i',R_i')$ which satisfy
  (\ref{classeq01}).  Observe that Conditions (i)-(iii) imply that if $(Y_{ix_i},P_{ix_i})$ is a non-trivial finite order, then
$(Y_{(i+1)x_{i+1}},P_{(i+1)x_{i+1}})$ is either a prime tournament and not an arc, or else it is trivial.  Either of these implies that
$$h_{i+1} : (h_i^{-1}(x_i),R|h^{-1}_i(x_i)) \to (f_i^{-1}(x_i),R_{i+1}|f^{-1}_i(x_i))$$
is a maximum order quotient map.
 \end{proof}
 \vspace{.5cm}

 \begin{lem}\label{classlem06} Assume that $(X,R)$ is a prime tournament which is not an arc and that $x \in X$.
 If $x$ is not initial, terminal or balanced, then it is contained in a cycle.  In particular,
 if $X$ is finite, then $R$  is point cyclic. \end{lem}

 \begin{proof} Assume that $(X,R)$ is a prime tournament and $x \in X$ is not contained in a cycle. Then for every $a, b \in X, \ \ a \ha x \ha b$ implies
 $a \ha b$. If $x$ is neither terminal nor initial, then both $R^{\circ -1}(x)$  and $R^{\circ}(x)$
are nonempty. If $x$ is not left balanced, then $R^{\circ -1}(x)$ is closed and $R^{\circ -1}(x) \mapsto 0$ and $R(x) \mapsto 1$
 is a quotient map to an arc. Since $(X,R)$ is prime, it is an arc.  Similarly, if $x$ is not right balanced, then
 $(X,R)$ is an arc.

 In a finite prime tournament, every point is isolated and so no point is  balanced. If it is not an arc, then it has no initial
 nor terminal point.

 \end{proof}
 \vspace{.5cm}

  \begin{add}\label{classadd07} In the  construction of Theorem \ref{classtheo05}, assume that every $(Y_{ix_i},P_{ix_i})$ is a finite tournament.
 A point $x \in X$ is a cycle point if and only if it satisfies the condition:
     \begin{itemize}
 \item[(iv'')] For  infinitely many $i \in \N$, with  $(Y_{ix_i},P_{ix_i})$ is a prime tournament which is not an arc.
 \end{itemize}

So $(X,R)$ is wac if and only if (iv'') or (iv') holds for every $x \in X$.
 \end{add}

 \begin{proof} Because the $(Y_{ix_i},P_{ix_i})$'s are finite, $X = IS$. Then Lemma \ref{classlem06} implies that conditions (iv) and (iv'')
 are equivalent.

 \end{proof}
 \vspace{.5cm}

 \begin{exes}\label{exes05}Uncountably many distinct arc cyclic tournaments on the Cantor set.\end{exes}

 Let $(Y_0,P_0)$ and $(Y_1,P_1)$ be the regular tournaments with $|Y_0| = 3$ and $|Y_1| = 5$. So
 $Y_0$ consists of a single $3-$cycle. Each of these is an arc cyclic, prime tournament. This is easy to check but we will verify
 these statements in the next section.
 \vspace{.25cm}

 (a) Let $\th \in \{ 0, 1 \}^{\N}$. Let $\{ (X_i,R_i,f_i) \} $ be the inverse system with
 $(X_1,R_1)  =  (Y_{\th_1},P_{\th_1})$ and $(X_{i+1},R_{i+1}) = (X_i,R_i) \ltimes (Y_{\th_{i+1}},P_{\th_{i+1}})$.
 That is, we use the construction of Theorem \ref{classtheo05} with $(Y_{ix_i},P_{ix_i}) = (Y_{\th_{i+1}},P_{\th_{i+1}})$
 for all $i \in \N$ and $x_i \in X_i$. It follows from the theorem that the inverse limit $(X_{\th},R_{\th})$ is an arc cyclic tournament on
 a Cantor set. Furthermore, if $\th \not= \th'$, then uniqueness of the classifiers implies that $(X_{\th},R_{\th})$
 is not isomorphic to $(X_{\th'},R_{\th'})$. If we exclude the countable set of $\th$'s which are eventually $0$ or eventually $1$,
 then remaining uncountable family of tournaments are all group tournaments associated with closed game subsets on the same group, namely a product of
 a countable number of $\Z/3\Z$'s with a countable number of $\Z/5\Z$'s. Alternatively, we can use the group structure which is
 the product of the $3-$adics with the $5-$adics.
 \vspace{.25cm}

 (b) The tournament $(Y_1,P_1)$ is a group tournament on the cyclic group $\Z/5\Z$ and the only automorphisms of $(Y_1,P_1)$ are
 translations by elements of the group, see, e.g. \cite{A20} Theorem 3.9. Choose $x_1, x_2 \in \Z/5\Z$ with $x_1 \not= e \not= x_2$
 and $x_2 \not= x_1 \not= x_2^{-1}$. For example, choose $x_1, x_2$ the two distinct members of the game
 subset $A$. There is then no automorphism of $(Y_1,P_1)$ which maps the pair $\{ e, x_1 \}$ to $\{ e, x_2 \}$.
 Now let $(X_1,R_1) = (Y_1,P_1)$ and let $(X_2,R_2) = (X_1,R_1) \ltimes \{ (Y_{1x},P_{1x}) \}$ with
 $(Y_{1x},P_{1x}) = (Y_1,P_1)$ for $x = e, x_1$ and $= (Y_0,P_0)$ otherwise. Alternatively, let
 $(X'_2,R'_2) = (X_1,R_1) \ltimes \{ (Y'_{1x},P'_{1x}) \}$ with
 $(Y'_{1x},P'_{1x}) = (Y_1,P_1)$ for $x = e, x_2$ and $= (Y_0,P_0)$ otherwise. Now fix $\th \in \{ 0, 1 \}^{\N}$.
 Let $(X,R) = (X_2,R_2) \ltimes (X_{\th },R_{\th })$ and $(X',R') = (X'_2,R'_2) \ltimes (X_{\th },R_{\th })$.
 Then $(X,R)$ and $(X',R')$ are  arc cyclic tournaments on the Cantor set which are not isomorphic despite the fact that the sets
 $\{ (Y_{ix_i},P_{i,x_i}) : x \in X \} $ and $\{ (Y_{ix_i},P_{i,x_i}) : x \in X' \} $ are equal for every level $i$.
 \vspace{.25cm}

 (c)  Let $Y_2 = \{ a_1, a_2, b_1, b_2, c \}$ and on it define the tournament $P_2$ to consist of
  \begin{align}\label{exeq} \begin{split}
  (a_1, b_1), \ \ (b_2, a_2), \quad (a_1, b_2), &\ \ (a_2, b_1), \quad (a_1, a_2), \ \ (b_1, b_2), \\
  (c, a_1), \ \ (c, a_2), &\quad (b_1,c), \ \ (b_2,c).
  \end{split}\end{align}

 We have $3-$cycles $\{ a_1, b_i, c \}, \{ c, a_i, b_1 \}$ for $i = 1,2$ and $\{ b_1, b_2, a_2 \}$. Thus, every arc is in a
 $3-$cycle except for $(a_1, a_2)$. Also $b_2 \in P_2(a_1) \cap P_2^{-1}(a_2)$.

 It easily follows that $(Y_2,P_2)$ is prime and the maximal arc cyclic subsets are $Y_2 \setminus \{ a_1 \}$ and $Y_2\setminus \{ a_2 \}$.
 Notice that the restriction of $P_2$ to neither of these subsets is arc cyclic.

 Let $\{ (X_i,R_i,f_i) \} $ be the inverse system with
 $(X_1,R_1) \ = \ (Y_{2},P_{2})$ and $(X_{i+1},R_{i+1}) = (X_i,R_i) \ltimes (Y_{2},P_{2})$.

 It follows from Addendum \ref{classadd07} that for the limit system $(X,R)$ every point of $X$ is a cycle point and so $(X,R)$ is wac.
 On the other hand, if $x, x' \in X$ with $x_i = x'_i$ and $x_{i+1} = (x_i,a_1), x'_{i+1} = (x_i,a_2)$, then the arc $(x, x')$ is not
 contained in any $3-$cycle in $X$. It follows that $X$ contains no nonempty, open, arc cyclic subset.
 \vspace{1cm}

  \section{ \textbf{Prime Tournament Constructions}}\vspace{.5cm}

  Throughout our examples below, for a set $J$ when we consider the product $J \times \{ -1, +1 \}$ we will write for
  $a \in J \ \ a- = (a,-1)$ and $a+ = (a,+1)$ and similarly write
  $J\pm$ for $J \times \{ \pm 1 \}$.
  \vspace{.25cm}

  \subsection{ \textbf{Doubles and Reduced Doubles}}

  \begin{exes}\label{exesnew01}\end{exes}

  (a)  For $(J,P)$ a finite tournament, we follow \cite{A20} Section 6, to define the
  \emph{double}\index{double}\index{tournament!double} $2(J.P) = (2J,2P)$ \index{$2(J,P)$}  to be a tournament
  on $2J = \{ 0 \} \cup J \times \{ -1, +1 \}$.

  The tournament $2P$ is defined as follows.
  \begin{align}\label{exeq01} \begin{split}
a \in J \quad &\Longrightarrow \quad a- \ha a+, \ a+ \ha 0,\ 0 \ha a- \ \ \text{in} \  2P. \\
a \ha b \ \ \text{in} \ P \quad &\Longrightarrow \quad  a+ \ha b+, \ a- \ha b-, \ b+ \ha a-, \ b- \ha a+\ \ \text{in} \ 2P.
\end{split}\end{align}

The \emph{reduced double}\index{reduced double}\index{double!reduced}\index{tournament!double!reduced}
$2'(J,P) = (2'J,2'P)$ \index{$2'(J,P)$}is the restriction of the double to $J \times \{ -1, +1 \}$.  That is, we remove the point $0$.
Thus, the double of a trivial tournament is a $3-$cycle and its reduced double is an arc.

We will call the tournament $(J,P)$ \emph{irreducible}\index{irreducible}\index{tournament!irreducible} if for every
pair $a  \not= b$ in $J$ there exists $c \in J$ such that either $\{ a, b \} \subset P^{\circ}(c)$ or
$\{ a, b \} \subset P^{\circ -1}(c)$.  We will explain later the reason for the label. Clearly, a tournament is irreducible
when for every $a, b \in J, \ \ Q(a,b) \not= J$.

\begin{theo}\label{extheo01} For a finite tournament $(J,P)$ the double $2(J,P)$ is regular, arc cyclic and prime with
$|2J| = 2 |J| + 1$.

If $(J,P)$ is irreducible, then the reduced double $2'(J,P)$ is arc cyclic and prime with $|2'J| = 2 |J|.$
\end{theo}

\begin{proof} A double is always regular and so is arc cyclic. Directly, observe that if $a \ha b $ in $P$, then
$\{ a-, b-, b+ \}$ and $\{ a+, b+, a- \}$ are $3-$cycles in $2'J$. So if $U$ is a $Q$ invariant subset for $2'(J,P)$ and
any pair among the four points  $\{ a-, b-, a+, b+ \}$ other than $\{ a+, b- \}$ is contained in $U$ then all four points are
contained in $U$.

Furthermore, $\{ a+, 0, b- \}$ and $\{ a+, 0, a- \}$ are $3-$cycles in $2J$. It easily follows directly that the double is
arc cyclic and prime.

If $\{ a, b \} \subset P^{\circ}(c)$, then $\{ a+, c-, b- \}$ is a $3-$cycle in $2'J$. If
$\{ a, b \} \subset P^{\circ -1}(c)$, then $\{ a+, c+, b- \}$ is a $3-$cycle in $2'J$. So if $\{ a+, b- \}$ is contained in $U$
then either $\{ c-, b- \}$ or $\{ a+, c+ \}$ is contained in $U$ as well.  In the first case
it follows as above that all four points of $\{ c-, b-, c+, b+ \}$
are contained in $U$ and in the second, all four points of $\{ c-, a-, c+, a+ \}$ are contained in $U$.  It easily follows that
the reduced double is arc cyclic and prime.

The cardinality results are obvious.

 \end{proof}
 \vspace{.5cm}

 \begin{cor}\label{excor02} For every odd number $2n+1 \ge 3$ there  are arc cyclic, prime tournaments of order $2n + 1$.
For every even number $2n \ge 8$ there are arc cyclic, prime tournaments of order $2n$. \end{cor}

\begin{proof} Beginning with any tournament of order $n$, including the trivial tournament with $n = 1$, the double of a tournament of
order $n$ is an arc cyclic, prime tournaments of order $2n + 1$.

Now begin with any tournament $(J_0,P_0)$. First attach two additional points $m, M$ to get $(J_1,P_1)$ with $J_1 = J_0 \cup \{m, M \}$,
and with $P_1$ extending $P_0$ so that $m$ is initial and $M$ is terminal in $J_1$.
 Now attach an additional point $p$ to get $(J_2,P_2)$ with $J_2 = J_1 \cup \{ p \}$ and with $P_2$ extending $P_1$ so that
 $p \ha m, \ p  \ha M$ and $x_0 \ha p$ in $P_2$ for some $x_0 \in J_0$. We check that $(J_2,P_2)$ is irreducible.

 The point $M$ is still terminal for $P_2$ and so any pair which does
 not include $M$ is contained in $P_2^{-1}(M)$. Any pair $\{ x, M \}$ with $x \in J_0$ is contained in $P_2(m)$. This
 takes care of all pairs except for $\{ m, M \} \subset P_2(p)$ and $\{ p, M \}\subset P_2(x_0)$.

 Thus, $(J_2,P_2)$ is irreducible and so its reduced double
 is arc cyclic and prime. The smallest case of this is with $(J_0,P_0)$ trivial. In that case $n = |J_2| = 4$ and so $2n = 8$.
 \end{proof}

 There also exists a prime tournament of order 6 which can be obtained from a regular tournament of order 7 by removing a suitable point.
 However, any tournament of order 4 has either a $3-$cycle or an arc as a quotient and so is not prime.
 \vspace{.5cm}

 If $J$ is the odd cyclic group $\Z/(2n + 1)\Z$ and the tournament $\widehat{A}$ is associated with the game subset $A = \{1, \dots, n \}$, then
the tournament $(J,\widehat{A})$ is isomorphic to the double of the order $(I,L)$ of length $n$, see \cite{A20} Example 6.5. Hence, it is
arc cyclic and prime. In particular, with $n = 2$ this applies to the unique regular tournament of order $5$. Notice that
if $2n+1$ is not a prime number, then Proposition \ref{prop3.03} implies that there is a game subset whose associated tournament is
isomorphic to a non-trivial lexicographic product and so is not prime.

On the other hand, if the odd order group $J$ is a non-cyclic group with any game subset or a cyclic group $\Z/(2n + 1)\Z$ with
game subset $A'$ such that $(J,\widehat{A'})$ is not isomorphic to $(J,\widehat{A})$ above, then by \cite{A20}Theorem 3.18, the tournament
$(J,\widehat{A'})$ is irreducible, as well as regular. Hence, its reduced double is prime.
 \vspace{.25cm}

(b)  For any topological tournament $(J,P)$ a  \emph{generalized reduced double}\index{generalized reduced double}
\index{double!generalized reduced}\index{tournament!double!generalized reduced} $2'(J,P) = (2'J, 2'P)$ is a topological tournament
on $J \times \{ -1, +1 \}$ such that the following conditions are satisfied.

\begin{itemize}
\item[(i)] The map $x+ \mapsto x$ is an isomorphism from the restriction to $J+ = J \times \{ +1 \}$ to $(J,P)$ and
$x- \mapsto x$ defines an isomorphism from the restriction to $J- = J \times \{ -1 \}$ to $(J,P)$.
\vspace{.25cm}

\item[(ii)]  The set $(2'P)\cap (J- \times J+)$ is a surjective relation from $J-$ to $J+$.  That is, for every $a \in J$,
there exist $a', a'' \in J$ such that $a'- \ha a+$ and $a- \ha a''+$ in $2'P$, i.e. $(2'P)^{-1}(J+) \supset J-$ and
$(2'P)(J-) \supset J+$.
\vspace{.25cm}
\end{itemize}

For the ordinary reduced double of a finite tournament, $a- \ha a+$ for all $a \in J$ implies condition (ii).

The lexicographic product of the arc on $\{ -1, +1 \}$ with $-1 \ha +1$  together with $(J,P)$ is a generalized reduced double.
However we will be primarily interested
in the cases when $2'(J,P)$ is prime.

 \vspace{.5cm}
 \subsection{ \textbf{Compact Countably Infinite Tournaments}}
\begin{exes}\label{exes06} \end{exes}
 \vspace{.25cm}

 (a) Let $\N^*$ \index{$\N^*$} be the one point compactification of the set $\N$ via the point $\infty$ at infinity.
\vspace{.25cm}

 We define\index{$N_0$}    the tournament $N_0 = (\N^*, L_0)$ with $L_0$ the linear order on $\N^*$, i.e.
 \begin{equation}\label{exeq01a}
 i \ha j \quad \Longleftrightarrow \quad i < j \quad \text{including } \ \ j = \infty.
 \end{equation}
 We will write $\bar N_0$\index{$\bar N_0$}  for the reverse tournament $(\N^*, L_0^{-1})$.

 The tournament $N_0$ has an arc quotient with infinitely many quotient maps to the arc.

 Notice that for any $i \in \N$, $|L_0^{-1}(i)| = i$. It follows that $N_0$ is rigid, i.e. the only automorphism of $N_0$ is the identity.
  \vspace{.5cm}

    We define \index{$N_1$}   the tournament $N_1 = (\N^*, L_1)$ with $L_1$ the linear order on $\N^*$ adjusted by reversing the arcs $(i, i+1)$
   for all $i \in \N$. Thus,
    \begin{equation}\label{exeq01b}
 i+1 \ha i \ha j \quad \Longleftrightarrow \quad i+1 < j \quad \text{including } \ \ j = \infty.
 \end{equation}
 The reverse tournament is $\bar N_1 = (\N^*, L_1^{-1})$.\index{$\bar N_1$}

 \begin{theo}\label{extheo01a} The tournament $N_1$ is prime and rigid. Furthermore, the restriction
 to any interval $\{k : i \le k \le j \}$ is  prime provided $j - i \ge 2$.\end{theo}

 \begin{proof}  Let $U$ be a closed, non-trivial $Q$ invariant subset of $\N^*$.

 Observe first that for any $i \in \N, \ \ c_i =  \{ i, i+2, i+1 \}$ is a $3-$cycle.

 Assume $i < j $ are in $U$.

 If $1 < i$, then $i \ha i-1 \ha j$ implies that $i-1 \in U$ and so, inductively,
 $i' \in U$ for all $i' < i$.

 If $j < \infty$, then $i \ha j+1 \ha j$ implies that $j+1 \in U$ and so
 $j' \in U$ for all $j' > j$ with $j' < \infty$. Hence, if $j = i+1, \ \ \N \subset U$.

 If $j = i+2$, then $i+1 \in U$ because of the $3-$cycle $c_i$. Again $\N \subset U$.

 If $j > i+2$, then  $i \ha k \ha j$ for all $k$ with $i+1 < k < j-1$ implies that such $k$ are in $U$.  Thus,
 $U$ contains every point of $\N$ except possibly $i+1$ and, if $j$ is finite, $j-1$. When $j$ is finite, $j+1 \in U$
 and so $i \ha j-1 \ha j+1$ implies $j-1 \in U$.  The $3-$cycle $c_i$ then implies that $i+1 \in U$.

 It follows that $\N \subset U$. Since $U$ is closed, $\infty \in U$.

 A similar argument shows that the restriction to an interval containing at least three points is prime.

 Assume that $h : N_1 \to N_1$ is a continuous tournament map with image non-trivial. Since $N_1$ is prime and not an arc, it follows
 that $h$ is a tournament isomorphism onto its image. Hence, $h(\infty) = \infty$. Furthermore, (\ref{exeq01b}) implies that
 $h(i+1) = h(i) + 1$. Thus, with $k = h(1)$ we have $h(i) = k + i - 1$ for all $i$.  In particular, if $h$ is surjective, it is
 the identity.

 \end{proof}
  \vspace{.5cm}

(b) Let $2\N^*$ \index{$2\N^*$}  be the one-point compactification of $ \N \times \{ -1, +1 \} $ by the point $\infty  $ at infinity. We will use
the label \index{$2N$}   $2N = (2\N^*, 2L)$ for a tournament which satisfies
  \begin{align}\label{exeq02aa} \begin{split}
  \N+ \ = \ (2L)^{\circ -1}(\infty), &\quad \N- \ = \ (2L)^{\circ}(\infty), \\
(i-, i+) \ \in \ 2L \quad &\text{for all} \ \ i \in \N.
\end{split}\end{align}

There are two important examples
\vspace{.25cm}

We define
 \index{$2N_0$}   the tournament $2N_0 = (2\N^*, 2L_0)$  so that the restriction of $2N_0$ to
 $\N^*+ = \N \times \{+1 \} \cup \{\infty \}$ is isomorphic to $N_0$ by $i+ \mapsto i$ and
the restriction of $2N_0$ to $\N^*- = \N \times \{-1 \} \cup \{\infty \}$ is isomorphic to $\bar N_0$ by $i- \mapsto i$. In addition,
  \begin{equation}\label{exeq02}
 i- \ha i+, (i+2)+ \qquad i+ \ha j- \quad \text{for all} \ \ j \not= i, i-2.
\end{equation}

 \begin{theo}\label{extheo01b} The tournament $2N_0$ is arc cyclic, prime and rigid. \end{theo}

\begin{proof} If $i < j$ and $i+2 \not= j$, then $$\{i-,  i+, j- \},\{ i+, j+, i- \}, \{ i+, \infty, i- \}, \{(i+2)+, \infty, i- \}$$ are $3-$cycles.
Thus, $2N_0$ is arc cyclic.

Assume $U$ is a non-trivial, closed $Q$ invariant subset. If any pair in $\{ i-, i+, (i+1)+,(i+1)- \}$ except $\{ (i+1)+,(i+1)- \}$
is contained in $U$, then
all four points are contained in $U$. Proceeding upward, we obtain $j-, j+ \in U$ for all $j \ge i$  as well as $\infty \in U$.
If $\{ (i+1)+,(i+1)- \} \subset U$, then because $j-, j+ \in U$ for all $j \ge i+1$  we have $\infty \in U$ because $U$ is closed. Because
$\{(i+2)+, \infty, i- \}$ is a $3-$cycle, it follows that $i- \in U$. Since $i-, (i+1)- \in U$ it follows that $\{ i-, i+, (i+1)+,(i+1)- \} \subset U$.
Thus, it follows that $2\N^* \subset U$ and so $2N_0$ is prime.

Because any automorphism of $2N_0$ would have to fix $\infty$ and because $N_0$ is rigid, it follows  that $2N_0$ is rigid.

 \end{proof}

\vspace{.25cm}

If we fix $n \in \N$ with $n > 3, \ \ 2L_0$ contains the countable set of arcs
$A_n \ = \ \{ ((i+k)+,i-) : i \in \N, 3 < k \le n \}$.
 If we reverse the arcs in any subset of $A_n$ we still have an arc cyclic, prime tournament.  An isomorphism between two such would have to
 be the identity on $2\N^*$ by rigidity of $N_0$ and $\bar N_0$. Thus, for distinct subsets of $A_n$
 the resulting tournaments are not isomorphic. In this way we obtain an uncountable number of distinct,
 countably infinite, compact, arc cyclic, prime tournaments each with a single non-isolated point.
  \vspace{.5cm}

   We  define
\index{$2N_1$}  the tournament $2N_1 = (2\N^*, 2L_1)$  so that the restriction of $2N_1$
 to $\N^*+ $ is isomorphic to $N_1$ by $i+ \mapsto i$ and
the restriction  to $\N^*- $ is isomorphic to $\bar N_1$ by $i- \mapsto i$. In addition,
 \begin{equation}\label{exeq02b}
 i- \ha i+, \qquad i+ \ha j- \quad \text{for all} \ \ j \not= i.
 \end{equation}

 \begin{theo}\label{extheo01c} The tournament $2N_1$ is arc cyclic, prime and rigid. The restriction to
 $\{ k-, k+ : i \le k \le j \}$ is arc cyclic and prime provided $j - i \ge 2$. \end{theo}

 \begin{proof} If $i+1 < j \le \infty$, then
$$\{ i-, i+, j- \}, \{ i-, i+, j+ \}, \{ (i+1)-, (i+1)+, i+ \}, \{ (i+1)+, i-, (i+1)- \}$$
are  $3-$cycles with $j- = j+ = \infty$ when $j = \infty$.
It follows that $2N_1$ is arc cyclic.

 Assume $U$ is a non-trivial, closed $Q$ invariant subset. If $U$ contains two points of $N^*+$, then because $N_1$ is prime, it
 follows that $N^*+ \subset U$.  Similarly, if $U$ contains two points of $N^*-$, then  $N^*- \subset U$. If either of these occurs then
 from the cycles it contains all $i+, i-$ and so $2\N^* \subset U$.

 Now we use the  cycles listed above.

 If $U$ contains $i+, i-$ for some
 $i$, then it contains $j+$ and $j-$ for all $j > i+1$ and so again  $2\N^* \subset U$.

 Now assume $U$ contains $i+, j-$ with $i \not= j$. If $j = i+1$, then $(i+1)+ \in U$. If $j = i-1$, then $i- \in U$. If $j > i+1$, then
 $i- \in U$. If $j < i-1$, then $j- \in U$. From the earlier computations it follows that $U = 2\N^*$ in these casees as well.

 Thus, $2N_1$ is prime.

 A similar computation works for the restriction.

 An automorphism must fix $\infty$. Again because $N_1$ is rigid, it follows that $2N_1$ is rigid.

  \end{proof}
  \vspace{.5cm}

 \subsection{ \textbf{Adjusting Lexicographic Products}}
  \begin{exes}\label{exesnew02}  \end{exes}

 (a) We assume that $(J,P)$ is a topological tournament with a generalized reduced double $2'(J,P)$ which is
  prime and  arc cyclic.  We also assume that $J$ does not have both an initial point and a terminal point. For example, in the finite case we may use
  $(J,P)$ any regular, irreducible tournament as in that case Theorem \ref{extheo01} says that
  the reduced double $2'(J,P)$ is arc cyclic and  prime.

  We will start with a topological lexicographic product and then alter the arc connections over certain pairs in the
  base of the product.
  \vspace{.25cm}

 We begin with the topological lexicographic product of $N_0 \ltimes \{ (Y_a,S_a) : a \in \N^* \}$ with $(Y_i,S_i) = (J,P)$ for all $i \in \N$
  and $(Y_{\infty},S_{\infty})$ trivial.  So the total space $Y = (\N \times J) \cup \{ \infty \}$.

   Leaving the other arcs unchanged we define $(Y,S)$ so that for each $i \in \N$, the  restriction
 $S|[\{ i, i+1 \} \times J]$ is isomorphic to $2'(J,P)$ by the map $(i,x) \mapsto x-$ and $(i+1,x) \mapsto x+$ for $x \in J$.

 \begin{theo}\label{extheo01d} The tournament $(Y,S)$ is prime with non-isolated terminal point $\infty$. Every point of $Y$ except for
 $\infty$ has an arc cyclic neighborhood. Furthermore, if $i < j$, then the restriction of $S$ to $ \{ k : i \le k \le j \} \times J$ is
 prime and locally arc cyclic.\end{theo}

 \begin{proof} For each point $(i,x)$  $\{ i, i+1 \} \times J$ is an arc cyclic neighborhood by assumption
 on $2'(J,P)$. Hence, $(Y,S)$ is an almost wac tournament.  It follows from Addendum \ref{qadd04a}
   that any closed, non-trivial $Q$ invariant subset $U$ is
 clopen.  If $U$ contains two points of  $\{ i, i+1 \} \times J$, then because $2'(J,P)$ is prime, $U$ contains
 $ \{ i, i+1 \} \times J$ and if $i > 1$, then
 $U$ contains $ \{ i-1, i \} \times J$ as well. Proceeding upwards and downwards,
 we obtain $\N \times J \subset U$ and since $U$ is closed, $\infty \in U$.

 Now assume $(i,x), (j,y) \in U$ with $j > i+1$ so that $(i,x) \ha (j,y)$. If $J$ does not have a terminal point, then there exists
 $x' \in J$ so that $(x,x') \in P^{\circ}$. Then $(i,x) \ha (i,x') \ha (j,y)$ and so $(i,x') \in U$ and as above $U = Y$.
 If $J$ does not have an initial point, then there exists
 $y' \in J$ so that $(y',y) \in P^{\circ}$. Then $(i,x) \ha (j,y') \ha (j,y)$ and so $(j,y') \in U$ and as above $U = Y$.

 If $(i,x), \infty \in U$ then for any $j > i+1$, $(i,x) \ha (j,x) \ha \infty$ and so $(j,x) \in U$. As before this implies $U = Y$.

  The same arguments work for the restriction to $ \{ k : i \le k \le j \} \times J$.

   \end{proof}
  \vspace{.5cm}

  The maximal arc cyclic subsets of $L$ are all of the form $\{i, i+1 \} \times J$. In particular, the isomorphism class of the restriction
  $ \{ k : i \le k \le j \} \times J$ is determined by the length $j - i$ since the restriction has exactly $j - i$ maximal arc cyclic subsets.

 If we had used $N_1$ instead of $N_0$ in the above construction we would have obtained the same tournament $(Y,S)$
  since $N_1$ was obtained from $N_0$ by reversing the $(i, i+1)$ arcs.

    Notice that if $(Y,S)$ is a compact tournament with a terminal point $M$ which is not isolated, and so is left balanced,
  it cannot happen that every arc not connected to $M$ is contained in a $3-$cycle. For suppose that $\{ y_n \}$ is a sequence
  in $Y \setminus \{ M \}$ which converges to $M$ and that $x \in  Y \setminus \{ M \}$.  Since $M \in S^{\circ}(x)$, eventually
  $x \ha y_n $. Suppose $z_n \in Y$ with $z_n \ha x \ha y_n $.
 We may assume $\{ z_n \}$ converges to a point $z$ so that $z \ \underline{\ha} \ x $ and, in particular, $z \not= M$.
Hence, $z \ha M$. Since $(z_n,y_n)$ converges $(z,M)$, eventually $(z_n, y_n) \in S^{\circ}$ and so eventually
$\{ z_n, x, y_n \}$ is not a $3-$cycle. That is, eventually the pair $\{ x, y_n \}$ is not contained in any  $3-$cycle.

For the arc $ (\{ 0, 1 \}, L)$ with $\{ 1 \}  = L^{\circ}(0)$ we consider the lexicographic product with  $(Y,S)$, defining
  $(\tilde Y, \tilde S) = (\{ 0, 1 \},L) \ltimes (Y,S)$.  The first coordinate projection map, $\pi$, is an arc quotient map.
  However, we have another prime quotient map
  $\tilde \pi : (\tilde Y, \tilde S) \to (Y,S)$ given by, with $x \in L$:
   \begin{equation}\label{exeq04}
  \tilde \pi(0,x) \ = \ x, \quad \text{and} \quad  \tilde \pi(1,x) \ = \ \infty.
\end{equation}

  Thus, the prime quotients of $(\tilde Y, \tilde S)$ exist, but are not unique.

 \vspace{.5cm}

  (b) Now assume that $(Z,P)$ is a topological tournament with\\ $\{ Z-, Z+ \}$ a partition of $Z$ by two disjoint clopen subsets
  such that the relation $P \cap (Z- \times Z+)$ is surjective. We will write $P\pm$ for the restriction $P|Z\pm$.

  For example, we may use $(Z,P)$ equal to the generalized double $2'(J,P)$ as in part (a).

  Now let $2N = (2\N^*, 2L)$ be a tournament satisfying (\ref{exeq02aa}).

  To define $(K,T)$ we begin with the topological lexicographic product of
  $2N \ltimes \{ (Y_{a},S_{a}) : a \in 2\N^* \ \}$ such that for all $i \in \N$
   \begin{align}\label{exeq04aa}\begin{split}
 (Y_{i+},S_{i+}) \ = \ &(Z+,P+), \\ (Y_{i-},S_{i-}) \ = \ &(Z-,P-)
 \end{split}\end{align}
  and with $(Y_{\infty},S_{\infty})$ trivial. Thus, the total space $K = (\N+ \times Z+) \cup (\N- \times Z-) \cup \{ \infty \}$.

  Leaving the other arcs unchanged we define $(K,T)$ so that for each $i \in \N$, the  restriction to
 $[(\{ i- \} \times Z-) \cup (\{  i+ \} \times Z+)]$ is isomorphic to $(Z,P)$ by the map $(i-,z-) \mapsto z-$
 for $z- \in Z-$ and $(i+,z+) \mapsto z+$ for $z+ \in Z+$.

  Recall that we defined a section for a topological lexicographic product. In this case, given any  function $\tilde \xi : \N+ \cup \ \N- \to Z$ with
  $\tilde \xi(\N+)  \subset Z+$ and $\tilde \xi(\N-)  \subset Z-$, the associated section $\xi : 2\N^* \to K$ is defined by
  $\xi(i\pm) = (i\pm, \tilde \xi(i\pm))$ and
  $\xi(\infty) = \infty$.  Lemma \ref{lemlex01b} says that any section $\xi$ is continuous and induces an isomorphism from
   $2N$ to the restriction of the corresponding lexicographic product to the image of $\xi$. We will restrict ourselves to
   sections which satisfy the condition

     \begin{equation}\label{exeq04ab}
     \tilde \xi(i-) \ha \tilde \xi(i+) \ \ \text{in} \ \ P \ \ \text{for all} \ \ i \in \N.
     \end{equation}

   This will imply that $\xi$ is a tournament isomorphism from $2N$ to the restriction $T|j(2\N^*)$.

    \begin{theo}\label{extheo01e} If $2N$ and $(Z,P)$ are both arc cyclic (or both prime) tournaments
    then $(K,T)$ is an arc cyclic (resp. prime) tournament. \end{theo}

    \begin{proof} Let $U$ be a closed, non-trivial $Q$ invariant subset of $K$.

    First consider a pair of points in $[(\{ i- \} \times Z-) \cup (\{ , i+ \} \times Z+)]$ for some $i$.
    If $(Z,P)$ is arc cyclic then such a pair is contained in an arc in $[(\{ i- \} \times Z-) \cup (\{ , i+ \} \times Z+)]$
      Furthermore, if  $(Z,P)$ is prime and $U$ contains such a pair,
    then it contains all of $[(\{ i- \} \times Z-) \cup (\{ , i+ \} \times Z+)]$.

    Given any other sort of pair, the assumption that $P \cap (Z- \times Z+)$ is surjective implies that there
    exists a section $\xi$ which contains the pair and, in particular, so that $\xi$ satisfies condition (\ref{exeq04ab}).
   If $2N$ is arc cyclic, then any such pair is contained in a
    $3-$cycle in $j(2\N^*)$. It follows that $(K,T)$ is arc cyclic when $2N$ and $(Z,P)$ are arc cyclic.

   If $2N$ is prime, it follows that if $U$ contains a pair of points in $j(2\N^*)$ then it contains all of $j(2\N^*)$.

    Hence, $U$ contains $j(2\N^*)$ for some section $\xi$.  If $\xi'$ is another section which agrees with $\xi$ at some pair of points,
    then $U$ contains $\xi'(2\N^*)$. By thus varying the sections, we see that $U$ contains $(\N+ \times Z+) \cup (\N- \times Z-)$.
    Since $U$ is closed, it
  contains all of $K$. Thus, $(K,T)$ is prime when
    $2N$ and $(Z,P)$ are prime.

   \end{proof}
  \vspace{.5cm}

  For two special cases we can use $2N = 2N_0$ and $2N = 2N_1$ which are arc cyclic, prime tournaments by Theorems \ref{extheo01b} and
  \ref{extheo01c}. We use the labels $(K,T_0)$ and $(K,T_1)$ for these special cases. Thus we have

   \begin{cor}\label{excor01ea} If $(Z,P)$ is an arc cyclic, prime tournament, then the
    tournaments $(K,T_0)$ and $(K,T_1)$ are arc cyclic and prime. \end{cor}
 \vspace{.5cm}

 If we use $(Z,P)$ equal to the reduced double of part (a), then  $K = [(\N- \ \cup \ \N+) \times J] \cup \{ \infty \}$.
\vspace{.5cm}
\subsection{ \textbf{The Attachment Construction}}

   \begin{df}\label{exdef04}For a tournament $(X,R)$ a subset $E$ of $X$ is called a \emph{spanning set}\index{spanning set} when
   it satisfies the following equivalent conditions.
   \begin{itemize}
   \item[(i)] The set $E$
   meets every input and output set, i.e.  for all $x \in X$, $R^{\circ}(x) \cap E$ and $R^{\circ -1}(x) \cap E$ are nonempty.

   \item[(ii)] There does not exists $x \in X$ such that either $E \subset R(x)$ or $E \subset R^{-1}(x)$.

 \item[(iii)] The  images   $R^{\circ}(E)$ and $(R^{\circ})^{-1}(E)$ each equal all of $X$.

 \item[(iv)] For every $x \in X$ there exist $x', x'' \in E$ such that $x' \ha x \ha x''$.
  \end{itemize}\end{df}
 \vspace{.5cm}

 If $E$ is a spanning set, then $Q(E \times E) = X$. Conversely, if $(X,R)$ is balanced and $E$ is open, then $Q(E \times E) = X$
 implies that $E$ is a spanning set, because in that case for any $x \in E$, both $R^{\circ}(x)$ and $R^{\circ -1}(x)$ meet $E$.

  If $E$ is a spanning set, then any subset which contains $E$ is a spanning set. We will be primarily interested in sets $E$ such that
  both $E$ and its complement are spanning sets.

 Spanning sets need not be large. Let $(J,P)$ be a finite tournament with three points $a, b, c \in J$ with $b \ha c $.
 It is easy to check that in the double $2(J,P)$
 each $3-$cycle $\{ 0, a-, a+ \}$ $\{ b-, b+, c+ \}$ is a spanning set and so each has a spanning set complement as well.
 If $(J,P)$ has no initial point, then $J-$ and $J+$ are complementary spanning sets for the reduced double $2'(J,P)$.

   \begin{prop}\label{exprop05} Let $(J,P)$ be a finite, regular tournament of size $2n + 1$. There are at least
   ${2n + 1 \choose n} - (2n + 1)(2n + 2)$ separate spanning sets $A$ with size $|A| = n$. If
   $n \ge 6$, then there are at least $n(2n+1)(n+1)$ such sets. For each such
   spanning set, the complement is a spanning set as well.\end{prop}

   \begin{proof} Each $P(x)$ and $P^{-1}(x)$ has size $n + 1$ and so contains $n + 1$ subsets of size $n$. Hence,
   there are at most $2 (2n + 1)(n + 1)$ sets of size $n$ which are contained in some $P(x)$ or $P^{-1}(x)$.
   Hence, $J$ contains at least ${2n + 1 \choose n} - (2n + 1)(2n + 2)$
   subsets of  size $n$ which are not contained in any $P(x)$ or $P^{-1}(x)$.  We can write the difference as
     \begin{equation}\label{exeq05}
    (2n+1)(n+1)[ (\frac{2n}{n+1} \frac{2n-1}{n} \cdots \frac{n+5}{6}\frac{n+4}{10}\frac{n+3}{12})(n+2) - 2]
 \end{equation}
Cancelling the initial $2$ into the $12$ and observing that $2n - 1 \ge n + 1$ when $n \ge 2$,
we see that when  $n \ge 6$, then the parenthesized expression is greater than $1$.

If $B$ is the complement of one of these sets $A$, then $|B| = n + 1$. So if $B$ is contained in some $P(x)$, then it equals $P(x)$ and
so $A$ is disjoint from $P^{\circ}(x)$, contra the assumption that $A$ is a spanning set.  Similarly, $B$ cannot be contained in any
$P^{-1}(x)$.  Thus, $B$ is a spanning set.

  \end{proof}
 \vspace{.5cm}

 Clearly, if $(X,R)$ has an initial or terminal point, then it does not admit a
  spanning set.

   \begin{theo}\label{extheo05a} If the topological tournament $(X,R)$ has no initial or terminal point, then
   the entire space $X$ is a spanning set.

   Assume that the compact topological tournament $(X,R)$ has no initial or terminal point and $X$ has no isolated points. If $U$ is an open
   spanning set, then $U$ contains a pairwise disjoint sequence of finite subsets $\{ H_i : i \in \N \}$ all with the same cardinality
   and such that each is a spanning set. \end{theo}

   \begin{proof}  Assume that the topological tournament $(X,R)$ has no initial or terminal point. It is clear that $X$ is a spanning set.

   If $U$ is an open spanning set, then for every $z \in X$, there exist $z-, z+ \in U$ such that $z \in R^{\circ}(z-) \cap R^{\circ -1}(z+)$.
   Choose  $\{ U_{z-}, U_z, U_{z+} \}$ a thickening of $\{ z-, z, z+ \}$ so that for every $x- \in U_{z-}, x+ \in U_{z+}$ we have
   $U_z \subset R^{\circ}(x-) \cap R^{\circ -1}(x+) $. By intersecting with $U$ we may assume that
   $U_{z-},  U_{z+} \subset U$.

    Let $\{ U_{z_j} : j = 1, \dots, k \}$ be a subcover of $X$. If
   $x_j- \in U_{z_j-}, x_j+ \in U_{z_j+},$ then $\{ x_j- : j = 1, \dots, k \} \cup \{ x_j+ : j = 1, \dots, k \}$ is a spanning set
   with cardinality $2k$.

   Now assume that $X$ has no isolated points.  We can then choose for each $j$ sequences of distinct points
   $\{ x_{ij}- : i \in \N \}$ in $U_{z_j-}$ and $\{ x_{ij}+ : i \in \N \}$ in $U_{z_j+}$. Since there are no isolated points, every open set is
   uncountable and so we can inductively make the choices so that for each $j$ none the points of $\{ x_{ij}- : i \in \N \} \cup \{ x_{ij}+ : i \in \N \} $
are contained in $\bigcup_{k< j} \ \{ x_{ik}- : i \in \N \} \cup \{ x_{ik}+ : i \in \N \}$.

   Let $H_i =   \{ x_{ij}- : j = 1, \dots, k \} \cup \{ x_{ij}+ : j = 1, \dots, k \}$ to define the pairwise disjoint sequence of spanning sets
   each with cardinality $2k$.

  \end{proof}
 \vspace{.5cm}

 An $n-$fold \emph{partition}\index{partition} $\{ C_1, \dots, C_n \}$ of a space $X$ is a cover by $n$ pair-wise disjoint clopen sets.
 It is called \emph{proper}\index{partition!proper}\index{proper partition} when no $C_i$ is empty.

A \emph{spanning set partition}\index{spanning set partition}  is a $2-$fold partition $\{ E, F \}$ of  $X$ by a pair  of complementary
spanning sets.

 \begin{prop}\label{exprop06a} Let $(X,R)$ be a topological tournament and $\{ E, F \}$ be a $2-$fold partition of $X$.
The following conditions are equivalent.
   \begin{itemize}
   \item[(i)] The restriction $R \cap (E \times F)$ is a surjective relation from $E$ to $F$.

   \item[(ii)] For every $a \in E, \ R(a) \cap F \not= \emptyset$ and for every $b \in F, \ R^{-1}(b) \cap E \not= \emptyset$.

   \item[(iii)] The images $R(E) \supset F$ and $R^{-1}(F) \supset E$.

   \item[(iv)] For every $a \in E,$ and $b \in F$, there exists $a' \in E, b' \in F$ such that $a \ha b',$ and $a' \ha b$.
 \end{itemize}
These conditions imply that neither $E$ nor $F$ is empty.

 Furthermore, the following conditions are equivalent.
   \begin{itemize}
  \item[(i)]  $R \cap [(E \times F) \cup (F \times E)]$ is a surjective relation on $X$.

   \item[(ii)] $R \cap (E \times F)$ is a surjective relation from $E$ to $F$ and $R \cap (F \times E)$ is a surjective relation from $F$ to $E$.

   \item[(iii)] $R(E) \cap R^{-1}(E) \supset F$ and $R(F) \cap R^{-1}(F) \supset E$.
   \end{itemize}

   If $\{ E, F \}$ is a spanning set partition,  then the restriction $R|[(E \times F) \cup (F \times E)]$ is a
surjective relation on $X$.  Conversely, if $(X,R)$ is balanced and $R|[(E \times F) \cup (F \times E)]$ is a
surjective, then $(E,F)$ is a spanning set partition.  \end{prop}

\begin{proof} The equivalences are easy to check. Definition \ref{exdef04}(iii) shows that a spanning set partition
satisfies $R(E) \cap R^{-1}(E) \supset F$ and $R(F) \cap R^{-1}(F) \supset E$. The converse holds when $(X,R)$ is balanced because
if a point $x$ is balanced, then $R^{\circ}(x)$ and $R^{\circ -1}(x)$ meet $E$ whenever $x $ is in the interior of $E$.

  \end{proof}
 \vspace{.5cm}

Now we develop the \emph{attachment construction}\index{attachment construction}. We begin with two examples.

 \begin{prop}\label{exprop06} Let $(X,R)$ be a topological tournament and $\{ E, F \}$ be a $2-$fold partition of $X$.

 Given a point $u$ not in $X$ let $X' \ = \ X \cup \{ u \}$ with $u$ isolated and define the topological tournament $R'$ on $X'$ by
 \begin{equation}\label{exeq06}
R'|X \ = \ R, \qquad \text{and} \quad R'^{\circ}(u) \ = \ F, \ \ R'^{\circ -1}(u) \ = \ E.
\end{equation}
When $X$ is compact, $X'$ is compact. If $(X,R)$ is wac, then $(X',R')$ is wac.

Assume that the relation $R \cap (F \times E)$ is surjective.

If $(X,R)$ is arc cyclic, then $(X',R')$ is arc cyclic.

If $(X,R)$ is prime, then $(X',R')$ is prime. \end{prop}

\begin{proof} The compactness and wac results are obvious.

If $b \in F$, there exists $a \in E$ with $b \ha a $ and if
$a \in E$, there exists $b \in F$ with $b \ha a$. In each case, $\{ a, u, b\}$ is a $3-$cycle.

It follows that $(X',R')$ is arc cyclic when $(X,R)$ is.

Now assume that $(X,R)$ is prime and that $U$ is a non-trivial, closed  $Q$ invariant subset of $X'$.

If any pair in $X$ is contained in $U$, then all of $X$ is contained in $U$. Using the above $3-$cycles we see that $u \in U$ as well.

If $\{ x, u \} \subset U$ for some $x \in X$, then the above $3-$cycles show that there exists $x'$ in the complementary member of the
pair $\{ E, F \}$ with $x' \in U$.  Again since $(X,R)$ is prime, $X \subset U$.

  \end{proof}
 \vspace{.5cm}

 \begin{prop}\label{exprop07} Let $(X,R)$ be a topological tournament and $\{ E, F \}$ be a $2-$fold partition of $X$.

  Given distinct points $u, v$ not in $X$ let $X'' \ = \ X \cup \{ u, v \}$ with $u, v$ isolated and define the tournament $R''$ on $X''$ by
 \begin{equation}\label{exeq07}\begin{split}
R''|X \ = \ R, \qquad \text{and} \quad R''^{\circ}(u) \ = \ F \cup \{ v \}, \ \ R''^{\circ}(v) \ = \ E, \\
\text{so that} \qquad R''^{\circ -1}(u)  \ = \ E, \ \ R''^{\circ -1}(v)  \ = \ F \cup \{ u \}. \hspace{1cm}
\end{split}\end{equation}
When $X$ is compact, $X''$ is compact. If $(X,R)$ is wac, then $(X'',R'')$ is wac.

Assume that  $R(E) \supset F$ and
 $R^{-1}(E) \supset F$, that is, for all $b \in F$, there exist $a, a' \in A$ such that $a \ha b \ha a'$.
 This assumption includes the possibility that $F = \emptyset$ and so $ EA = X$.

 If  $(X,R)$ is arc cyclic, then $(X'',R'')$ is arc cyclic.

Assume, in addition, that $F$ is nonempty.

If $(X,R)$ is prime, then $(X'',R'')$ is prime. \end{prop}

\begin{proof} The compactness and wac results are again obvious.

Now assume that  $R(E) \supset F$ and
 $R^{-1}(E) \supset F$.

If $a \in E$, then $\{ a, u, v \}$ is a $3-$cycle.

If $b \in F$, we may choose $a, a' \in E$ such that $a \ha b \ha a'$, so that
$\{ a, b, v\}$ and $\{ b, a', u \}$ are $3-$cycles.

It follows that $(X'',R'')$ is arc cyclic if $(X,R)$ is.

Observe that if $F = \emptyset$, then $E = X$ is a closed $Q$ invariant subset of $X''$. Smashing it to a point we see that
$(X'',R'')$ has a $3-$cycle quotient.

Now assume that $F$ is nonempty and that $(X,R)$ is prime. Notice that since $F$ is nonempty,
there are at least two points in $E$ and so at least three points in $X$.
It follows that $(X,R)$ is not an arc. Let $U$ be a non-trivial, closed $Q$ invariant subset of $X''$.

If $U$ contains any pair in $X$, then it contains all of $X$ because $(X,R)$ is prime.  Then since $F$ is nonempty, the
three cycles $\{ a, b, v\}$ and $\{ b, a', u \}$ imply that $u, v \in U$.

If $\{ b, v \} \subset U$ with $b \in B$, then $a \in U$ and if
 or $\{ b, u \}  \subset U$ then $a' \in U$.  Since two points of $X$ are in $U$, all of $X''$ is contained in $U$, again.

 Now assume $\{ a, u \} \subset U$ or $\{ a, v \} \subset U$ with $a \in A$, then the $3-$cycle $\{ a, u, v \}$ for all $a \in E$
 first implies first that both $u$ and $v$ are in $U$ and then that all of $E \subset U$. If $b \in B$, then the cycle $\{ a, b, v\}$
 implies that $b \in U$. Thus, $F \subset U$ and so $X'' = U$.

 Thus, $(X'',R'')$ is prime.

\end{proof}

Let $(Y,S)$ be a tournament containing isolated points $u, v$ with $(u,v) \in S$.  Assume that there does not exist $y \in Y \setminus \{ u,v \}$
such that either $y \in S(u) \cap S(v)$ or $y \in S^{-1}(u) \cap S^{-1}(v)$. With $E = S^{\circ -1}(u) = S^{\circ}(v)$ and
so $F = S^{\circ}(u) \setminus \{ v \} = S^{\circ -1}(v) \setminus \{ u \}$, we see that $(Y,S)$ is isomorphic to $(X'',R'')$
with $X = Y \setminus \{ u,v \}$ and $R = S|X$. Following \cite{A20} we then call $(Y,S)$ \emph{reducible}\index{reducible}
\index{tournament!reducible} via $\{ u, v \}$. If $(Y,S)$ is finite and no such pair $u, v$ exists, then, as above, we call $(Y,S)$
\emph{irreducible}\index{irreducible}\index{tournament!irreducible}.
\vspace{.5cm}

For a topological tournament $(X,R)$ we will call two  $2-$fold partitions
$(E_1,F_1)$ and $(E_2,F_2)$ distinct when $E_1 \not= E_2$. Note that
 \begin{equation}\label{exeq07a}
 E_1 \ \not= \ E_2 \quad \Leftrightarrow \quad  F_1 \ \not= \ F_2 \quad \Leftrightarrow \quad [(E_1 \cap F_2) \cup (E_2 \cap F_1)] \ \not= \ \emptyset.
 \end{equation}
This allow the possibility, which we will frequently use, that $E_1 = F_2$ and $F_1 = E_2$.

For the attachment construction we begin with two topological tournaments $(Y,S)$ and $(X,R)$. In the resulting tournament,
$(Y,S)$ and $(X,R)$ play symmetric roles, but it is convenient to use an asymmetric construction method. We assume that $X$ and $Y$
are disjoint.

Let $\{ C_i : i = 1, \dots, n \}$ be a proper $n-$fold partition of $Y$. Let
$\{ (E_i,F_i) : i = 1, \dots, n \}$ be a list of $n$ pairwise distinct $2-$fold partitions of $X$. Define
 $Z = X \cup Y$ with the topology on $Z$ so that $X$ and $Y$ are clopen subsets of $Z$ with their initial topologies the
 relative topologies from $Z$. Define the topological tournament
 $$T \subset Z \times Z = (Y \times Y) \cup (X \times X) \cup [(X \times Y) \cup (Y \times X)]$$
 by
 \begin{equation}\label{exeq08}
T \ = \ S \cup R \cup ( \bigcup_i [(E_i \times C_i) \cup (C_i \times F_i)]).
\end{equation}

We call $(Z,T)$ the \emph{attachment}\index{attachment}\index{tournament!attachment}
  of $(Y,S)$ to $(X,R)$ via $\{ C_i : i = 1, \dots, n \}$ and $\{ (E_i,F_i) : i = 1, \dots, n \}$.

  Clearly,  if $(Y,S)$ and $(X,R)$ are both compact, wac or locally arc cyclic, then $(Z,T)$ satisfies the corresponding property.

  For example, the tournament $(X'',R'')$ of Proposition \ref{exprop07} is the attachment of the
  arc  $Y = \{ u, v \}$ with $ u \ha v $ to $(X,R)$ with $C_1 = \{ u \}, C_2 = \{ v \}$ and
  $(E_1,F_1) = (E,F), (E_2,F_2) = (F,E)$.

  The tournament $(X',R')$ of Proposition \ref{exprop06} is the attachment of the trivial tournament on
  $Y = \{ u \} $ to $(X,R)$ with $C_1 = \{ u \}$ and $(E_1,F_1) = (E, F)$.

\begin{theo}\label{extheo08} Let $(Z,T)$ be the attachment
  of $(Y,S)$ to $(X,R)$ via $\{ C_i : i = 1, \dots, n \}$ and $\{ (E_i,F_i) : i = 1, \dots, n \}$.
  Assume that $n \ge 2$ and that for each $(E_i,F_i)$ the relation $R \cap (F_i \times E_i)$ is surjective (and so neither $E_i$ nor
  $F_i$ is empty), e.g. it suffices that each $(E_i,F_i)$ be a spanning partition.

   If $(Y,S)$ and $(X,R)$ are arc cyclic, then $(Z,T)$ is arc cyclic.

 If $(Y,S)$ and $(X,R)$ are prime, then $(Z,T)$ is prime.

 If $Y$ is finite with $|Y| = n$ so that each $C_i$ is a singleton, and $(X,R)$ is prime, then $(Z,T)$ is prime.

 If $X = \bigcup_i E_i$, then $T \cap (X \times Y)$ is a surjective relation from $X$ to $Y$.

  If $X = \bigcup_i F_i$, then $T \cap (Y \times X)$ is a surjective relation from $Y$ to $X$.
\end{theo}

\begin{proof} If $c \in C_i$ and $b \in F_i$, then there exists $a \in E_i$ such that $b \ha a $.
If $c \in C_i$ and $a \in E_i$, then there exists $b \in F_i$ such that $b \ha a$.  In each case,  $\{ c, b, a \}$ is a $3-$cycle.

It follows that  $(Z,T)$ is arc cyclic if $(Y,S)$ and $(X,R)$ are arc cyclic.

Now assume that $(X,R)$ is prime and that $U$ is a nontrivial, closed $Q$ invariant subset of $Z$.

If any pair of $X$ is in $U$, then $X \subset U$ because $(X,R)$ is prime and so the above $3-$cycles imply that every $c \in C_i$ is in $U$.
Thus, $Z \subset U$.

If some pair $\{ c, a \} \subset U$ with $c \in C_i$ and $a \in E_i$, then the above $3-$cycles show that
there exists $b \in F_i \cap U$. Since some pair in $X$ is contained in $U$ again $U = Z$. Similarly, if
$\{ c, b \} \subset U$ with $c \in C_i$ and $b \in F_i$ we have $U = X$.

There remains the case when some pair of $Y$ is contained in $U$.

Case 1: Assume that $(Y,S)$ is prime.  It then follows that $Y \subset U$. Let $c_1 \in C_i, c_2 \in C_j$ with $i \not= j$. Recall that the $C_i$'s
are nonempty and $n \ge 2.$   There exists $x \in (E_i \cap F_j) \cup (E_j \cap F_i)$ because the $2-$fold partitions of $X$
are distinct. It follows that $T^{\circ}(x)$ and $T^{\circ -1}(x)$ meet $U$ and so $x \in U$. Since $\{ c_2, x \} \subset U$ it follows as above that
$U = Z$.

Case 2: Assume that $|Y| = n$. If a pair $\{ c_1, c_2 \} \subset U$ with $c_1 \not= c_2$ in $Y$, there exist
$i \not= j$ such that $\{ c_1 \} = C_i $ and $\{ c_2 \} = C_j $.  Since $i \not= j$ we may choose $x$ as in Case 1, and so
 obtain that $U = Z$.

 If $y \in C_i$, then with $x \in E_i, x' \in F_i$ we have $x \ha y \ha x'$. If $x \in X$ and $\bigcup_i E_i = X$, then $x \in E_i$ for
 some $i$ and so $x \ha y$ for $y \in C_i$. Hence, $T \cap (X \times Y)$ is surjective.  Similarly if $\bigcup_i F_i = X$, then
 $T \cap (Y \times X)$ is surjective.

 \end{proof}

\textbf{ Remark:} A \emph{Special Case} which we will use repeatedly has $n = 2$, with $(E_1, F_1)$  a spanning partition and
$(E_2, F_2) = (F_1, E_1)$, in which case, of course,
 $E_1 \cup E_2 = F_1 \cup F_2 = X$.

\vspace{.5cm}

\begin{exes}\label{exes07aa} Constructing Generalized Reduced Doubles\end{exes}
Let $(X,R)$ be a compact  tournament with no initial or terminal point. Assume that $X$ totally disconnected with no isolated points.
By Proposition \ref{qprop09aa} the assumption that $X$ be totally disconnected is redundant when
$(X,R)$ is prime.

Let $\{ E, F \}$ be a spanning partition, which exists by Theorem \ref{extheo05a}. Define

\begin{equation}\label{exeq09}
C_1 \ = \ E_1 \ = \ F_2  \ = \  E, \quad \text{and} \quad C_2  \ = \ F_1 \ = \ E_2 \ = \  F.
\end{equation}

  We let $X\pm = X \times \{\pm 1 \}$, writing, as before, $x\pm$ for $(x, \pm 1)$ with $x \in X$. The tournaments $(X\pm, R\pm)$ are
  defined so that each is a copy of $(X,R)$ via the isomorphisms $x\pm \mapsto x$.

  We obtain a generalized reduced double  $2'(X,R) = (2'X, 2'R)$, as defined in Example \ref{exesnew01} (b), by using
   the attachment of $(X+,R+)$ to $(X-,R-)$ via $\{ C_1, C_2\}$ and $\{ (E_1,F_1),  (E_2,F_2)\}$.
  This is an example of the Special Case mentioned in the Remark after Theorem \ref{extheo08}. In particular,
  $2'R \cap [(X- \times X+) \cup (X+ \times X-)]$ is a surjective relation on $2'X$.

  It follows that if $(X,R)$ is wac, locally arc cyclic, arc cyclic or
  prime, then $2'(X,R)$ satisfies the corresponding property.

\vspace{1cm}

  \section{ \textbf{Prime Tournament Examples }}\vspace{.5cm}

  We will show that $(\Z[2], \widehat{A})$, the standard tournament of $2-$adics, is a prime tournament.  This provides us with an example of
  a prime, arc cyclic tournament on a Cantor set.  We will use it to construct other such examples. However, we require an invariant
  which will allow us to distinguish among such examples.  What we will use for such a tournament $(X,R)$ is the collection
  of the  almost wac tournaments which are the restrictions of $R$ to the subsets $R(x)$ as $x$ varies over $X$. In particular, we will look
  at the prime quotients of these restrictions. \vspace{.25cm}

Now recall that we regard the additive  group of $2-$adic integers,  $\Z[2]$
  as the product $ \{ 0. 1 \}^{\N}$ with addition of two sequences pointwise but with carrying to the right.
  We write $\0 = 000\dots $ for the identity element, instead of $e$, and we write the group additively.  Thus,
  $\Z[2]$ is a topological group on a Cantor set. In fact, as it is the inverse limit of the finite rings $\Z/2^i \Z =  \{ 0. 1 \}^{i} $,
  $\Z[2]$ is a topological integral domain with $\1 = 100\dots $ the multiplicative identity.
  Two elements of $\Z[2]$ are congruent mod $2^i$ when their projections to $\Z/2^i \Z$ are equal, or
  equivalently, they have the same first $i$ coordinates. In particular,
  $x \in \Z[2]$ is \emph{even}, i.e. there exists $x'$ such that $x = 2 x'$, if and only if $x_1 = 0$. Otherwise, $x_1 = 1$ and $x$ is
  \emph{odd} with $x - \1 $ even.

   With $\bar 0 = 1, \bar 1 = 0$ we defined $\bar y$ for $y \in \Z[2]$ by $(\bar y)_i = \overline{y_i}$ and saw that
   $y + \bar y + \1 = \0 $ and so $-y = \bar y + \1 .$  If $y = 0^{i-1} 1 z$, then $-y = 0^{i-1} 1 \bar z$.
   We defined  $A_i = \{ 0^{i-1} 1 0 z : z \in \Z[2]\}$, a clopen
   subset with $-A_i = \{ 0^{i-1} 1 1  z : z \in \Z[2]\}$. We then defined the game subset $A = \{ \0 \} \cup (\bigcup_{i} A_i)$. We use the label
   $\widehat{A}$ for the tournament associated with $A$.

   The set $A_1$ consists of the elements $x \in \Z[2]$ which are congruent to $1$ mod 4. This is a multiplicative subgroup of $\Z[2]$ and it
   is easy to check that multiplication by any element of $A_1$ is an additive group isomorphism which preserves each $A_i$ and so
   is a tournament automorphism for $\widehat{A}$.

   Define the shift map $\s$ on $\Z[2]$ by $\s(y)_i = y_{i+1}$. Algebraically, $\s$ is given by
   \begin{equation}\label{adiceq01aa1}
  \s(y) \ = \ \begin{cases} \ \ y/2  \ \text{if} \ y \ \text{is even}, \\ \ (y - \1 )/2 \ \text{if} \ y \ \text{is odd}.\end{cases}
\end{equation}

For $k \in \N $ and $ w \in \{ 0, 1 \}^k $ let $I_w = \{ z \in \Z[2] : z_i = w_i$ for $i = 1, \dots, k \}$.  This is the mod $2^i$
congruence class associated with $w$.  Thus,
$\s^k : I_w \to \Z[2] $ is a bijection with inverse $x \mapsto wx$.

  Observe that for all $i,j  \in \N$ with $j > i+1$ and all $x, y \in \Z[2]$
  \begin{align}\label{adiceq01}\begin{split}
(i) &\quad  0^{j-1}1x + 0^{i-1}1 z \ = \  0^{i-1}1 y,  \quad  \text{with} \ \ z + 0^{j-i-1}1x = y, \\
(ii) &\quad 0^{i-1}10x + 0^{i-1}10z \ = \  0^i1y \quad  \ \text{with} \ \ z + x = y, \\
(iii) &\quad 0^{i-1}11x + 0^{i-1}11z \ = \  0^i1y \quad  \ \text{with} \ \ z + x + \1 = y, \\
(iv) &\quad 0^{i-1}10x + 0^{i}1\ep z \ = \  0^{i-1}11y \quad  \ \text{with} \ \ \ep z + x = y \ (\ep = 0,1).
\end{split}\end{align}
\vspace{.25cm}

Recall from Example \ref{exes06} (a) the tournament $N_1 = (\N^*,L_1)$ and its inverse $\bar N_1 = (\N^*,L_1^{-1})$.

   \begin{theo}\label{adictheo01} (a) The $2-$adic group tournament $(\Z[2],\widehat{A})$ is an arc cyclic, prime tournament on a Cantor set.
\vspace{.25cm}

   (b) The homeomorphism $inv$ on $\Z[2]$ given by $inv(x) = -x$, i.e. multiplication by $- \1 $, is an isomorphism from  $(\Z[2],\widehat{A})$ to
   $ (\Z[2], \widehat{A}^{-1})$.
\vspace{.25cm}

   (c) For each $k \in \N, \ w \in \{ 0, 1 \}^k $, the shift $\s^{k}  : I_w \to \Z[2]$ is an isomorphism from the restriction
    $\widehat{A}|I_w$  to    $\widehat{A}$ .
   \vspace{.25cm}

   (d) For each $x \in \Z[2]$ the restriction of $\widehat{A}$ to $\widehat{A}(x)$ is isomorphic to the topological lexicographic product
   $\bar N_1 \times \{ (Y_a,S_a) \}$ with $(Y_i, S_i) = (\Z[2],\widehat{A})$ for $i \in \N $ and with $(Y_{\infty}, S_{\infty)}$ trivial.
   The projection map to $\bar N_1 = (\N^*, L_1^{-1})$ is given by $A_i \mapsto i$ and $\0 \mapsto \infty$.

    The restriction of $\widehat{A}$ to $\widehat{A}^{-1}(x)$ is isomorphic to the topological lexicographic product
   $N_1 \times \{ (Y_a,S_a) \}$ with $(Y_i, S_i) = (\Z[2],\widehat{A})$ for $i \in \N $ and with $(Y_{\infty}, S_{\infty)}$ trivial.
   The projection map to $N_1 = (\N^*, L_1)$ is given by $-A_i \mapsto i$ and $\0 \mapsto \infty$.

   In each case, the two-level product is the classifier for the restriction.
    \end{theo}

   \begin{proof} Observe first that $(\Z[2],\widehat{A})$  is  arc cyclic by Theorem \ref{theocycle02}.

   (b): Clearly, $x - y \in A$ if and only if $inv(x) - inv(y) \in - A $. Recall that $\widehat{- A} = \widehat{A}^{-1}$.

   (c): Note  that $z \in A_i$ if and only if $0^kz \in A_{i+k}$. The result follows because
      $w x - w y = 0^{k}(x - y)$.

   (d): From (\ref{adiceq01})(i)-(iii) we see that for  all $i,j  \in \N$ with $j > i+1$
   \begin{align}\label{adiceq01ac}\begin{split}
   x \in A_i \ \ \  \text{and} \ \ x' \in A_j \cup - A_j \quad \ \  &\Rightarrow \ \ (x', x) \in \widehat{A}^{\circ}, \\
   x \in -A_i \ \ \text{and} \ \ x' \in A_j \cup - A_j \quad \ &\Rightarrow \ \ (x, x') \in \widehat{A}^{\circ}, \\
   x \in A_i \ \ \text{and} \ \ x' \in A_{i+1} \cup - A_{i+1} \ \ &\Rightarrow \ \ (x, x') \in \widehat{A}^{\circ}, \\
   x \in -A_i \ \ \text{and} \ \ x' \in A_{i+1} \cup - A_{i+1} \ \ &\Rightarrow \ \ (x, x') \in \widehat{A}^{\circ}.
   \end{split}\end{align}

    From (\ref{adiceq01})(iv) it follows that
      for $x \in A_i$ and $ x' \in -A_i$,
     \begin{equation}\label{adiceq01ae}
     (x,x') \in \widehat{A}^{\circ} \ \ \text{if } \ \ x_{i+2} = x'_{i+2} \ \ \text{and} \ \  (x',x) \in \widehat{A}^{\circ} \ \ \text{otherwise}.
     \end{equation}

   Since translation by $-x$ is an automorphism of $\widehat{A}$ we may restrict attention to
   $x = \0 $.

   Any neighborhood of $\0 $ contains  $A_i \cup -A_i$ for $i$ sufficiently large. Hence, $A_i \mapsto i$ and $\0 \mapsto \infty$ is a
   continuous surjection  from $A$ onto $\N^*$.

   From (\ref{adiceq01ac}) it follows that for the restriction of $\widehat{A}$ to $A$,
   we have for $j > i+1$ that $A_j \ha A_i$ and
   $A_i \ha A_{i+1}$. We see, first, that in $A$ each $A_i$ is a $Q$ invariant clopen subset and, second, that
   the quotient tournament is isomorphic to
   $\bar N_1 = (\N^*, L_1^{-1})$ which is prime by Theorem \ref{extheo01a}.

   Since $(\Z[2],\widehat{A})$ is arc cyclic
   Theorem \ref{wactheo02} (f) implies that the quotient map induces an isomorphism of $\widehat{A}$ on $A$ with
   the lexicographic product $\bar N_1 \times \{ (Y_a,S_a) \}$ with $(Y_i, S_i)$ the restriction of $\widehat{A}$
   to $A_i$ which is, by (c), isomorphic to $(\Z[2],\widehat{A})$.

  The proof for the restriction to $-A$ is similar or can be obtained using the isomorphism $inv$.

   (a): A section $\xi : \N^* \to A$ is a choice function with $\xi(i) \in A_i$ and with $\xi(\infty) = \0 $. A section
   $\bar \xi : \N^* \to -A$ is a choice function with $\bar \xi(i) \in -A_i$ and with $\bar \xi(\infty) = \0 $. By
   Lemma \ref{lemlex01b} each section is continuous. Each $\xi$ is a tournament isomorphism from $\bar N_1$ to the
   restriction of $\widehat{A}$ to the image $\xi(\N^*)$, and each $\bar \xi$ is a tournament isomorphism from $ N_1$ to the
   restriction of $\widehat{A}$ to the image $\bar \xi(\N^*)$.

   To prove that $(\Z[2],\widehat{A})$ is prime, we let $U$ be a non-trivial $Q$ invariant subset. By translation we may assume that
   $\0 \in U$. Since the tournament is arc cyclic, $U$ is clopen. Hence, for sufficiently large $i \in \N$,
   $A_i \cup -A_i \subset U$. It follows that for any section $\xi$ or $\bar \xi$, infinitely many points of the image are contained in $U$.
   The restriction of $\widehat{A}$ to each image is prime and so each entire image is contained in $U$. For any $x \in A_i$
   there is a section $\xi$ with $\xi(i) = x$. Hence, $x \in U$.  For any $x \in -A_i$
   there is a section $\bar \xi$ with $\bar \xi(i) = x$. Hence, $x \in U$. Thus, $U = \Z[2]$ and so the tournament is prime.

   This implies that all of the $(Y_i,S_i)$'s are all prime and so the above lexicographic product is the second stage of the classifier construction.
Since all the points of $\Z[2]$ are non-isolated, the classifier system terminates at this second level.

     \end{proof}

\textbf{Remark:} It follows from the uniqueness of the classifiers, that if $h$ is any automorphism of $(\Z[2],\widehat{A})$ such that
$h(\0 ) = \0 $, then $h(\pm A_i) = \pm A_i$ for all $i \in \N$.

 \vspace{.5cm}

 For any $j \in \N$, we define the complementary subsets $D_j, \bar D_j$ by:
 \begin{equation}\label{adiceq02}
D_j \ = \ \{ x \in \Z[2] : x_j = 0 \}, \qquad \bar D_j \ = \ \{ x \in \Z[2] : x_j = 1 \}.
\end{equation}

\begin{prop}\label{adicprop02} (a) For each $j \in \N$ the $2-$fold partition $\{ D_j, \bar D_j \}$ is a spanning set partition.
\vspace{.25cm}

(b) For any $x \in \Z[2] \ \ \widehat{A}^{\circ}(x) \ = \ \bigcup_i \ (x + A_i)$ and we have
\begin{itemize}
\item If $x \in D_j$ (or $ x \in \bar D_j$), then for all $i > j$, $(x + (\pm A_i)) \subset D_j$ (resp.  $(x + ( \pm A_i)) \subset \bar D_j$)
and $x + A_j \subset \bar D_j$ (resp. $x + A_j \subset D_j$).
\vspace{.25cm}

\item If $x \in D_j \cap D_{j-1}$ or $x \in \bar D_j \cap \bar D_{j-1}$, then
$x + A_{j-1} \subset D_j$ and $x + (- A_j) \subset \bar D_j$. If $x \in \bar D_j \cap D_{j-1}$ or $x \in D_j \cap \bar D_{j-1}$, then
$x + A_{j-1} \subset \bar D_j$ and $x + (- A_j) \subset  D_j$.
\vspace{.25cm}

\item For any $x \in \Z[2]$, if $i < j-1$, then
$(x + (\pm A_i)) \cap D_j$ and $(x + (\pm A_i)) \cap \bar D_j$ are nonempty.
\end{itemize}
\vspace{.25cm}

(c) The restrictions $\widehat{A}|D_1$ and $\widehat{A}|\bar D_1$ are each isomorphic to $(\Z[2],\widehat{A})$
via the maps $x \mapsto \ep x$ for $\ep = 0, 1$.
\vspace{.25cm}

(d) Define the map $h: Z[2] \times \{ -1, +1 \} \to Z[2]$ by $x- \mapsto 0x$ and $x+ \mapsto 1x$.
The map $h$ is a homeomorphism. Letting $T$ be the topological
tournament on $Z[2] \times \{ -1, +1 \}$ such that $h$ is a tournament isomorphism to $\widehat{A}$ on  $Z[2]$, we obtain
a generalized reduced double  $2'(\Z[2],\widehat{A})$ which is itself isomorphic to $(\Z[2],\widehat{A})$.
\end{prop}

\begin{proof} (a): Adding $0^{j-1}10y$ or $0^{j-1}11y$ to any $x$ in one of the partition elements yields a point in the opposite element.
Adding $0^k10y$ or $0^k11y$ with $k \ge j$ yields a point in the same element. Hence, $\widehat{A}^{\circ}(x)$ and $\widehat{A}^{\circ -1}(x)$
each meet both $D_j$ and $\bar D_j$.

(b): Because the translation $\ell_x$ is an automorphism of $(\Z[2],\widehat{A})$ it follows that
$$ \widehat{A}^{\circ}(x) \ = \ x + \widehat{A}^{\circ}(\0 ) \ = \ x + A^{\circ} \ = \  \bigcup_i (x + A_i).$$
The remaining results are easy to check directly using $A_i = \{ 0^{i-1}10z : z \in \Z[2] \}$.

(c): $D_1 = I_0$ and $\bar D_1 = I_1$  and so the results follow from Theorem \ref{adictheo01}(c).

(d): This is clear from (c). This is an example of the reduced double construction from Example \ref{exes07aa}.

\end{proof}

 \vspace{.5cm}

 On $\Z[2]$ we define the \emph{twist map}\index{twist map}  $\t_j$ by
 \begin{equation}\label{adiceq02a}
 \t_j(w0x) \ = \ w1(x+ \1 ), \quad \t_j(w1x) \ = \ w0x \quad \text{with} \ \ w \in \{0, 1 \}^{j-1}, x \in \Z[2].
 \end{equation}

 \begin{prop}\label{adicprop02aaa} The twist map $\t_j$ is an automorphism of $(\Z[2],\widehat{A})$ such that
 $\t_j(D_j) \ = \ \bar D_j$ and $\t_j(\bar D_j) \ = \ D_j$. \end{prop}

 \begin{proof} It is clear that $\t_j$ is a homeomorphism which interchanges $D_j$ and $\bar D_j$. To check that it maps arcs to arcs
 we must consider a number of cases of the effect of adding $0^{k-1}10z$.
 \begin{itemize}
 \item[(i)] ($k > j$):  For $\ep = 0, 1$ if $w \ep x + 0^{k-1}10z = w \ep y$, then
$w \bar \ep x +  0^{k-1}10z = w \bar \ep y$ and $w \bar \ep (x + \1 ) +  0^{k-1}10z = w \bar \ep (y + \1 )$.
\vspace{.25cm}

 \item[(ii)] ($k = j$): If $w0x + 0^{j-1}10z = w1y$ (and so $x + 0z = y$ and $x_1 = y_1$), then
 $\t_j(w0x) = w1(x + \1 ) + 0^{j-1}1(0z - 2) = w0y = \t_j(w1y)$.

 If $w1x + 0^{j-1}10z = w0y$ (and so $x + \1 + 0z = y$ and  $\bar x_1 = y_1$), then
 $\t_j(w0x)  + 0^{j-1}1(0z - 2)  = w1(x + \1 ) + 0^{j-1}1(0z - 2) = w0(y + \1 ) = \t_j(w0y)$.
\vspace{.25cm}

  \item[(iii)] ($k = j-1$): With $w' \in \{0, 1 \}^{j-2}$ we have:

  If $w'00x + 0^{j-2}10z = w'10y$, then $w'01(x + \1 ) + 0^{j-2}10z = w'11(y + \1 )$.

  If $w'10x + 0^{j-2}10z = w'01y$, then $w'11(x + \1 )  + 0^{j-2}1(0z - 2) = w'00y$.

   If $w'01x + 0^{j-2}10z = w'11y$, then $w'00x + 0^{j-2}10z = w'10y$

 If $w'11x + 0^{j-2}10z = w'00y$, then $w'10x + 0^{j-2}1(0z + 2) = w'01(y + \1 )$.
 \vspace{.25cm}

 \item[(iv)] ($k < j-1$): With $p \in \{0, 1 \}^{j-k-1}$ we have the following.

 If $w0x + 0^{k-1}10pz = w'0y$, then $w0 + 0^{k-1}10p = w'0$ and $x + z = y$ (no carry to the $j+1$ place), and so
 $w1(x + \1 ) +  0^{k-1}10pz= w'1(y + \1 )$ since $w1 + 0^{k-1}10p = w'1$.

 If $w0x + 0^{k-1}10pz = w'1y$, then $w0 + 0^{k-1}10p = w'1$ and $x +  z = y$ (no carry to the $j+1$ place), and so
 $w1(x + \1 ) +  0^{k-1}10p(z - 2) = w'0y$ since $w1 + 0^{k-1}10p = w'01$.

  If $w1x + 0^{k-1}10pz = w'0y$, then $w1 + 0^{k-1}10p = w'01$ and $x + \1 + z = y$ (carry to the $j+1$ place), and so
 $w0x +  0^{k-1}10p(z + 2) = w'1(y + \1 )$ since $w0 + 0^{k-1}10p = w'1$.

   If $w1x + 0^{k-1}10pz = w'1y$, then

   EITHER,  $w1 + 0^{k-1}10p = w'1$ and $x +  z = y$ (no carry to the $j+1$ place), and so
 $w0x +  0^{k-1}10pz = w'0y$ since $w0 + 0^{k-1}10p = w'0$.

 OR, $w1 + 0^{k-1}10p = w'11$ and $x + \1  + z = y$ (carry to the $j+1$ place), and so
$w0x +  0^{k-1}10pz = w'0y$ since $w0 + 0^{k-1}10p = w'01$.
\vspace{.25cm}
 \end{itemize}

 For all of these cases,  $x \ha x'$ implies $\t_j(x) \ha \t_j(x')$.

\end{proof}

 \vspace{.5cm}

\begin{ex}\label{adicexes02} Arc cyclic, prime tournaments on a Cantor set via attachment. \end{ex}

Fix $j, k \in \N$.   We let  $Z = \Z[2] \times \{-1, +1 \}$ labelling $Z\pm = \Z[2] \times \{ \pm \}$. As usual we write $z\pm$ for
$(z, \pm 1) \in Z \pm$ and we write $  \widehat{A}\pm$
   for the tournament $\widehat{A}$ on $Z\pm$. In general, for any  $B \subset \Z[2]$ we write $B \pm$ for the
   copy of the subset in $Z \pm$. In particular, we let $D_j+$ and $\bar D_j+$ be the copy of $D_j$ and $\bar D_j$ in $Z+$ and
   $D_k-$ and $\bar D_k-$ be the copy of $D_k$ and $\bar D_k$ in $Z-$.

   We define the topological tournament $P[j,k]$ on $Z$, regarding
   $$ Z \times Z \ = \ (Z+ \times Z+)  \cup  (Z- \times Z-)  \cup  (Z- \times Z+)  \cup  (Z+ \times Z-).$$
   Define
  \begin{align}\label{exeq01neww01a}\begin{split}
P[j,k] \ = \ \widehat{A}+ \ &\cup \ \widehat{A}- \ \cup\\
[(D_j- \times D_j+)  \cup  (\bar D_j- \times \bar D_j+)] \ &\cup \ [(D_j+ \times \bar D_j-)  \cup  (\bar D_j+ \times  D_j-)].
 \end{split}\end{align}
 Thus, the restriction of $P[j,k]$ to $Z\pm$ is independent of $j,k$.

 Define the \emph{twist map} $\t_{j,k}$ and the \emph{interchange map} $ \r_{j,k}$ on $Z$ by
 \begin{align}\label{exeq01neww01b}\begin{split}
 \t_{j,k}(z+) \ = \ (\t_j(z))+, &\quad \t_{j,k}(z-) \ = \ (\t_k(z))-, \\
  \r_{j,k}(z+) \ = \ z-, &\quad \r_{j,k}(z-) \ = \ (\t_j(z))+.
  \end{split}\end{align}

 \begin{theo}\label{extheo10a} For each $j,k \in \N$, $(Z,P[j,k])$ is an arc cyclic, prime tournament on the Cantor set.

 The twist map $\t_{j,k}$ is an automorphism
 of $(Z,P[j,k])$ which interchanges each $D_j+$ with $\bar D_j+$ and $D_k-$ with $\bar D_k-$.

 The interchange map $\r_{j,k}$ is an isomorphism from $(Z,P[j,k])$ to $(Z,P[k,j])$.
 \end{theo}

 \begin{proof} The tournament $(Z,P[j,k])$ is the attachment of $(\Z[2],\widehat{A}) = (Z+,\widehat{A}+)$ to
 $(\Z[2],\widehat{A}) = (Z-,\widehat{A}-)$ via $\{ C_1 = D_j+, C_2 = \bar D_j+ \}$ and
 $\{(E_1,F_1) = (D_k-, \bar D_k-), (E_2, F_2) = (\bar D_k-, D_k-)\}$. Since $(\Z[2],\widehat{A})$ is arc cyclic and prime
 it follows that $(Z,P[j,k])$ is arc cyclic and prime.

 From Proposition \ref{adicprop02aaa} it follows that $\t_{j,k}$ and $\r_{j,k}$ preserve arcs which are contained in $Z+$ or $Z-$.
 Let $z- \in Z-$ and $z'+ \in Z+$.
 If $z- \in D_k-$ and $z+ \in D_j+$ so that $(z-,z'+) \in P[j,k]$, then
 $\t_{j,k}(z-) \in \bar D_k-$ and $\t_{j,k}(z'+) \in \bar D_j+$ so that $(\t_{j,k}(z-),\t_{j,k}(z'+)) \in P[j,k]$.
 $\r_{j,k}(z-) \in \bar D_k+$ and $\r_{j,k}(z'+) \in D_j-$ so that $(\r_{j,k}(z-),\r_{j,k}(z'+)) \in P[k,j]$. The remaining
 three possibilities are similar.

 \end{proof}
 \vspace{.25cm}

 Now we analyze the wac tournaments which are the restrictions to $P[j,k](x)$ for $x \in Z$. Because of the twist and interchange
 automorphisms, we need only consider the case when $x = z+ \in  D_j+$.

 \begin{equation}\label{exeq01neww02}
P[j,k]^{\circ}(z+) \ = \ \bar D_k- \ \cup \ [\bigcup_{i \in \N} \ (z + A_i)+].
\end{equation}

 From Proposition \ref{adicprop02} we have $z + A_i \subset D_j$ for all $i > j$, and
 $z + A_j \subset \bar D_j$ Furthermore, $z + A_{j-1} \subset D_j$ if $z \in D_{j-1}$ and $z + A_{j-1} \subset \bar D_j$ if $z \in \bar D_{j-1}$.
  For each $i < j-1$, $z + A_i$ meets both $D_j$ and $\bar D_j$.

  It follows that for the restriction of $P[j,k]$ to $P[j,k](z+)$ the set $\bar D_k-$ and each $(z + A_i)+$ for $i \ge j-1$ is a $Q$ invariant subset.
   So we obtain
 a quotient by smashing each to a point, which we will label $\bar d_k -$ and $i+$ for $i \ge j-1$. We label quotient tournament
 $(Z_{z+},T_{z+})$ with
  \begin{equation}\label{exeq01neww03}
 Z_{z+} \ = \ \{ z+ \}  \ \cup \  \{ \bar d_k - \} \ \cup \ \{ i+ : i \ge j-1 \} \ \cup \ [\bigcup_{i=1}^{j-2} \ (z + A_i)+].
 \end{equation}

  Case 1 ($j = 1$): The map $\bar d_k \mapsto 1 \ \ i+ \mapsto i+1, z+ \mapsto \infty $ is an isomorphism from $(Z_{z+},T_{z+})$
  onto the prime tournament $\bar N_1 = (\N^*,V_1^{-1})$. The next stage of the classifier for the restriction to $P[j,k](z+)$
 is the topological lexicographic product  $\bar N_1 \ltimes \{ (Y_a,S_a) \}$ with $(Y_1,S_1)$ the restriction of $\widehat{A}$ to $\bar D_k$,
$(Y_i,S_i) = (\Z[2], \widehat{A})$ for each $i > 1$ and $(Y_{\infty},S_{\infty})$  trivial. Note that since the almost wac tournament
which is the restriction to $P[j,k](z+)$ does not have an arc quotient it does have a classifier, see the Remark after Theorem \ref{classtheo02}.
   \vspace{.25cm}

Case 2 ($j = 2$) : The restriction of $T_{z+}$ to $\N^*+$ and the restriction to $\{ \bar d_k- \} \cup \N^{*} \setminus \{ 1 \}$ are each isomorphic
to $\bar N_1$ with $1+ \ha \bar d_k-$ if $x \in D_1$ and $\bar d_k- \ha 1$ if $x \in \bar D_1$. In either case, the pair $\{ \bar d_k-, 1 \}$ is a $Q$
invariant subset on which $T_{z+}$ restricts to an arc. Smashing the two points together we again obtain $\bar N_1$ as the prime quotient of
the restriction to
$P[j,k](z+)$. However, in this case, the next stage of the classifier for the restriction to
$P[j,k](z+)$
 is the topological lexicographic product  $N_1 \ltimes \{ (Y_a,S_a) \}$ with $(Y_1,S_1)$ the arc on $\{ \bar d_k-,1 \}$,
$(Y_i,S_i) = (\Z[2], \widehat{A})$ for each $i > 1$ again and $(Y_{\infty},S_{\infty})$  trivial. Over the arc at the next stage, one fiber is isomorphic to
$(\Z[2], \widehat{A})$ and the other is isomorphic to the restriction of $\widehat{A}$ to $\bar D_k$.
 \vspace{.25cm}

 Case 3 ($j > 2$) : The tournament $(Z_{z+},T_{z+})$ is  prime. The restriction of $T_{z+}$ to the Cantor set portion of $Z_{z+}$, which is
 $\bigcup_{i=1}^{j-2} \  (z + A_i)+$ is prime when $j = 3$ and when $j > 3$ it
 has a prime quotient which is the restriction of $\bar N_1$ to the set $\{ 1, \dots, j-2 \}$ (which is an arc when $j = 4$).

 \begin{proof}    A section for $(Z_{z+},T_{z+})$ is a map $\xi$
 from $\N^*$ to $Z_{z+}$ with $\xi(\infty) = z+, \xi(i) = i+$ for $i \ge j-1$ and $\xi(i) \in (x + A_i)+$ for $i \le j-2$. A section induces an
 isomorphism from $\bar N_1$ on $\N^*$ to the restriction of $T_{z+}$ on the image of $\xi$.

 Let $U$ be  a closed, non-trivial $Q$ invariant subset. It is clopen in $Z_x$ because $(Z_{z+},T_{z+})$ is
 almost wac.

 If $U$ contains any point in $(z + A_i)+$ for some $i \le j-2$, then from Lemma \ref{qlem01a} it follows that
 $U$ contains additional points of $(z + A_i)+$ because the restriction to $(z + A_i)+$ is balanced. Since the restriction
 is also prime, it then follows that $(z + A_i)+ \subset U$.
 In that case, $\bar d_k \in U$ since $z + A_i$ meets both $D_j$ and $\bar D_j$.
 If $y \in (z + A_i) \cap D_j$, then $\{ y+, \bar d_k, j+ \}$ is a $3-$cycle in $Z_{z+}$ and so $j+ \in U$. Furthermore, there exists a section through
 $y+$ and $j+$.

    The restriction to the image of any section $\xi$ is  prime and so if $U$ meets two points in the image, then
    it contains the entire image. By varying the section we see that  $(z + A_i)+ \subset U$ for all $i \le j-2$. Thus, if $U$ meets
    any $(z + A_i)+$ with $i \le j-2$ or meets any section in two points, then we have
     $U = Z_{z+}$.

    If $\bar d_k-, k+ \in U$ with $k \ge j$, then for $y \in (z + A_{j-2}) \cap D_j$ we have $k+ \ha y+ \ha \bar d_k$ and so $y+ \in U$. As two points
    of a section lie in $U$, we have $U = Z_{z+}$. Finally, if $\bar d_k-, (j-1)+ \in U$  and $y' \in (z + A_{j-2}) \cap \bar D_j$, then
    $\bar d_k- \ha y'+ \ha (j-1)+$ and so again two points of a section are in $U$ and $U = Z_{z+}$.

    Thus, $(Z_x, T_x)$ is prime.

    Now restrict to the Cantor set portion of $Z_x$. If $j = 3$, then the Cantor set portion
    is $z + A_1$ whose restriction is isomorphic to $(\Z[2],\widehat{A})$ and so is prime. For $j > 3$, each $(z + A_i)+$ is $Q$ invariant
     the resulting quotient is clearly
    restriction of $\bar N_1$ to the set $\{ 1, \dots, j-2 \}$. This is an arc when $j = 4$ and is prime in any case.

    \end{proof}

    \begin{lem}\label{exlem10b} The restriction of $\widehat{A}$ to the subsets $D_2 $ and
    $\bar D_2$ of $\Z[2]$ is isomorphic to the lexicographic product
    $\{ 0, 1 \} \times (\Z[2],\widehat{A})$ where $\{ 0, 1 \}$ is the arc with $0 \ha 1$. \end{lem}

    \begin{proof} $D_2$ is the disjoint union of $I_{00}$ and $I_{10}$ with $I_{00} \ha I_{10}$. These are $Q$ invariant for
    the restriction to $D_2$ and so the quotient is the arc $\{ 0, 1 \}$. The restriction of $\widehat{A}$ to each of
    $I_{00}$ and $I_{10}$ is isomorphic to $(\Z[2],\widehat{A})$.

    By he twist map $\t_2$, the restriction to $\bar D_2$ is isomorphic to
    the restriction to $D_2$.

    \end{proof}

    \begin{theo}\label{extheo10c} As the pair $(j,k)$ varies over the set $\{ (1,1), (1,2), (2,2) \} $ \\$\cup \{ (j,k) : j \le k,$ and $j \not= 2 \}$,
    no two of the tournaments $(Z,P[j,k])$ are isomorphic. \end{theo}

    \begin{proof} From the interchange isomorphism we see that $(Z,P[j,k])$  and $(Z,P[k,j])$ are isomorphic for any $j, k \in \N$.

    With $j = 1, 2$ the prime quotients of $P[j,k]|P[j,k](z+)$ are isomorphic to $\bar N_1$ by Case 1 and Case 2, above.

    We distinguish between $P[1,1], P[1,2]$ and $P[2,2]$ by looking at the classifiers for the restrictions to $P[j,k]|P[j,k](z+)$
    and applying Lemma \ref{exlem10b}. Notice that $(Z,P[1,1])$ is isomorphic to $(\Z[2],\widehat{A})$ as it is the same
    as the generalized reduced double constructed in part (d) of Proposition \ref{adicprop02}.  Furthermore, none of these can be isomorphic
    to any $(Z,P[j,k])$ with $j \ge 3$ since some of the prime quotients in the latter case have Cantor set portions.

    For $j, j' \ge 3$ and any $k, k'$ the restriction to the Cantor set portion of the prime quotients of $P[j,k]|P[j,k](z+)$ and
    $P[j',k']|P[j,k](z+)$ have in turn different prime quotients when $j \not= j'$ by Case 3.  In particular,
    $(Z,P[j,1])$ and $(Z,P[j',1])$ are not isomorphic if $j \not= j'$. Furthermore, for $j, j', k, k' \ge 3$ with $j \le k$ and $j' \le k'$,
    $(Z,P[j,k])$ is isomorphic to $(Z,P[j',k'])$ only when $j = j'$ and $k = k'$.

       \end{proof}

      \textbf{ Remark:} Distinguishing between $(Z,P[1,k])$ and $(Z,P[2,k])$ for $k \ge 3$ would require an analysis of the
      restriction of $\widehat{A}$ to $D_k$ analogous to that of Lemma \ref{exlem10b} and we have not bothered with it.

     \vspace{.5cm}

     Thus, we obtain a countable infinity of distinct arc cyclic prime tournaments on the Cantor set. While additional examples can be constructed
     using more complicated attachments, this method will still only yield a countable family of tournaments.
     If we begin with a countable family of tournaments $(Y,S)$ and $(X,R)$ we will only be able to construct
     countably many new examples because a Cantor set contains only
     countably many clopen sets.
     \vspace{.5cm}

  \begin{ex}\label{adicexes01} Uncountably many arc cyclic, prime tournaments on a Cantor set,  \end{ex}

  We now follow Example \ref{exesnew02} (b) by beginning with $2N_1 = (2\N^*, 2L_1)$
 which is arc cyclic and prime.

  We build tournaments
  indexed by $\th \in \N^{\N}$. On $\N^{\N}$ we define the shift map $\s$ by $\s(\th)_i = \th_{i+1}$.

  Recall that in Example \ref{adicexes02}
   we let  $Z = \Z[2] \times \{-1, +1 \}$ labelling $Z\pm = \Z[2] \times \{ \pm \}$ with $  \widehat{A}\pm$
   the tournament $\widehat{A}$ on $Z\pm$ and $D_j\pm, \ \bar D_j\pm$  the copies of $D_j, \ \bar D_j$ in $Z \pm$.

   In Example \ref{adicexes02} we defined $P[j] = P[j,j]$ on $Z$ by
  \begin{align}\label{exeq01neww}\begin{split}
P[j] \ = \ \widehat{A}+ \ &\cup \ \widehat{A}- \ \cup\\
(D_j- \times D_j+) \ \cup \ (\bar D_j- \times \bar D_j+) \ &\cup \ (D_j+ \times \bar D_j-) \ \cup \ (\bar D_j+ \times  D_j-).
 \end{split}\end{align}
 The twist map $\t_{j,j}$ is an automorphism
 of $(Z,P[j])$ which interchanges each $D_j\pm$ with $\bar D_j\pm$.

  To define $(K,T[\th])$ we begin with the topological lexicographic product of
  $2N_1 \ltimes \{ (Y_{a},S_{a}) : a \in 2\N^* \ \}$ such that for all $a \in \N+ \ \cup \ \N-$,
  $(Y_{a},S_{a}) = (\Z[2], \widehat{A})$
  and with $(Y_{\infty},S_{\infty})$ trivial.  The underlying space $K = [(\N+ \ \cup \ \N-)\times \Z[2]] \cup \{ \infty \}$.

  Leaving the other arcs unchanged we define $T[\th]$ so that for each $i \in \N$, the  restriction to
 $\{ i- , i+\} \times \Z[2]$ is isomorphic to $(Z,P[j])$ with $j = \th_i$ by the map $(i-,z) \mapsto z- \in Z-$
  and $(i+,z) \mapsto z+ \in Z+$.

 From Theorem \ref{extheo01e} we see that $(K,T[\th])$
is an arc cyclic, prime tournament on a Cantor set for each $\th$.

 We compute the prime quotients of the almost wac tournaments which are the restrictions of $T[\th]$ to $T[\th](x))$  for all $x \in K$.

 We note first the following which is obvious from the way the tournaments were obtained from the lexicographic products.

 \begin{lem}\label{adiclem03} Let $i \in \N$ and let $K'$ be a closed subset of $K$.

 If $K'$ is disjoint from $ \{ i+ \} \times Z[2]$,
 then $K' \cap ( \{ i- \} \times Z[2])$ is a clopen subset of $K'$ which, if it is nonempty,  is $Q$ invariant for
 $(K', T[\th]|K')$.

 If $K'$ is disjoint from $ \{ i- \} \times Z[2]$,
 then $K' \cap ( \{ i+ \} \times Z[2])$ is a clopen subset of $K'$ which, if it is nonempty,  is $Q$ invariant for
 $(K', T[\th]|K')$. \end{lem}

 \vspace{.5cm}

 Case 1 ($x = \infty$) : $T[\th](\infty) = \{ \infty \} \ \cup \ [\bigcup_i \ (\{ i- \} \times Z[2]) $. From Lemma \ref{adiclem03} it
 follows that each $\{ i- \} \times Z[2]$ is a $Q$ invariant subset.  Smashing each to a point we obtain the prime quotient
 which is isomorphic to $\bar N_1$. Similarly, the prime quotient of the restriction to $T[\th]^{-1}(\infty)$ is
 isomorphic to $N_1$.
  \vspace{.25cm}

 Case 2 ($ x = (i-,z) \in \{ i- \} \times Z[2], j = \th_i$) : Because of the twist map automorphism, we may assume that $z \in D_j$.
 \begin{align}\label{exeq02neww}\begin{split}
 T[\th](x) \ = \ (\{ i- \} \times  &\widehat{A}(z)) \ \cup (\{ i+ \} \times D_j) \ \cup\\
  (\{ (i+1)- \} \  &\cup \  \{ k- : k < i-1 \}) \times \Z[2]).
 \end{split}\end{align}
 
 If $i > 2$, then $\{ 1- \} \times \Z[2]$ is a $Q$ invariant subset for the restriction to $T[\th](x)$ and in the quotient
 it is an isolated terminal point.  
 
 If $i = 1,2$, then the
set $\{ (i+1)- \} \times \Z[2]$ is a $Q$ invariant subset for the restriction to $T[\th](x)$ and in the quotient
 it is an isolated terminal point. 
 
 In either case, the restriction of $T[\th]$ to $T[\th](x)$ has an arc quotient.
  \vspace{.25cm}

   Case 3 ($ x = (i+,z) \in \{ i+ \} \times Z[2], j = \th_i$) :  Again we may assume that $z \in D_j$.
  \begin{align}\label{exeq03neww}\begin{split}
   T[\th](x) \ = \ (\{ i+ \} \times  &\widehat{A}(z)) \ \cup (\{ i- \} \times \bar D_j) \ \cup \
  \{ (i+1)- \} \times \Z[2]  \ \cup \\
 [\ \bigcup_{k < i}  \ \{ j- \} \times \Z[2]\ ] \ &\cup \ [\bigcup_{k > i+1} \ \{ k-, k+ \} \times \Z[2]\ ] \ \cup \ \{ \infty \}.
  \end{split}\end{align}

 The restriction of $T[\th]$ to $\bigcup_{k > i+1} \ \{ k-, k+ \} \times \Z[2] \cup \{ \infty \}$
 is isomorphic to $T[\s^{i+1}(\th)]$  by that map $k\pm \mapsto (k-i-1)\pm$. Hence, this restriction is an arc cyclic, prime tournament.

 For each $k < i$ and for $k = i+1$ we smash $\{ k- \} \times \Z[2]$ to a point which we label $k-$.

 For each $k \ge j-1$, $\{i+ \} \times (x + A_k)$ is a $Q$ invariant set in $(\{ i+ \} \times \widehat{A}(z)) \ \cup (\{ i- \} \times \bar D_j)$
 and hence in $T[\th](x)$. We smash each to a
 point which we label $(i+,k)$. Similarly, $\{ i- \} \times \bar D_j$ is a $Q$ invariant set in
 $(\{ i+ \} \times  \widehat{A}(z)) \ \cup (\{ i- \} \times \bar D_j)$ and hence in $T[\th](x)$.  We smash it to a point which we label
 $(i-, \bar d)$.

 As we saw when we analyzed Example \ref {adicexes02}, the prime quotient of the restriction
 to $(\{ i+ \} \times  \widehat{A}(z)) \ \cup (\{ i- \} \times \bar D_j)$
 is $$K' \ = \  \{(i-, \bar d) \} \ \cup \ \{ (i+, k) : k \ge j-1 \} \ \cup \ (\bigcup_{k < j-1} \ \{ i+ \} \times \{ z + A_k \}).$$

  Thus, we have a quotient $(K',T')$ of the restriction to $T[\th](x)$ with $K'$
 the union of three pieces. Fix $a \in \Z[2]$.
 \begin{align}\label{exeq04neww}\begin{split}
 K'_1 \ &= \ (\bigcup_{k > i+1} \ \{ k-, k+ \} \times \Z[2]) \cup \{ \infty \}, \\
 K'_2 \ &= \ \{ 1-, 2-, \dots, (i-1)-, (i- , \bar d), (i+1)-,\\ &((i+2)-,a), ((i+3)-,a), \dots, \infty \}. \\
 K'_3 \ &= \ \{ (i-, \bar d) \} \cup \{ (i+,k) : j-1 \le k \le \infty \} \cup (\bigcup_{k < j-1} \ \{ i+ \} \times \{ z + A_k \}).
   \end{split}\end{align}

Recall that $x + A_j \subset \bar D_j$ while $x + A_k \subset D_j$ for $k > j$.
Hence, for any $k > j+1$,
  $\{ (i-, \bar d), (i+, j), (k-, a) \}$ is a $3-$cycle in $K'_2 \cup K'_3$.

   On $K'_1$ the restriction of the quotient of $T[\th]$ is isomorphic to $T[\s^{j+1}(\th)]$. On $K'_2$
   the restriction of the quotient of $T[\th]$ is
   isomorphic to $\bar N_1$. On $K'_3$ the restriction is also isomorphic to the
   prime quotient of $x$ in $(Z,P[j])$. Thus, the restriction of the quotient of $T[\th]$ to each
   set is prime.

     Now let $U$ be a non-trivial $Q$ invariant subset.

    If $U$ contains two points of $K'_{\ep}$ for $\ep = 1,2,3$ then it contains $K'_{\ep}$.
   $K'_2 \cap K'_1$ is infinite. Hence, $K'_1 \subset U$ or $K'_2 \subset U$
   implies $K'_1 \cup K'_2 \subset U$. Since $(i-, \bar d), (k-, a) \in U$ the above $3-$cycle implies the $(i+, j) \in U$.
   Since $(i-, \bar d), (i+, j) \in U \cap K'_3$ it follows that
    $K'_3 \subset U$.

    On the other hand, if $K'_3 \subset U$ then $(i-, \bar d), (i+, j) \in U$ implies $(k-, a) \in U$  for all
    $k > j+1$ and so $K'_1 \cup K'_2 \subset U$. Thus, in all these cases,
   $U = K'$.

  If $U$ contains a point of $K'_1$, then because $K'_1$ is balanced, it follows from Lemma \ref{qlem01a}
   that $U$ contains two points of $K'_1$ and so, as above, $U = K'$.

  If $U$ contains a point of $\{ i+ \} \times \{ z + A_k \}$ for some $k < j-1$, then because $z + A_k$ is balanced, it follows from Lemma \ref{qlem01a}
  again
   that $U$ contains two points of $K'_3$ and so, as above, $U = K'$.

  Finally, suppose that $x \in U \cap K'_3 \setminus \{ (i-, \bar d) \}$  and $\ell- \in U \cap K'_2$ for some $\ell = 1,2, \dots, i-1$ or $i+1$. Since
  $x \in \{ i+ \} \times \Z[2]$,
  , $x \ha (k+,b) \ha \ell-$ for all $k > j+1$ and $b \in \Z[2]$. Hence, $ (k+, b) \in U \cap K'_1$ for all such $k$ and $b$.
  As before, it follows that $U = K'$.

  Hence, $(K',T')$ is a prime topological tournament. Notice that it contains a Cantor set and a countable number of isolated points.
\vspace{.25cm}

To summarize, the almost wac tournaments which are restrictions of $T[\th]$ to the sets $T[\th](x)$ have arc quotients except when $x = \infty$ or
$x \in \{ i+ \} \times \Z[2]$. In the latter case, the restriction has a prime quotient which contains a Cantor set and a countable
set of isolated points. In the former, the restriction has a prime quotient which is isomorphic to $\bar N_1$. Thus, $\infty$ is the
unique point of $K$ such that the restriction of $T[\th]$ to the set $T[\th](x)$ has a countably infinite prime quotient.
\vspace{.5cm}

 \begin{theo}\label{theo01neww} The tournaments $(K,T[\th])$  are arc cyclic, prime tournaments on a Cantor set for all $\th \in \N^{\N}$ and
 with no two of them isomorphic. \end{theo}

 \begin{proof} Suppose $h : (K,T[\th]) \to (K,T[\th'])$ is an isomorphism. Because of the above characterization of the point $\infty$ it
 follows that $h(\infty) = \infty$. It then follows that $h$ is an isomorphism from the restriction $T[\th]|T[\th](\infty)$ to
$T[\th']|T[\th'](\infty)$ and from $T[\th]|T[\th]^{-1}(\infty)$ to $T[\th']|T[\th']^{-1}(\infty)$. From Case 1 above, the uniqueness
of classifiers and the rigidity of $N_1$ and $\bar N_1$ it must follow that $h$ maps each $\{ i \pm \} \times \Z[2]$ into itself.

Hence, for each $i$, $h$ induces an isomorphism from $(Z,P[\th_i])$ to $(Z,P[\th'_i])$.  But Theorem \ref{extheo10c} implies that
these are isomorphic only when $\th_i = \th'_i$.  Since $\th_i = \th'_i$ for all $i$, $\th = \th'$.

 \end{proof}

 Thus, we have obtained an uncountable family of distinct arc cyclic, prime tournaments on the Cantor set.
\vspace{.5cm}

\begin{exes}\label{exes08}Limit points of the set of isolated points in prime tournaments.\end{exes}

For the examples below, let $(Z,P)$ be a compact, arc cyclic, prime tournament
with $Z$ metrizable. Since the tournament is prime, $Z$ is totally disconnected.
We assume that $e$ is a point of $Z$  which has a clopen neighborhood $G_0$ no point of which is isolated and so is a Cantor set. Since
$(Z,P)$ is arc cyclic,
every point of $G_0$ is a cycle point.

Let $E = P^{\circ}(e)$ and $F = P^{\circ -1}(e)$. Since $E$ is open and $e$ is a $G_{\d}$ point we can choose
an increasing sequence of clopen sets $\{ E_i : i \in \N \}$ with $\bigcup_i \ E_i  \ = \ E $. Let $F_i = P(E_i) \cap P^{-1}(e)$.

\begin{lem}\label{exlemnewww01} $\{ F_i : i \in \N \}$ is an increasing sequence of clopen sets with  $\bigcup_i \ F_i  \ = \ F $.
For each $i \in \N, \ \ P \cap (E_i \times F_i)$ is a surjective relation from $E_i$ to $F_i$. \end{lem}

\begin{proof} By asymmetry, $e \not\in F_i$ and $E_i \cap F_i = \emptyset$. Therefore,  $F_i = P^{\circ}(E_i) \cap F$ and so
$F_i$ is open as well as closed.

If $b \in F$, then because $(Z,P)$ is arc cyclic, there exists $a \in Z$ such that $\{ b, e, a \}$ is a $3-$cycle. Since
$a \in E$, we have $a \in E_i$ for some $i$ and so $b \in F_i$. Hence,  $\bigcup_i \ F_i  \ = \ F $.

If $b \in F_i$, then from its definition, there exists $a \in E_i$ such that $a \ha b$. If $a \in E_i$,
 then because $(Z,P)$ is arc cyclic, there exists $b \in Z$ such that $\{ b, e, a \}$ is a $3-$cycle.
 Hence, $b \in F_i$.  Thus, $P \cap (E_i \times F_i)$ is a surjective relation.

 \end{proof}
  \vspace{.25cm}

 By compactness, there exists $i_0$ such that $E_{i_0} \cup F_{i_0} \cup G_0 = Z$. We renumber the sequences, labelling $E_{i+i_0-1}$ as $E_i$
 and  $F_{i+i_0-1}$ as $F_i$.  Let $G_i = Z \setminus (E_i \cup F_i)$. Thus, $\{ G_i \}$ is a decreasing sequence of clopen sets
 each contained in $G_0$ and with $\bigcap_i \ G_i \ = \ \{ e \}$. Since every point of $G_i$ is a cycle point, the restriction
$P|G_i$ is balanced for all $i$.
\vspace{.25cm}

(a) Let $(Y,S)$ be a prime tournament with no initial point, but with a non-isolated terminal point $M$. As it is not isolated,
it is left balanced. Examples are $N_1$
of Example \ref{exes06} (a) or
 $(Y,S)$ from Example \ref{exesnew02} (a).

Since $(Y,S)$ is prime, $Y$ is totally disconnected. We can choose a strictly decreasing sequence of clopen subsets $\{ G'_i : i \in \N \}$
with $G'_1 \ = \ Y$ and $\bigcap_i \ G'_i \ = \ \{ M \}$.  Let $H_i = G'_i \setminus G'_{i+1}$ so that $\{ H_i \}$ is a pairwise disjoint
sequence of nonempty clopen subsets of $Y$ with union equal to $Y \setminus \{ M \}$.

We initially assume that $Y$ and $Z$ are disjoint.

We define the compact space $X = (Y \cup Z)/\{M, e\}$ by identifying the point $M$ in $Y$ with the point $e $ in $Z$. 
That is, we smash the pair $\{ M, e \}$ to
a point which we will label $e$. We now regard
$Y$ and $Z$ as subsets of $X$ so that $Y \cap Z = \{ e \}$.  Notice that the isolated points of $X$ are the
isolated points of $Y$ or $Z$ since $M$ is not isolated in $Y$ and $e$ is not isolated in $Z$.

 With $$X \times X \ = \ (Y \times Y) \cup (Z \times Z) \cup [(Y \times Z) \cup (Z \times Y)]$$
we define the tournament $R$ on $X$ as the following union of a countable number of closed sets.
\begin{equation}\label{exeq11a}
R \ = \ S \  \cup \ P \ \cup \ \
\bigcup_i [H_i \times (E_i \cup G_i) \ \cup  \ F_i \times H_i].
\end{equation}

\begin{theo}\label{extheo09} The tournament $(X,R)$ is a prime topological tournament.

If every non-isolated point of $Y$ except for $M$ is a cycle point, then every non-isolated point of $X$ is a cycle point and so $(X,R)$ is wac.

If every point of $Y$ except for $M$ has an arc cyclic subset neighborhood in $Y$, then
every point of $X$ except for $M = e $ has an arc cyclic subset neighborhood in $X$.
\end{theo}

\begin{proof} Let $\{ (z_n,w_n) \}$ be a  sequence in $R$ which converges to a point $(z,w)$ of $X \times X$. If the sequence lies infinitely
often in $Y \times Y$ and so in $S$, then the limit point lies in $S \subset R$. Similarly, if the sequence lies infinitely often
in $Z \times Z$ the limit lies in $P \subset R$.

So we may assume that the sequence lies entirely in $Y \times Z$ or $Z\times Y$. If $\{ \ell_n \}$ is  the sequence in $Y$
let $i_n$ be defined by $\ell_n \in H_{i_n}$. If for some $i \in \N$
$i_n = i$ infinitely often then we may assume that $\ell_n \in H_i$ for all $n$ by going to a subsequence. Because
$[H_i \times (E_i \cup G_i)] \cup [F_i \times H_i]$ is a closed set, it follows that $(z,w)$ is in this set and so in $R$.

Otherwise,  $i_n$ tends to $\infty$.  So for every $j \in \N \ \ \{ \ell_n \}$ is eventually
in $G'_j$.  That is, $\{ \ell_n \}$ converges to $M = e$.

 Now let $\{ z_n \}$ be the sequence in $Z$.
If $z_n \in G_{i_n}$ infinitely often, then since $i_n \to \infty$ and $\bigcap_i G_i = \{ e\}$, it follows that $\{ z_n \}$
converges to $ e$.  That is,
 the limit point $(z,w) \ = \ (e , e) \in R$. Otherwise, either $z_n \in P(e)$ infinitely often and with limit in $a \in P(e)$ or it lies in
$P^{-1}(e)$ infinitely often with limit $b \in P^{-1}(e)$. The limit point is then either $(e , a)$ or $(b, e )$ both of which are in $P \subset R$.

Thus, $R$ is closed and so $(X,R)$ is a topological tournament.

Now let $U$ be a non-trivial, closed $Q$ invariant subset of $X$.

If two points of $Y$ are in $U$, then because $(Y,S)$ is prime, $Y \subset U$ and, in particular, $e = M$ is in $U$.

 Now if there exists $ g \in G_0 \cap U$ (and this includes the case $g = e$),
 then because $P|G_0$ is balanced, there exist other points in $G_0 \cap U$ by Lemma \ref{qlem01a}.
 Because $(Z,P)$ is prime, $Z \subset U$.

 If $(\ell, a) \in R^{\circ}$ with $\ell \in H_i \cap U, a \in E_i \cap U$,  then there exists $b \in F_i$ such that
 $\{ \ell, a, b \}$ is a $3-$cycle in $X$. If $(b, \ell) \in R^{\circ}$ with $\ell \in H_i \cap U, b \in F_i \cap U$,
 then there exists $a \in E_i$ such that  $\{ \ell, a, b \}$ is a $3-$cycle in $X$. Hence, $a, b \in U$ in each case and
 so again $Z \subset U$.

 If two points of $Z$ are in $U$, then $Z \subset U$ and, in particular, $e = M$ is in $U$. If $\ell \in H_i$ and
 $b \in   F_i$, then $b \ha \ell \ha M = e $. Hence, $\ell \in U$. Thus, $Y \subset U$.

 So in any case $U = X$ which implies that $(X,R)$ is prime.

 Because every non-isolated point of $Z$ is a cycle point in $Z$, it is a cycle point in $X$.  This includes $ e = M$.  Hence, if
 every non-isolated point of $Y \setminus \{ M \}$ is a cycle point in $Y$, it follows that every non-isolated point of $X$ is a cycle point in $X$ and
 so $(X,R)$ is wac.

 The local arc cyclicity result is clear.

 \end{proof}
 \vspace{.5cm}

(b)  Let $2N = (2\N^*,2L)$ be a countably infinite, arc cyclic, prime tournament from Example  \ref{exes06}(b), e.g. use either
$2N_0 = (2\N^*,2L_0)$  or $2N_1 = (2\N^*,2L_1)$.
As in (a) above, we define the compact space $X$ by identifying the point $\infty$ in $2\N^*$ with the point $e $ in $Z$. We will regard
$2\N^*$ and $Z$ as subsets of $X$ and use $e $ to label the point $\infty = e $. Thus, $X$ contains a countable number of isolated points
$\N+ \ \cup \ \N-$
with limit point $e $ which lies in a Cantor set.

Define
the tournament $R$ on $X$ as the following union of a countable number of closed sets.
\begin{align}\label{exeq11}\begin{split}
R \ = \ &2L \cup \ P \ \cup \\
\bigcup_i [\{ i+ \} \times  &(E_i \cup G_{i})] \cup [F_i \times \{ i+ \}] \quad \cup \\
\bigcup_i [\{ i- \} \times  &E_i] \cup [(F_i \cup G_i) \times \{ i- \}].
\end{split} \end{align}

\begin{theo}\label{extheo09b} The tournament $(X,R)$ is an arc cyclic, prime topological tournament with isolated points
the set $\N+ \ \cup \ \N-$ which has limit point $ \infty = e$. \end{theo}

\begin{proof} The proof that $R$ is closed and so that $(X,R)$ is a topological tournament is similar to the proof in (a) above.
The proof that $(X,R)$ is prime is also  similar to the proof in (a).

Both $2N$ and $(Z,P)$ are arc cyclic, prime tournaments. Thus any arc in $\N^*$ or $Z$ is contained in a $3-$cycle.

For $g \in G_i$, $\{ i+, g, i- \}$ is a $3-$cycle in $X$. Given $a \in E_i$, there exists $b \in F_i$ and given $b \in F_i$ there
exists $a \in E_i$ such that $\{ i+, a, b \}$ and $\{ i-, a, b \}$ are $3-$cycles in $X$.  Thus, $(X,R)$ is arc cyclic.

 \end{proof}
 \vspace{.5cm}

 Call a tournament $(X,R)$ \emph{almost locally arc cyclic}\index{tournament!almost locally arc cyclic}
 if only
 finitely many points of $X$ do not have an arc cyclic subset neighborhood. We call these the \emph{exceptional points}\index{exceptional points}.
 Since an isolated point has a trivial arc cyclic neighborhood, any exceptional point is non-isolated.

 \begin{theo}\label{extheo10} Let $(J,P)$ be a finite tournament.
 \begin{itemize}
 \item[(a)] There exists an arc cyclic, prime tournament $(X,R)$ with $X$ countably infinite and with finitely many non-isolated points.
 Furthermore, the restriction  $(F, R|F)$ to the set $F$ of non-isolated points is isomorphic to $(J,P)$.
\vspace{.25cm}

 \item[(b)] There exists an arc cyclic, prime tournament $(X,R)$ such that $X$ contains a
 Cantor subset $C$ and countably many isolated points. Each of the - only finitely many - limit points of the isolated points is contained in $C$.
Furthermore, the restriction  $(F, R|F)$ to the set $F$ of limit points points of the isolated points is isomorphic to $(J,P)$.
\vspace{.25cm}

 \item[(c)] There exists an arc cyclic, prime tournament $(X,R)$ such that $X$ contains a
 Cantor subset $C$ and countably many isolated points. Each of the limit points of the isolated points is contained in $C$.
 The set $F$ of the limit points of the isolated points is countably infinite and
 the restriction  $(F, R|F)$ to the set $F$ of limit points points of the isolated points is isomorphic to $N_1 = (\N^*,L_1)$.
\vspace{.25cm}

\item[(d)]  There exists an almost locally arc cyclic, wac, prime tournament $(X,R)$ with $X$ a Cantor set.
 Furthermore, the restriction  $(F, R|F)$ to the set $F$ of exceptional points is isomorphic to $(J,P)$.
 \end{itemize}
 \end{theo}

 \begin{proof} In cases (a), (b) and (d) we proceed by induction on $n = |J|$.

 (a): For $n = 1$ with $(J,P)$ trivial, we use  $2N_1 = (2\N^*,2L_1)$ from Example  \ref{exes06}(b). It is
 a countably infinite, arc cyclic, prime tournament
 with $\infty$ the single non-isolated point.

 We will need  $2-$fold partitions $(E,F)$ of $2\N^*$ such that $2L_1 \cap [(E \times F) \cup (F \times E)]$ is a surjective relation on $2\N^*$.

 For $n, m > 2$ let
 \begin{align}\label{exeq12}\begin{split}
 A+ \ = \ \{ i+ : i &\leq n \}, \quad A- \ = \ \{ i- : i \leq n \}, \\
B+ \ = \ \{ i+ : n < i &\leq n + m \}, \quad B- \ = \ \{ i- : n < i \leq n + m \}, \\
 C \ = \ K \setminus (A+  \cup \ A-  \cup \  &B+  \cup  B-) \ = \ \{ i+, i- : n + m < i \} \cup \{ \infty \}.
 \end{split} \end{align}

 It is easy to check that the restriction of $2L_1$ to each of the following subsets is a surjective relation:
 \begin{displaymath}\begin{split}
 A- \times \ A+, \quad  B+ \times \ (B-  \cup \ C), \\
 A+ \times \ B+, \quad B- \times \ (A-  \cup \ B+), \\
 C \times \ A- \hspace{2cm}
 \end{split}\end{displaymath}

 Hence, for $n, m > 2$ and $(E,F) = (A- \ \cup \ B+, A+ \ \cup \ B- \ \cup \ C)$ we have that
$2L_1 \cap [(E \times F) \cup (F \times E)]$ is a surjective relation on $2\N^*$.

Now let $(J',P')$ be a tournament with $|J'| > 1$ and $a \in J'$. Let $J = J' \setminus \{ a \}$ and let $P = P'|J$.
Let $J_- = P'^{\circ -1}(a), J_+ = P'^{\circ}(a)$ so that $J $ is the disjoint union of $J_-$ and $J_+$.

By induction hypothesis, there exists an arc cyclic, prime tournament $(X,R)$ with $X$ countably infinite and with finitely many non-isolated points.
 Furthermore, the restriction  $(F, R|F)$ to the set $F$ of non-isolated points is isomorphic to $(J,P)$. Using the isomorphism we
 identify $(R, R|F)$ with $(J,P)$ and so regard $J$ as a subset of $X$. Choose $C_1$ a proper clopen subset of $X$ which contains $J_-$ and is
 disjoint from $J_+$ and let $C_2 = X \setminus C_1$.

 Let $(E_1, F_1) = (E, F)$ and $(E_2, F_2) = (F, E)$. These are distinct $2-$fold partitions of
 $2\N^*$ with $2L_1 \cap (F_i \times E_i)$ surjective for $i = 1, 2$.

 Let $(X',R')$ be the attachment
  of $(X,R)$ to $(2\N^*,2L_1)$ via $\{ C_i : i = 1, 2\}$ and $\{ (E_i,F_i) : i = 1,2 \}$. It follows from Theorem \ref{extheo08}
  that $(X', R')$ is arc cyclic and prime.

  Identifying the point $a \in J'$ with $ \infty $ in $2\N^*$, we see that $J'$ is the set of non-isolated points $F'$ in $X'$ and that
  $R'|F'$ equals $(J',P')$, completing the induction.
\vspace{.25cm}

(b): For $n = 1$ we begin with the tournament $(Y,S)$ obtained as in Example \ref{exes08} (b)  by connecting
$2N_1 = (2\N^*,2L_1)$ to an arc cyclic, prime tournament $(Z,P)$ on the Cantor set, identifying $\infty \in 2\N^*$ with $e \in Z$.
By Theorem \ref{extheo09b}
this is an arc cyclic, prime tournament and
 the point $\infty = e $ is the unique limit point of the isolated points of $Y$, which are exactly those in $2\N^*$.

 Since we have chosen $Z$ with no isolated points, there exists, by Theorem \ref{extheo05a} a spanning partition $\{Z_1,Z_2 \}$ for $(Z,P)$,
 labelled so that $e \in Z_1$.
With $(E,F)$ as in part (a) it
 follows that $(\hat E, \hat F) =  (E \cup Z_2, F \cup Z_1)$ is a $2-$fold partition of $Y$ with
$S \cap [(\hat E \times \hat F) \cup (\hat F \times \hat E)]$  a surjective relation relation on $Y$.
Notice that since $Z_1$ is a clopen neighborhood of $e $
in $Z$ and $F$ is a clopen neighborhood of $\infty $ in $2\N^*$, it follows that $\hat F$ is clopen in $Y$.

As before let $(J',P')$ be a tournament with $|J'| > 1, a \in J', J = J' \setminus \{ a \}$ and  $P = P'|J$.
Again let $J_- = P'^{\circ -1}(a), J_+ = P'^{\circ}(a)$.

By induction hypothesis, there exists an arc cyclic, prime tournament $(X,R)$ such that $X$ contains a Cantor subset and countably many isolated points
 with each limit point of the isolated points contained in $C$. There are only finitely many of these.
Furthermore, the restriction  $(F, R|F)$ to the set $F$ of limit
points points of the isolated points is isomorphic to $(J,P)$. Using the isomorphism we again
 identify $(R, R|F)$ with $(J,P)$ and so regard $J$ as a subset of $X$. Choose $C_1$ a proper clopen subset of $X$ which contains $J_-$ and is
 disjoint from $J_+$ and let $C_2 = X \setminus C_1$.

 Let $(E_1, F_1) = (\hat E, \hat F)$ and $(E_2, F_2) = (\hat F, \hat E)$. These are distinct $2-$fold partitions of
 $Y$ with $S \cap (F_i \times E_i)$ surjective for $i = 1, 2$.

 Let $(X',R')$ be the attachment
  of $(X,R)$ to $(Y,S)$ via $\{ C_i : i = 1, 2\}$ and $\{ (E_i,F_i) : i = 1,2 \}$. It follows from Theorem \ref{extheo08}
  that $(X', R')$ is arc cyclic and prime.

  Identifying the point $a \in J'$ with $ \infty = e$ in $Z$, we see that $J'$ is the set of limits of isolated points $F'$ in $X'$ and that
  $R'|F'$ equals $(J',P')$, completing the induction.
  \vspace{.25cm}

  (c) Again we begin with $2N_1 = (2\N^*,2L_1)$.  We will use the construction of Example \ref{exesnew02}(b) which adjusts the lexicographic
  product via an arc cyclic, prime topological tournament $(W,S)$ with $\{ W+, W- \}$ a partition of $W$.

  For $(W+,S+)$ we will use the tournament from Example \ref{exes08} (b) obtained by attaching $2N_1 = (2\N^*,2L_1)$ to the arc cyclic, prime tournament
  $(Z,P)$.  We assume that $Z$ is a Cantor set. So $(W+,S+)$ is an arc cyclic, prime tournament with a countable set of
  isolated points and a single limit point $\infty = e $ in the Cantor set $Z$.  Let $\{ C_1, C_2 \}$ be an arbitrary $2-$fold proper partition
  of $W+$. For $(W-,S-)$  we will use $(Z,P)$  and we choose a spanning partition $\{ Z_1,Z_2 \}$.
  We define the pair of spanning partitions $(E_1,F_1) = (Z_1, Z_2)$ and
  $(E_2,F_2) = (Z_2, Z_1)$. We then let $(W,S)$ be the attachment of $(W+,S+)$ to $(W-,S-)$ via $\{ C_1, C_2 \}$ and
  $ (E_1,F_1),(E_2,F_2)$.  By Theorem \ref{extheo08} we see, as usual, that $(W,S)$ is a prime, arc cyclic tournament and that
  $S|(W- \times \ W+)$ is a surjective relation. We now proceed as in Example \ref{exesnew02}(b)to
  obtain the tournament $(K,T)$. It then follows from Theorem \ref{extheo01e} that the tournament $(K,T)$
  is an arc cyclic, prime tournament. The isolated points are those of
  $\bigcup_i \ \{i+ \} \times W+$ with limit points $$F \ = \  \{ \infty \in 2\N^* \} \cup \{ (i+, \0 = \infty) \in \{i+ \} \times Z+ : i \in \N \}.$$

  It is clear that the restriction of $T$ to $F$ is isomorphic to $N_1 = (\N^*, L_1)$.
\vspace{.25cm}

  (d): For $n = 1$ we begin with the standard $2-$adic example $(J,P) = (\Z[2],\widehat{A})$ which we regard as its own
  reduced double $2'(J,P)$ following Proposition \ref{adicprop02} (d). We then proceed as in Example  \ref{exesnew02}(a).
   The result is an almost locally arc cyclic tournament with
  a single exceptional point $\infty$ which is terminal.  Then as in \ref{exes08}(a) we identify $\infty $ with $\0 $ in $\Z[2]$ to
  obtain a tournament $(Y,S)$ on the Cantor set. By Theorem \ref{extheo09} the tournament is wac and almost locally arc cyclic with
  $\infty = \0 $ the only exceptional point.

 As before let $(J',P')$ be a tournament with $|J'| > 1, a \in J', J = J' \setminus \{ a \}$ and  $P = P|J$.
Again let $J_- = P'^{\circ -1}(a), J_+ = P'^{\circ}(a)$.

By induction hypothesis, there exists an almost locally arc cyclic, wac, prime tournament $(X,R)$ with $X$ a Cantor set.
 Furthermore, the restriction  $(F, R|F)$ to the set $F$ of exceptional points is isomorphic to $(J,P)$.
 Using the isomorphism we again
 identify $(R, R|F)$ with $(J,P)$ and so regard $J$ as a subset of $X$.

Since $X$ has no isolated points we can apply Theorem \ref{extheo05a} to choose disjoint finite sets $H_1, H_2$  both disjoint from
$J$ as well, and such that $H_1$ and $H_2$ are each spanning sets for $(X,R)$. Choose $E$ a clopen subset of $X$ such that
$J_- \cup H_1 \subset E$ and $E$ is disjoint from $J_+ \cup H_2$. With $F = X \setminus E$, it follows that $(E,F)$ is
a spanning set partition of $X$. Let $(E_1, F_1) = (E, F)$ and $(E_2, F_2) = (F, E)$. These are distinct $2-$fold partitions of
 $X$ with $R \cap (F_i \times E_i)$ surjective for $i = 1, 2$.

Now let $C_1$ be a proper clopen subset of $Y$ which contains the point $\infty = \0 $ and let $C_2$ be its -nonempty- complement.

 Let $(X',R')$ be the attachment
  of $(Y,S)$ to $(X,R)$ via $\{ C_i : i = 1, 2\}$ and $\{ (E_i, F_i) : i = 1,2 \}$. It follows from Theorem \ref{extheo08}
  that $(X', R')$ is prime. Since $(X,R)$ and $(Y,S)$ are wac and almost locally arc cyclic, it follows that
  $(X',R')$ is wac and almost locally arc cyclic.

    Identifying the point $a \in J'$ with $ \infty = \0 $ in $W$, we see that $J'$ is the set of exceptional points $F'$ in $X'$ and that
  $R'|F'$ equals $(J',P')$, completing the induction.

 \end{proof}

 \vspace{1cm}

  \section{\textbf{Semi-Prime Tournaments}}
  \vspace{.5cm}

  While the results can be extended to the non-metric case, it will be convenient to restrict to metrizable spaces in this section.
  As the spaces are assumed to be totally disconnected, we will assume that each is equipped with an ultrametric labelled $u$.

  \begin{df}\label{semidef01} A topological tournament $(X,R)$ is called
  \emph{semi-prime}\index{semi-prime tournament}\index{tournament!semi-prime}
     when $X$ is a compact, totally disconnected space
  and there exists $\ep > 0$ such that every non-trivial $Q$ invariant subset of $X$ has diameter at least $\ep$. \end{df}
  \vspace{.5cm}

  \begin{theo}\label{semitheo02} Let $R$ be a topological tournament on a compact, totally disconnected space $X$.
  \begin{itemize}
  \item[(a)] If $(X,R)$ is a prime tournament or if $X$ is finite, then $(X,R)$ is semi-prime.

  \item[(b)] If $(X,R)$ is a semi-prime tournament and $A$ is a non-empty clopen subset of $X$, then the restriction
  $R|A$ is a semi-prime tournament on $A$.

  \item[(c)] Assume $h : (X_2,R_2) \to (X_1,R_1)$ is a quotient map. If $(X_2,R_2)$ is a semi-prime, wac tournament,
  then $(X_1,R_1)$ is a semi-prime, wac tournament.  Furthermore, there exists a finite set $H$ of isolated points
  of $X_1$ such that $h^{-1}(y)$ is a singleton set for   all $y \in X_1 \setminus H$. In particular, if $X_1$ has infinitely
  many isolated points, then $X_2$ has infinitely many isolated points.
  \end{itemize}
  \end{theo}

  \begin{proof} (a): If $R$ is prime, then $X$ is the only non-trivial $Q$ invariant subset of $X$.  If $X$ is finite, then
  there exists $\ep > 0$ such that $u(x,x') \ge \ep$ whenever $x \not= x'$.
\vspace{.25cm}

  (b): Let $\ep_1 > 0$ be a lower bound for the diameters of  non-trivial $Q$ invariant subsets of $X$. With $B = X \setminus A$
  let $\ep_2 > 0$ such that $(x,x') \in A \times B$ implies $u(x,x') \ge \ep_2$.  By  Theorem \ref{theo2.07} there exists
  $\ep > 0$ with $\ep < \ep_1$ such that if $u(x,x') \ge \ep_2$, then $\{ V_{\ep}(x), V_{\ep}(x') \}$ is a thickening for
  $\{ x, x' \}$. In particular, if $(x,x') \in A \times B$ then $x \ha x'$ implies $V_{\ep}(x) \times V_{\ep}(x') \subset R^{\circ}$.
  Otherwise, $V_{\ep}(x) \times V_{\ep}(x') \subset R^{\circ -1}$.  It follows that if $U \subset A$ is non-trivial and $Q$ invariant for
  the restriction $R|A$, then $diam \ U < \ep$ would imply  $U \subset V_{\ep}(x)$ for $x \in U$. It would then follow that
  $U$ is  $Q$ invariant in $X$ with respect to $R$. Since the diameter of $U$ is less than $\ep_1$, this cannot happen.

  Thus, the diameter is at least $\ep$ for any subset of $A$ which is non-trivial and $Q$ invariant with respect to $R|A$.  That is,
  $(A,R|A)$ is semi-prime.
  \vspace{.25cm}

  (c): By Theorem \ref{wactheo02}(f) $(X_1,R_1)$ is wac since $(X_2,R_2)$ is. Furthermore, if $y = h(x)$ is non-isolated, then
  $\{ x \} = h^{-1}(y)$ and $x$ is non-isolated. Let $\ep > 0$ be a lower bound for the diameters of  non-trivial $Q$ invariant subsets of $X_2$.
  By compactness, we can choose $U$ an open subset of $X_1$ with $y \in U$ such that the diameter of $h^{-1}(U)$ is less than $\ep$.
  If for $y' \in U$, the set $h^{-1}(y')$ were not a singleton, then $y'$ would be isolated and so $h^{-1}(y')$ would be a non-trivial
  $Q$ invariant subset of $X_1$ with diameter less than $\ep$.  As this does not happen, it follows that each $h^{-1}(y')$ is a singleton.
  That is, $h$ restricts to a continuous bijection from $h^{-1}(U)$ to $U$.  By Theorem \ref{wactheo02}(f) $h$ is an open map and so
  the restriction to $h^{-1}(U)$ is a homeomorphism onto $U$.

  We can choose for each non-isolated point $y \in X_1$ an open set $U_y$ such that the restriction of the projection $h$ to
 $h^{-1}(U_y)$ is a homeomorphism to $U_y$. The collection $\{ U_y : y \in X_1$
 non-isolated $\} \cup \{ \{ z \} : z \in X_1 $ isolated $\}$ is an
 open cover of $X_1$.  Let $\{ U_{y_1}, \dots U_{y_k} \} \cup \{ \{z_1 \}, \dots, \{ z_{\ell} \} \}$ be a finite subcover. For each
 $y \in U_{y_1} \cup \dots \cup U_{y_k}$, $h^{-1}(y)$ is trivial. Thus, $h^{-1}(y)$ is non-trivial only for $y$ in some subset
  $H$ of $\{ z_1, \dots, z_{\ell} \}$.

  If $X_1$ has infinitely many isolated points, then for infinitely many isolated points $y \in X_1$, the clopen set
  $h^{-1}(y) \subset X_2$ is a singleton and these are isolated points of $X_2$.

  If $(X_1,R_1)$ were not semi-prime, then we could choose a sequence $\{ B_n \}$ of non-trivial $Q$ invariant subsets with
  $diam \ B_n \ \to \ 0$. By going to a subsequence we may assume that the sets converge to a singleton $\{ y \}$ in $X_1$.
  As it is a limit, the point $y$ is non-isolated.
  With the open set $U$ chosen as above, eventually we would have $B_n \subset U$. By  Theorem \ref{qtheo03}, $h^{-1}(B_n)$ is
  $Q$ invariant in $X_2$ and it is non-trivial since $B_n$ is. However, the diameter of  $h^{-1}(B_n)$ is bounded by the
  diameter of $h^{-1}(U)$ which is smaller than $\ep$. The contradiction implies that $(X_1,R_1)$ is semi-prime.

  \end{proof}

 \begin{theo}\label{semitheo03}  Let $\{(X_i,R_i,f_i) \}$ be a classifier system for a wac tournament $(X,R)$ with maps $\{ h_i \}$.
 If $(X,R)$ is a semi-prime tournament, then the system terminates at some finite level $n$.  That is, for some
 $n \in \N$ the map $h_n : X \to X_n$ is a homeomorphism inducing an isomorphism from $(X,R)$ to $(X_n,R_n)$. \end{theo}

 \begin{proof} As before, let $\ep > 0$ be a lower bound for the diameters of  non-trivial $Q$ invariant subsets of $X$.

By Theorem \ref{classtheo02} we can regard $(X,R)$ as the inverse limit of the system which implies that for every $x \in X$
$\{ x \} \ = \ \bigcap_{i \in \N} \ h_i^{-1}(x_i)$ from which it follows that $1_X $ is the intersection of the
decreasing sequence  $ \{ (h_i \times h_i)^{-1}(1_{X_i})  \subset X \times X \}$.
By compactness, there exists $n \in \N$ such that $(h_n \times h_n)^{-1}(1_{X_n}) \subset V_{\ep}$.

If for some $x \in X$, it happened that $h_n^{-1}(x_n)$ were not a singleton, then $x_n$ would be isolated in $X_n$ and
so $h_n^{-1}(x_n)$ would be a non-trivial $Q$ invariant subset of $X$.  Since $h_n^{-1}(x_n) \times h_n^{-1}(x_n) =  (h_n \times h_n)^{-1}(x_n,x_n)$
it would follow that $h_n^{-1}(x_n)$ has diameter less than $\ep$. Since this does not happen, it follows that each $h_n^{-1}(x_n)$ is a
singleton.  Thus, $h_n$ is a continuous bijection and so is a homeomorphism by compactness. Since $h_n$ maps $R$ to $R_n$, it is
an isomorphism $(X,R)$ to $(X_n,R_n)$.

\end{proof}
\vspace{.5cm}

We will call the  minimum $n$ such that $h_n$ is a bijection the \emph{terminal level}\index{terminal level} for $(X,R)$.

\begin{lem}\label{semilem04}  Let $(X_2,R_2) = (X_1,R_1) \ltimes \{ (Y_x,S_x) : x \in X_1 \}$ be a topological lexicographic product with
$(X_1,R_1)$ and each $(Y_x,S_x)$ wac tournaments so that $(X_2,R_2)$ is wac. The tournament $(X_2,R_2)$ is semi-prime if and only if
 $(X_1,R_1)$ and each $(Y_x,S_x)$ is semi-prime and, in addition, $(Y_x,S_x)$ is trivial except for a finite subset of isolated points
 $x \in X_1$. \end{lem}

 \begin{proof} It follows from Lemma \ref{classlem05a} that $(X_2,R_2)$ is wac.

 If $(X_2,R_2)$ is semi-prime then from  Theorem \ref{semitheo02} it follows that the quotient $(X_1,R_1)$ and the restriction to
 each $Q$ invariant subset $\{ x \} \times Y_x$ for $x$ isolated in $X_1$ is a semi-prime tournament. Moreover, by Theorem \ref{semitheo02}
 $\pi^{-1}(x)$ is non-trivial only for $x$ in some finite set $H \subset X_1$ . Thus, $Y_x$ is non-trivial only for  $x$ in the finite set $H$.

 Now assume that the base and fibers are semi-prime tournaments and that $Y_x$ is non-trivial only for $x \in H$.
 Let $u_1$ be an ultrametric on $X_1$. Replacing an ultrametric $u_2$ on $X_2$ by $\max(u_2, \pi^*u_1)$ we may assume that
 $u_1(x,x') \le u_2((x,y),(x',y'))$ for $(x,y), (x',y') \in X_2$. Choose $\ep > 0$ a lower bound for the $u_1$ diameter of the non-trivial
 $Q$ invariant subsets of $X_1$ and for the $u_2$ diameter of the non-trivial $Q$ invariant subsets of any of the $Y_{x}$'s for $x \in H$.

 Let $A$ be a non-trivial $Q$ invariant subset of $X_2$.  If $A$ is contained in some $Y_{x}$ then $x \in H$ and the diameter
 of $A$ is at least $\ep$ because
 $Q$ invariance for $X_2$ implies $Q$ invariance for $Y_{x}$. Otherwise, $\pi(A)$ is a non-trivial $Q$ invariant subset of $X_1$ and
 so it has $u_1$ diameter at least $\ep$.  Hence, the $u_2$ diameter of $A$ is at least $\ep$.

 Thus, $(X_2,R_2)$ is semi-prime.

 \end{proof}

 \begin{theo}\label{semitheo05} A wac tournament $(X,R)$ is semi-prime if and only if there exists a wac prime tournament $(X',R')$
 and a clopen subset $A \subset X'$ such that $(X,R)$ is isomorphic to the restriction $(A,R'|A)$.
 If $(X,R)$ is arc cyclic, then $(X',R')$ can be chosen to be arc cyclic.\end{theo}

 \begin{proof} By Theorem \ref{semitheo02} the restriction of a prime tournament to a clopen subset is semi-prime.

 Now assume that $(X,R)$ is wac and semi-prime and let $\{(X_i,R_i,f_i) \}$ be a classifier system for  $(X,R)$ with maps $\{ h_i \}$.
 We apply Theorem \ref{semitheo03} and prove the result by induction on the terminal level $n$.

 If $n = 1$, then $(X,R)$ which is isomorphic to $(X_1,R_1)$ is either itself a prime tournament other than an arc, or else
 it is a finite order.  If $(X,R)$ is prime we use $(X',R') = (X,R)$.  If $(X,R)$ is any finite tournament, then we obtain
 $(X',R')$ by using Theorem \ref{extheo08}, or when $|X| = 1$ Proposition \ref{exprop06}.

 Now, inductively, assume the result when the terminal level is $n$ and assume that $(X,R)$ has terminal level $n+1$ and so
 we will use $h_{n+1}$ as an identification regarding $(X,R) = (X_{n+1},R_{n+1})$. Then
 using the lexicographic construction for the classifier, we write $(X,R)$ as the topological lexicographic product
 $(X_n,R_n) \ltimes \{  (Y_x,S_x) : x \in X_n \}$. By Theorem \ref{semitheo02} again, the quotient map $h_n : (X,R) \to (X_n,R_n)$
 shows that $(X_n,R_n)$ as well as $(X,R)$ is a semi-prime and wac. So Lemma \ref{semilem04} implies that
 $(Y_x,S_x)$ is trivial except for $x \in H$ with $H$ a nonempty finite set of isolated points of $X_n$. In addition, for each
 $x \in H$, either $(Y_x,S_x)$ is a prime tournament other than an arc, or else it is a finite order and so $Y_x$ is finite.

 The  wac, prime tournament $(X_n,R_n)$ clearly has terminal level $n$ and so the induction hypothesis implies that there
 exists a wac, prime tournament $(Z_1,T_1)$ with $X_n$ a clopen subset of $Z_1$ such that $R_n = T_1|X_n$. If $(X,R)$ is arc cyclic, then
 the quotient $(X_n,R_n)$ is arc cyclic and we can choose $(Z_1,T_1)$ arc cyclic.

 If we let $(Z_2,T_2)$ be the topological lexicographic product
 $(Z_1,T_1) \ltimes \{ (Y_x,S_x) : x \in Z_1 \}$ with $(Y_x,S_x)$ as before when $x \in H$, and
 with $(Y_x,S_x)$ trivial otherwise.  Clearly, $X$ is a clopen subset of $Z_2$ with $R = T_2|X$. If $(X,R)$ is arc cyclic, then
 $(Y_x,S_x)$ is arc cyclic for each $x \in H$ and so $(Z_2,T_2)$ is arc cyclic.

 First we choose a proper partition $\CC = \{ C_i : i = 1, \dots, m  \}$ of $Z_2$ as follows. If for $x \in H, (Y_x,S_x)$
 is an infinite prime tournament, then we choose as  two members of $\CC$  a proper $2-$fold partition of the clopen set
 $\{ x \} \times Y_x$. If for $x \in H, Y_x$ is finite,
 then we choose as members of $\CC$ the singleton subsets of the finite set of isolated points $\{ x \} \times Y_x$. The
 complement of union of all of these is  the clopen set $Z_2 \setminus \bigcup \{ \{ x \} \times Y_x : x \in H \}$.
 If it is nonempty, then  it is adjoined as the last member of $\CC$.

 Now let $(Z_3,T_3)$ be an arc cyclic, prime tournament on a Cantor set, e.g. we may use $(\Z[2],\widehat{A})$. By
 Theorem \ref{extheo05a} we can choose distinct spanning set partitions of $Z_3$: $\{ (E_i,F_i) : i = 1,\dots, m \}$.
 We obtain $(X',R')$ by attaching $(Z_2,T_2)$ to $(Z_3,T_3)$ via the proper partition $\CC = \{ C_i \} $ of $Z_2$ and
 the associated partitions $\{ (E_i,F_i) \}$ of $Z_3$.

 It is clear that $(X',R')$ is wac and by Theorem \ref{extheo08} it is arc cyclic if $(Z_2,T_2)$ is.

 Now let $U$ be a non-trivial $Q$ invariant subset for $(X',R')$.

 By Lemma \ref{qlem01a} if $U$ contains a point of $Z_3$, then because $(Z_3,T_3)$ is balanced, it contains two points of $Z_3$.
 If $U$ contains two points of $Z_3$ then $Z_3 \subset U$
 because $(Z_3,T_3)$ is prime. If $z \in C_i$, then there exist $a \in E_i, b \in F_i$ with $b \ha a$ for $T_3$. It follows
 that $\{ z, b, a \}$ is a $3-$cycle for $R'$ and so $z \in U$. Thus, $U = X'$.

 Now suppose that distinct points $z_1, z_2$ lie in $ U \cap Z_2$.  If for some $i \not= j$, $z_1 \in C_i, z_2 \in C_j$, then there
 exists $x \in (E_i \cap F_j)$ or $x \in (E_j \cap F_i$ because the spanning partitions are distinct. By relabelling we may
 assume the first. Then $z_2 \ha x \ha z_1$ in $R'$ and so $x \in U$.  As above, it then follows that $U = X'$.

If $z_1$ and $z_2$ do not lie in the same set  $\{ x \} \times Y_x$, then there is a
section $\xi : Z_1 \to Z_2$ which contains both  $z_1$ and $z_2$. Since $(Z_1,T_1)$ is prime,
the image of the section lies in $U$.  No section is entirely contained in
a member of $\CC$ and so there exists two points $z_1'$ and $z_2'$ in $U \cap Z_2$ which lie in different elements of $\CC$.

 If $z_1$ and $ z_2$ are both in $\{ x \} \times Y_x$ with
 $Y_x$ finite, then they do not lie in the same member of $\CC$.

 Finally, if $z_1, z_2 \in \{ x \} \times Y_x$ with $(Y_x,S_x)$ an infinite prime tournament, then $Y_x \subset U$.
 Since $Y_x$ contains two different members of $\CC$ we can choose two points $z_1', z_2'$ of $\{ x \} \times Y_x$ which
 lie in different elements of $\CC$.  Thus, in this case as well, $U = X'$.

 We have proved that $(X',R')$ is prime as required.

 \end{proof}

  \begin{theo}\label{semitheo06} Let $\{(X_i,R_i,f_i) \}$ be a classifier system for a wac tournament $(X,R)$ with maps $\{ h_i \}$. Assume that
  $X$ has only finitely many isolated points. The tournament $(X,R)$ is semi-prime if and only if every $X_i$ has only finitely many isolated points and
  there exists a terminal level $n$, i.e. $h_n : (X,R) \to (X_n,R_n)$ is an isomorphism. In particular, this
  applies if $X$ is a Cantor set. \end{theo}

  \begin{proof} We use the lexicographic construction for the classifier. In any case, $(X_1,R_1)$ is either prime or finite and so is
  semi-prime.

  If  $X_i$ has only finitely many isolated points and $(X_i,R_i)$ is semi-prime, then Lemma \ref{semilem04} implies that
   $(X_{i+1},R_{i+1})$ is semi-prime.  It follows by induction that if every $X_i$ has only finitely many isolated points, then
   every $(X_i,R_i)$ is semi-prime. If there is a terminal level $n$, then the isomorphism $h_n$ implies that $(X,R)$ is semi-prime and
   has only finitely many isolated points.

   Now assume that $(X,R)$ is semi-prime.  It has a terminal level by Theorem \ref{semitheo03}.
   If for some $i$, $(X_i,R_i)$ has infinitely many isolated points, then Theorem \ref{semitheo02} (c) implies that
   $X$ has infinitely many isolated points.

   \end{proof}

 \vspace{1cm}

 \section{\textbf{Appendix}}

\subsection{\textbf{Alternative Game Subsets for the 2-adics}}

\vspace{.5cm}

We return to  the additive  group of $2-$adic integers,
  which we regard as the product $\Z[2] =  \{ 0. 1 \}^{\N}$.

   For $\ep \in \Z[2] = \{ 0, 1 \}^{\N}$  we define $A(\ep)$ by
     \begin{equation}\label{adiceq01aa}
          A(\ep)_i \  = \ \begin{cases} -A_i \ \ \text{ when} \ \ \ep_i = 1, \\
   \ \ \ A_i \ \ \text{ when} \ \ \ep_i = 0. \end{cases} \quad \text{and} \quad A(\ep) = \{ \0 \} \cup (\bigcup_{i} A(\ep)_i).
   \end{equation}
   So we can write $A(\ep)_i = (-1)^{\ep_i}A_i$.

   Thus, $A(\ep)$ is a closed game subset for $\Z[2]$ and we let $\widehat{A(\ep)}$ be the associated topological tournament. The original
   subset $A$ is $A(\ep)$ with $\ep = \0 $.  Letting $\bar \ep $ be given by $(\bar \ep)_i = \overline{\ep_i}$, then the complementary
   game subset    $-A(\ep) = A(\bar \ep)$. Recall that for any game subset $B$ we have $\widehat{-B} = \widehat{B}^{-1}$.

   \begin{theo}\label{adictheo01xx}
%
\vspace{.25cm}

   (a) For each $k \in \N, \ w \in \{ 0, 1 \}^k $, the shift $\s^{k}$ is an isomorphism from the restriction
    $(I_w,\widehat{A(\ep)}|I_w)$  to
   $(\Z[2],\widehat{A(\s^{k}(\ep))}$ .
   \vspace{.25cm}

   (b) For any $\ep \in \{ 0, 1 \}^{\N}$ there is a topological tournament isomorphism $h[\ep] : (\Z[2],\widehat{A}) \to (\Z[2],\widehat{A(\ep)})$
   with $h[\ep](\0 ) = \0 $ and for all $k \in \N$, the following diagram commutes.
    \begin{equation}\label{adicdiag}\begin{CD}
  I_{0^k} \  @>h[\ep]>> \  I_{0^k} \\
  @V \s^{k} VV  @VV \s^{k} V \\
  \Z[2] \ \  @>>h[\s^{k}(\ep)]> \ \ \Z[2]
  \end{CD}
  \end{equation}
    \end{theo}

   \begin{proof}    (a): Note  that
    \begin{equation}\label{adiceq01ab} \begin{split}
   0x \in A(\ep)_{i+1} = (-1)^{\ep_{i+1}}A_{i+1} \ \  \Longleftrightarrow \\  x \in (-1)^{\ep_{i+1}}A_{i} = (-1)^{(\s(\ep)_{i}}A_{i} = A(\s(\ep))_i.
   \end{split}\end{equation}
   Because $0^{i-1}1\ep x - 0^{i-1}1\ep y = 0^{i+1}(x - y)$ for $\ep = 0, 1$, the result follows by induction.

   (b): From (\ref{adiceq01})(i)-(iii) we see, as in (\ref{adiceq01ac}) that for all $\ep \in \{ 0, 1 \}^{\N}$ and  all $i,j  \in \N$ with $j > i+1$
   \begin{align}\label{adiceq01acxx}\begin{split}
   x \in A_i(\ep) \ \ \text{and} \ \ x' \in A_j(\ep) \cup (- A_j(\ep)) \ \ &\Rightarrow \ \ (x', x) \in \widehat{A(\ep)}^{\circ}, \\
   x \in -A_i(\ep) \ \ \text{and} \ \ x' \in A_j(\ep) \cup (- A_j(\ep)) \ \ &\Rightarrow \ \ (x, x') \in \widehat{A(\ep)}^{\circ}, \\
   x \in A_i(\ep) \ \ \text{and} \ \ x' \in A_{i+1}(\ep) \cup (- A_{i+1}(\ep)) \ \ &\Rightarrow \ \ (x, x') \in \widehat{A(\ep)}^{\circ}, \\
   x \in -A_i(\ep) \ \ \text{and} \ \ x' \in A_j(\ep) \cup (- A_j(\ep)) \ \ &\Rightarrow \ \ (x, x') \in \widehat{A(\ep)}^{\circ}.
   \end{split}\end{align}

    From (\ref{adiceq01})(iv) it follows that
    \begin{align}\label{adiceq01adxx}\begin{split}
   x \in A_i(\ep) \ \ \text{and} \ \ x' \in -A_i(\ep)  \ \ &\Rightarrow  \ \ (x, x') \in \widehat{A(\ep)}^{\circ} \\
  \text{if either } \  x_{i+2} = x'_{i+2} \ \text{and} \ \ep_i = \ep_{i+1} \ \ &\text{or} \
  \ x_{i+2} = \bar x'_{i+2} \ \text{and} \ \ep_i = \bar \ep_{i+1}, \\
   \text{ and} \ \ (x', x) \in &\widehat{A(\ep)}^{\circ} \ \ \text{otherwise}.
 \end{split}\end{align}

     Now assume that for some $\ep$ and all $p \in \N$,
      the isomorphisms $h[\s^p(\ep)] : (\Z[2],\widehat{A}) \to (\Z[2],\widehat{A(\s^p(\ep)})$ have been defined so that
     the diagrams (\ref{adicdiag}) commute with $\ep$ replaced by $\s^p(\ep)$.
     We now define $h[\ep]$.

      First, observe that $h[\ep]$ is defined on $I_0$ from \ref{adicdiag}. Hence, for $x, x' \in I_0$, that $(x, x') \in \widehat{A}$
      if and only if $(h[\ep](x),h[\ep](x')) \in \widehat{A(\ep)} $.  Provided that $h[\ep]$ is chosen to map $\pm A_1$ to
      $\pm A(\ep)_1$, it will follow from  (\ref{adiceq01acxx}) that if $x \in A_1, x' \in A_j \cup (-A_j)$ with $j > 2$,
      then $(x,x') \in \widehat{A}$ and $(h[\ep](x),h[\ep](x')) \in\widehat{A(\ep)}$ with the reverse directions for $x \in -A_1$.
      Furthermore, if
      $x \in A_1, x' \in A_2 \cup (-A_2)$, then $(x',x) \in \widehat{A}$ and $(h[\ep](x'),h[\ep](x)) \in\widehat{A(\ep)}$  with
      the reverse directions for $x \in -A_1$.

      It remains to define $h[\ep]$ on $A_1 \cup -A_1$
        the definition  depends on the values of $\ep_1$ and $\ep_2$:
     \begin{align}\label{adiceq01af}\begin{split}
     \text{For} \ \ (\ep_1,\ep_2) = (0,0), &\quad 10z \mapsto 10 h[\s^2(\ep)](z), \ \  11x \mapsto 11 h[\s^2(\ep)](z). \\
     \text{For} \ \ (\ep_1,\ep_2) = (0,1), &\quad 10z \mapsto 10 (h[\s^2(\ep)](z) + \1 ), \ \  11x \mapsto 11 h[\s^2(\ep)](z). \\
  \text{For} \ \ (\ep_1,\ep_2) = (1,0), &\quad 10z \mapsto 11 (h[\s^2(\ep)](z) + \1 ), \ \  11z \mapsto 10 h[\s^2(\ep)](z). \\
  \text{For} \ \ (\ep_1,\ep_2) = (1,1), &\quad 10z \mapsto 11 h[\s^2(\ep)](z), \ \  11z \mapsto 10 h[\s^2(\ep)](z).
   \end{split}\end{align}

   Observe that for $x = wz, x' = wz',$ then $(x, x') \in \widehat{A}$ if and only if $(00z, 00z') \in \widehat{A}$
   if and only if $(h[\s^2(\ep)](z),h[\s^2(\ep)](z')) \in \widehat{A(\s^2(\ep))}$ if and only if
   $(w'h[\s^2(\ep)](z),w'h[\s^2(\ep)](z')) \in \widehat{A(\ep)}$ and if and only if
   $(w'(h[\s^2(\ep)](z) + \1 ),w'(h[\s^2(\ep)](z') + \1 )) \in \widehat{A(\ep)}$ for $w, w' \in \{0, 1 \}^2$.
It follows that for $x, x' \in \pm A_1$, that $(x,x') \in \widehat{A}$ if and only if \\
$(h[\ep](x),h[\ep](x')) \in\widehat{A(\ep)}$.

Lastly, to show that for $10z \in A_1, 11z' \in -A_1, \ \ (10z,11z') \in \widehat{A}$ if and only if
$(h[\ep](10z),h[\ep](11z')) \in\widehat{A(\ep)}$ we use (\ref{adiceq01adxx}) and (\ref{adiceq01ae}).
We observe that for any $z$ and $\ep$, $h[\s^2(\ep)](z)_1 = z_1$, i.e. either both $z$ and $h[\s^2(\ep)](z)$
are  even or both are odd. Now suppose $(10z,11z') \in \widehat{A}$ and so by (\ref{adiceq01ae})
$z_1 = z'_1$. If $\ep_1 = 0$, $h[\ep](10z) \in A_1, h[\ep](11z) \in -A_1$ and so
by (\ref{adiceq01adxx}) $(10h[\s^2(\ep)](z),11h[\s^2(\ep)](z')) \in\widehat{A(\ep)}$ if $\ep_2 = 0$ and
$(11h[\s^2(\ep)](z'),10h[\s^2(\ep)](z)) \in \widehat{A(\ep)}$ if $\ep_2 = 1$. So in the latter case
$(10(h[\s^2(\ep)](z) + \1 ),11h[\s^2(\ep)](z')) \in \widehat{A(\ep)}$. With similar arguments for the two cases with
$\ep_1 = 1$.

This construction requires that we know the isomorphisms $h[\s^p(\ep)]$ for all $p \in \N$. We begin with $\ep = \0 $ for which
$A(\ep) = A$ and we use the identity with $h[\0 ](x) = x$ for all $x$. The construction then yields $h[ \1 ]$. Continuing on
we obtain the definition of $h[\ep]$ for any $\ep$ with $\ep_j = 0$ for $j$ sufficiently large. This set is $\Z_+ = \N \cup \{ 0 \}$ regarded as
as subset of $\Z[2]$.

Recall that $x \cong x' (mod \ 2^k)$ when $x_i = x'_i$ for all $i \le k$ in $\N$.
\vspace{.25cm}

CLAIM: Assume for  $\ep, \ep' \in \Z_+$ that $\ep \ \cong \  \ep' \ (mod \ 2^{k-1})$.

(a) For $x, x' \in \Z[2], \ \ x \ \cong \  x' \ (mod \ 2^{k})$ if and only if $h[\ep](x) \ \cong \  h[\ep'](x')\  (mod \ 2^{k})$.

(b) If $x \ \ncong \ x' \ (mod \ 2^{k-1})$, then $(x,x') \in \widehat{A(\ep)}$ if and only if \\ $(x,x') \in \widehat{A(\ep')}$.

\begin{proof} (a): For $k = 1$ the assumption on $\ep$ and $\ep'$ is vacuous and the result follows because $x \ \cong \  h[\ep](x) (mod \ 2)$
for all $x$ and $\ep$. For $k > 1$ the result is clear for $\ep = \ep' = \0 $.  Assume the result for $\ep, \ep'$ with $\ep_j = \ep'_j = 0$ for
all $j > N$ and we prove the result when $\ep_j = \ep'_j = 0$ for
all $j > N+1$ and so we can apply the result for $\s^p(\ep) \ \cong \  \s^p(\ep') (mod \ 2^{k-p-1})$ with $p \ge 1$. If $x = 0z, x' = 0z' \in I_0$, then
\begin{align}\label{new}\begin{split}
x \ \cong \  x' \ (mod \ 2^k) \ &\Longleftrightarrow \ z \ \cong \  z' \ (mod \ 2^{k-1})
 \Longleftrightarrow \\
h(\s(\ep))(z) \ \cong \  &h(\s(\ep'))(z') \ (mod \ 2^{k-1})  \Longleftrightarrow \\
h[\ep](x) = 0h(\s(\ep)(z) &\ \cong \  0h(\s(\ep')(z') =  h[\ep'](x') \ (mod \ 2^k).
\end{split}\end{align}

If $x, x' \in \pm A_1 $, then $h[\ep](x) \in \pm A(\ep)_1 $ and $h[\ep'](x') \in \pm A(\ep')_1$. If $\ep_1 = \ep'_1$, then $A(\ep)_1 = A(\ep')_1$ and
so $h[\ep](x)  \ \cong \   h[\ep'](x') \ (mod \ 2^2)$, proving the result for $k = 2$.

For $k > 2$, $\ep_1 = \ep'_1$ and $\ep_2 = \ep'_2$  and  because the result holds for $\s^2(\ep) \ \cong \  \s^2(\ep') \ (mod \ 2^{k-3})$ and
in (\ref{adiceq01af}) $z \ \cong \  z' \ (mod \ 2^{k-2})$ we obtain the result from the definition (\ref{adiceq01af}) for $x$ and $x'$.
\vspace{.25cm}

(b) : If $x \ \ncong \ x' \ (mod \ 2^{k-1})$, then $(x,x') \in \widehat{A(\ep)}$ if and only if $x' - x \in A(\ep)_i$ for some $i < k-1$.
Since $\ep \ \cong \  \ep' \ (mod \ 2^{k-1})$ we have $A(\ep)_i = A(\ep')_i$ for all $i \le k-1$.

\end{proof}

Now given an arbitrary $\ep \in \Z[2]$ define $\ep^n$ by
$ \ep^n_i \ = \ \begin{cases} \ep_i \ \text{for} \ i \le n \\ \ 0 \  \ \text{for} \ i > n. \end{cases}$
From part (a) of the Claim we have $h[\ep^n](x) \ \cong \  h[\ep^m](x) \ (mod \ 2^k)$ provided $n,m \ge k$. Thus, we can define
$h[\ep](x)_i =  h[\ep^n](x)_i$ for all $n \ge i$. From part (b) of the Claim it then follows that $(x,x') \in \widehat{A}$ if and only if
$(h[\ep](x),h[\ep'](x)) \in  \widehat{A(\ep)}$.

     \end{proof}

 \vspace{.5cm}

Thus we have an uncountable set of game subsets $A(\ep)$ all of whose associated tournaments are isomorphic.

 \vspace{1cm}

 \subsection{\textbf{Sections Over the Cantor Set}}

  \vspace{.5cm}

  \begin{theo}\label{Cantor1} Let $f : X \to Y$ be a continuous, open surjection from a compact metric space $X$ onto a totally
  disconnected space $Y$.
  There exists a continuous map $r : Y \to X$ such that $f \circ r = 1_C$. \end{theo}

  \begin{proof} By replacing the metric $d$ on $X$ by $\min (d, 1)$ we may assume that $X$ has diameter at most $1$.

  We define a decreasing sequence of open subsets  $\{ Z_n \}$ of $X$ and successively refining clopen partitions $\A_n$ of $C$
  such that for each $U \in \A_n$ the open set $f^{-1}(U) \cap Z_n$ has diameter at most $1/n$ and is mapped by $f$ onto $U$.

  Begin with $Z_1 = X$ and $\A_1 = \{ Y \}$.

  Given $Z_n$ and $\A_n$ we choose for each $U \in \A_n$ an open cover $\B(U)$ of $f^{-1}(U) \cap Z_n$ by subsets of diameter at most
  $1/(n+1)$. Now choose a clopen partition $\A_{n+1}|U$ of $U$ which refines the open cover $\{ f(B) : B \in \B(U) \}$.

  For each $U' \in \A_{n+1}|U$ choose a $B(U') \in \B(U)$ such that $U' \subset f(B(U'))$ and so $f^{-1}(U') \cap B(U')$ maps
  onto $U'$ and has diameter at most $1/(n+1)$. Let $\A_{n+1} = \bigcup \{ \A_{n+1}|U : U \in \A_n \}$ and
  $Z_{n+1} = \bigcup \{ f^{-1}(U') \cap B(U') : U' \in \A_{n+1} \}$.

  Note that the closure $\overline{Z_{n}} = \bigcup \{ \overline{f^{-1}(U) \cap Z_n} : U \in \A_n \}$. Define
  $Z = \bigcap_n \overline{Z_{n}}$.

  For $x \in Y$, let $U_n(x)$ denote the member of $\A_n$ which contains $x$. $f^{-1}(x) \cap Z = \bigcap_n \overline{f^{-1}(U_n) \cap Z_n} $
  which is a singleton since $\overline{f^{-1}(U_n) \cap Z_n}$ has diameter at most $1/n$.

  Hence, the restriction $f|Z$ is a continuous bijection which is therefore a homeomorphism by compactness.
  We define $r = (f|Z)^{-1}$.

  \end{proof}

  It follows that in Proposition \ref{prop3.03} if $G_2$ is a compact metrizable group mapping onto a totally disconnected  group, then
  the lift $j$ and the retraction $p$ can be chosen to be continuous.

 \begin{cor}\label{Cantor2} If a compact group $H$ acts on a totally disconnected compact metric space $X$, then the
 quotient space of orbits $ Y = \{ Hx : x \in X \}$ is totally disconnected and there exists a continuous selection
 $r : Y \to X$ with $r(Hx) \in Hx$ for all $x$. \end{cor}

 \begin{proof} With respect to the diagonal action of $H$ on $X \times X$, the diagonal
 $1_X = H 1_X$ is the intersection $ \bigcap \ HV$, as $V$ varies over the closed neighborhoods of the diagonal. Hence, if $V_1$ is
 a neighborhood of the diagonal, then for some such $V \ \ H V \subset V_1$ and so $V \subset \bigcap \{ (h \times h)^{-1}(V_1) : h \in H \}$.
 That is, the action is equicontinuous. Hence, if $u$ is an ultra-metric on $X$, we can replace it by $\max \{ h^*u : h \in H \}$
 where $h^*u(x,y) = u(hx, hy)$. That is, we may assume that $u$ is $H$ invariant.

 Now on the quotient space define $\bar u$ by
 $\bar u (Hx, Hy) = \min \{ u(x_1, y_1) : x_1 \in Hx, y_1 = Hy \}$. If the minimum is achieved at the pair $(x_1,y_1)$ and
 $x_2 \in Hx$, there exists $h \in H$ such that $hx_1 = x_2$ and so with $y_2 = hy_1$ we have $\bar u(Hx, Hy) = u(x_2, y_2)$.
 In particular, if $\pi : X \to Y$ is the projection with $\pi(x) = Hx$, then $\pi(V^u_{\ep}(x)) = V^{\bar u}_{\ep}(\pi(x))$
 for all $x \in X$. Note that for $x, y, z \in X$, there exist $x_1 \in Hx, y_1 \in Hy, z_1 \in Hz$ such that
 $\bar u(Hx, Hy) = u(x_1, y_1)$ and $\bar u(Hy, Hz) = u(y_1, z_1)$. Hence,
 \begin{align}\begin{split}\label{eqCantor}
 \max ( \bar u(Hx, Hy), \bar u(Hy, Hz)) \ = \  &\max( u(x_1, y_1),u(y_1, z_1))\\
  \ge \  u(x_1,z_1) \  \ge \  &\bar u(Hx, Hz).
 \end{split}\end{align}
 Clearly, $\bar u(Hx, Hy) = 0$ if and only if $Hx = Hy$ and $\bar u$ is symmetric. Since $\pi^* \bar u \le u$, it follows that
 $\bar u$ is a continuous ultra-metric on the quotient space $Y$ and so it induces the quotient topology by compactness.
 Hence, the quotient is totally disconnected.

 Since  $\pi(V^u_{\ep}(x)) = V^{\bar u}_{\ep}(\pi(x))$, it follows that $\pi$ is an open map and so the selection exists by
 Theorem \ref{Cantor1}.

 \end{proof}

 This result generalizes Lemma \ref{lem3.04a}.

\vspace{1cm}

\bibliographystyle{amsplain}

\printindex

 \end{document}